\documentclass[11pt]{article}

\newif\ifsplit
\splitfalse

\usepackage{amsmath,amsfonts,amssymb,amsthm}
\usepackage{xcolor}
\usepackage{enumitem}
\usepackage{authblk}
\usepackage{makecell}
\usepackage{graphicx}

\usepackage{mathtools}

\RequirePackage[colorlinks,citecolor=blue,urlcolor=blue,linkcolor=red,filecolor=black]{hyperref}
\ifsplit
    \usepackage{xr-hyper}
\fi

\theoremstyle{plain}
\newtheorem{proposition}{Proposition}[section]
\newtheorem{theorem}[proposition]{Theorem}

\newtheorem{conjecture}[proposition]{Conjecture}
\theoremstyle{definition}
\newtheorem{lemma}[proposition]{Lemma}
\newtheorem{remark}[proposition]{Remark}
\newtheorem{definition}[proposition]{Definition}
\newtheorem{example}[proposition]{Example}
\newtheorem{examples}[proposition]{Examples}

\newtheorem{condition}[proposition]{Condition}
\newcommand{\nc}{\newcommand}
\nc{\md}{\mathrm{d}}
\nc{\I}{{\mathbf 1}}
\nc{\remove}[1]{}

\nc{\bU}{\mathbb{U}}
\nc{\cU}{\mathcal{U}}
\nc{\bN}{{\mathbf N}}
\nc{\bM}{{\mathbf M}}
\nc{\cB}{{\mathcal B}}
\nc{\cM}{{\mathcal M}}
\nc{\cC}{{\mathcal C}}
\nc{\cL}{{\mathcal L}}
\nc{\cF}{{\mathcal F}}
\nc{\cK}{{\mathcal K}}
\nc{\cA}{{\mathcal A}}
\nc{\R}{{\mathbb R}}
\nc{\Q}{{\mathbb Q}}
\nc{\C}{{\mathbb C}}
\nc{\M}{{\mathcal M}}
\nc{\N}{{\mathbb N}}
\nc{\cN}{{\mathcal N}}
\nc{\Z}{{\mathbb Z}}
\nc{\bF}{{\mathbf F}}
\nc{\bw}{{\mathbf w}}
\nc{\tC}{\tilde{C}}
\nc{\tc}{\tilde{c}}
\nc{\hC}{\hat{C}}
\nc{\hc}{\hat{c}}
\nc{\tphi}{\tilde{\phi}}
\nc{\tvphi}{\tilde{\varphi}}
\nc{\hvphi}{\hat{\varphi}}
\nc{\tPhi}{\tilde{\Phi}}
\nc{\Psif}{\Psi^{!}}
\nc{\Ff}{F^{!}}
\nc{\tbN}{\tilde{\mathbf{N}}}
\nc{\tx}{\tilde{x}}
\nc{\ty}{\tilde{y}}
\nc{\talpha}{\tilde{\alpha}}
\nc{\tf}{\tilde{f}}
\nc{\tGamma}{\tilde{\Gamma}}
\nc{\tmu}{\tilde{\mu}}

\DeclareMathOperator{\supp}{supp}
\nc{\BP}{\mathbb{P}}
\nc{\BE}{\mathbb{E}}
\nc{\BQ}{\mathbb{Q}}
\nc{\BX}{\mathbb{X}}

\DeclareMathOperator{\diam}{{diam}}
\DeclareMathOperator{\conv}{{conv}}
\DeclareMathOperator{\BV}{{\mathbb Var}}
\DeclareMathOperator{\BC}{{\mathbb Cov}}

\usepackage{soul}
\usepackage[normalem]{ulem}
\setstcolor{red}

\numberwithin{equation}{section}

\usepackage[backend=biber,style=numeric, maxbibnames=6, giveninits=true, backref=true, url=false, isbn=false, doi=false]{biblatex}
\AtEveryBibitem{\clearfield{note}}

\title{Persistence of asymptotic variance under transport:\\ from hyperfluctuation to stealthy hyperuniformity}

\author[1,2]{Luca Lotz}
\author[3,1,4]{Michael A. Klatt}

\affil[1]{German Aerospace Center (DLR), Institute of Frontier Materials on Earth
and in Space,\linebreak Functional, Granular, and Composite Materials, Linder Höhe, 51170
Cologne, Germany}
\affil[2]{Institute for Stochastics, Karlsruhe Institute of Technology, Englerstr. 2, 76131 Karlsruhe, Germany.}
\affil[3]{German Aerospace Center (DLR), Institute for AI Safety and Security,\linebreak Wilhelm-Runge-Str. 10, 89081 Ulm, Germany}
\affil[4]{Department of Physics, Ludwig-Maximilians-Universität München,\linebreak Schellingstr. 4, 80799 Munich, Germany}

\date{August 2026}

\begin{document}

\maketitle
\begin{abstract}\noindent
    We introduce $p$-uniformity to characterize the scaling of density fluctuations in spatial random systems in $\R^d$, ranging from hyperfluctuation to stealthy hyperuniformity. Our central theorem establishes sufficient conditions to preserve $p$-uniformity under transport. The first condition, a finite $(d+p)$-th moment of the transport distance, allows for a Taylor expansion of the transport. The second condition controls the corresponding terms. We thus solve a previously stated open problem; indeed we extend it, since our result applies to a general $p$-uniform source in any dimension, and the source and transport may be dependent. As an application, we construct new classes of point processes that are isotropic and $p$-uniform with arbitrarily high $p$, and that can be simulated in linear time. 
    We thus achieve three-dimensional isotropic hyperuniform samples with an unprecedented system size of $10^9$ points. 
    We conclude with an outlook on a converse statement.
\end{abstract}

{
\hypersetup{hidelinks}
\tableofcontents
}
\bigskip


\section{Introduction}

Let $\Phi$ be a stationary random complex measure, e.g., a point process, on $\R^d$ that is locally square-integrable, i.e., $\BE[|\Phi(B)|^2]<\infty$ for all bounded Borel sets $B \subseteq \R^d$. Let $\lambda_d$ denote the Lebesgue measure on $\R^d$, let $f:\R^d\to[0,\infty)$ be a test function with $\|f\|_2=1$, and define $f_r(x):=f(\frac{x}{r})$ for $x\in\R^d, r>0$. Then we call
\begin{equation}\label{e: intro asym var}
    \sigma_\Phi^2:=\lim_{r\to\infty} \frac{\BV[\Phi(f_r)]}{r^d}
\end{equation}
the \textit{asymptotic variance} of $\Phi$ if it exists. 
A common example of $f$ is the indicator function of a ball normalized by its volume, where the ball acts as an `observation window'. Alternatively, $f$ can be a smooth function with compact support or at least with a sufficiently fast decay, which represents a `diffuse observation window'. The asymptotic variance $\sigma_\Phi^2$ quantifies the mass fluctuations of the random measure in the limit of large observation windows. For example, if $\Phi$ is a point process, $\sigma_\Phi^2$ is the limit of the variance of the number of points in the observation windows rescaled by volume. We are especially interested in the case $\sigma_\Phi^2=0$, where mass fluctuations are suppressed at large scales.
Such a random measure is called \textit{hyperuniform}~\cite{Torquato_Stillinger_2003, Torquato_2018}. 

Prominent examples of hyperuniform point processes include lattices and
quasicrystals~\cite{Torquato_Stillinger_2003, Oğuz_Socolar_Steinhardt_Torquato_2017, Björklund_Hartnick_2024}, 
as well as ergodic, isotropic, and locally disordered examples~\cite{Torquato_2018},
e.g., determinantal point processes~\cite{Soshnikov_2000} 
or the one-component plasma~\cite{leble_two-dimensional_2026}.
Hyperuniformity is of great interest to physics~\cite{gabrielli_glass-like_2002,
  Torquato_Stillinger_2003,
  Gabrielli_Jancovici_Joyce_Lebowitz_Pietronero_Sylos_Labini_2003,
  torquato_hyperuniformity_2016,
  Torquato_2018,
  haberko_transition_2020,
  rissone_long-range_2021,
  wilken_random_2021,
  galliano_two-dimensional_2023,
  Shih_Casiulis_Martiniani_2024,
  maire_hyperuniform_2025,
  vanoni_effective_2026},
materials science~\cite{gorsky_engineered_2019,
  zheng_disordered_2020,
  yu_engineered_2021,
  klatt_wave_2022}.
It also increasingly attracts interest in
mathematics~\cite{Ghosh_Lebowitz_2017, 
  Ghosh_Lebowitz_2018,
  Klatt_Last_Yogeshwaran_2020, 
  Mastrilli_Błaszczyszyn_Lavancier_2024,
  Dereudre_Flimmel_Huesmann_Leblé_2024, 
  Björklund_Hartnick_2024,
  Lachièze-Rey_Yogeshwaran_2024,
  Butez_Dallaporta_García-Zelada_2024,
  Lachièze-Rey_2025a,
  Flimmel_2025,
  Huesmann_Leblé_2026}.

Recently, a particular focus has lain on the construction and simulation
of large-scale samples of isotropic amorphous (i.e., non-crystalline) hyperuniform
point processes (for $d\geq2$)~\cite{Uche_Stillinger_Torquato_2004, Uche_Torquato_Stillinger_2006, Batten_Stillinger_Torquato_2008, Shih_Casiulis_Martiniani_2024, Thomassey_Lachièze-Rey_Shapira_2026}. Such algorithms are
especially important for the study of hyperuniformity because it is a
long-range property and hence takes strong advantage of large system
sizes. However, the results were so far mainly limited to numerical
studies. Of special interest are point processes with rapid decay of mass fluctuations 
in growing observation windows (which can be observed in~\eqref{e: intro asym
var} if $f$ is chosen to be smooth) because such processes exhibit a
strong reduction of fluctuations already at intermediate
window sizes. This behavior extends the range of the corresponding physical effects.

Seemingly independent, the theoretical question about the persistence of
hyperuniformity under transport has attracted considerable attention. A
random \textit{transport} $K$ moves mass from one random measure $\Phi$, the
\textit{source}, to another random measure $K\Phi$, the
\textit{destination}. Practically, one can think of a random transport
as a random movement of mass, e.g., as a displacement of points.
So far, there have been two complementary approaches to develop
conditions that guarantee the persistence of hyperuniformity under
transport. The first establishes the persistence of hyperuniformity
under a simple $d$-th moment condition on the transport distance, but is
limited to dimension $d\leq2$~\cite{Dereudre_Flimmel_Huesmann_Leblé_2024}. The second works in any
dimension, but is limited by a restriction in long-range dependence~\cite{Flimmel_2025, Klatt_Last_Lotz_Yogeshwaran_2025}. 
However, both approaches only control the asymptotic variance itself and cannot guarantee the persistence 
of the decay rate beyond a very limited degree.
A more comprehensive control of the decay rate is not only of
theoretical but also of practical interest, since it enables the
construction of non-crystalline hyperuniform point processes with
rapidly decaying number fluctuations.

Here, we fill this gap. We begin by formalizing an order in the asymptotic behavior of $\BV[\Phi(f_r)]$. Usually, this behavior is quantified using the scaling exponent~\cite[Section 5]{Torquato_2018}, also called the hyperuniformity exponent~\cite{Mastrilli_Błaszczyszyn_Lavancier_2024, Lachièze-Rey_2025b}. However, this scaling exponent is not well-defined for every stationary random measure $\Phi$. Hence, we introduce the more general but closely connected concept of $p$-uniformity in Section~\ref{s: p-uniformity}. Under our assumptions, there is a unique positive semi-definite signed measure $\beta_\Phi$, called the \textit{covariance measure}, such that
\begin{equation}
    \BC[\Phi(f), \Phi(g)] = \beta_\Phi(f\star g)
\end{equation}
for all bounded functions $f,g$ with bounded support, where $\star$ is the tilted convolution defined as in~\eqref{e: tilted convolution}. Its Fourier transform $\hat\beta_\Phi$ is a locally finite measure and called the \textit{Bartlett spectral measure} or diffraction measure~\cite{Bartlett_1963, Björklund_Hartnick_2024}. For $p\in[-d, \infty)$, we call $\Phi$ $p$\textit{-uniform} if, as $\varepsilon\to0$, 
\begin{equation}\label{e: uniform intro}
    \frac{\hat\beta_\Phi(B_\varepsilon)}{\varepsilon^d} = O(\varepsilon^p).
\end{equation}
If~\eqref{e: uniform intro} holds with $o$ instead of $O$, we call $\Phi$ \textit{beyond} $p$-\textit{uniform}. Moreover, we call $\Phi$ $\infty$-uniform if there is an $\varepsilon>0$ such that $\hat\beta_\Phi(B_\varepsilon)=0$. So $\infty$-uniformity coincides with stealthy hyperuniformity~\cite{Batten_Stillinger_Torquato_2008,Torquato_Zhang_Stillinger_2015}. A similar technique is used in~\cite{Oğuz_Socolar_Steinhardt_Torquato_2017}, where $\hat\beta_\Phi(B_\varepsilon)$ is denoted by $Z(\varepsilon)$. If $\Phi$ is $p$-uniform and $p<\infty$, then the fraction on the RHS (right-hand side) of~\eqref{e: intro asym var} can be bounded by $cr^{-p}$ for some $c>0$ given that $f$ is sufficiently smooth. If $p=\infty$, then it decays faster than any polynomial. Hence, $\Phi$ is hyperuniform, i.e., $\sigma_\Phi^2=0$, iff $\Phi$ is beyond $0$-uniform. By definition, if $\Phi$ is $p$-uniform, then it is also $q$-uniform for all $q\in[-d, p]$. 

More concretely, if we assume that $\hat\beta_\Phi$ has a density $S_\Phi$, called the \textit{spectral density}, and that $S_\Phi$ has a scaling exponent $p\in(-d,\infty)$, i.e., 
\begin{equation}
    S_\Phi(k) = a\|k\|^p + o(\|k\|^p),
\end{equation}
as $\|k\|\to0$, for some $a> 0$, then $\Phi$ is $q$-uniform iff $q\in[-d,p]$. If $\Phi$ is a point process, this spectral density, rescaled by the intensity, is also called the \textit{structure factor}.

Another benefit of the notion of $p$-uniformity is that we can connect it to the results from~\cite{Ghosh_Lebowitz_2017, Lachièze-Rey_2025a} to show that $(2p+d)$-uniformity implies $p$-rigidity for any $p\in\N_0\cup\{\infty\}$; see~\eqref{e: p-rigid} or~\cite{Ghosh_Peres_2017, Lachièze-Rey_2025a} for the definition of rigidity.

Next, we establish conditions under which $p$-uniformity is preserved under transports; see Section~\ref{s:main results}.
Hence, we also answer the previously mentioned open question about the persistence of rapid decay of mass fluctuations. We start with showing that a $d$-th moment condition on the transport distance guarantees that local square integrability persists under transport; see Subsection~\ref{ss:Square integrability of destinations} and partially also~\cite{Dereudre_Flimmel_Huesmann_Leblé_2024}. Then, in Subsection~\ref{ss:Persistence of uniformity under transport}, we show that a $(d+p)$-th moment condition on the transport distance allows for a form of Taylor expansion of the transport. Concretely, given the source $\Phi$, the transport $K$, and the smooth, compactly supported test function $f$, the destination $K\Phi$ evaluated at the scaled test function $f_r$ (defined as in \eqref{e: intro asym var}) admits the expansion
\begin{equation}\label{e: finite expansion intro}
    K\Phi(f_r) = \sum_{q=0}^{\left\lceil\frac{d+p}{2}\right\rceil-1} \frac{1}{q!} \Psi_q(f_r^{(q)}) + o\big(r^{\frac{d-p}{2}}\big)
\end{equation}
in $L^2$ as $r\to\infty$. The tensor-valued stationary random complex measures $\Psi_q$ have explicit random densities with respect to the source $\Phi$, i.e., 
\begin{equation}\label{e: Psi_q intro}
    \Psi_q(dy) = \bigg(\int (x-y)^{\otimes q} \, K(y, dx)\bigg)\, \Phi(dy)\quad y\in\R^d,
\end{equation}
where $\otimes$ denotes the tensor product.
Here, the random transport $K$ is identified with its transport kernel, and the measure $K(y, \cdot)$ dictates how mass of the source $\Phi$ at the location $y\in\R^d$ is relocated by the transport.
If an exponential moment of the transport distance is also finite, we even obtain a full Taylor series, i.e.,
\begin{equation}
    K\Phi(f) = \sum_{q=0}^{\infty} \frac{1}{q!} \Psi_q(f^{(q)})
\end{equation}
given that the Fourier transform of $f$ is smooth and compactly supported.
We exploit this new decomposition of transports in our main Theorem~\ref{t:main theorem}. There, we show that it suffices to control the first $\lceil \frac{d+p}{2}\rceil$ terms of this expansion to guarantee that $p$-uniformity persists under transport. While the theorem applies for any $d\in\N$ and $p\in[-d,\infty]$, we can specifically recover the results from~\cite{Dereudre_Flimmel_Huesmann_Leblé_2024} by choosing $d=1$ and $p\in\{0,1\}$ or $d=2$ and $p=0$, since the $0$-th order term of the expansion is trivial for a standard transport; see~\eqref{e: Psi_q intro}.

This connection is no coincidence. Some parts of the proofs even build upon~\cite{Dereudre_Flimmel_Huesmann_Leblé_2024}. However, we introduce multiple new crucial elements. The most important is the multiple-term expansion~\eqref{e: finite expansion intro}, whereas in~\cite{Dereudre_Flimmel_Huesmann_Leblé_2024} only the cases are considered where a single term is sufficient. This allows us to also formulate a condition for the persistence of hyperuniformity for $d\geq3$, which was an open question of theirs. Critically, we also go beyond hyperuniformity and can ensure the persistence of $p$-uniformity for higher $p$. Finally, we substantially broaden the applicability by allowing the source to be an arbitrary random complex measure and the transport to be not only spatially correlated with itself but also dependent on the source.

In Subsection~\ref{ss:Persistence of uniformity under spatial interpolation}, we expand on the concept of transport in a more physical way. We interpolate the source and destination spatially continuously. It turns out that our conditions ensuring the persistence of $p$-uniformity from the source directly to the destination also imply that $p$-uniformity is preserved along the way. Conversely, the persistence of $p$-uniformity under spatial interpolation almost implies these conditions.

In Section~\ref{s:applications}, we primarily apply our results to construct new classes of point processes. The first class, defined in Subsection~\ref{ss: pp from invariant partitions}, is the aforementioned class of non-crystalline point processes for $d\geq2$ that can be $p$-uniform for any $p$. They can even be simulated in linear time, and the algorithm parallelizes perfectly. Thus, we solve also this previously mentioned open problem. For the construction, we leverage \textit{fair tilings} like in~\cite{Klatt_Last_Lotz_Yogeshwaran_2025} to reduce the global to a local problem. A tiling is called fair if every cell has the same volume. We construct a new class that can be simulated in linear time. The technique expands on the STIT-tilings that are created by iteratively dividing cells using random hyperplanes~\cite{Nagel_Weiss_2003, Nagel_Weiss_2005}. We deviate by only dividing cells such that both newly created cells have the same volume. We call the invariant distribution of such a Markov chain on the space of fair tessellations a \textit{fair STIT}. If we choose the splitting direction uniformly, the resulting fair STIT is ergodic and isotropic, i.e., non-crystalline. In the second step, we apply an idea from~\cite{Gabrielli_Joyce_Torquato_2008}. Let $n,m\in\N$. We place (random) points $X_1,...,X_n$ in every cell $C$ such that the first $m$ moments of the cell and the points coincide, i.e.,
\begin{equation}\label{e: averaging set intro}
    \frac{1}{\lambda_d(C)} \int_C x^{\otimes q}\, dx = \frac{1}{n} \sum_{j=1}^n X_j^{\otimes q}, \quad q\in\{0,...,m-1\}.
\end{equation}
Then we apply our main Theorem \ref{t:main theorem} to show that the resulting point process is beyond $(2m-d)$-uniform; see Figure~\ref{fig:illustration}. This procedure can be simulated very efficiently, since each cell can be treated independently. The locations of the points in each cell can also be considered to be the nodes of an equal-weight quadrature formula of the cell. These are specifically called \textit{averaging sets} or $t$-designs~\cite{Seymour_Zaslavsky_1984, Delsarte_Goethals_Seidel_1977} and connect our construction to the theory of hyperuniformity in compact spaces~\cite{Brauchart_Grabner_Kusner_2019}. In Appendix~\ref{s: averaging sets}, we contribute to the theory of averaging sets and, e.g., show that for every $d\in\N, m\in\N_0$, there is an $n_0\in\N$ such that for every $n\geq n_0$ there is an averaging set of size $n$ and order $p$ for every bounded convex $C\subseteq \R^d$, showing that the previous construction is indeed possible for $n$ sufficiently large.

To demonstrate the practical feasibility and numerical stability of the proposed construction, we have implemented the algorithm and performed an extensive numerical study. Specifically, we have simulated samples containing $10^9$ points in dimensions $2$ and $3$ for different values of the parameter $m$. This system size is unprecedented in dimension $3$; see, e.g.,~\cite{Uche_Torquato_Stillinger_2006, Torquato_Kim_Klatt_2021, Shih_Casiulis_Martiniani_2024}. For these samples, we have also calculated the empirical structure factor and estimated the scaling exponent; see Figure~\ref{fig: empirical structure factor}. The obtained values are in agreement with our lower bound of $2m-d$ from Theorem~\ref{t:main theorem}. Indeed, they exceed this lower bound by up to $d$, as was expected. Both the implemented code and the resulting data sets will be made available in  open-access repositories.

\begin{figure}[p]
  \centering
  \includegraphics[width=\textwidth]{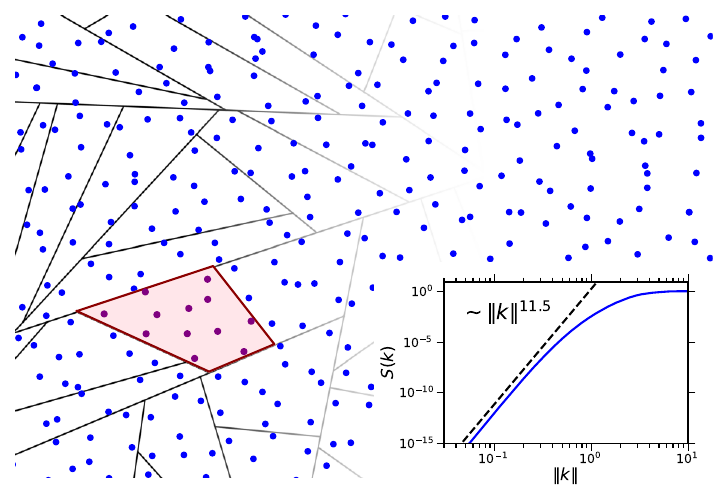}%
  \caption{Illustration of our construction of an isotropic point process
  that is hyperuniform with an exceptionally large scaling exponent. In
  black and fading out from left to right, we show a realization of a
  random fair tiling, i.e., a random division of space into cells of equal
  volume; specifically, a realization of a fair STIT; see
  Subsection~\ref{ss: pp from invariant partitions}. To obtain the point
  process shown in blue, we concentrate the mass that is uniformly
  distributed over each cell into $n$ points while preserving the first
  $p$ moments, i.e., satisfying~\eqref{e: averaging set intro}. Here,
  $n=12$ and $p=6$; see the cell highlighted in red. By our main
  Theorem~\ref{t:main theorem} about the persistence of $p$-uniformity
  under transport, this point process is beyond $(2p-d)$-uniform, i.e.,
  beyond $10$-uniform in this case. A simulation study with $50,000$
  samples of around $50,000$ points each yields the estimated structure
  factor shown in the inset. The data suggests that the point process is
  even \linebreak$11.5$-uniform. A point process with such a large scaling exponent
  is approximately stealthy within floating-point precision even up to
  intermediate wavenumbers $\|k\|$. Moreover, the linear scaling of the
  algorithm enables the generation of samples several orders of magnitude
  larger.}
  \label{fig:illustration}
\end{figure}

We also extend our method to other fair tilings. Under a moment condition, we show that the arbitrary placement of a singular point in each cell yields a beyond $(2-d)$-uniform point process, and the placement of a singular point at the centroid yields a beyond $(4-d)$-uniform point process. Hence, the former always yields a hyperuniform point process for $d\leq 2$ (see also~\cite{Dereudre_Flimmel_Huesmann_Leblé_2024}), and the latter for $d\leq4$. As a special case, we conclude that the centers of mass of the tiles of a pinwheel tiling form a beyond $2$-uniform, hence also class I hyperuniform, point process~\cite{Radin_1994}. For the remaining dimensions, we construct counterexamples in Appendix~\ref{s: Hyperfluctuating examples}; see partially also~\cite{Dereudre_Flimmel_Huesmann_Leblé_2024}. 

Our second class of new point processes, defined in Subsection~\ref{ss: pp from averaging sets}, is similar in nature, but its construction is based on a progenitor point process instead of a fair tiling. For every point of the progenitor point process, which is assumed to be $\tilde{p}$-uniform for some $\tilde{p}\geq p$, we independently place a random cluster of new points. Again, these clusters form averaging sets, but this time not of cells but of some arbitrarily chosen finite measure. In this way, the $\tilde{p}$-uniformity of the progenitor point process can be reduced by a controlled amount. While this construction is a lot simpler, this simplicity allows us to not only prove that the result is $p$-uniform for some $p\in 2\N_0$, but also to show that this value of $p$ is maximal. As examples for clusters, we bring up averaging sets of trigonometric manifolds and spherical $t$-designs~\cite{Delsarte_Goethals_Seidel_1977}. An example from~\cite{Lachièze-Rey_2025a} arises as a special case for $d=2$. By generalizing this principle---placing random finite measures instead of random point clusters---we further construct, to the best of our knowledge, the first nonnegative random field that is $\infty$-uniform, i.e., stealthy hyperuniform, and that also has a continuous spectral density.

In Subsection~\ref{ss: Persistence of uniformity under displacements}, we do not construct a new class of $p$-uniform point processes. Instead, we more generally specialize our main theorem to the case where the transport is a displacement field. We derive a mixing condition for the persistence of $p$-uniformity for $p\leq 2$ under displacement. Our conditions are strictly weaker than those from~\cite{Flimmel_2025} and practically also weaker than those from~\cite{Klatt_Last_Lotz_Yogeshwaran_2025}. Then, we prove a more specialized version for Gaussian displacements independent of the source. There, we also derive a condition that is strictly weaker than those in~\cite{Flimmel_2025, Klatt_Last_Lotz_Yogeshwaran_2025}. The conditions in our result can now be considered essentially sharp. We conclude the subsection by rigorously proving for $d=1$ that the destination of a perturbation by a Gaussian random field independent of the source cannot be beyond $4$-uniform under very basic assumptions. This statement was already formulated in~\cite{Gabrielli_2004}, but their argument was not mathematically rigorous and omitted the necessary conditions. We expect that the assertion also to holds for $d\geq2$ under suitable conditions, but the restriction to $d=1$ simplifies the proof significantly.

In Subsection~\ref{ss: Lattices at positive temperature}, we apply our results to make another statement from the physics literature mathematically rigorous for $d\geq3$. We show that, while a perfect lattice is $\infty$-uniform, i.e., stealthy hyperuniform, the lattice at positive temperature is not even beyond $0$-uniform, i.e., not hyperuniform. The proof generalizes to all materials with an asymptotically linear dispersion relation.

In the final Section~\ref{s: prospects}, we give an outlook on how our concepts can be extended in a different direction. Transports can always be associated with so-called transport costs. From this perspective, our main theorem implies that a transport that has a low cost in a certain sense preserves $p$-uniformity. In particular, when we fix the source to be the $\infty$-uniform Lebesgue measure (or lattice like in~\cite{Dereudre_Flimmel_Huesmann_Leblé_2024}), it is natural to also ask the reverse question. Does $p$-uniformity of a random measure also imply a certain form of low transport distance to the Lebesgue measure? Here, a low transport distance of two random measures is synonymous with the existence of a low-cost transport from one to another. If the statement is true, it will open the possibility to (almost) characterize $p$-uniformity in a transport sense. Actually, for $d=2$ and $p=0$, the question has already been answered in~\cite{Lachièze-Rey_Yogeshwaran_2024, Huesmann_Leblé_2026}, and in~\cite{Butez_Dallaporta_García-Zelada_2024} even more generally for $d\geq2$, $d+p\leq 2$. For the remaining cases, we introduce a new form of transport cost closely connected to the conditions in our main theorem. We think that this transport cost should be the key to answering the question conclusively for every dimension $d\in\N$ and every $p\in[-d,\infty)$. We support this claim by showing that at least all $\infty$-uniform, i.e., stealthy hyperuniform, random measures have this sort of low transport distance to the Lebesgue measure.

\section{Preliminaries}

In this section, we recall some foundational definitions and results on random measures and transport that we rely on throughout the paper. More specifically, we are interested in stationary random measures and their second moment properties, which also involves the Bartlett spectral measure; see~\cite{Bartlett_1963}. We allow for the random measures to be signed (and even complex-valued) since these objects come up naturally in Section~\ref{s:main results}, even when we apply the Theorems from there to nonnegative random measures. Further, the physics community has also been interested in the asymptotic properties of complex-valued random measures recently; see, e.g.,~\cite{Torquato_Kim_Klatt_Car_Steinhardt_2026}. Transports are mainly handled in the form of transport kernels. Most elements can also be found in the Preliminaries of~\cite{Klatt_Last_Lotz_Yogeshwaran_2025}, but some objects are treated in greater generality here. For a more in-depth look into the theory of random measures, we recommend~\cite{Kallenberg_2017}.

\subsection{Random complex measures}\label{ss: random complex measures}

Let $\BX$ be a metric space. Most of the time, we simply assume that $\BX=\R^d$ for some $d\in\N$. We define $\bM_+ := \bM_+(\BX)$ as the space of \textit{locally finite measures} on $\BX$. The space of all \textit{locally absolutely finite complex measures} on $\BX$ is then simply defined as
\begin{equation}
    \bM := \bM(\BX) := \{\mu_1-\mu_2 + i(\mu_3-\mu_4):\mu_1,\mu_2,\mu_3,\mu_4\in\bM_+(\BX)\}.
\end{equation}
Given $\mu\in \bM(\BX)$, its total variation measure is denoted by $|\mu|$. It is the smallest element of $\bM_+(\BX)$ such that $|\mu(B)|\leq |\mu|(B)$ for all $B\in\cB_b(\BX)$, where $\cB_b(\BX)$ is the set of all bounded Borel sets of $\BX$. If $\BX=\R^d$, we also use $\cB_b^d:=\cB_b(\R^d)$. If, e.g., $\mu\in \bM(\BX)$ has a density $f:\BX\to \C$ with respect to some $\nu\in\bM_+(\BX)$, then $|\mu|$ has the density $|f|$ with respect to $\nu$. For a measurable function $f:\BX\to\C$ which is integrable with respect to $|\mu|$, we extend the notation by
\begin{equation}
    \mu(f) := \int f(x)\, \mu(dx).
\end{equation}
Proceeding, we equip $\bM(\BX)$ with the sigma field $\cM:=\cM(\BX)$ generated by the mappings $\bM(\BX)\to\C; \mu\mapsto \mu(B)$ with $B\in\cB_b(\BX)$. We have $\bM_+(\BX)\in\cM(\BX)$, and further prominent examples are the \textit{counting measures} $\bN:=\bN(\BX):= \{\mu\in\bN(\BX):\mu(B)\in\N_0 \text{ for } B\in\cB_b(\BX)\}$ and \textit{simple counting measures} $\bN_s:=\bN_s(\BX):= \{\mu\in\bM(\BX):\mu(\{x\})\in\{0,1\} \text{ for } x\in\BX\}$; see, e.g.,~\cite[Subsection 2.1]{Last_Penrose_2018}.

A \textit{random complex measure} is a random element $\Phi$ of $\bM$ with respect to some fixed probability space $(\Omega, \cA, \BP)$, which is assumed to be rich enough to allow for the construction of all random elements we consider in this paper. In particular, a random element of $\bM_+$ is called \textit{(nonnegative) random measure}, a random element of $\bN$ is called \textit{point process}, and a random element of $\bN_s$ is called \textit{simple point process}. For convenience, we use the notation $\Phi(\omega, B):=\Phi(\omega)(B)$ and $\Phi(B):=\Phi(\cdot, B)$ for a random complex measure $\Phi$ and $\omega\in\Omega, B\in\cB_b$. We do the same for other integrable sets and functions. A convenient way to check if a measurable function $f:\BX\to[0,\infty]$ is integrable with respect to a (nonnegative) random measure $\Phi$ is to check if it is integrable with respect to its \textit{intensity measure} $\BE[\Phi]$, since then by the \textit{Campbell formula},
\begin{equation}\label{e: campbell}
    \BE[\Phi(f)] = \BE[\Phi](f),
\end{equation}
see, e.g.,~\cite[Proposition 2.7]{Last_Penrose_2018}.
For a random complex measure $\Phi$ and a complex-valued function $f$, the same formula holds if $|f|$ is integrable with respect to $\BE[|\Phi|]$.

\subsubsection*{Invariance}

Going forward, we are mostly interested in the case $\BX=\R^d$ and so-called stationary random complex measures. A random complex measure $\Phi$ on $\R^d$ is called \textit{stationary} if $\theta_x\Phi\overset{d}{=}\Phi$ for all $x\in\R^d$, where $\theta_x\Phi:=\Phi-x:=\Phi(\cdot+x)$. Sometimes, we also assume \textit{isotropy}, which is defined in the same way with the group of shift-operators replaced by the group of isometries of $\R^d$. To be able to handle both concepts at once, and to make sure that multiple objects are not only individually but also jointly stationary, we introduce invariance.

Suppose that $G$ is a group that acts measurably on $\BX$ and $\Omega$, and preserves $\BP$, i.e., $\BP(gA)=\BP(A)$ for $g\in G,A\in\cA$. Then a random complex measure $\Phi$ on $\BX$ is called \textit{invariant} (with respect to the group action) if
\begin{equation}
    \Phi(g\omega, gB)= \Phi(\omega, B),\quad g\in G, \omega\in\Omega, B\in\cB_b(\BX). 
\end{equation}
If $\BX=\R^d$, and $\Phi$ is invariant with respect to the shift-operators, then $\Phi$ is stationary, and if $\Phi$ is invariant with respect to the isometries on $\R^d$, then $\Phi$ is isotropic.
From now on, if not stated otherwise, we assume that $\BX=\R^d$, and that $G=\{\theta_x:x\in\R^d\}$ is the group of shift-operators with the natural group action on $\R^d$. Further, we do not construct the shift explicitly in most examples, as it is clear how to define it from context. For an example of a proper construction; see~\cite[Theorem 4.3]{Klatt_Last_Lotz_Yogeshwaran_2025}.

As a first-moment property of an invariant complex random measure $\Phi$ on $\R^d$ with locally finite $\BE[|\Phi|]$, we obtain that
\begin{equation}
    \BE[\Phi]=\gamma\lambda_d
\end{equation}
for some $\gamma\in\C$, which is called the \textit{intensity}. The measure $\lambda_d$ denotes the Lebesgue measure on $\R^d$. Therefore, $\Phi(f)$ is well defined for every absolutely integrable function $f\in L^1(\R^d, \C)$. Further, if $\Phi$ is almost surely nonnegative and not almost surely the zero measure, then $\gamma>0$. In this case, we can also define the \textit{first-order Palm probability measure} $\BP_0^\Phi$ by
\begin{equation}
    \BP_0^\Phi(A) = \frac{1}{\gamma}\BE\bigg[\int_{[0,1)^d} \I\{\theta_xA\}\, \Phi(dx)\bigg],\quad A\in\cA,
\end{equation}
and formulate the \textit{refined Campbell theorem}:
\begin{equation}\label{e: refined Campbell 1}
    \iint f(\omega, x)\, \Phi(\omega, dx)\, \BP(d\omega) = \gamma\iint f(\theta_{-x}\omega, x)\, \BP_0^\Phi(d\omega)\, dx
\end{equation}
for any measurable $f:\Omega\times\R^d\to [0,\infty]$; see, e.g.,~\cite[Appendix A]{Klatt_Last_Lotz_Yogeshwaran_2025}. For a complex-valued function, the equation holds if either side is finite with respect to $|f|$. If $\Phi$ is a simple point process, one can informally think about $\BP_0^\Phi$ as the conditional probability measure with the condition being that $\Phi$ has a point in $0$. Generally, if we write $\BE_0^\Phi$, we refer to the expectation with respect to the first-order Palm probability measure $\BP_0^\Phi$.

\subsubsection*{Second moment properties}

While the intensity suffices to analyze the first-order properties of an invariant random complex measure, we are mostly interested in second-order properties in this paper. To this end, the \textit{reduced second moment measure} is the second-order equivalent of the intensity. Suppose that $\Phi$ is an invariant random complex measure on $\R^d$ which is locally absolutely square-integrable, i.e., $\BE[|\Phi|(B_1)^2]<\infty$, where $B_1:=\{x\in\R^d:\|x\|<1\}$ is the open ball of radius $1$ around $0$. Then we can define its reduced second moment measure by
\begin{equation}
    \alpha_\Phi(B) := \BE\bigg[\iint \I\{x\in[0,1)^d, x-y\in B\}\, \Phi(dx)\, \overline\Phi(dy)\bigg],\quad B\in\cB_b.
\end{equation}
Under these assumptions, it is locally finite and even translation-bounded, i.e., there exists a $c>0$ such that $\alpha_\Phi(B_1+x)\leq c$ for all $x\in\R^d$.
If $\Phi$ is nonnegative almost surely, this definition also works if $\Phi$ is not locally absolutely square-integrable, but then $\alpha_\Phi$ is not locally finite. To formulate the connection of this measure to the second moment, let us introduce the \textit{convolution} and \textit{tilted convolution}:
\begin{align}
    (f\ast g)(y) &:= \int f(x)g(y-x)\, dx,\\
    (f\star g)(y) &:= \int f(x) \overline{g(x-y)}\, dx,\quad y\in\R^d,\label{e: tilted convolution}
\end{align}
where $f,g:\R^d\to\C$ such that the integrals above converge. If $f,g$ are nonnegative, we allow the value of $\infty$ and do not require the integrals to converge. In the same way, we can define these operations for locally absolutely finite complex measures $\mu,\nu$ on $\R^d$, i.e., if the following is well defined, then
\begin{align}
    (\mu \ast \nu)(B) &:= \int \nu(B-x) \, \mu(dx),\\
    (\mu \star \nu)(B) &:= \int \overline{\nu(-B+x)} \, \mu(dx),\quad B\in\cB_b^d.
\end{align}
With an application of Fubini's theorem, we obtain the \textit{second moment formula}
\begin{equation}\label{e: second moment formula}
    \BE\big[\Phi(f)\overline{\Phi(g)}\big] = \alpha_\Phi(f\star g),
\end{equation}
for $f,g\in L^1(\R^d, \C)$ such that either side is absolutely integrable. If $f,g, \Phi$ are nonnegative, the equality even holds if integrability is not given. Coming back to the integrable case, we can also derive the \textit{covariance formula} 
\begin{equation}\label{e: covariance formula}
    \BC[\Phi(f),\Phi(g)] = \beta_\Phi(f\star g),
\end{equation}
where the covariance of complex random variables $X,Y$ is defined by $\BC[X\overline{Y}]-\BE[X]\overline{\BE[Y]}$, and the \textit{covariance measure} $\beta_\Phi$ by
\begin{equation}
    \beta_\Phi := \alpha_\Phi - |\gamma|^2\lambda_d
\end{equation}
with $\gamma$ being the intensity of $\Phi$. Just like in the first-order case, there is also a \textit{refined Campbell theorem} available here if $\Phi$ is an invariant (nonnegative) random measure. Then there exists a family of \textit{second-order Palm probability measures} $(\BP_{0,y}^\Phi)_{y\in\R^d}$, which is uniquely defined $\alpha_\Phi$-almost everywhere such that
\begin{equation}\label{e: refined Campbell 2}
    \iint f(\omega, x, y)\, \Phi^2(\omega, d(x,y))\, \BP(d\omega) = \iiint f(\theta_{-x}\omega, x, x+y)\, \BP_{0,y}^\Phi(d\omega)\, \alpha_\Phi(dy)\, dx
\end{equation}
for any measurable $f:\Omega\times\R^d\times\R^d\to [0,\infty]$. For a complex-valued function, the equation holds if either side is finite with respect to $|f|$. Generally, if we write $\BE_{0,y}^\Phi$, we refer to the expectation with respect to the second-order Palm probability measure $\BP_{0,y}^\Phi$. From the combination of both refined Campbell theorems, we can also derive that 
\begin{equation}
    \alpha_\Phi=\gamma\BE_0^\Phi[\Phi].
\end{equation}
Proofs can be found, e.g., in~\cite{Klatt_Last_Lotz_Yogeshwaran_2025}.

\subsubsection*{Bartlett spectral measure}

Basically, the Bartlett spectral measure allows for an analysis of the second-order properties of a random complex measure in the Fourier domain. It is named after M. S. Bartlett since he brought it up in~\cite{Bartlett_1963}. To approach it, we first recall the Fourier transform of a function $f\in L^1(\R^d,\C)$ as
\begin{equation}
    \hat{f}(k) := \int e^{-i\langle k, x\rangle}f(x)\, dx, \quad k\in\R^d,
\end{equation}
with bounded $\hat{f}\in\cC(\R^d, \C)$.
Similarly, the Fourier transform can be defined for a complex measure $\mu$ with finite total variation, i.e., $|\mu|(\R^d)<\infty$, as
\begin{equation}
    \hat{\mu}(k) := \int e^{-i\langle k, x\rangle}\, \mu(dx),
\end{equation}
and again $\hat{\mu}\in\cC(\R^d, \C)$ is bounded. Further, the Fourier transform can be extended to so-called tempered distributions. However, we only need to extend it to positive semi-definite locally finite complex measures. A locally finite complex measure $\mu$ on $\R^d$ is called \textit{positive semi-definite} if
\begin{equation}
    \mu(f\star f) \geq 0
\end{equation}
for every bounded and compactly supported function $f:\R^d\to\C$. In this case, the Fourier transform $\hat\mu$ of $\mu$ is a locally finite and translation-bounded measure that satisfies
\begin{equation}
    \mu(f\star g) = (2\pi)^{-d}\hat\mu\big(\hat{f}\,\overline{\hat{g}}\big)
\end{equation}
for all bounded and compactly supported functions $f,g:\R^d\to\C$; see~\cite[Section 4]{Berg_Forst_2012}. A measure $\mu$ is called translation-bounded if there is a $c>0$ such that 
\begin{equation}\label{e: translation-bounded}
    \mu(B_1+x) \leq c, \quad x\in\R^d.
\end{equation}
As a consequence, if $f\in L^1(\R^d, [0,\infty))$ is decreasing in every coordinate direction, it is also $\mu$ integrable for every translation-bounded measure $\mu$. If $\mu$ has finite total variation and is positive semi-definite, the two notations for the Fourier transform clash. In this case, the bounded continuous \textit{function} $\hat\mu$ is the density of the locally finite \textit{measure} $\hat\mu$. However, it is always clear from context which object we are referring to.

Now, we can apply this Fourier transform to the reduced second moment measure $\alpha_\Phi$, and the correlation measure $\beta_\Phi$ of an invariant locally absolutely square-integrable random complex measure $\Phi$, as~\eqref{e: second moment formula} and~\eqref{e: covariance formula} yield that both are positive semi-definite. In particular, $\hat\beta_\Phi$ is called the \textit{Bartlett spectral measure} of $\Phi$. We also obtain the following relations:
\begin{align}\label{e: spectral second moment formula}
    \BE\big[\Phi(f)\overline{\Phi(g)}\big] &= (2\pi)^{-d}\hat\alpha_\Phi\big(\hat{f}\, \overline{\hat{g}}\big),\\
    \BC\big[\Phi(f),\Phi(g)\big] &= (2\pi)^{-d}\hat\beta_\Phi\big(\hat{f}\, \overline{\hat{g}}\big), \label{e: spectral covariance formula}
\end{align}
for every $f,g\in L^1(\R^d, \C)$ such that the integral on the RHS converges absolutely. If $f=g$ and the integral on the RHS is infinite, then we can also derive that $\BE[|\Phi(f)|^2]=\infty$. The measures $\hat\alpha_\Phi$ and $\hat\beta_\Phi$ are not only locally finite but also translation-bounded; see, e.g.,~\cite[Proposition 4.9]{Berg_Forst_2012}.
Hence, we obtain that $\BE[|\Phi(f)|^2]<\infty$ for every $f\in L^1(\R^d, \C)$ such that $|\hat{f}(k)|\leq (1+\|k\|)^{-\frac{d}{2}}\rho(\|k\|)^{-\frac{1}{2}}$, $k\in\R^d$, for some strongly log-dominating function $\rho$; see Definition~\ref{d:strongly log-dominating}. In particular, this second moment is finite if $\hat{f}$ has compact support.

If the Bartlett spectral measure $\hat\beta_\Phi$ has a density with respect to the Lebesgue measure $\lambda_d$, then we call this density the \textit{spectral density} $S_\Phi$. If $\Phi$ is a point process with intensity $\gamma>0$, the rescaled $\gamma^{-1} S_\Phi$ is also called the \textit{structure factor}. Given the spectral density $S_\Phi$, we can express the covariance formula as
\begin{equation} \label{e: sf covariance formula}
    \BC\big[\Phi(f),\Phi(g)\big] = \frac{1}{(2\pi)^d}\int \hat{f}(k)\overline{\hat{g}(k)} S_\Phi(k)\, dk
\end{equation}
for every $f,g\in L^1(\R^d, \C)$ such that the integral on the RHS converges absolutely. If the covariance measure $\beta_\Phi$ has finite total variation, i.e., $|\beta_\Phi|(\R^d)<\infty$, then the spectral density always exists and is a bounded continuous function. Hence, the covariance formula yields that $\BE[|\Phi(f)|^2]<\infty$ for every $f\in L^1(\R^d, \C)\cap L^2(\R^d, \C)$.

\subsubsection*{Tensor-valued random complex measures}

Since we apply Taylor's theorem to functions with $d$-dimensional domain in the later sections, we not only have to work with complex-valued measures but also tensor-valued measures. However, we limit ourselves to just handling these objects component-wise and do not derive analogues of the reduced second moment measure, covariance measure, and Bartlett spectral measure for the full objects, even though it would be possible; see, e.g.,~\cite[Section 20]{Yaglom_1987}.

We write $(\C^d)^{\otimes q}$, $q\in\N_0$, for the space of rank $q$ tensors of $\C^d$. We equip it with the scalar product and norm it inherits from the natural embedding into the space $\C^{d^q}$. The corresponding norm is called the Frobenius norm. Note that scalars are rank $0$ tensors, vectors are rank $1$ tensors, and matrices are rank $2$ tensors. For two tensors $x\in(\C^d)^{\otimes q_1}$, $y\in(\C^d)^{\otimes q_2}$ and $q_1,q_2\in\N_0$, we denote the tensor product by $x\otimes y$, which then is a rank $q_1+q_2$ tensor. Further, for a tensor $x$ and $q\in\N_0$, we write $x^{\otimes q}$ for the $q$-times tensor product of $x$ with itself. Note that we have $\|x\otimes y\|\leq \|x\|\|y\|$, $\|x^{\otimes q}\|\leq \|x\|^q$ for two tensors $x,y$ and $q\in\N_0$ by the submultiplicativity of the Frobenius norm. We access the components of the tensors via indices, i.e.,
\begin{equation}
    x_{j_1,...,j_q} := \langle x, e_{j_1}\otimes \cdots \otimes e_{j_q}\rangle , \quad x\in(\C^d)^{\otimes q}, j_1,...,j_q\in\{1,...,d\},q\in\N,
\end{equation}
where $e_j$ is the $j$-th unit vector of $\R^d$ for $j\in\{1,...,d\}$.
Given a bounded tensor valued function $f$, we simply define $\|f\|_\infty:=\|g\|_\infty$, where $g$ is the $[0,\infty)$-valued concatenation of $f$ and the Frobenius norm.

Given a $(\C^d)^{\otimes q}$-valued complex measure $\mu$ and a function $f:\C^d\to \C^{\otimes q}$, we integrate $f$ with respect to $\mu$ using the inner product, i.e.,
\begin{equation}\label{e: tensor valued measure}
    \mu(f) := \sum_{i_1,...,i_q\in\{1,...,d\}} \mu_{i_1,...,i_q}(f_{i_1,...,i_q}).
\end{equation}

\subsection{Transport kernels}

In the following sections, we analyze the effect of transports on specific second moment properties of invariant random complex measures. These transports are handled as so-called transport kernels. Let $(\BX, \mathcal{X})$ and $(\mathbb{Y}, \mathcal{Y})$ be measurable spaces. Then $K:\BX\times\mathcal{Y}\to\C$ is called a \textit{kernel} from $\BX$ to $\mathbb{Y}$ if $K(x,\cdot)$ is a complex measure for every $x\in\BX$ and $K(\cdot, B)$ is measurable for every $B\in\mathcal{Y}$. If $K(x,\cdot)$ has finite total variation for every $x\in\BX$, then we also call it a \textit{transport kernel}. If $K(x,\cdot)$ is a probability measure for every $x\in\BX$, then we also call it a \textit{probability kernel}. To simplify notation, we define $|K|(x, B) := |K(x)|(B)$ for $x\in\BX, B\in\mathcal{Y}$, and if $|K|(x)(\mathbb{Y})$ is bounded independent of $x\in\BX$, then we say that the kernel has \textit{uniformly bounded total variation}. Note that in~\cite{Klatt_Last_Lotz_Yogeshwaran_2025} transport kernels are always assumed to be probability kernels, which is not the case here. Besides that, most of the following is introduced there in the same way in~\cite{Last_Thorisson_2009}.

Now, if we suppose that $\phi$ is a locally finite complex measure on a metric space $\BX$, that $K$ is a transport kernel from $\BX$ to $\BX$, and that the integral in
\begin{equation}
    K\phi(B) := \int K(x, B)\, \phi(dx)
\end{equation}
converges for every $B\in\cB_b(\BX)$, then we call $K$ a \textit{transport kernel} of a \textit{transport} from the \textit{source}~$\phi$ to the \textit{destination} $K\phi$. If $K$ is a probability kernel, then one can think about this transport as taking the mass of $\phi$ from every location $x\in\BX$ and distributing it according to $K(x)$. Note that two kernels $K, L$ can also be concatenated by
\begin{equation}
    KL(x, \cdot) := \int K(y, \cdot)\, L(x, dy),
\end{equation}
and that the concatenation is associative. 
This concept can then be extended to the transport of random complex measures. Note that a random complex measure $\Phi$ on a metric space $\BX$ can be interpreted as a kernel from $\Omega$ to $\BX$. A \textit{(random) transport kernel} $K$ would then be a kernel from $\Omega\times\BX$ to $\BX$, and if
\begin{equation}
    |K||\Phi|(\omega, \cdot) := \int |K|(\omega, x, \cdot)\, |\Phi|(\omega, dx)
\end{equation}
is locally finite, then 
\begin{equation}
    K\Phi(\omega, \cdot) := \int K(\omega, x, \cdot)\, \Phi(\omega, dx)
\end{equation}
is a well-defined random complex measure on $\BX$. From now onward, we also use\linebreak $K_x:=K(x):=K(\cdot, x, \cdot)$ for $x\in\BX$ to simplify the notation. Moreover, operations that act on measures are applied $\omega$-wise to $K_x$. Let us again focus on the case $\BX=\R^d$. There, such (random) transport kernels can be obtained from the disintegration of a so-called transport, which is a random measure on $\R^d\times\R^d$; see, e.g.,~\cite[Proposition 2.2]{Klatt_Last_Lotz_Yogeshwaran_2025}. Elsewhere, a transport is also called coupling, see, e.g.,~\cite{Villani_2008}, or translocation; see, e.g.,~\cite{Kantorovitch_1958}. Alternatively, a transport kernel can be a random element of the Skorokhod space of all càdlàg-functions from $\R^d$ to the (metrized) space of finite complex measures; see~\cite[Subsection 5.3]{Klatt_Last_Lotz_Yogeshwaran_2025} for the special case that the complex measures are probability measures and also~\cite{Janson_2021}. The latter case is particularly important if one wants to assume that the transport kernel and the source are independent.

Just like with complex random measures, we are again interested in invariance as a form of (joint) stationarity or isotropy. A random transport kernel $K$ from $\R^d$ to $\R^d$ is called \textit{invariant} if 
\begin{equation}
    K(g\omega, gx, gB) = K(\omega, x, B)\quad g\in G, \omega\in\Omega, x\in\R^d, B\in\cB^d,
\end{equation}
where $G$ is the group that acts on $\Omega$ and $\R^d$ as before. Sometimes it is more convenient to work with the kernel $K^\ast$ defined by
\begin{equation}
    K^\ast(\omega, x, B) := K(\omega, x, B+x), \omega\in\Omega, x\in\R^d, B\in\cB^d,
\end{equation}
which describes the transport relative to the position $x$. Note that $K$ and $K^\ast$ uniquely determine each other, and that invariance of $K$ is equivalent to 
\begin{equation}
    K^\ast(g\omega, gx,B) = K^\ast(\omega, x, B)\quad g\in G, \omega\in\Omega, x\in\R^d, B\in\cB^d.
\end{equation}
Invariance is further compatible with concatenation, i.e., if $K, L$ are invariant transports and if $\Phi$ is an invariant random complex measure, then $LK$ and $K\Phi$ are also invariant. coming to the most important case that $G$ is the group of translations, then we also obtain that $K\Phi$ has a finite intensity if $\Phi$ has finite intensity and $K$ has uniformly bounded total variation. If $K$ is a probability kernel, then the intensity even persists under transport.

Coming to the second moment properties, if $K$ is a transport kernel and $\Phi$ is a (nonnegative) random measure such that $|K|\Phi$ is locally square-integrable, then we can leverage Palm theory and the refined Campbell theorems~\eqref{e: refined Campbell 1},~\eqref{e: refined Campbell 2} to obtain the formulas
\begin{align}
    \alpha_{K\Phi}(dz) &= \int \BE_{0,y}^\Phi\big[K_y\star K_0\big](dz)\, \alpha_\Phi(dy)\label{e: alpha transport formula}\\
    \beta_{K\Phi}(dz) &= \int \Big(\BE_{0,y}^\Phi\big[K_y\star K_0\big] - \big(\BE_{0}^\Phi\big[K_0\big]+y\big)\star \BE_{0}^\Phi\big[K_0\big]\Big)(dz)\, \alpha_\Phi(dy)\nonumber\\
    &\qquad + \Big(\big(\BE_0^\Phi[K_0]\star \BE_0^\Phi[K_0]\big) \ast \beta_\Phi\Big)(dz).\label{e: beta transport formula}
\end{align}
If $K$ and $\Phi$ are independent, then $\BE_{0,y}^\Phi$ and $\BE_0^\Phi$ can simply be replaced by $\BE$. In this case, the formula even holds when $\Phi$ is not assumed to be nonnegative. For a proof of the special case that $K$ is a probability kernel; see~\cite[Lemma 3.2]{Klatt_Last_Lotz_Yogeshwaran_2025}.

An important example for a transport kernel $K$ is one where $K_x$ is a Dirac measure for every $x\in\R^d$. Then $K$ is called a \textit{displacement kernel}. Another possibility is a transport where the mass is not preserved. For example, one does not move anything and simply adds a weight, i.e., $K_x= Z(x)\delta_x$ with $Z(x)\in\C$ for $x\in\R^d$. In this case, the transport is associated with the $\C$-valued \textit{random field} $Z$ on $\R^d$. Sometimes we identify $Z$ with the the random complex measure $Z\lambda_d:=K\lambda_d$. A typical example is a \textit{GRF (Gaussian random field)}, where every finite collection of random variables of the random field is distributed according to some multivariate normal distribution. An invariant GRF $Z$ is called centered if its mean $\mu:=\BE[Z(0)]$ is $0$. For measurability reasons, we require $Z$ to lie in the Skorokhod space defined earlier. Hence, it has continuous paths, and the covariance function $C(y):= \BE[Z(y)\overline{Z(0)}]$ is also continuous. On the other hand, one can construct a centered GRF for any covariance function $C$ which is Hölder continuous of some positive degree. Therefore, for any finite measure $\nu$ on $\R^d$ such that $k\mapsto\|k\|^\vartheta$ is integrable for some $\vartheta>0$, one can choose $Z$ in such a way that $\beta_Z=\nu$. Finally, we can extend these $\C$-valued random fields to $\C^d$- or even $(\C^{d})^{\otimes q}$-valued random fields just like we extended random complex measures to tensor-valued complex random measures. For a vector-valued random field, the expectation is a vector and the covariance is a matrix. We can handle a tensor-valued GRF in the same way if we consider that we can identify $(\C^{d})^{\otimes q}$ and $\C^{d^q}$. For more details; see, e.g.,~\cite[Section 5]{Khoshnevisan_2002}.

\section{Definition and properties of \textit{p}-uniformity} \label{s: p-uniformity}

In this section, we formalize a concept from physics, commonly referred to as the scaling exponent (see, e.g.,~\cite[Subsection 5.3]{Torquato_2018}), by introducing $p$-uniformity. This framework applies to all locally square-integrable invariant random complex measures and encompasses the concepts of hyperuniformity~\cite{Torquato_Stillinger_2003}, the three classes of hyperuniformity~\cite{Zachary_Torquato_2009, Torquato_2018}, its scaling exponent (also called hyperuniformity exponent in~\cite{Mastrilli_Błaszczyszyn_Lavancier_2024, Lachièze-Rey_2025b}), and stealthy hyperuniformity~\cite{Batten_Stillinger_Torquato_2008,Torquato_Zhang_Stillinger_2015}. Similar concepts were proposed in~\cite{Roca_2023, Lachièze-Rey_2025b}; however, their definitions are less well suited to our purposes, particularly in certain important edge cases.

We begin directly with the definition of $p$-uniformity, which relies on the Bartlett spectral measure as in~\cite{Oğuz_Socolar_Steinhardt_Torquato_2017}, where the latter is accessed through the function $Z$. While a definition only utilizing real space quantities would also be possible, it is more convenient to formulate the definition in Fourier space. Later, in Theorem~\ref{t:asymptotic variance test function}, we establish the corresponding real-space characterization.

\begin{definition}
Let $p\in[-d,\infty)$. Suppose that $\Phi$ is a locally square-integrable invariant random complex measure, i.e., $\BE[|\Phi(B_1)|^2]<\infty$. If
\begin{equation}
    \frac{\hat\beta_\Phi(B_\varepsilon)}{\varepsilon^d} = O(\varepsilon^p)
\end{equation}
as $\varepsilon\to 0$, then $\Phi$ is called $p$\textit{-uniform}.
Similarly, if
\begin{equation}
    \frac{\hat\beta_\Phi(B_\varepsilon)}{\varepsilon^d} = o(\varepsilon^p)
\end{equation}
as $\varepsilon\to 0$, then $\Phi$ is called \textit{beyond} $p$\textit{-uniform}.
Further, if $\hat\beta_\Phi(B_\varepsilon)=0$ for some $\varepsilon>0$, then $\Phi$ is called $\infty$\textit{-uniform with radius} $\varepsilon$. Note that $B_\varepsilon:=\{x\in\R^d:\|x\|<\varepsilon\}$ is the open ball of radius $\varepsilon$ around $0$.

For $p<\infty$, if $\Phi$ is $p$-uniform but not beyond $p$-uniform, then it is called \textit{solely} $p$\textit{-uniform}. Moreover, if $\Phi$ is $\infty$-uniform with radius $\varepsilon$ but not $\infty$-uniform with radius $\vartheta$ for any $\vartheta>\varepsilon$, then it is called $\infty$\textit{-uniform} \textit{with maximal radius} $\varepsilon$. 

Finally, a tensor-valued random complex measure is called (beyond) $p$-uniform (with radius $\varepsilon$) if all of its components are.
\end{definition}

It should immediately be noted that hyperuniformity, i.e., the property that
\begin{equation}
    \frac{\BV[\Phi(B_r)]}{r^d}\xrightarrow{r\to\infty} 0,
\end{equation}
is equivalent to beyond $0$-uniformity. For $p<0$, a solely $p$-uniform random measure is hyperfluctuating, i.e., the above limit is infinite. For $p>0$, $p$-uniformity is related to the three classes of hyperuniformity and the so-called hyperuniformity exponent. Moreover, $\infty$-uniformity is equivalent to stealthy hyperuniformity. Basically, if the Bartlett spectral measure $\hat\beta_\Phi$ has a density with respect to the Lebesgue measure, i.e., the spectral density $S_\Phi$ is well defined, and if for some $p>-d$ and $a\geq 0$, as $k\to 0$,
\begin{equation}
    S_\Phi(k) = a\|k\|^p + o(\|k\|^p),
\end{equation}
then $\Phi$ is $p$-uniform. Further, $\Phi$ is also beyond $p$-uniform iff $a=0$.
Details and proofs follow after we establish some basic properties. To get a better grasp on the concept of $p$-uniformity, one can also take a look at the Examples~\ref{ex: uniformity degrees}, where $p$-uniform point processes are given for various $p\in[-d,\infty]$.

\begin{remark}\label{r: uniformity order}
    The concept of $p$-uniformity admits an order, and $p$-uniformity for higher $p$ always implies $p$-uniformity for lower $p$ in the following sense.
    Let $p,q\in[-d, \infty]$ and $\varepsilon,\vartheta>0$. Suppose that $\Phi$ is a locally square-integrable invariant random complex measure.
    \begin{itemize}
        \item If $\Phi$ is beyond $p$-uniform, then $\Phi$ is also $p$-uniform.
        \item If $\Phi$ is $p$-uniform and $p>q$, then $\Phi$ is also beyond $q$-uniform.
        \item If $\Phi$ is $\infty$-uniform with radius $\varepsilon$ and $\varepsilon\geq \vartheta$, then $\Phi$ is also $\infty$-uniform with radius $\vartheta$.
    \end{itemize}
    This order also implies that if $\Phi$ is solely $p$-uniform, then $\Phi$ is $q$-uniform iff $q\leq p$. Moreover, if $\Phi$ is $\infty$-uniform with maximal degree $\varepsilon$, then $\Phi$ is $\infty$-uniform with degree $\vartheta$ iff $\vartheta\leq \varepsilon$.
\end{remark}

Further, one may wonder why $p$-uniformity is defined only for $p\geq -d$. The reason is that $(-d)$-uniformity is already trivial in the sense that every locally square-integrable invariant random complex measure is $(-d)$-uniform. By the ordered nature of the concept, $p$-uniformity remains trivial for $p<-d$. Still, the following remark gives an insight into beyond $(-d)$-uniformity, which turns out to be directly related to pseudo-ergodicity. A locally square-integrable invariant random complex measure is called \textit{pseudo-ergodic} if $\lim_{r\to\infty}\frac{\Phi(B_r)}{\lambda_d(B_r)}$ is deterministic. We choose to consider the $L^2$-limit here, but others are possible as well. Note that, by the $L^2$-ergodic theorem, ergodicity implies pseudo-ergodicity; see, e.g.,~\cite[Section 25]{Kallenberg_2021}.

\begin{proposition}\label{p:pseudo ergodic}
    Let $\Phi$ be a locally square-integrable invariant random complex measure. Then
    \begin{equation}
        \frac{\hat\beta_\Phi(\{0\})}{(2\pi)^d} = \lim_{r\to\infty}\frac{\BV[\Phi(B_r)]}{\lambda_d(B_r)^2} = \BV\bigg[\lim_{r\to\infty}\frac{\Phi(B_r)}{\lambda_d(B_r)}\bigg],
    \end{equation}
    where the limit in the variance is taken with respect to the $L^2$-norm and is well defined. In particular, $\Phi$ is always $(-d)$-uniform, and $\Phi$ is beyond $(-d)$-uniform iff $\Phi$ is pseudo-ergodic.
\begin{proof}
    First, we see that by~\eqref{e: spectral covariance formula},
    \begin{align*}
        \frac{\BV[\Phi(B_r)]}{\lambda_d(B_r)^2} &=(2\pi)^{-d}\kappa_d^{-2} r^{-2d} \int |\hat\I_{B_r}(k)|^2 \, \hat\beta_\Phi(dk)\\
        &= (2\pi)^{-d}\kappa_d^{-2} \int |\hat\I_{B_1}(rk)|^2 \, \hat\beta_\Phi(dk)\\
        &= \frac{\hat\beta_\Phi(\{0\})}{(2\pi)^d} + (2\pi)^{-d}\kappa_d^{-2} \int_{\R^d\setminus \{0\}} |\hat\I_{B_1}(rk)|^2 \, \hat\beta_\Phi(dk),
    \end{align*}
    where $\kappa_d:=\lambda_d(B_1)$ is the volume of the unit ball.
    Further, we know that $\hat\beta_\Phi$ is translation-bounded (recall \eqref{e: translation-bounded}), and $\hat\I_{B_1}(k) \leq c_2(1+\|k\|)^{-\frac{d+1}{2}}$ as $\|k\|\to\infty$, for some $c_2>0$; see Examples~\ref{ex: fourier smooth}. Since translation-bounded measures behave just like the Lebesgue measure with respect to decreasing functions, one can apply the theorem of dominated convergence. Therefore, 
    \begin{equation*}
        (2\pi)^{-d}\kappa_d^{-2} \int_{\R^d\setminus \{0\}} |\hat\I_{B_1}(rk)|^2 \, \hat\beta_\Phi(dk) \xrightarrow{r\to\infty} 0,
    \end{equation*}
    which yields the first equality. The second equality follows from the $L^2$-ergodic theorem, whereby $\frac{\Phi(B_r)}{\lambda_d(B_r)}$ converges in $L^2$ as $r\to\infty$; see, e.g.,~\cite[Section 25]{Kallenberg_2021}.
\end{proof}
\end{proposition}

Having treated the extreme case of $(-d)$-uniformity, we can come to the real-space characterization of $p$-uniformity for $p\in(-d,\infty)$. The case $p=\infty$ is treated in Theorem~\ref{t: stealthy characterization}. The connection between real space and Fourier space in the context of hyperuniformity was already explored in~\cite{Torquato_Stillinger_2003}, where hyperuniformity itself was first introduced. As seen in the following theorem, the characterization involves a test function. In these early versions, the test function was assumed to be an indicator function, in particular the indicator function of the unit ball. In this case, the relation only holds one-to-one for $p$-uniformity with $p\in(0,1)$, which is also called class III hyperuniformity; see~\cite[Subsection 5.3]{Torquato_2018}. For other indicator functions, the relation can break down even earlier; see, e.g.,~\cite{Kim_Torquato_2017}. More recently, a criterion for the decay of the Fourier transform of the indicator function was introduced, which guarantees that the characterization of $p$-uniformity with respect to this indicator function works up to a certain $p$; see~\cite[Theorem 3.7]{Björklund_Hartnick_2024}. Still, the characterization was limited to $p<1$, which comes from the fact that, at this point, only indicator functions were considered. The following theorem generalizes the connection between real space and Fourier space with respect to $p$-uniformity, building upon their proof. For higher degrees, smoother test functions are needed, but for lower degrees, indicator functions of nicely behaved sets are still sufficient. Similar generalizations were made in~\cite{Björklund_Byléhn_2024,Mastrilli_Błaszczyszyn_Lavancier_2024,Krishnapur_Yogeshwaran_2024,Björklund_Byléhn_2025, Lachièze-Rey_2025b}. It should also be noted that much earlier, while not explicitly stated, these techniques were already utilized in~\cite{Ghosh_Lebowitz_2018}. Our proof is inspired by~\cite[Theorem 3.6 and 3.7]{Björklund_Hartnick_2024} and requires us to make the following extension to a definition of theirs.

\begin{definition}[Fourier-smoothness]\label{d: fourier smoooth}
    Let $p\in[-d, \infty)$ and $f\in L^1(\R^d, \C)$. Suppose that there is a $c>0$ such that
    \begin{equation}\label{e: fourier smooth 1}
        |\hat{f}(k)| \leq c \frac{1}{(1+\|k\|)^{\frac{d+p}{2}}}, \quad k\in\R^d.
    \end{equation}
    If $p\leq 0$, then also suppose that, as $r\to\infty$,
    \begin{equation}\label{e: fourier smooth 2}
        \int_{B_2^c} \sup_{s\in B_1} |\hat{f}(r(k+s))|^2\, dk = O(r^{-(d+p)}).
    \end{equation}
    Then $f$ is called \textit{Fourier-smooth with exponent} $p$.
\end{definition}
Note that~\eqref{e: fourier smooth 2} also holds for $p>0$ since it is then implied by~\eqref{e: fourier smooth 1}.
The most commonly used $\I_{B_1}$ is Fourier-smooth with exponent $1$, and $\I_{[0,1)^d}$ is only Fourier-smooth with exponent $-d+2$; see Examples~\ref{ex: fourier smooth}. Schwartz functions, e.g., are Fourier-smooth with any exponent, i.e., in particular functions in $C_c^\infty(\R^d, \C)$. We provide a deeper insight in Appendix~\ref{s: Fourier-smoothness}.

Just like the indicator function was spread in Proposition~\ref{p:pseudo ergodic}, we spread the test functions using the following notation.

\begin{definition}\label{d: f_r}
    Let $f:\R^d\to\C$ be a function and $r>0$. Then $f_r:\R^d\to\C$ is defined by
    \begin{equation}
        f_r(x) := f(\tfrac{x}{r}),\quad x\in\R^d.
    \end{equation}
\end{definition}

This notation is consistent with indicator functions, since $f_r=\I_{rB}$ if $f=\I_B$, and in particular, $f_r=\I_{B_r}$ if $f=\I_{B_1}$.

\begin{remark}\label{r: f_r formulas}
    Let $f\in L^1(\R^d, \C), r>0$. Then
    \begin{equation}
        \hat{f_r}(k) = r^d \hat{f}(rk),\quad k\in\R^d.
    \end{equation}
    Further, if $f\in \cC^q(\R^d, \C)$ for $q\in\N$, then
    \begin{equation}
        f_r^{(q)}(x) = \frac{1}{r^q} f^{(q)}(\tfrac{x}{r}), \quad x\in\R^d,
    \end{equation}
    where $f^{(q)}$ denotes the (possibly tensor-valued) $q$-th derivative of $f$.
    Note that we always differentiate or take the Fourier transform after scaling $f$.
\end{remark}

\begin{theorem}\label{t:asymptotic variance test function}
Let $p\in[-d,\infty)$. Suppose that $\Phi$ is a locally square-integrable invariant random complex measure. Further, suppose that $f\in L^1$ is Fourier-smooth with exponent $p+\vartheta$ for some $\vartheta>0$. Then
\begin{equation}\label{e:asymptotic variance test function}
    \frac{\BV[\Phi(f_r)]}{r^d} = O(r^{-p})\quad \text{as } r\to\infty.
\end{equation}
if $\Phi$ is $p$-uniform. The converse holds under the additional assumption that $\lambda_d(f)\neq0$. In both cases, the same holds for beyond $p$-uniformity with $o$ instead of $O$.
\begin{proof}
    Remark~\ref{r: f_r formulas} and~\eqref{e: spectral covariance formula} yield that
    \begin{align}
        \frac{\BV[\Phi(f_r)]}{r^{d-p}} &= (2\pi)^{-d}r^{-d+p} \int |\hat{f}_r(k)|^2 \, \hat\beta_\Phi(dk) \nonumber\\
        &= (2\pi)^{-d}r^{d+p} \int |\hat{f}(rk)|^2 \, \hat\beta_\Phi(dk) \nonumber\\\label{e:linear statistic bound 1}
        &= (2\pi)^{-d}r^{d+p} \int_{B_2} |\hat{f}(rk)|^2 \, \hat\beta_\Phi(dk) + (2\pi)^{-d}r^{d+p} \int_{B_2^c} |\hat{f}(rk)|^2 \, \hat\beta_\Phi(dk)
    \end{align}
    for $r>0$. Since $\hat\beta_\Phi$ is translation-boundeds, see \eqref{e: translation-bounded}, and since $f$ is Fourier-smooth, we can apply \eqref{e: fourier smooth 2}. Hence, there is a $c_1>0$ such that, as $r\to\infty$,
    \begin{equation}
        r^{d+p} \int_{B_2^c} |\hat{f}(rk)|^2 \, \hat\beta_\Phi(dk) \leq c_1r^{d+p} \int_{B_2^c}\sup_{s\in B_1} |\hat{f}(r(k+s))|^2 \, dk = O(r^{-\vartheta}).
    \end{equation}
    Therefore, it only remains to treat the fist summand in \eqref{e:linear statistic bound 1}. We can apply the other Fourier-smoothness property~\eqref{e: fourier smooth 1} of $f$ to obtain that there is a $c_2>0$ such that
    \begin{align*}
        r^{d+p} \int_{B_2} |\hat{f}(rk)|^2 \, \hat\beta_\Phi(dk) &\leq c_2 r^{d+p} \int_{B_2} (1 + r \|k\|)^{-(d+p+\vartheta)} \, \hat\beta_\Phi(dk) \\
        &\leq c_2 r^{d+p} \sum_{j=1}^{2\lceil r\rceil} \int_{B_{\frac{j}{r}} \setminus B_{\frac{j-1}{r}}} (1 + r \|k\|)^{-(d+p+\vartheta)} \, \hat\beta_\Phi(dk) \\
        &\leq c_2 r^{d+p} \sum_{j=1}^{2\lceil r\rceil}  j^{-(d+p+\vartheta)} \big(\hat\beta_\Phi\big(B_{\frac{j}{r}}\big) - \hat\beta_\Phi\big(B_{\frac{j-1}{r}}\big) \big), \quad r>0.
    \end{align*}
    By reordering the sum and applying
    \begin{equation*}
        (j+1)^{-(d+p+\vartheta)} \geq j^{-(d+p+\vartheta)} - (d+p+\vartheta)j^{-(d+p+\vartheta+1)},\quad j>0,
    \end{equation*}
    we can further calculate that
    \begin{align}
        r^{d+p} \int_{B_2} |\hat{f}(rk)|^2 \, \hat\beta_\Phi(dk) &\leq c_2 r^{d+p} \sum_{j=1}^{2\lceil r\rceil}  (j^{-(d+p+\vartheta)} - (j+1)^{-(d+p+\vartheta)}) \hat\beta_\Phi\big(B_{\frac{j}{r}}\big)\nonumber\\
        &\leq c_2(d+p+\vartheta) r^{d+p} \sum_{j=1}^{2\lceil r\rceil}  j^{-(d+p+\vartheta+1)} \hat\beta_\Phi\big(B_{\frac{j}{r}}\big)\nonumber\\\label{e:linear statistic bound 2}
        &= c_2(d+p+\vartheta) \sum_{j=1}^{2\lceil r\rceil}  j^{-(1+\vartheta)} \big(\tfrac{j}{r}\big)^{-(d+p)} \hat\beta_\Phi\big(B_{\frac{j}{r}}\big), \quad r>0.
    \end{align}
    If $\Phi$ is $p$-uniform, then $\big(\tfrac{j}{r}\big)^{-(d+p)} \hat\beta_\Phi\big(B_{\frac{j}{r}}\big)$ is uniformly bounded for $r\geq 1$ and $j\in\{1,...,2\lceil r\rceil\}$. In combination with~\eqref{e:linear statistic bound 1},~\eqref{e:linear statistic bound 2} and the fact that $\sum_{j=1}^\infty j^{-(1+\vartheta)} <\infty$, we obtain~\eqref{e:asymptotic variance test function}. If $\Phi$ is also beyond $p$-uniform, the convergence of the bound~\eqref{e:linear statistic bound 2} to $0$ follows from the dominated convergence theorem, which concludes the first part of the proof also in this case.
    
    Now assume that $\lambda_d(f)\neq0$. Hence, $\hat{f}(0) = \lambda_d(f) \neq 0$, and because $\hat{f}$ is continuous, as $f\in L^1$, we can conclude that there exist $c_3, s>0$ such that $|\hat{f}(k)|^2\geq c_3$ for $k\in B_s$. By Remark~\ref{r: f_r formulas}, this bound results in $|\hat{f}_r(k)|^2 \geq c_3 r^{2d}$ for $k\in B_{\frac{s}{r}}$, $r>0$. Therefore,
    \begin{equation}\label{e: bound spectral by real}
        \frac{\hat\beta_\Phi\big(B_{\frac{s}{r}}\big)}{\big(\tfrac{s}{r}\big)^{d+p}} \leq s^{-(d+p)}c_3^{-1} \frac{\int |\hat{f}_r(k)|^2 \, \hat{\beta}_\Phi(dk)}{r^{d-p}} = s^{-(d+p)}(2\pi)^dc_3^{-1} \frac{\BV[\Phi(f_r)]}{r^{d-p}},\quad r>0.
    \end{equation}
    Choosing $\varepsilon:= \frac{s}{r}$ yields that assuming~\eqref{e:asymptotic variance test function} implies that $\Phi$ is $p$-uniform, and that assuming~\eqref{e:asymptotic variance test function} with $o$ instead of $O$ implies that $\Phi$ is beyond $p$-uniform.
\end{proof}
\end{theorem}

Under stricter assumptions on the test function, we can also calculate the limit variance.

\begin{proposition}\label{p:asymptotic variance test function}
    In the setting of Theorem~\ref{t:asymptotic variance test function}, assume $p>-d$. Also assume that either $d=1$, $f$ is isotropic, or $\Phi$ is isotropic. If further $\varepsilon^{-(d+p)}\hat\beta_\Phi(B_\varepsilon)$ converges as $\varepsilon\to0$, and $x\mapsto x_jf(x)$ is Fourier-smooth with exponent $p+2+\vartheta$ for every $j\in\{1,...,d\}$, then
\begin{equation}\label{e: limit var formula}
    \frac{\BV[\Phi(f_r)]}{r^{d-p}} \xrightarrow{r\to\infty}\frac{d+p}{d\kappa_d(2\pi)^d} \int|\hat{f}(k)|^2\|k\|^p\,dk\lim_{\varepsilon\to0}\frac{\hat\beta_\Phi(B_\varepsilon)}{\varepsilon^{d+p}} .
\end{equation}

\begin{remark}
    These additional conditions from Proposition~\ref{p:asymptotic variance test function} reduce to $p>-d$ if $f$ is assumed to be an isotropic Schwartz function. Further, the case $p=-d$ can be treated with the \linebreak $L^2$-ergodic theorem as in Proposition~\ref{p:pseudo ergodic}. Also note that if $p\in 2\N_0$ (or $p\in 4\N_0$ for the second equality), then
    \begin{equation}\label{e: laplace version}
        \frac{1}{(2\pi)^d}\int|\hat{f}(k)|^2\|k\|^p\,dk = \big\|f^{\big(\tfrac{p}{2}\big)}\big\|_2^2 = \big\| \Delta^{\tfrac{p}{4}}f\big\|_2^2,
    \end{equation}
    where $\Delta$ is the Laplace operator.
\end{remark}

\begin{proof}[Proof of Proposition~\ref{p:asymptotic variance test function}] 
    Assume that $p>-d$ and that $\varepsilon^{-(d+p)}\hat\beta_\Phi(B_\varepsilon)$ converges as $\varepsilon\to0$. Then, similarly to~\eqref{e:linear statistic bound 1}, we can show
    \begin{equation*}
        \frac{\BV[\Phi(f_r)]}{r^{d-p}} = (2\pi)^{-d}r^{d+p} \int_{B_1} |\hat{f}(rk)|^2 \, \hat\beta_\Phi(dk) + o(1)
    \end{equation*}
    as $r\to\infty$. Now we define $g:[0,\infty)\to[0,\infty)$ and the measure $\mu$ on $[0,\infty)$ by
    \begin{align*}
        g(s)&:= \int_{\partial B_1} |\hat{f}(sk)|^2\,\sigma_d(dk), \quad s\in[0,\infty),\\
        \mu(B)&:= \hat\beta_\Phi\big(\{x\in\R^d:\|x\|\in B\}\big),\quad B\in\cB([0,\infty)),
    \end{align*}
    where $\sigma_d$ is the unique rotation invariant probability measure on $\partial B_1$. Since we assume that $x\mapsto x_jf(x)$ is Fourier-smooth with exponent $p+2+\vartheta$ for each $j\in\{1,...,d\}$, we can conclude that $g\in C^1$, and we do not only obtain that there is a $c_4>0$ such that $g(s)\leq c_4 (1+|s|)^{-(d+p+\vartheta)}$ but also such that $g'(s)\leq c_4 (1+|s|)^{-(d+p+1+\vartheta)}$. If we assume that either $f$ or $\Phi$ is isotropic, then $\hat{f}$ or $\hat\beta_\Phi$ is isotropic and
    \begin{align*}
        (2\pi)^{-d}r^{d+p} \int_{B_1} |\hat{f}(rk)|^2 \, \hat\beta_\Phi(dk) &= (2\pi)^{-d}r^{d+p} \int_0^1 g(rs) \, \mu(ds)\\
        &= (2\pi)^{-d}r^{d+p} \bigg(\mu([0,1))g(r)- \int_0^1 rg'(rs) \mu(B_s)\, ds\bigg)\\
        &= (2\pi)^{-d}\mu([0,1))r^{d+p}g(r)- (2\pi)^{-d}\int_0^r g'(s) s^{d+p}\frac{\mu(B_{\frac{s}{r}})}{(\frac{s}{r})^{d+p}}\, ds,
    \end{align*}
    where we used partial integration and that $\mu(\{0\})=0$ by Proposition~\ref{p:pseudo ergodic}.
    \begin{align*}
        (2\pi)^{-d}r^{d+p} \int_{B_1} |\hat{f}(rk)|^2 \, \hat\beta_\Phi(dk) &= (2\pi)^{-d}r^{d+p} \bigg(\mu([0,1))g(r)- \int_0^1 rg'(rs) \mu(B_s)\, ds\bigg)\\
        &\xrightarrow{r\to\infty} -(2\pi)^{-d}\int_0^\infty g'(s)s^{d+p}\, ds \lim_{\varepsilon\to0}\frac{\hat\beta_\Phi(B_\varepsilon)}{\varepsilon^{d+p}}\\
        &= \frac{d+p}{(2\pi)^d} \int_0^\infty g(s) s^{d+p-1}\, ds \lim_{\varepsilon\to0}\frac{\hat\beta_\Phi(B_\varepsilon)}{\varepsilon^{d+p}}\\
        &= \frac{d+p}{d\kappa_d(2\pi)^d} \int_0^\infty |\hat{f}(k)|^2 \|k\|^p\, ds \lim_{\varepsilon\to0}\frac{\hat\beta_\Phi(B_\varepsilon)}{\varepsilon^{d+p}},
    \end{align*}
    where we applied the theorem of dominated convergence and partial integration. Hence, all assertions have been shown.
\end{proof}
\end{proposition}

Since the indicator function of the unit ball $\I_{B_1}$ is Fourier-smooth with exponent $1$, we can use it to rigorously establish the connection between our spectral definition of $p$-uniformity and the classical notion of hyperuniformity and hyperfluctuation; see the following remark. In contrast, the indicator function of the unit cube is only guaranteed to work for $p$-uniformity with $p<-d+2$, which is negative for $d\geq2$. Consequently, the limit in~\eqref{e: hyperuniform scaling ball} may not be $0$ if $d\geq 2$, even if $\Phi$ is hyperuniform. For $d\geq3$, it may even be $\infty$. This observation has first been made in~\cite{Kim_Torquato_2017}.

Further, when $\Phi$ is assumed to be a (weighted) point process of intensity $1$, Theorem \ref{t:asymptotic variance test function} reveals a connection to the theory of numerical integration. Since $\BE[\Phi(f)] = \lambda_d(f)$ for all $f\in L^1$ by \eqref{e: campbell}, $\Phi(f)$ is an unbiased estimator for $\lambda_d(f)$. In that sense, $\Phi$ is is $p$-uniform iif the mean squared error of the corresponding numerical integration scheme which evaluates in every $x\in\Phi$ and applies the weight $\Phi(\{x\})$ satisfies
\begin{equation}
    \BV[\Phi_r(f)] = O(r^{-(d+p)})\quad \text{as } r\to\infty
\end{equation}
for all $f\in L^1$, which are Fourier smooth with exponent $p+\vartheta$ for some $\vartheta>0$, and where $\Phi_r(B):= \frac{1}{r^d} \Phi(rB)$ for $r>0, B\in\cB^d$; see, e.g.,~\cite{Stehr_Kiderlen_2020}.

\begin{remark}
    Suppose that $\Phi$ is a locally square-integrable invariant random complex measure. Then $\Phi$ is hyperuniform, i.e.,
    \begin{equation}\label{e: hyperuniform scaling ball}
        \frac{\BV[\Phi(B_r)]}{r^d} \xrightarrow{r\to\infty} 0,
    \end{equation}
    iff $\Phi$ is beyond $0$-uniform. Further, $\Phi$ is not $0$-uniform iff
    \begin{equation}
        \limsup_{r\to\infty}\frac{\BV[\Phi(B_r)]}{r^d} = \infty.
    \end{equation}
    Hence, if $\Phi$ is hyperfluctuating, i.e.,
    \begin{equation}
        \frac{\BV[\Phi(B_r)]}{r^d} \xrightarrow{r\to\infty} \infty,
    \end{equation}
    then $\Phi$ is not $0$-uniform. Finally, if $\Phi$ is solely $0$-uniform, then one would typically say that $\Phi$ exhibits classical or Poisson-like behavior since the variance $\BV[\Phi(B_r)]$ scales like the volume $\lambda_d(B_r)$ as $r\to\infty$.
    These assertions follow directly from an application of Theorem~\ref{t:asymptotic variance test function} with $p=0$ and $f=\I_{B_1}$. 
\end{remark}

Going further, we can relate our notion of $p$-uniformity to the scaling exponent of the spectral density close to $0$ that is typically considered by the physics community; see, e.g.,~\cite[Subsection 5.3]{Torquato_2018}. In the mathematics community, this scaling exponent has also been called the hyperuniformity exponent if it is positive; see~\cite{Mastrilli_Błaszczyszyn_Lavancier_2024, Lachièze-Rey_2025b}. 

\begin{proposition}\label{p:mixing condition uniformity}
    Let $p\in(-d,\infty)$. Suppose that $\Phi$ is a locally square-integrable invariant random complex measure, and that $\hat\beta_\Phi$ admits a density $S_\Phi$ (close to $0$) satisfying
    \begin{equation}\label{e: sf asymptotics}
        S_\Phi(k) = a \|k\|^p + o(\|k\|^p)
    \end{equation}
    as $k\to0$ for some $a\geq0$.
    Then $\Phi$ is $p$-uniform. Furthermore, $\Phi$ is beyond $p$-uniform iff $a=0$. If one supposes that, for some $\vartheta>0$, $f\in L^1$ is Fourier-smooth with exponent $p+\vartheta$ and $x\mapsto x_jf(x)$ is Fourier-smooth with exponent $p+2+\vartheta$ for each $j\in\{1,...,d\}$, then
    \begin{equation}\label{e: limit var density}
        \frac{\BV[\Phi(f_r)]}{r^{d-p}} \xrightarrow{r\to\infty} \frac{a}{(2\pi)^d} \int|\hat{f}(k)|^2\|k\|^p\, dk.
    \end{equation}
    The limit terms can also be expressed without the Fourier transform as in~\eqref{e: laplace version} if $p\in2\N_0$.
    \begin{proof}
        Under the assumption that $\hat\beta_\Phi$ has the density $S_\Phi$ close to $0$, one obtains that
        \begin{equation}\label{e: limit var beta density}
            \frac{\hat\beta_\Phi(B_\varepsilon)}{\varepsilon^{d+p}} \xrightarrow{\varepsilon\to0} \frac{d\kappa_d}{d+p} a.
        \end{equation}
        Hence, $\Phi$ is $p$-uniform and beyond $p$-uniform iff $a=0$. To obtain~\eqref{e: limit var density}, we can almost apply~\eqref{e: limit var formula} from Proposition~\ref{p:asymptotic variance test function}. What is missing is that we do not assume that $\Phi$ is isotropic here. However,~\eqref{e: sf asymptotics} yields that the component of $\hat\beta_\Phi$ which contributes to the limit is isotropic.
    \end{proof}
\end{proposition}

We can give integrability conditions to guarantee a behavior as in~\eqref{e: sf asymptotics}. These can also be found in~\cite{Krishnapur_Yogeshwaran_2024}.
\begin{proposition}
    In the setting of Proposition~\ref{p:mixing condition uniformity}, if $p\in[0,\infty)\setminus2\N_0$, then one can guarantee that $S_\Phi$ exists and has an expansion of the form~\eqref{e: sf asymptotics} with $a=0$ if 
    \begin{equation}\label{e: correlation moment integrability}
        \int (\|y\|+1)^p\, |\beta_\Phi|(dy) < \infty,
    \end{equation}
    and if 
    \begin{equation}\label{e: correlation zero integral}
        \int y^{\otimes q}\, \beta_\Phi(dy) = 0
    \end{equation}
    for $q\in2\N_0$ with $q < p$. If $p\in2\N_0$ and if $\Phi$ is isotropic or $p=0$, then under the same conditions one obtains~\eqref{e: sf asymptotics} with
    \begin{equation}
        a = \frac{(-1)^{\frac{p}{2}}}{p!}\int y_1^p\, \beta_\Phi(dy).
    \end{equation}
    If $p\in2\N$ and $\Phi$ is not isotropic, then the coefficient, and even the leading order in the asymptotics of $S_\Phi$ can be direction-dependent. Still, $\Phi$ is beyond $p$-uniform iff~\eqref{e: correlation zero integral} holds also for $q=p$. If $p\in[-d, 0)$,~\eqref{e: correlation moment integrability} cannot guarantee that $\hat\beta_\Phi$ has a density, but it still implies that $\Phi$ is beyond $p$-uniform.
    \begin{proof}
         If~\eqref{e: correlation moment integrability} holds for $p\geq 0$, then $\hat\beta_\Phi$ has a continuous density. Further, since 
        \begin{equation*}
            e^{-i\langle k, y\rangle} = \sum_{q=0}^{\lfloor p \rfloor} \frac{(-i)^q}{q!} \langle k,y\rangle^q + o(|\langle k, y\rangle|^p), \quad k, y\in\R^d,
        \end{equation*}
        as $\langle k,y\rangle\to 0$, one can calculate that
        \begin{equation}
            S_\Phi(k) = \sum_{q=0}^{\lfloor p \rfloor} \frac{(-i)^q}{q!} \int y^{\otimes q}\, \beta_\Phi(dy)\,  k^{\otimes q} + o(\| k\|^p)
        \end{equation}
        as $k\to 0$. Hence, the assertion follows from the assumption of~\eqref{e: correlation zero integral}. Note that isotropy of $\Phi$ implies isotropy of $\beta_\Phi$, whereby $\int \langle k, x\rangle^p\, \beta_\Phi(dy)$ is constant for $k\in\partial B_\varepsilon$ for every $\varepsilon>0$.

        If~\eqref{e: correlation moment integrability} holds for $p\in[-d,0)$, then 
        \begin{equation*}
            \frac{\BV[\Phi(B_r)]}{r^{d-p}} = \int r^p\frac{\I_{B_r}\star \I_{B_r}(y)}{r^d} \, \beta_\Phi(dy)
            \xrightarrow{r\to \infty} 0
        \end{equation*}
        by the theorem of dominated convergence. Hence, $\Phi$ is beyond $p$-uniform by Theorem~\ref{t:asymptotic variance test function}.
    \end{proof}
\end{proposition}

Further, we can connect $p$-uniformity to the three classes of hyperuniformity; see~\cite{Zachary_Torquato_2009} and~\cite[Subsection 5.3]{Torquato_2018}. In Theorem~\ref{t:asymptotic variance test function}, one assumes that the test function is always smooth enough with respect to $p$. However, if the test function is not sufficiently smooth, the decay of the variance is limited even if the random complex measure is $p$-uniform for very hight $p$, as it can be seen in the following theorem. If one considers the test function $f=\I_{B_1}$, one obtains the three classes of uniformity: class III, when the $p$-uniformity of $\Phi$ is the limiting factor for the decay of the variance, i.e., if $\Phi$ is $0$-uniform but ont $1$-uniform; class I, when the smoothness of $\I_{B_1}$ is the limiting factor, i.e., if $\Phi$ is $p$-uniform for some $p>1$; and class II, where both factors (can) contribute at the transition, i.e., if $\Phi$ is solely $1$-uniform. The case that $\Phi$ is beyond $1$-uniform but not $p$-uniform for any $p>1$ is not covered by the established definitions. The real-space behavior of class I hyperuniform point processes is found in \eqref{e:class I hyperuniformity}, that of class II in \eqref{e:class II hyperuniformity}, and that of class III was already treated in Theorem \ref{t:asymptotic variance test function}. One can find similar results to the following theorem in the survey~\cite{Lachièze-Rey_2025b}, and in the special case that $f=\I_{B_1}$ in~\cite{Torquato_2018, Coste_2021}. In the special case that $p=1$ and that $\Phi$ is a point process, the lower bound, i.e., the last part of the following theorem, was first proven in~\cite{Beck_1987} and then generalized in \cite{Björklund_Byléhn_2024, Lachièze-Rey_2025b}.

\begin{theorem}\label{t:hyperunifomrity classes}
Suppose that $\Phi$ is a locally square-integrable invariant random complex measure. Further, let $p,q\in(-d,\infty)$, assume that $\Phi$ is $p$-uniform, and suppose that $f\in L^1$ is Fourier-smooth with exponent $q$. If $p>q$, then
\begin{equation}\label{e:class I hyperuniformity}
    \frac{\BV[\Phi(f_r)]}{r^d} = O(r^{-q})
\end{equation}
as $r\to\infty$. If $p=q$, then
\begin{equation}\label{e:class II hyperuniformity}
    \frac{\BV[\Phi(f_r)]}{r^d} = O(\log(r)r^{-q})
\end{equation}
as $r\to\infty$. 
Moreover, assume that $p=q>0$ and that there are $\tilde c, a\in(0,1)$ such that 
\begin{equation}\label{e: oscillation uniformer}
\int_a^1|\hat{f}(sk)|^2\, ds\geq \tilde{c}(1+\|k\|)^{-(d+q)}, \quad k\in \R^d.
\end{equation}
Then~\eqref{e:class I hyperuniformity} can only hold if $\Phi$ is beyond $p$-uniform, and it can only hold with $o$ instead of $O$ if $\Phi$ is $\lambda_d$ up to a deterministic constant factor.

\begin{remark}
   The additional property~\eqref{e: oscillation uniformer} is fulfilled if $f=\I_{B_1}$ and $q=1$, or if \linebreak $\lim_{\|k\|\to\infty}|\hat{f}(k)|(1+\|k\|)^{\frac{d+q}{2}}$ exists, is not $0$, and $\lambda_d(f)\neq0$.
\end{remark}

\begin{proof}[Proof of Theorem~\ref{t:hyperunifomrity classes}]
    Let $p,q\in(-d,\infty)$, and suppose that $f\in L^1$ is Fourier-smooth with exponent $q$. The proof is mostly analogous to that of Theorem~\ref{t:asymptotic variance test function}. Similar to~\eqref{e:linear statistic bound 1}, we obtain that
    \begin{equation}\label{e:lin stat bound 1}
        \frac{\BV[\Phi(f_r)]}{r^{d-q}} = (2\pi)^{-d}r^{d+q} \int_{B_2} |\hat{f}(rk)|^2 \, \hat\beta_\Phi(dk) + O(1)
    \end{equation}
    as $r\to\infty$. Similar to~\eqref{e:linear statistic bound 2}, we can further bound the asymptotically relevant summand by
    \begin{equation}\label{e:lin stat bound 2}
        (2\pi)^{-d}r^{d+q} \int_{B_2} |\hat{f}(rk)|^2 \, \hat\beta_\Phi(dk) \leq c_1(d+q) \sum_{j=1}^{2\lceil r\rceil}  j^{-1} \big(\tfrac{j}{r}\big)^{-(d+q)} \hat\beta_\Phi\big(B_{\frac{j}{r}}\big), \quad r>0,
    \end{equation}
    for some $c_1>0$. Additionally, we suppose that $\Phi$ is $q+\vartheta$-uniform for some $\vartheta\geq 0$, whereby there is some $c_2>0$ such that $\hat{\beta}_\Phi(B_r)\leq c_2 r^{d+q+\vartheta}, r\leq 3$. We can apply these bounds to obtain
    \begin{align*}
        \sum_{j=1}^{2\lceil r\rceil}  j^{-1} \big(\tfrac{j}{r}\big)^{-(d+q)} \hat\beta_\Phi\big(B_{\frac{j}{r}}\big) &\leq c_2r^{-\vartheta}\sum_{j=1}^{2\lceil r\rceil} j^{\vartheta-1}\\
        &\leq \I\{\vartheta>0\}c_2\vartheta^{-1} + \I\{\vartheta=0\} c_2(\log(r)+1 + \log(4)), \quad r>1.
    \end{align*}
    As $r\to\infty$, the bound is $O(1)$ if $\vartheta>0$, and $O(\log(r))$ if $\vartheta=0$, which, in combination with~\eqref{e:lin stat bound 1} and~\eqref{e:lin stat bound 2}, proves this part of the assertion.

    Now also assume that $q>0$ and that there is a $c_3>0$ such that $|\hat{f}(k)|\geq c_3 (1+\|k\|)^{-\frac{d+q}{2}}$ for $k\in \R^d$. The more general case follows afterwards. If $\Phi$ is not deterministic, there is $c_4>0$ such that $\hat\beta_\Phi(B_{c_4})>0$. Hence,
    \begin{align*}
        \frac{\BV[\Phi(f_r)]}{r^{d-q}} = (2\pi)^{-d}r^{d+q} \int |\hat{f}(rk)|^2\,\hat\beta_\Phi(dk) &\geq (2\pi)^{-d}c_3^2 r^{d+q} \int (1+\|rk\|)^{-(d+q)}\,\hat\beta_\Phi(dk)\\
        &\geq (2\pi)^{-d}2^{-1}c_3^2 c_4^{-(d+q)}\hat\beta_\Phi(B_{c_4}) \\
        &> 0,\quad r>c_4^{-1},
    \end{align*}
    whereby~\eqref{e:class I hyperuniformity} cannot hold with $o$ instead of $O$.
    
    For the other part, assume that~\eqref{e:class I hyperuniformity} holds. If $f$ were Fourier-smooth with a higher exponent, Theorem~\ref{t:asymptotic variance test function} would imply that $\Phi$ is $q$-uniform. However, examining~\eqref{e: bound spectral by real} in the proof, we see that the Fourier-smoothness of $f$ is irrelevant for this part of the assertion. It remains to show that $\Phi$ in this case is not only $q$-uniform but also beyond $q$-uniform. If we assume that this is not the case, then there must be a sequence $(\varepsilon_n)_{n\in\N}$ and a $c_5>0$ such that $\varepsilon_n\to 0$ and $\varepsilon_n^{-(d+q)}\hat\beta(B_{\varepsilon_n})\to c_5$ as $n\to\infty$. Without loss of generality, we can assume that there is a $c_6\in(0,1)$ such that $\varepsilon_{n+1}\leq c_6 \varepsilon_n$ for $n\in\N$. Then once can show that $\varepsilon_n^{-(d+q)}\hat\beta(B_{\varepsilon_n}\setminus B_{\varepsilon_{n+1}})$ is asymptotically lower bounded by $c_5(1-c_6)>0$ as $n\to\infty$. Thus, without loss of generality, there is a $c_7>0$ such that $\hat\beta(B_{\varepsilon_n}\setminus B_{\varepsilon_{n+1}})\geq c_7 \varepsilon_n^{d+q}$ for $n\in\N$. Hence, 
    \begin{align*}
        \frac{\BV[\Phi(f_r)]}{r^{d-q}} &\geq (2\pi)^{-d}c_3^2 r^{d+q} \int (1+\|rk\|)^{-(d+q)}\,\hat\beta_\Phi(dk)\\
        &\geq (2\pi)^{-d}c_3^2 \sum_{n=1}^\infty r^{d+q} (1+r\varepsilon_n)^{-(d+q)} \hat\beta(B_{\varepsilon_n}\setminus B_{\varepsilon_{n+1}})\\
        &\geq (2\pi)^{-d}c_3^2c_7 \sum_{n=1}^\infty \frac{1}{(1+r^{-1}\varepsilon_n^{-1})^{d+q}}\\
        &\xrightarrow{r\to\infty} \infty,
    \end{align*}
    which contradicts the assumption. Therefore, $\Phi$ must be beyond $q$-uniform. 
    
    Finally, it remains to show that the weaker assumption of~\eqref{e: oscillation uniformer} also suffices compared to the stronger assumption made in the previous part of the proof. For this fact note that
    \begin{align*}
        \frac{1}{r(1-a)}\int_{ar}^{r} \frac{\BV[\Phi(f_s)]}{s^{d-q}}\, ds &= \frac{1}{r(1-a)(2\pi)^d}\iint_{ar}^{r} s^{d+q} |\hat{f}(sk)|^2\, ds\,\hat\beta_\Phi(dk)\\
        &= \frac{1}{(1-a)(2\pi)^d}  r^{d+q} \iint_{a}^{1} s^{d+q} |\hat{f}(srk)|^2\, ds\,\hat\beta_\Phi(dk)\\
        &\geq \frac{\tilde{c}a^{d+q}}{(1-a)(2\pi)^d} r^{d+q}\int (1+\|rk\|)^{-(d+q)}\,\hat\beta_\Phi(dk),\quad r>0.
    \end{align*}
    For the lower bound, the previous methods apply, and imply that the integrand on the left-most side must be greater than or equal to the lower bound for some $s\in[ar, r]$, which concludes what was left to show.
    
    To see that one cannot expect more than the fact that~\eqref{e:class I hyperuniformity} does not hold in this case, e.g., a lower bound for the $\limsup$ which is analogous to the upper bound from~\eqref{e:class II hyperuniformity}, consider the case that $\hat\beta_\Phi=\sum_{n\in\N} 2^{-2n^a}(\delta_{2^{-n^a}} + \delta_{-2^{-n^a}})$ for $a>1$. Hence, not all random measures that lie between class III and class I hyperuniformity have to be class II hyperuniform as it is defined in~\cite[Subsection 5.3]{Torquato_2018}.
\end{proof}
\end{theorem}

Having covered the real-space implications of $p$-uniformity for $p<\infty$, let us come to a real-space characterization of $\infty$-uniformity as an analogue to Theorem~\ref{t:asymptotic variance test function}. The argument comes from~\cite{Ghosh_Lebowitz_2018}.

\begin{theorem}\label{t: stealthy characterization}
Suppose that $\Phi$ is a locally absolutely square-integrable invariant random complex measure. Then $\Phi$ is $\infty$-uniform iff there is a function $f\in L^1(\R^d, \C)$ with $\lambda_d(f)\neq0$ such that $\BV[\Phi(f)]=0$. Further, if $\hat{f}(k)\neq 0$ for all $k\in B_{\varepsilon}$ and some $\varepsilon>0$, then $\Phi$ is $\infty$-uniform with radius $\varepsilon$. Conversely, if $\Phi$ is $\infty$-uniform with radius $\varepsilon>0$, then $\BV[\Phi(f)]=0$ for all $f\in L^1(\R^d, \C)$ with $\hat{f}(k)=0$ for all $k\in B_\varepsilon^c$.
\begin{proof}
    Basically, the assertions follow directly from the definition of $\infty$-uniformity with radius $\varepsilon$ and the fact that by the covariance formula~\eqref{e: spectral covariance formula},
    \begin{equation}
        \BV[\Phi(f)] = \frac{1}{(2\pi)^d}\int |\hat{f}(k)|^2\, \hat\beta_\Phi(dk).
    \end{equation}
    For the first assertion, one should also note that if $\lambda_d(f)\neq 0$, then $|\hat{f}(0)|^2 > 0$, which stays true in a neighborhood of $0$ since $f\in L^1$, whereby $\hat{f}$ is continuous.
\end{proof}
\end{theorem}

For lattices, $f$ can be chosen as the indicator function of the unit cell. In general, one can always consider $f:=|\hat{\I}_{B_\frac{\varepsilon}{2}}|^2$ for $\varepsilon>0$. For example, for $d=1$ that yields
\begin{equation}
    f(x)=\varepsilon^2\frac{\sin(\tfrac{\varepsilon}{2}x)^2}{(\tfrac{\varepsilon}{2}x)^2},\quad x\in\R.
\end{equation}

While it can be seen in the previous statement that $\infty$-uniformity heavily restricts the randomness of complex random measures, the implications for nonnegative random measures are even stronger. The following upper bound was already derived in~\cite{Ghosh_Lebowitz_2018}, but they did not focus on the lower bound. Instead, they provided a bound for the hole size. A similar statement can be found in~\cite[Proposition 5.4]{Lachièze-Rey_2025b}.

\begin{proposition}\label{p: stealthy uniform bounds}
    There are dimensional constants $c_1,c_2,c_3>0$ such that for every locally absolutely square-integrable invariant random measure $\Phi$ with intensity $\gamma$, which is $\infty$-uniform with radius $\varepsilon>0$, it holds that
    \begin{equation}
        c_1\gamma r^d \leq \Phi(B_r) \leq c_2\gamma r^d,\quad r>\frac{c_3}{\varepsilon},
    \end{equation}
    almost surely.
    \begin{proof}
        For the upper bound, see~\cite{Ghosh_Lebowitz_2018}. For the lower bound assume that $\Phi$ is $\infty$-uniform with radius $\varepsilon>0$, and consider a function $f\in L^2(\R^d, (0,\infty))$ such that $\hat{f}(k)=0$ for $k\in B_1^c$ and $f(x)\leq c_4(1+\|x\|)^{-(d+1)}$ for $x\in\R^d$ for some $c_4>0$. To see that such a function $f$ exists consider the Fourier transform of $(\hat{\I}_{B_{\frac{1}{2}}}\star \hat{\I}_{B_{\frac{1}{2}}}) \hat{\I}_{B_1}$. Then, by Theorem~\ref{t: stealthy characterization} and the existence of the upper bound, for $r>\varepsilon^{-1}$,
        \begin{align*}
            c_4\Phi(B_r) \geq \Phi(f(\varepsilon_0\cdot)\I_{B_r}) &\geq \Phi(f(\varepsilon_0\cdot)) - c_4\int_{B_r^c} \frac{1}{(1+\varepsilon_0\|x\|)^{d+1}}\, \Phi(dx)\\
            &\geq \lambda_d(f) \gamma\varepsilon_0^{-d} -  c_5 \tfrac{1}{\varepsilon_0r} \gamma\varepsilon_0^{-d}\\
            &= \big(\lambda_d(f) -  c_5 \tfrac{1}{\varepsilon_0r}\big) \gamma\varepsilon_0^{-d}, \quad \varepsilon_0<\varepsilon,
        \end{align*}
        for some $c_5>0$.
        Hence, choosing $\varepsilon_0:=\frac{2c_5}{\lambda_d(f)r}$ yields the desired lower bound, since then,
        \begin{equation*}
            \Phi(B_r) \geq \frac{\lambda_d(f)^{d+1}}{2^{d+1}c_4c_5^{d}} \gamma r^d, \quad r > \frac{2c_5}{\lambda_d(f)\varepsilon}.
        \end{equation*}
    \end{proof}
\end{proposition}

One can also relate $p$-uniformity to rigidity. For $p\in\N_0$, a random complex measure $\Phi$ is called $p$-rigid if
\begin{equation}
    \int_B x^{\otimes q}\, \Phi(dx) \in \sigma(\Phi|_{B^c}) \quad \text{a.s., }q\in\{0,...,p\}, B\in\cB_b^d.\label{e: p-rigid}
\end{equation}
If $p=0$, $\Phi$ is called number rigid. Further, if 
\begin{equation}
    \Phi|_B \in \sigma(\Phi|_{B^c}),\quad \text{a.s., }B\in\cB_b^d,
\end{equation}
then $\Phi$ is called maximally (or strongly) rigid; see~\cite{Ghosh_Peres_2017, Lachièze-Rey_2025a}. Note that if $\Phi$ is $p$-rigid for every $p\in\N_0$, then it is also maximally rigid by a simple density argument. The foundational work that connects hyperuniformity and rigidity can again be found in~\cite{Ghosh_Lebowitz_2017, Ghosh_Lebowitz_2018}, on which the more recent and general work~\cite{Lachièze-Rey_2025a} builds upon. Our contribution in the following proposition is to show that $(2p+d)$-uniformity implies the sufficient conditions for $p$-rigidity from~\cite{Lachièze-Rey_2025a}. There, this connection was only established in special cases.

\begin{proposition}\label{p: rigidity}
    Let $p\in\N_0$. Suppose that $\Phi$ is a locally absolutely square-integrable invariant random complex measure. If $\Phi$ is $(2p+d)$-uniform, then $\Phi$ is $p$-rigid. Further, if $\Phi$ is $q$-uniform for all $q\in[-d,\infty)$, in particular, if $\Phi$ is $\infty$-uniform, then $\Phi$ is maximally rigid.
\end{proposition}

\begin{remark}\label{r: rigidity}
    In Proposition~\ref{p: rigidity}, one only has to consider the part of $\hat\beta_\Phi$ which is continuous with respect to the Lebesgue measure when checking $(2p+d)$-uniformity.
\end{remark}
    
\begin{proof}[Proof of Proposition~\ref{p: rigidity}]
    We only have to show the first part since the second part follows from the first.

    Let $p\in\N_0$ and suppose that $\Phi$ is a locally absolutely square-integrable invariant random complex measure which is $(2p+d)$-uniform. Let $S_\Phi$ be the density of (the part of) $\hat\beta_\Phi$ which is continuous with respect to the Lebesgue measure. Hence, there is a $c_1>0$ such that
    \begin{equation}\label{e: SF integral bound}
        \frac{\int_{B_\varepsilon}S_\Phi(k)\, dk}{\varepsilon^{2(d+p)}} \leq \frac{\hat\beta_\Phi(B_\varepsilon)}{\varepsilon^{2(d+p)}} \leq c_1, \quad \varepsilon\in(0,1).
    \end{equation}
    Since later on we only use the bound of the integral over $S_\Phi$, one can see here that, as stated in Remark~\ref{r: rigidity}, the part of $\hat\beta_\Phi$ which is singular to $\lambda_d$ can be ignored when checking the $(2p+d)$-uniformity of $\Phi$ in this context.
    Let $Q:\R^d\to \C$ be a multivariate polynomial, i.e., there exist $a_m\in\C$ for $m\in\N_0^d$ such that all but finitely many $a_m=0$ and $ Q(k) = \sum_{m\in\N_0^d} a_mk^m$, where $k^m:= k_1^{m_1}\cdots k_d^{m_d}$.
    To show that the sufficient condition in~\cite[Theorem 1]{Lachièze-Rey_2025a} for $p$-rigidity of $\Phi$ is fulfilled, we have to show that 
    \begin{equation}\label{e: infinite fraction integral}
        \int_{B_\varepsilon} \frac{|Q(k)|^2}{S_\Phi(k)}\, dk = \infty
    \end{equation}
    for any $\varepsilon>0$ if there is some $m\in\N_0^d$ such that $a_m\neq0$ and $\|m\|_1\leq p$. By Remark~\ref{r: uniformity order}, we can use induction to, without loss of generality, assume that $a_m=0$ for all $m\in\N_0^d$ with $\|m\|_1<p$. In this case, there exists a $c_2>0$ such that
    \begin{equation}\label{e: poly upper bound}
        |Q(k)| \leq c_2 \|k\|^p, \quad k\in B_1.
    \end{equation}
    As a final preparation, we can apply the Cauchy-Schwarz inequality to obtain 
    \begin{align*}
        \int_{B_\varepsilon\setminus B_{\frac{\varepsilon}{2}}} |Q(k)|^2 \, dk &= \int_{B_\varepsilon\setminus B_{\frac{\varepsilon}{2}}} \frac{|Q(k)|}{\sqrt{S_\Phi(k)}} |Q(k)|\sqrt{S_\Phi(k)} \, dk\\
        &\leq \sqrt{\int_{B_\varepsilon\setminus B_{\frac{\varepsilon}{2}}} \frac{|Q(k)|^2}{S_\Phi(k)}\, dk\int_{B_\varepsilon\setminus B_{\frac{\varepsilon}{2}}} |Q(k)|^2S_\Phi(k)\, dk},\quad \varepsilon>0.
    \end{align*}
    An inversion of this inequality yields
    \begin{align*}
        \int_{B_\varepsilon\setminus B_{\frac{\varepsilon}{2}}} \frac{|Q(k)|^2}{S_\Phi(k)}\, dk &\geq \bigg(\int_{B_\varepsilon\setminus B_{\frac{\varepsilon}{2}}} |Q(k)|^2 \, dk\bigg)^2 \bigg(\int_{B_\varepsilon\setminus B_{\frac{\varepsilon}{2}}} |Q(k)|^2S_\Phi(k)\, dk\bigg)^{-1}\\
        &\geq \bigg(\varepsilon^d\int_{B_1\setminus B_{\frac{1}{2}}} |Q(\varepsilon k)|^2 \, dk\bigg)^2 \bigg(c_2^2\int_{B_\varepsilon} \|k\|^{2p}S_\Phi(k)\, dk\bigg)^{-1}\\
        &\geq \bigg(\int_{B_1\setminus B_{\frac{1}{2}}} \Big|\frac{Q(\varepsilon k)}{\varepsilon^p}\Big|^2 \, dk\bigg)^2 \bigg(c_2^2 \frac{\int_{B_\varepsilon} S_\Phi(k)\, dk}{\varepsilon^{2(d+p)}}\bigg)^{-1},\quad\varepsilon\in(0,1),
    \end{align*}
    where we applied~\eqref{e: poly upper bound} in the second line. Since
    \begin{equation*}
        \frac{Q(\varepsilon k)}{\varepsilon^p} \xrightarrow{\varepsilon\to 0} \sum_{\substack{m\in\N_0^d\\ \|m\|_1 = p}} a_mk^m
    \end{equation*}
    uniformly on $B_1$, and by~\eqref{e: SF integral bound}, there exists a $c_3>0$ such that 
    \begin{equation*}
        \int_{B_\varepsilon\setminus B_{\frac{\varepsilon}{2}}} \frac{|Q(k)|^2}{S_\Phi(k)}\, dk \geq c_3,\quad \varepsilon\in(0,1),
    \end{equation*}
    if there is an $m\in\N_0^d$ with $a_m\neq0$ and $\|m\|_1=p$. This fact guarantees~\eqref{e: infinite fraction integral} for any $\varepsilon>0$, which concludes the proof.
\end{proof}

Having investigated the real-space implications of $p$-uniformity, where Theorem~\ref{t:asymptotic variance test function} played an important role, there is another direction in which this theorem can be further investigated. For the characterization of $p$-uniformity in Theorem~\ref{t:asymptotic variance test function} to hold, it was assumed that the test function $f$ satisfies $\lambda_d(f)\neq0$. However, in the upcoming Section~\ref{s:main results}, derivatives play a major role, which cannot satisfy this property. Therefore, if one aims to find the highest $p$ for which a random complex measure is $p$-uniform, it is crucial to find cases in which the characterization in Theorem~\ref{t:asymptotic variance test function} still holds even if $\lambda_d(f)=0$.

\begin{proposition}\label{p:asymptotic variance test function zero integral}
    In the setting of Theorem~\ref{t:asymptotic variance test function}, if $\Phi$ is pseudo-ergodic and $p>-d$, the assumption of $\lambda_d(f)\neq0$ can be weakened to $f\neq0$ under one of the following four additional assumptions: $d=1$, $f$ is isotropic, $\inf_{k\in\partial B_r} |\hat{f}(k)|>0$ for some $r>0$, or $\Phi$ is isotropic.    
    \begin{proof}
        Assume that we are in the setting of Theorem~\ref{t:asymptotic variance test function} with $f\neq0$. Then by the covariance formula~\eqref{e: spectral covariance formula} and Remark~\ref{r: f_r formulas}, we obtain
        \begin{equation}\label{e:variance formula remark zero integral}
            \frac{\BV[\Phi(f_r)]}{r^d} = (2\pi)^{-d}r^d \int |\hat{f}(rk)|^2 \, \hat{\beta}_\Phi(dk).
        \end{equation}
        If $d=1$, or if $f$ is isotropic, then $|\hat{f}|$ is isotropic. In the other hand, if $\Phi$ is isotropic, then $\hat{\beta}_\Phi$ is isotropic. In this case, one can replace $\hat{f}$ on the RHS by
        \begin{equation}\label{e: replacement function iso}
            k\mapsto \sqrt{\int_{\partial B_1} |\hat{f}(\|k\|s)|^2 \,\sigma_d(ds)} ,\quad k\in\R^d,
        \end{equation}
        without changing the value of the integral. The measure $\sigma_d$ refers to the unique rotation-invariant probability measure on the surface of the unit ball. This new function is also isotropic. Therefore, without loss of generality, we can assume that $\hat{f}$ is isotropic in this case as well. Now, as $f$ is integrable, $\hat{f}$ is continuous (which also holds for the replacement function~\eqref{e: replacement function iso}), and as $f\neq 0$, we can conclude that $\inf_{k\in\partial B_r} |\hat{f}(k)|>0$ for some $r>0$. Hence, we can work under this assumption for the remaining part of the proof. 

        Let us assume that, as $r\to\infty$,
        \begin{equation}\label{e:variance asymptotic remark zero integral}
            \frac{\BV[\Phi(f_r)]}{r^d} = O(r^{-p}).
        \end{equation}
        Therefore, by~\eqref{e:variance formula remark zero integral}, there are $r_0>0,c_1>0$ such that
        \begin{equation*}
            r^{d+p} \int |\hat{f}(rk)|^2 \, \hat{\beta}_\Phi(dk) \leq c_1, \quad r\geq r_0.
        \end{equation*}
        As $\inf_{k\in\partial B_r} |\hat{f}(k)|$ is continuous in $r>0$, the additional assumption implies that there are $c_2,c_3>0, a\in(0,1)$ such that $|\hat{f}|^2 \geq c_2 (\I_{B_{c_3}} - \I_{B_{a c_3}})$. Hence,
        \begin{equation*}
            c_2 r^{d+p} \big(\hat{\beta}_\Phi\big(B_{\frac{c_3}{r}}\big) - \hat{\beta}_\Phi\big(B_{\frac{ac_3}{r}}\big)\big) \leq c_1, \quad r\geq r_0.
        \end{equation*}
        By a change of variables, we obtain that there exists an $\varepsilon_0>0$ such that
        \begin{equation*}
             \hat{\beta}_\Phi(B_{a\varepsilon}) \geq \hat{\beta}_\Phi(B_{\varepsilon}) - \varepsilon^{d+p} \frac{c_1}{c_2 c_3^{d+p}},\quad \varepsilon\leq \varepsilon_0.
        \end{equation*}
        Now an induction over $n\in\N$ yields that, for $\varepsilon\leq \varepsilon_0$,
        \begin{align*}
            \hat{\beta}_\Phi(\{0\}) \xleftarrow{n\to\infty} \hat{\beta}_\Phi(B_{a^n\varepsilon}) &\geq \hat{\beta}_\Phi(B_{\varepsilon}) - \varepsilon^{d+p} \frac{c_1}{c_2 c_3^{d+p}} \sum_{j=0}^{n-1} a^{j(d+p)}\\
            &\xrightarrow{n\to\infty} \hat{\beta}_\Phi(B_{\varepsilon}) - \varepsilon^{d+p} \frac{c_1}{c_2 c_3^{d+p}(1-a^{d+p})}.
        \end{align*}
        As $\Phi$ is (pseudo-) ergodic, we have that $\hat{\beta}_\Phi(\{0\})=0$ by Proposition~\ref{p:pseudo ergodic}, whereby
        \begin{equation*}
            \hat{\beta}_\Phi(B_{\varepsilon}) \leq \varepsilon^{d+p} \frac{c_1}{c_2 c_3^{d+p}(1-a^{d+p})},\quad \varepsilon\leq\varepsilon_0,
        \end{equation*}
        which implies $p$-uniformity of $\Phi$. Hence, the proof is concluded for the $O$-case. If we assume that~\eqref{e:variance asymptotic remark zero integral} also holds with $o$ instead of $O$, then we can choose $c_1$ arbitrarily small by choosing $r_0$ sufficiently large, i.e., $\varepsilon_0$ sufficiently small, and obtain beyond $p$-uniformity of $\Phi$. 
    \end{proof}
\end{proposition}

Note that of these four alternative additional assumptions, only $d=1$ and isotropy of $\Phi$ can be fulfilled if $f$ is a (partial) derivative. Still, these two cases are important in applications. Further, the assumption that $\Phi$ is pseudo-ergodic is essential, as this is equivalent to the fact that $\Phi$ is beyond $(-d)$-uniform by Proposition~\ref{p:pseudo ergodic}, and for $(-d)$-uniformity, the situation is very different. There, as it can be seen in the following proposition, any test function $f$ with $\lambda_d(f)=0$ leads to a real-space behavior which is indistinguishable from that of a $(-d)$-uniform random measure.

\begin{proposition}\label{p:asymptotic variance test function 0}
Suppose that $\Phi$ is a locally square-integrable invariant random complex measure. Further, suppose that $f\in L^1$ is Fourier smooth with exponent $-d+\vartheta$ with $\lambda_d(f)=0$, where $\vartheta>0$. Then, as $r\to \infty$,
\begin{equation}
    \frac{\BV[\Phi(f_r)]}{r^d} = o(r^d)
\end{equation}
\begin{proof}
Choose $\Psi:= \Phi - \lim_{r\to\infty}\frac{\Phi(B_r)}{\lambda_d(B_r)} \cdot \lambda_d$. Note that the limit is taken with respect to the $L^2$-norm and is well defined by the $L^2$-ergodic theorem; see, e.g.,~\cite[Section 25]{Kallenberg_2021}. We can see that $\Phi(f_r)=\Psi(f_r)$, as $\lambda_d(f_r) = r^d\lambda_d(f)=0$, $r>0$. Further, by construction, $\Psi$ is pseudo-ergodic and therefore beyond $(-d)$-uniform by Proposition~\ref{p:pseudo ergodic}. Now the assertion follows from Theorem~\ref{t:asymptotic variance test function}.
\end{proof}
\end{proposition}

Another important implication of Theorems~\ref{t:asymptotic variance test function} and~\ref{t: stealthy characterization} is that $p$-uniformity persists under the linear combination of random measures, even if they are correlated.

\begin{proposition}\label{p: uniform linear combination}
    Let $p\in[-d, \infty], \varepsilon>0$. Suppose that $\Phi$ and $\Psi$ are locally absolutely square-integrable invariant random complex measures, which can be correlated, and let $a,b\in\C$. Assume that $\Phi$ and $\Psi$ are both (beyond) $p$-uniform (with radius $\varepsilon$). Then $a\Phi+b\Psi$ is also (beyond)\linebreak $p$-uniform (with radius $\varepsilon$).

    Further, if $\Phi$ is (beyond) $p$-uniform (with radius $\varepsilon$), $\Psi$ is not, and $b\neq 0$, then $a\Phi+b\Psi$ is not (beyond) $p$-uniform (with radius $\varepsilon$). In this case, if $p<\infty$, if either of the following limits exist, we have
        \begin{equation}\label{e:lin comb limit}
            \lim_{r\to\infty} \frac{\BV[(a\Phi+b\Psi)(f_r)]}{r^{d-p}} = \lim_{r\to\infty} b^2\frac{\BV[\Psi(f_r)]}{r^{d-p}},
        \end{equation}
        where $f$ is chosen as in Theorem~\ref{t:asymptotic variance test function}.
    \begin{proof}
        Let $p\in[-d, \infty)$, and suppose that $f\in\C_c^\infty(\R^d, [0,\infty))$ and $f\neq0$. Then, by the Cauchy-Schwarz inequality,
        \begin{align*}
            \Big(\sqrt{b\BV[\Psi(f_r)]} - \sqrt{a\BV[\Phi(f_r)]}\Big)^2 &\leq \BV[(a\Phi+b\Psi)(f_r)]\\ &\leq \Big(\sqrt{b\BV[\Psi(f_r)]} + \sqrt{a\BV[\Phi(f_r)]}\Big)^2
        \end{align*}
        for $r>0$. Now, the second inequality yields the first assertion, the first inequality yields the second assertion, and together they yield~\eqref{e:lin comb limit} with an application of Theorem~\ref{t:asymptotic variance test function}. For $p=\infty$, the assertions follow similarly from Theorem~\ref{t: stealthy characterization}.
    \end{proof}
\end{proposition}

We close this section by giving $p$-uniform examples for different $p\in[-d,\infty]$. This question is trivial when it comes to random signed measures, as a Gaussian random field that is solely $p$-uniform can be constructed easily for any $p\in[-d,\infty]$. Therefore, all given examples are nonnegative random measures and, in particular, even simple point processes.

\begin{examples}\label{ex: uniformity degrees}
    The following examples are ordered from $p$-uniform with low $p$ to high $p$. All but the first are ergodic.
    \begin{enumerate}[label=(\alph*)]
        \item Suppose that $N$ is a nondeterministic, nonnegative random variable such that $\BE[N^{-2d}]<\infty$. Further, suppose that $U\sim\cU([0,1)^d)$ is independent of $N$. Then the randomly scaled stationary lattice $\sum_{z\in\Z^d} \delta_{N(z+U)}$ is solely $(-d)$-uniform. This is a simple consequence of Proposition~\ref{p:pseudo ergodic} and the fact that the spatial mean of the point process is $N^{-d}$.
        \item A stationary Poisson hyperplane intersection process is solely $(-d+1)$-uniform. This fact can be simply derived from Theorem~\ref{t:asymptotic variance test function} and the asymptotics of the variance derived in~\cite{Lothar_Heinrich_Hendrik_Schmidt_Volker_Schmidt_2006}; see also~\cite{Klatt_Last_2022}.
        \item A stationary Poisson process is solely $0$-uniform, which is a direct consequence of the fact that the Bartlett spectral measure is a multiple of the Lebesgue measure in this case.
        \item The Ginibre point process is solely $2$-uniform. The same holds for many other hyperuniform determinantal point processes, and in general, nontrivial determinantal point processes cannot be beyond $2$-uniform; see, e.g.,~\cite[Subsection 3.3]{Lachièze-Rey_2025b}.
        \item The zeros of the planar Gaussian analytic function are solely $4$-uniform; see~\cite{Sodin_Tsirelson_2002}.
        \item The stationary lattice $\Z^d+U$ with $U\sim\cU([0,1)^d)$ is $\infty$-uniform with maximal radius $2\pi$. This can be directly derived from the fact that it has the Bartlett spectral measure $\sum_{z\in\Z^d\setminus\{0\}} \delta_{2\pi z}$, which is a consequence of the Poisson summation formula.
    \end{enumerate}
\end{examples}

This list leaves large gaps on the scale without proper examples. On the left end of the scale, we construct ergodic solely $p$-uniform point processes with an arbitrarily low $p<0$ in Appendix~\ref{s: Hyperfluctuating examples}.  However, our main interest lies in the construction of ergodic point processes which fill the gap between the degree of $4$ and $\infty$. A construction given in~\cite{Lachièze-Rey_2025a} which was proven to result in a solely $2p$-uniform point process for any $p$ that is prime for $d=2$. This construction can be artificially extended to higher dimensions. In contrast, see Subsection~\ref{ss: pp from averaging sets}. There, we explore a wider class of point processes which can be solely $2p$-uniform for any $p\in\N$ natively in any dimension. Moreover, since the construction in~\cite{Lachièze-Rey_2025a} is a subclass of the class explored in Subsection~\ref{ss: pp from averaging sets}, we also prove that the restriction to $p$ prime imposed there is not necessary. While the examples in Subsection~\ref{ss: pp from averaging sets} are a natural derivation from the results of the following Section~\ref{s:main results}, they are also strongly connected to an idea from~\cite{Gabrielli_Joyce_Torquato_2008}. Another class of point processes, that can be $p$-uniform for arbitrarily high $p$, and that is arguably larger and more relevant in applications, is explored in Subsection~\ref{ss: pp from invariant partitions}. The trade-off is that, when we obtain $p$-uniformity in the latter class of point processes, we usually cannot expect the process to be solely $p$-uniform. All of these examples are derived from the theory we build in the following section on the effect of transportation on $p$-uniformity.

\section{On \textit{p}-uniformity under transports}\label{s:main results}

The effect of transportation on $p$-uniformity can only be studies if both the source and destination have a locally finite second moment. While it is easy to see that a probability kernel preserves a locally finite first moment, the same does not hold for the second moment. Further complications arise if the transport kernel is not assumed to be a probability kernel. In the following subsection, we derive sufficient conditions for the preservation of a locally finite second moment under transportation. It turns out that these conditions even lay a foundation for the results on the preservation of $p$-uniformity in the subsections thereafter.

\subsection{Persistence of square integrability under transport}\label{ss:Square integrability of destinations}

There is no single function that makes the central Condition~\ref{c:condition square-integrable general} for our first Theorem~\ref{t:main square-integrable} sharp. However, the following class of functions does.

\begin{definition}[Strongly log-dominating]\label{d:strongly log-dominating}
Let $\rho:[0,\infty)\to [0,\infty)$ be an increasing function that satisfies
\begin{equation}\label{e:strongly log-dominating}
    \int_0^\infty \frac{1}{r \rho(r) + 1} \, dr < \infty.
\end{equation}
Then $\rho$ is called a \textit{strongly log-dominating} function.
\end{definition}

As suggested by the name, $\rho$ is not a strongly log-dominating function if $\rho(r)=\log(r+1)$, $r\geq0$. However, many other slowly increasing functions are admissible. For instance, $\rho$ is strongly log-dominating if $\rho(r)= r^\vartheta$ or even $\rho(r)=\log(r+1)^{1+\vartheta}, r\geq0$, for some $\vartheta>0$.

In the following, we work with two different integrability conditions. One is for the more general case, where the transport kernel $K$ can be non-mass-preserving and dependent on the random measure $\Phi$ it transports. The other is specifically adapted to the case that $K$ is a probability kernel (or of uniformly bounded total variation) and independent of $\Phi$. Both conditions can be considered sharp within their respective context, as it can be seen in Remark~\ref{r:square-integrable sharpness} and~\ref{r:ind square-integrable sharpness}. If not stated otherwise, we assume that $K$ is an invariant (random) transport kernel, and that $\Phi$ is a locally absolutely square-integrable invariant random complex measure. In this general setting, we use the following integrability condition.

\begin{condition}\label{c:condition square-integrable general}
    For some strongly log-dominating function $\rho$, it holds that the invariant random measure $\Psi_{d,\rho}$ defined by
    \begin{equation}\label{e: remainder process definition}
        \Psi_{d,\rho}(dy) := \bigg(\int \sqrt{\|x\|^d\rho(\|x\|) + 1}\, |K^\ast|(y, dx)\bigg)\, |\Phi|(dy), \quad y\in\R^d,
    \end{equation}
    is locally square-integrable.
\end{condition}
This condition limits how much mass can be transported over large distances. 
In some special cases, e.g., if $K$ is a probability kernel (i.e., mass-preserving) and independent of $\Phi$, we can also use the following condition, which is a natural extension of the condition used in~\cite{Dereudre_Flimmel_Huesmann_Leblé_2024}.
\begin{condition}\label{c:condition square-integrable Markov}
    Assume that $K$ has uniformly bounded total variation, that $K$ and $\Phi$ are independent, and that    \begin{equation}\label{e:condition square-integrable Markov}
        \BE\bigg[\int \|x\|^{d} \, |K_0|(dx) \bigg] < \infty.
    \end{equation}
\end{condition}
In this case, we can interpret the condition directly as a moment condition on the typical transport distance.

Now we are ready to state the main theorem of this subsection. Afterwards, we show how the general integrability conditions can be rephrased or simplified in certain contexts.

\begin{theorem}\label{t:main square-integrable}
Suppose that $\Phi$ is a locally absolutely square-integrable invariant random complex measure, and that $K$ is an invariant transport kernel. Further, assume that either Condition~\ref{c:condition square-integrable general} or Condition~\ref{c:condition square-integrable Markov} holds.
Then $K\Phi$ is locally absolutely square-integrable.
\begin{proof}
If Condition~\ref{c:condition square-integrable general} holds, the assertion follows from Lemma~\ref{l:remainder2case1}. If Condition~\ref{c:condition square-integrable Markov} holds, the assertion follows from Lemma~\ref{l:remainder2case2}. In both cases, we apply the lemmas with $p=0$. The lemmas each actually include a stronger statement, which we do not need here but in the following subsection. Also, while Lemmas~\ref{l:remainder2case1} and~\ref{l:remainder2case2} make the same statements under different conditions, the proofs differ completely. The proof of Lemma~\ref{l:remainder2case2} follows a technique that is adapted from a proof in~\cite{Dereudre_Flimmel_Huesmann_Leblé_2024}.
\end{proof}

\end{theorem}

As seen in Remark~\ref{r:square-integrable sharpness}, Condition~\ref{c:condition square-integrable general} is sharp in the sense that for any $\rho:[0,\infty)\to[0,\infty)$ with 
\begin{equation}
    \int_0^\infty \frac{1}{r\rho(r) + 1} \, dr = \infty,
\end{equation}
one can find $\Phi$ and $K$ such that Condition~\ref{c:condition square-integrable general} holds but $K\Phi$ is not locally square-integrable. Condition~\ref{c:condition square-integrable Markov} is sharp in a similar sense.

\begin{proposition}\label{p:independent square-integrable}
In the setting of Theorem~\ref{t:main square-integrable}, Condition~\ref{c:condition square-integrable general} is equivalent to the existence of an $\varepsilon>0$ such that
\begin{equation}
    \int_{B_{\varepsilon}} \BE_{0,y}^{|\Phi|}\bigg[ \int\sqrt{ \|x\|^d \rho(\|x\|)+1}\, |K^\ast|(y,dx) \int\sqrt{ \|x\|^d \rho(\|x\|)+1}\, |K^\ast|(0,dx) \bigg] \alpha_{|\Phi|}(dy)  < \infty.
\end{equation}
Further, this condition is implied by
\begin{equation}\label{e: condition sq int simplified}
    \BE_0^{|\Phi|}\bigg[|\Phi|(B_\varepsilon)\bigg(\int\sqrt{ \|x\|^d \rho(\|x\|)+1}\, |K_0|(dx) \bigg)^2\bigg] < \infty.
\end{equation}
\end{proposition}

\begin{remark}\label{r:independent square-integrable}
    In Proposition~\ref{p:independent square-integrable}, if $\Phi$ and $K$ are independent, then~\eqref{e: condition sq int simplified} simplifies to
    \begin{equation}
        \BE\bigg[\bigg(\int\sqrt{ \|x\|^d \rho(\|x\|)+1}\, |K_0|(dx) \bigg)^2\bigg] < \infty.
    \end{equation}
    Note that if we additionally assume that the total variation of $K$ is uniformly bounded, this condition is neither strictly stronger nor strictly weaker than Condition~\ref{c:condition square-integrable Markov}.
\end{remark}

\begin{proof}[Proof of Proposition~\ref{p:independent square-integrable}]
For convenience, let
\begin{equation*}
    F(y) := \int \sqrt{\|x\|^d\rho(\|x\|) + 1}\, |K^\ast|(y,dx), \quad y\in\R^d.
\end{equation*}
Then, by the refined Campbell theorem~\eqref{e: refined Campbell 2}, with $\Psi_{d, \rho}$ defined as in Condition~\ref{c:condition square-integrable general},
\begin{align*}
    \BE\big[\Psi_{d,\rho}(B_\frac{\varepsilon}{2})^2\big]&= \BE\bigg[ \bigg(\int_{B_\frac{\varepsilon}{2}} F(y) \, |\Phi|(dy)\bigg)^2  \bigg] \nonumber\\
    &= \int (\I_{B_\frac{\varepsilon}{2}} \star \I_{B_\frac{\varepsilon}{2}})(y) \BE_{0,y}^{|\Phi|}[ F(y) F(0) ] \, \alpha_{|\Phi|}(dy) \nonumber\\
    &\leq \frac{\kappa_d \varepsilon^d}{2^d} \int_{B_{\varepsilon}} \BE_{0,y}^{|\Phi|}[ F(y)F(0)] \, \alpha_{|\Phi|}(dy)\nonumber\\
    &\leq \int (\I_{B_\varepsilon} \star \I_{B_\varepsilon})(y) \BE_{0,y}^{|\Phi|}[ F(y) F(0) ] \, \alpha_{|\Phi|}(dy)\nonumber \\
    &= \BE\big[\Psi_{d,\rho}(B_\varepsilon)^2\big],
\end{align*}
which proves the first part. The second part follows with the refined Campbell theorem~\eqref{e: refined Campbell 1} from
\begin{align*}
    \BE\big[\Psi_{d,\rho}(B_\frac{\varepsilon}{2})^2\big]&= \BE\bigg[\int_{B_\frac{\varepsilon}{2}}\int_{B_\frac{\varepsilon}{2}} F(y)F(z) \, |\Phi|(dy)\, |\Phi|(dz)  \bigg] \nonumber\\
    &\leq 2 \BE\bigg[\int_{B_\frac{\varepsilon}{2}}|\Phi|(B_\frac{\varepsilon}{2})F(y)^2\, |\Phi|(dy)\bigg]\nonumber\\
    &= 2\gamma \int_{B_\frac{\varepsilon}{2}} \BE_0^{|\Phi|}[|\Phi|(B_\frac{\varepsilon}{2}-y)F(0)^2]\, dy,\nonumber\\
    &\leq \frac{\kappa_d \gamma \varepsilon^d}{2^{d-1}}\BE_0^{|\Phi|}[|\Phi|(B_\varepsilon)F(0)^2],
\end{align*}
where $\gamma\in[0,\infty)$ is the intensity of $|\Phi|$.
\end{proof}

Finally, we can also consider how the preservation of local square integrability is related to the Wasserstein distance, as was done in~\cite{Dereudre_Flimmel_Huesmann_Leblé_2024}. 
\begin{definition}\label{d: Wasserstein distance}
    Suppose that $\Phi, \Psi$ are invariant random (nonnegative) measures with the same finite intensity. For $p\in [1,\infty)$ define
    \begin{align}
        W_p(\Phi, \Psi) := \inf\bigg\{\bigg(\BE_0^\Phi\bigg[\int\|x\|^p\, K_0(dx)\bigg]\bigg)^{\frac{1}{p}}: K \text{ invariant proba} &\text{bility kernel} \nonumber\\
        &\text{with } \Psi=K\Phi\bigg\},
    \end{align}
    and for $p=\infty$ define
    \begin{align}
        W_\infty(\Phi, \Psi) := \inf\bigg\{\sup_{\substack{\omega\in A\\ x\in\supp(K_0(\omega))}} \|x\| : K &\text{ invariant probability kernel}\nonumber\\ &\text{with } \Psi=K\Phi,A\in \cF, \BP_0^\Phi(A)=1\bigg\},
    \end{align}
    then $W_p(\Psi, \Phi)$ is called the $p$- or $L^p$-Wasserstein distance.
\end{definition}

See~\cite{Huesmann_2016, Erbar_Huesmann_Jalowy_Müller_2025} for details like the fact that $W_p$ is a metric for $p\in[1,\infty]$.
Just like in~\cite{Dereudre_Flimmel_Huesmann_Leblé_2024}, we obtain the following relation of local square integrability and the Wasserstein distance to the Lebesgue measure.

\begin{remark}
    Suppose that $\Phi$ is an invariant random (nonnegative) measure with finite intensity $\gamma>0$. If $W_d(\Phi, \gamma\lambda_d)<\infty$, then $\Phi$ is locally square-integrable. This fact follows directly from Theorem~\ref{t:main square-integrable} with the application of Condition~\ref{c:condition square-integrable Markov}.
\end{remark}

Since we have also treated the persistence of local square integrability under transport for two arbitrary invariant random measures in Theorem~\ref{t:main square-integrable}, the question arises if one can simply replace $\lambda_d$ in the above remark by an arbitrary invariant random measure which is locally square-integrable. However, it turns out that a formulation of the conditions of Theorem~\ref{t:main square-integrable} in terms of the Wasserstein distance is quite inefficient in general. 

\begin{proposition}\label{p: wasserstein conditions general}
    Suppose that $\Phi, \Psi$ are invariant random (nonnegative) measures with the same finite intensity. Let $p\in[0,\infty)$. Assume that $\BE[\Phi(B_\varepsilon)^{2+p}]<\infty$ for some $\varepsilon>0$, and assume that $W_{d+\frac{d}{p}+\vartheta}(\Phi, \Psi)<\infty$ for some $\vartheta>0$. Then $\Psi$ is also locally square-integrable. The same holds if $\Phi(B_\varepsilon)$ is uniformly bounded for some $\varepsilon>0$, e.g., if $\Phi$ is $\infty$-uniform; see Proposition~\ref{p: stealthy uniform bounds}, and if $W_{d}(\Phi, \Psi)<\infty$.
    \begin{proof}
        For $p=0$, the assertion follows from the fact that
        \begin{equation*}
            \BE[\Psi(B_1)^2]\leq \BE[\Phi(B_{1+W_\infty(\Phi,\Psi)})^2]<\infty.
        \end{equation*}
        Hence, from now on we assume that $p\in(0,\infty)$, that $\BE[\Phi(B_\varepsilon)^{2+p}]<\infty$ for some $\varepsilon>0$, and that $W_{d+\frac{1}{p}+\vartheta}(\Phi, \Psi)<\infty$ for some $\vartheta>0$. Then there exists an invariant probability kernel $K$ such that $\Psi=K\Phi$ and 
        \begin{equation*}
            \BE_0^\Phi\bigg[\int\|x\|^{d+\frac{1}{p}+\vartheta}\, K_0(dx)\bigg] < \infty.
        \end{equation*}
        Also note that $\BE[\Phi(B_\varepsilon)^{2+p}]$ for some $\varepsilon>0$ implies that the same holds for every $\varepsilon>0$.
        Now, we can apply the Hölder inequality twice to obtain
        \begin{align*}
            &\BE_0^\Phi\bigg[\Phi(B_1)\bigg(\int \sqrt{\|x\|^{d+\vartheta\frac{p}{1+p}}+1}\, K_0(dx)\bigg)^2\bigg]\nonumber\\
            &\quad\ \ \leq \BE_0^\Phi\big[\Phi(B_1)^{1+p}\big]^{\frac{1}{1+p}}\BE_0^\Phi\bigg[\bigg(\int \sqrt{\|x\|^{d+\vartheta\frac{p}{1+p}}+1}\, K_0(dx)\bigg)^{2\frac{1+p}{p}}\bigg]^{\frac{p}{1+p}}\nonumber\\
            &\quad\ \ \leq \BE_0^\Phi\big[\Phi(B_1)^{1+p}\big]^{\frac{1}{1+p}}\BE_0^\Phi\bigg[\int \big(\|x\|^{d+\vartheta\frac{p}{1+p}}+1\big)^{\frac{1+p}{p}}\, K_0(dx)\bigg]^{\frac{p}{1+p}}\nonumber\\
            &\quad\ \ \leq \bigg(\frac{\BE\big[\Phi(B_{2})^{2+p}\big]}{\BE\big[\Phi(B_{\kappa_d^{-1}})\big]}\bigg)^{\frac{1}{1+p}}\BE_0^\Phi\bigg[2^{\frac{1}{p}}\int \|x\|^{d+\frac{d}{p}+\vartheta}\, K_0(dx)+2^{\frac{1}{p}}\bigg]^{\frac{p}{1+p}}\nonumber\\
            &\quad\ \ < \infty.
        \end{align*}
        Then, by choosing $\rho(r):= r^{\vartheta\frac{p}{1+p}}, r\geq0$, we can see that~\eqref{e: condition sq int simplified} from Proposition~\ref{p:independent square-integrable} holds, which allows the application of Theorem~\ref{t:main square-integrable} to show that $\Psi$ is locally square-integrable.

        For the last case that $\Phi(B_\varepsilon)$ is uniformly bounded for some $\varepsilon>0$, the assertion follows from the fact that one can use the technique from Appendix~\ref{ss:Reduction to Lebesgue} to find an invariant kernel $\tilde{K}$ which has uniformly bounded total variation such that $\Psi=\tilde{K}\lambda_d$ additionally to the invariant probability kernel $K$ which fulfills $\Phi=K\Psi$ and Condition~\ref{c:condition square-integrable Markov}. Then $K\tilde{K}, \lambda_d$ also fulfill \linebreak $K\tilde{K}\lambda_d=\Psi$ and Condition~\ref{c:condition square-integrable Markov}, whereby Theorem~\ref{t:main square-integrable} can be applied.
    \end{proof}
\end{proposition}

Further, this inefficiency in the formulation with respect to the Wasserstein distance is not caused by the methods we apply but is systemic. One can consider a random variable $R\geq 1$ such that $\BE[R^{\frac{2+p}{p}d}\log(R)]<\infty$ and $\BE[R^{\frac{2+p}{p}d}\log(R)^2]=\infty$, a random variable $U$ with $U\sim\cU([0,R)^d)$ conditioned on $R$, and define
\begin{align}
    \Phi(dy) &:= \frac{R^{\frac{1+p}{p}d}}{\|R\lfloor \frac{y-U}{R}\rfloor+U-y\|^d+1}\, dy,\\
    K(y) &:= \delta_{R\lfloor \frac{y-U}{R}\rfloor+U}, \quad y\in\R^d,
\end{align}
and prove similar to Remark~\ref{r:square-integrable sharpness} that the conditions of Proposition~\ref{p: wasserstein conditions general} are basically sharp.

\subsection{Persistence of \textit{p}-uniformity under transport} \label{ss:Persistence of uniformity under transport}

For the persistence of $p$-uniformity for $p>0$, we need stronger integrability conditions compared to the persistence of the locally finite second moment, but these conditions are natural extensions of the preceding Conditions~\ref{c:condition square-integrable general} and~\ref{c:condition square-integrable Markov}. The first two only have increased exponents and the third replaces the polynomial moment with an exponential one. Let $p\in[-d,\infty)$.
\begin{condition}\label{c:condition square-integrable general with p}
    For some strongly log-dominating function $\rho$, it holds that the invariant random measure $\Psi_{d+\max(p,0),\rho}$ defined by
    \begin{equation}\label{e: remainder process definition with p}
        \Psi_{d+\max(p,0),\rho}(dy) := \bigg(\int \sqrt{\|x\|^{d+\max(p,0)}\rho(\|x\|) + 1}\, |K^\ast|(y, dx)\bigg)\, |\Phi|(dy), \quad y\in\R^d,
    \end{equation}
    is locally square-integrable.
\end{condition}
\begin{condition}\label{c:condition square-integrable Markov with p}
    Assume that $K$ has uniformly bounded total variation, that $K$ and $\Phi$ are independent, and that 
    \begin{equation}\label{e:condition square-integrable Markov with p}
        \BE\bigg[\int \|x\|^{d+\max(p,0)} \, |K_0|(dx) \bigg] < \infty.
    \end{equation}
\end{condition}
If we instead want to preserve $\infty$-uniformity with radius $\varepsilon>0$, we assume the following condition, which is the strongest of the three.
\begin{condition}\label{c:condition square-integrable stealthy}
    It holds that the invariant random measure $\Psi_{\infty, \varepsilon}$ defined by
    \begin{equation}\label{e: remainder process definition stealthy}
        \Psi_{\infty,\varepsilon}(dy) := \bigg(\int e^{\varepsilon\|x\|}\, |K^\ast|(y, dx)\bigg)\, |\Phi|(dy), \quad y\in\R^d,
    \end{equation}
    is locally square-integrable.
\end{condition}
While we do not provide a version of this condition that assumes independence of $\Phi$ and $K$ and uniformly bounded total variation of $K$ analogous to Condition~\ref{c:condition square-integrable Markov with p}, this would be possible. Further, note that analogous to how Conditions~\ref{c:condition square-integrable general with p} and~\ref{c:condition square-integrable Markov with p} are polynomial moment assumptions on the typical transport distance, Condition \ref{c:condition square-integrable stealthy} is an exponential moment assumption. We can see this by rewriting and simplifying them analogous to Proposition~\ref{p:independent square-integrable} and Remark~\ref{r:independent square-integrable}.

Still, these integrability conditions alone do not suffice to preserve $p$-uniformity if $d+p>2$. In the special cases that $d=2$ with $p>0$ or $d\geq 3$ with $p=0$, this fact was already discovered in~\cite{Dereudre_Flimmel_Huesmann_Leblé_2024}. We close this gap by introducing a Taylor expansion of the transport; see Theorem~\ref{t:main theorem formula}. The terms of this expansion contain the random complex measures $\Psi_q$ defined in~\eqref{e:Psi_q definition}. In the following theorem, we show how conditions on these can ensure the persistence of $p$-uniformity for any $p\in[-d,\infty]$ and in any dimension $d\in\N$. Note that we also treat the case that $K$ and $\Phi$ are dependent, which has not been considered before to the best of our knowledge.

\begin{theorem}\label{t:main theorem}
Suppose that $\Phi$ is a locally absolutely square-integrable invariant random complex measure, and let $K$ be an invariant transport kernel. We distinguish the following two cases. 
\begin{enumerate}[label=(\roman*)]
    \item Let $p\in[-d,\infty)$, and assume that either Condition~\ref{c:condition square-integrable general with p} or Condition~\ref{c:condition square-integrable Markov with p} holds. Further, assume that, for $q\in\N_0$ with $q<\max(\tfrac{d+p}{2},1)$, the random complex (tensor-valued) measure $\Psi_q$ defined by
    \begin{equation}\label{e:Psi_q definition}
        \Psi_q(dy) := \bigg(\int x^{\otimes q} \, K^\ast(y, dx)\bigg)\, \Phi(dy), \quad y\in\R^d, 
    \end{equation}
    is (beyond) $(p-2q)$-uniform. Then $K\Phi$ is (beyond) $p$-uniform.
    
    \item Let $\varepsilon>0$, assume that Condition~\ref{c:condition square-integrable stealthy} holds, and assume that $\Psi_q$ is $\infty$-uniform with radius $\varepsilon$ for every $q\in\N_0$. Then $K\Phi$ is also $\infty$-uniform with radius $\varepsilon$.
\end{enumerate}
\begin{proof}
    The proof mainly relies on the more general result of the following Theorem~\ref{t:main theorem formula}. Let $n:= \max(\lceil\tfrac{d+p}{2}-1\rceil, 0)$, i.e., $n=\lfloor\tfrac{d+p}{2}\rfloor$ if $\tfrac{d+p}{2}\notin\N$, and $n=\tfrac{d+p}{2}-1$ if $\tfrac{d+p}{2}\in\N$. For the first part, i.e., when Condition~\ref{c:condition square-integrable general with p} or~\ref{c:condition square-integrable Markov with p} holds for some $p\in[-d,\infty)$, Theorem~\ref{t:main theorem formula} implies, with convergence in $L^2$ as $r\to\infty$, that
    \begin{equation}
        K\Phi(f_r) = \sum_{q=0}^{n}\frac{1}{q!} \Psi_q\big(f_r^{(q)}\big) + o\Big(r^{\tfrac{d-p}{2}}\Big),
    \end{equation}
    where $f\in\cC_c^\infty(\R^d, [0,\infty))$, $f\neq0$, and $f_r^{(q)}$ denotes the (possibly tensor-valued) $q$-th derivative of $f_r$. As defined in \eqref{e: tensor valued measure}, we evaluate the integral of a tensor-valued complex measure of a tensor-valued function using the inner product. We obtatin that, as $r\to\infty$,
    \begin{align}
        \frac{\BV[K\Phi(f_r)]}{r^{d-p}} &\leq (n+2)^2\sum_{q=0}^{n}\frac{\BV\big[\Psi_q\big(f_r^{(q)}\big)\big]}{q!^2r^{d-p}}  + o(1)\nonumber\\
        &\leq (n+2)^2\sum_{q=0}^{n}\frac{\BV\big[\Psi_q\big(f^{(q)}(\tfrac{\cdot}{r})\big)\big]}{q!^2r^{d-(p+2q)}}  + o(1),
    \end{align}
    where we also used Remark~\ref{r: f_r formulas}. Now the assertion follows from Theorem~\ref{t:asymptotic variance test function} since we assumed that $\Psi_q$ is (beyond) $(p-2q)$-uniform for $q\in\{0,...,n\}$.

    The second part, i.e., where we assume that Condition~\ref{c:condition square-integrable stealthy} holds for some $\varepsilon>0$, follows from the second statement in Theorem~\ref{t:main theorem formula}. Let $f\in L^1(\R^d,[0,\infty))$ be a function such that $\hat{f}\in\cC^\infty$ with $\supp(\hat{f})\subseteq B_\varepsilon$ and $\hat{f}(k)\neq0$ for all $k\in B_{\varepsilon-\vartheta}$ and some $\vartheta\in(0, \varepsilon)$. Then, by Theorem~\ref{t:main theorem formula}, with convergence in $L^2$,
    \begin{equation}\label{e:stealthy taylor}
        K\Phi(f) = \sum_{q=0}^{\infty}\frac{1}{q!} \Psi_q\big(f^{(q)}\big).
    \end{equation}
    Note that the Paley-Wiener theorem ensures that $f\in\cC^\infty$. Now Theorem~\ref{t: stealthy characterization} implies that the RHS of~\eqref{e:stealthy taylor} is deterministic, as we assumed that $\Psi_q$ is $\infty$-uniform with radius $\varepsilon$ for every $q\in\N_0$, and that $\hat{f}$, and therefore also $\widehat{f^{(q)}}$ has support in $B_\varepsilon$ for every $q\in\N_0$. Hence, the LHS (left-hand side) must also be deterministic, which implies that $K\Phi$ is $\infty$-uniform with radius $\varepsilon-\vartheta$ by Theorem~\ref{t: stealthy characterization}, as $\hat{f}(k)\neq 0, k\in B_{\varepsilon-\vartheta}$. Since $\vartheta$ can be chosen arbitrarily small, we obtain that
    \begin{equation}
        \hat\beta_{K\Phi}(B_\varepsilon) = \lim_{\vartheta\to 0} \hat\beta_{K\Phi}(B_{\varepsilon-\vartheta}) = 0,
    \end{equation}
    and therefore also that $K\Phi$ is $\infty$-uniform with radius $\varepsilon$.
\end{proof}
\end{theorem}

Note that, for the considered $q\in\N_0$, $\Psi_q$ is always well defined, invariant, and locally absolutely square-integrable under the integrability assumptions made in Theorem~\ref{t:main theorem}. Further, for $q\geq 1$, one can ignore the atom of $\hat\beta_{\Psi_q}$ in $0$ when checking $(p-2q)$-uniformity if it exists. This comes from the fact that derivatives always integrate to $0$, whereby $\widehat{f^{(q)}}(0) = 0$. Moreover, if $K$ is a probability kernel, then $\Psi_0=\Phi$, which is why we say that the theorem handles the persistence of $p$-uniformity. Concerning $\infty$-uniformity, we can weaken the assumption of Condition~\ref{c:condition square-integrable stealthy} with respect to $\varepsilon>0$ to just assuming it with respect to some arbitrarily small $\tilde\varepsilon>0$. However, we omit the proof since it necessitates replacing the Taylor expansion of $f$ by the Taylor expansion of $e^{\langle k, \cdot\rangle}f$ for some $k\in\R^d$, which makes the proof harder to read.

The following Theorem~\ref{t:main theorem formula} introduces the Taylor expansion of the transport and lays the foundation of the proof of the preceding Theorem~\ref{t:main theorem}. It is also of independent interest, as it allows for a more detailed analysis of the (asymptotic) behavior of the destination of the transport. The strongest version assumes Condition~\ref{c:condition square-integrable stealthy} and allows for representing the transport as a Taylor series. It would imply all of the preceding results if one were satisfied with the assumption of a finite exponential moment of the typical transport distance. However, under the weaker assumption of Condition~\ref{c:condition square-integrable general with p} or~\ref{c:condition square-integrable Markov with p}, one can still control the error of the finite expansion, which suffices for the analysis of $p$-uniformity for $p<\infty$. Recall Definition~\ref{d: f_r} and \eqref{e: tensor valued measure}.

\begin{theorem}\label{t:main theorem formula}
    Suppose that $\Phi$ is a locally square-integrable invariant random complex measure, and let $K$ be an invariant transport kernel. We distinguish the following two cases. 
    \begin{enumerate}[label=(\roman*)] 
        \item Let $p\in[-d,\infty)$, and assume that either Condition~\ref{c:condition square-integrable general with p} or Condition~\ref{c:condition square-integrable Markov with p} holds. Further, suppose that $f:\R^d\to\C$ is compactly supported and $\lfloor\tfrac{d+\max(p,0)}{2}+1\rfloor$-times continuously differentiable. Then, with convergence in $L^2$ as $r\to\infty$, we have
        \begin{equation}\label{e: decomposition polynomial moment}
            K\Phi(f_r) = \sum_{q=0}^{\big\lfloor\tfrac{d+p}{2}\big\rfloor}\frac{1}{q!} \Psi_q\big(f_r^{(q)}\big) + o\Big(r^{\tfrac{d-p}{2}}\Big)
        \end{equation}
        with $\Psi_q$ defined as in~\eqref{e:Psi_q definition} in Theorem~\ref{t:main theorem} and $f_r^{(q)}=r^{-q}f^{(q)}(\tfrac{\cdot}{r})$ being the (possibly tensor-valued) $q$-th derivative of $f_r$. If $\frac{d+p}{2}\in\N$, one can disregard the last summand.
        \item Assume that Condition~\ref{c:condition square-integrable stealthy} holds for some $\varepsilon>0$, and assume that $f\in L^1(\R^d, \C)$ is chosen such that $\hat{f}\in \cC_c^{d+1}$ with $\supp(\hat{f})\subseteq B_\varepsilon$. Then, with limits taken in $L^2$, we obtain
        \begin{equation}\label{e: decomposition exponential moment}
            K\Phi(f) = \sum_{q=0}^{\infty}\frac{1}{q!} \Psi_q\big(f^{(q)}\big).
        \end{equation}
    \end{enumerate}
    \begin{proof}
        We begin with the first part, where we assume that Condition~\ref{c:condition square-integrable general with p} or~\ref{c:condition square-integrable Markov with p} holds. We choose $f:\R^d\to\C$ such that it is compactly supported and $\lfloor\tfrac{d+\max(p,0)}{2}+1\rfloor$-times continuously differentiable. It can be easily seen that we can assume $ |f| \leq \I_{B_1}$ without loss of generality. Further, we define $\gamma:=\tfrac{d+\max(p,0)}{2}$. By Lemma~\ref{l:taylor}, there exists a mapping $R:(0,\infty)\times\R^d\times\R^d\to \C$ and a constant $c_1>0$ such that for all $r>0,x,y\in\R^d$,
        \begin{align}
            f_r(y+x) &= \sum_{k=0}^{\lfloor \gamma \rfloor} \frac{1}{q!} f_r^{(q)}(y) x^q + R(r,y,x),\\
            |R(r,y,x)| &\leq c_1 \frac{\|x\|^\gamma}{r^\gamma} g(r,x),\label{e:taylor2 proof}
        \end{align}
        where $g$ is bounded by $1$ and $g(r,x) \to 0$ as $r\to\infty$ for all $x\in\R^d$. Further, for all $r>0$, $x\in\R^d,y\in B_r^c$
        \begin{equation}
            R(r,y,x) = f_r(y+x).\label{e:taylor3 proof}
        \end{equation}
        An application of this decomposition of $f_r$ yields that
        \begin{align}
            K\Phi(f_r) &= \iint f_r(x+y)\, K^\ast(y,dx)\, \Phi(dy) \nonumber\\
            &= \iint \sum_{q=0}^{\lfloor \gamma\rfloor} \frac{1}{q!} f_r^{(q)}(y) x^{\otimes q} + R(r,y,x)\, K^\ast(y,dx)\, \Phi(dy) \nonumber\\
            &= \sum_{q=0}^{\lfloor \gamma \rfloor} \frac{1}{q!}\Psi_q(f_r^{(q)})  + \iint R(r,y,x)\, K^\ast(y,dx)\, \Phi(dy),\quad r>0.\label{e: incomplete decomposition polynomial moment}
        \end{align}
        We see that the sum on the RHS of~\eqref{e: decomposition polynomial moment} is already included in the sum on the right-most side of~\eqref{e: incomplete decomposition polynomial moment}. It remains to show that the remaining part goes to $0$ in $L^2$ as $r\to\infty$. By the triangle inequality, the convergence of each remaining summand can be treated individually.

        For $q\in\N$ with $\tfrac{d+p}{2}\leq q \leq \tfrac{d+\max(p,0)}{2}$, as $r\to\infty$, we obtain that
        \begin{equation*}
            \frac{1}{r^{d-p}} \BE\big[\Psi_q\big(f_r^{(q)}\big)^2\big] \leq \frac{1}{r^{2d}} \BE\big[\Psi_q\big(f^{(q)}(\tfrac{\cdot}{r})\big)^2\big] \to 0,
        \end{equation*}
        where we applied Remark~\ref{r: f_r formulas}, Proposition~\ref{p:asymptotic variance test function 0}, that derivatives always integrate to $0$, and that the Paley-Wiener theorem implies that all components of $f^{(q)}$ are Fourier-smooth with an exponent grater than $-d$. This derivation also proves that one can drop the last summand in~\eqref{e: decomposition polynomial moment} if $\frac{d+p}{2}\in\N$.

        The last remaining summand in~\eqref{e: incomplete decomposition polynomial moment} can be bounded further by splitting the outer integral. For $r>0$ it holds that
        \begin{align}
            &\bigg|\iint R(r,y,x)\, K^\ast(y,dx)\, \Phi(dy)\bigg| \nonumber\\
            &\quad \ \ \leq \int_{B_{2r}}\int |R(r,y,x)|\, |K^\ast|(y,dx)\, |\Phi|(dy) + \int_{B_{2r}^c}\int |R(r,y,x)|\, |K^\ast|(y,dx)\, |\Phi|(dy)\nonumber\\
            &\quad \ \ \leq \frac{c_1}{r^\gamma}\int_{B_{2r}}\int \|x\|^\gamma g(r, x)\, |K^\ast|(y,dx)\, |\Phi|(dy) + \int_{B_{2r}^c} |K|(y,B_r)\, |\Phi|(dy),
        \end{align}
        where we applied~\eqref{e:taylor2 proof} and~\eqref{e:taylor3 proof}.
        The first term can be treated using Lemma~\ref{l:condition reduction lebesgue}, as it shows that there exists an invariant nonnegative transport kernel $\tilde{K}$ such that
        \begin{equation*}
            \BE\bigg[\bigg(\int \|x\|^{\gamma}\, \tilde{K}_0(dx) \bigg)^2\bigg] < \infty
        \end{equation*}
        and
        \begin{equation*}
            \int_{B_{1}}\int \|x\|^\gamma g(r, x)\, |K^\ast|(y,dx)\, |\Phi|(dy) = \int \|x\|^\gamma g(r, x)\, \tilde{K}_0(dx).
        \end{equation*}
        Hence,
        \begin{align*}
            &\frac{1}{r^{d-p}}\BE\bigg[\bigg(\frac{c_1}{r^\gamma}\int_{B_{2r}}\int \|x\|^\gamma g(r, x)\, |K^\ast|(y,dx)\, |\Phi|(dy)\bigg)^2\bigg] \nonumber\\
            &\quad\ \ \leq c_1^2c_3\BE\bigg[\bigg(\int_{B_{1}}\int \|x\|^\gamma g(r, x)\, |K^\ast|(y,dx)\, |\Phi|(dy)\bigg)^2\bigg] \nonumber\\
            &\quad\ \ = c_1^2c_3\BE\bigg[\bigg(\int \|x\|^\gamma g(r, x)\, \tilde{K}_0(dx)\bigg)^2\bigg]\nonumber\\
            &\quad\ \ \xrightarrow{r\to\infty} 0,
        \end{align*}
        where $c_3$ is a dimensional constant that depends on how efficiently one can cover a larger ball by multiple smaller balls, and where we applied the theorem of dominated convergence in the last step, as $g(r, x) \leq 1$ and $g(r, x) \to 0$ as $r\to\infty$ for $r>0,x\in\R^d$. For the second term, we can directly apply Lemma~\ref{l:remainder2case1} or~\ref{l:remainder2case2} depending on the condition we assume. Hence, the first part of the assertion has been established, and we continue with the second part.

        Assume that Condition~\ref{c:condition square-integrable stealthy} holds with respect to $\varepsilon>0$, and suppose that $f\in L^1(\R^d, \C)$ such that $\hat{f}\in \cC_c^{d+1}$ with $\supp(\hat{f})\subseteq B_\varepsilon$ instead of what was assumed before. Then, as $B_\varepsilon$ is open, there also exists a $\varepsilon_0< \varepsilon$ such that $\supp(\hat{f})\subseteq B_{\varepsilon_0}$. By Lemma~\ref{l: taylor series}, we know that $f$ can be expanded as a Taylor series with infinite convergence radius, i.e., 
        \begin{equation}
            f(y+x) = \sum_{q=0}^\infty \frac{1}{q!} f^{(q)}(y)x^{\otimes q},\quad x,y\in\R^d,
        \end{equation}
        and at the same time, we have a bound on the $L^1$-norm of a modification of the derivatives, as there exists a $c_4>0$ such that
        \begin{equation}\label{e: l1 bound derivative}
            \int \sup_{u\in B_1}\|f^{(q)}(x+u)\|\, dx \leq c_4(q+1)^{d+1} \varepsilon_0^q,\quad q\in\N_0.
        \end{equation}
        We can directly insert the Taylor series into $K\Phi(f)$ to obtain a formula which is already very similar to~\eqref{e: decomposition exponential moment}. It holds that
        \begin{align}
            K\Phi(f) &= \iint f(x+y)\, K^\ast(y,dx)\, \Phi(dy) \nonumber\\
            &= \iint \sum_{q=0}^{\infty} \frac{1}{q!} f^{(q)}(y) x^{\otimes q}\, K^\ast(y,dx)\, \Phi(dy).
        \end{align}
        Hence, it only remains to show that one can change the order of integration. By Fubini's theorem, this works if the object is absolutely integrable. In parallel, we show that the series in~\eqref{e: decomposition exponential moment} converges absolutely in $L^2$, ensuring not only path-wise but also $L^2$-convergence since
        \begin{align}\label{e: abs conv bound}
            &\sqrt{\BE\bigg[\bigg(\iint \sum_{q=0}^{\infty} \frac{1}{q!} \|f^{(q)}(y)\| \|x\|^q\, |K^\ast|(y,dx)\, |\Phi|(dy)\bigg)^2\bigg]}\nonumber\\
            &\quad \ \ \leq \sum_{q=0}^{\infty} \frac{1}{q!}\sqrt{\BE\bigg[\bigg(\iint \|f^{(q)}(y)\| \|x\|^q\, |K^\ast|(y,dx)\, |\Phi|(dy)\bigg)^2\bigg]}.
        \end{align}
        Further, by Lemmas~\ref{l: simplification to lebesgue} and~\ref{l:condition reduction lebesgue}, there is an invariant nonnegative transport kernel $L$ such that $L(y)$ has support on $\overline{B_1}+y$ with $|\Phi|=L\lambda_d$, and such that the invariant nonnegative transport kernel $\tilde{K}:=|K|L$ satisfies
        \begin{equation}
            \BE\bigg[\bigg(\int e^{\varepsilon\|x\|}\, \tilde{K}_0(dx)\bigg)^2\bigg] < \infty.
        \end{equation}
        Hence, for $q\in\N_0$,
        \begin{align*}
            &\BE\bigg[\bigg(\iint \|f^{(q)}(y)\| \|x\|^q\, |K^\ast|(y,dx)\, |\Phi|(dy)\bigg)^2\bigg]\nonumber\\
            &\quad\ \ = \BE\bigg[\bigg(\iint \|f^{(q)}(z)\| \|x\|^q\, |K^\ast|(z,dx)\, L(y, dz)\, dy\bigg)^2\bigg]\nonumber\\
            &\quad\ \ \leq \BE\bigg[\bigg(\int \sup_{u\in B_1}\|f^{(q)}(y+u)\| \int \|x\|^q\, |K^\ast|(z,dx)\, L(y, dz) \, dy\bigg)^2\bigg]\nonumber\\
            &\quad\ \ = \BE\bigg[\bigg(\int\sup_{u\in B_1}\|f^{(q)}(y+u)\| \int \|x\|^q\, \tilde{K}^\ast(y,dx) \, dy\bigg)^2\bigg].
        \end{align*}
        We can further bound this quantity using the Cauchy-Schwarz inequality and~\eqref{e: l1 bound derivative}, whereby
        \begin{align*}
            &\BE\bigg[\bigg(\iint \|f^{(q)}(y)\| \|x\|^q\, |K^\ast|(y,dx)\, |\Phi|(dy)\bigg)^2\bigg]\nonumber\\
            &\quad\ \ \leq \iint \sup_{u,v\in B_1}\|f^{(q)}(y+u)\| \|f^{(q)}(z+v)\| \BE\bigg[\int \|x\|^q\, \tilde{K}^\ast(y,dx) \int \|x\|^q\, \tilde{K}^\ast(z,dx)\bigg] \, dy\,dz\nonumber\\
            &\quad\ \ \leq \bigg(\int \sup_{u\in B_1}\|f^{(q)}(y+u)\|\, dy\bigg)^2\BE\bigg[\bigg(\int \|x\|^q\, \tilde{K}_0^\ast(dx)\bigg)^2\bigg]\nonumber\\
            &\quad\ \ \leq c_4^2(q+1)^{2d+2} \varepsilon_0^{2q}\BE\bigg[\bigg(\int \|x\|^q\, \tilde{K}_0^\ast(dx)\bigg)^2\bigg].
        \end{align*}
        Applying this bound to the RHS of~\eqref{e: abs conv bound} together with another application of the Cauchy-Schwarz inequality yields
        \begin{align*}
            &\sum_{q=0}^{\infty} \frac{1}{q!}\sqrt{\BE\bigg[\bigg(\iint \|f^{(q)}(y)\| \|x\|^q\, |K^\ast|(y,dx)\, |\Phi|(dy)\bigg)^2\bigg]}\nonumber\\
            &\quad \ \ \leq \sum_{q=0}^{\infty} \frac{1}{q!}c_4(q+1)^{d+1} \varepsilon_0^{q} \sqrt{\BE\bigg[\bigg(\int \|x\|^q\, \tilde{K}_0^\ast(dx)\bigg)^2\bigg]}\nonumber\\
            &\quad \ \ = c_4\sum_{q=0}^{\infty}\frac{1}{q+1} \sqrt{\BE\bigg[\bigg(\int (q+1)^{d+2}\frac{ (\varepsilon_0\|x\|)^q}{q!}\, \tilde{K}_0^\ast(dx)\bigg)^2\bigg]}\nonumber\\
            &\quad \ \ \leq c_4\sqrt{\sum_{q=0}^{\infty}\frac{1}{(q+1)^2}} \sqrt{\BE\bigg[\sum_{q=0}^\infty\bigg(\int (q+1)^{d+2}\frac{ (\varepsilon_0\|x\|)^q}{q!}\, \tilde{K}_0^\ast(dx)\bigg)^2\bigg]}.
        \end{align*}
        Here, we can bound the sum of squares by the square of the sum, and obtain
        \begin{align*}
            &\sum_{q=0}^{\infty} \frac{1}{q!}\sqrt{\BE\bigg[\bigg(\iint \|f^{(q)}(y)\| \|x\|^q\, |K^\ast|(y,dx)\, |\Phi|(dy)\bigg)^2\bigg]}\nonumber\\
            &\quad \ \ \leq c_4\sqrt{\sum_{q=0}^{\infty}\frac{1}{(q+1)^2}} \sqrt{\BE\bigg[\sum_{q=0}^\infty\bigg(\int (q+1)^{d+2}\frac{ (\varepsilon_0\|x\|)^q}{q!}\, \tilde{K}_0^\ast(dx)\bigg)^2\bigg]}\nonumber\\
            &\quad \ \ \leq c_4\sqrt{\sum_{q=0}^{\infty}\frac{1}{(q+1)^2}} \sqrt{\BE\bigg[\bigg(\int \sum_{q=0}^\infty(q+1)^{d+2}\frac{ (\varepsilon_0\|x\|)^q}{q!}\, \tilde{K}_0^\ast(dx)\bigg)^2\bigg]}\nonumber\\
            &\quad \ \ \leq c_4c_5\sqrt{\sum_{q=0}^{\infty}\frac{1}{(q+1)^2}} \sqrt{\BE\bigg[\bigg(\int e^{\varepsilon\|x\|}\, \tilde{K}_0^\ast(dx)\bigg)^2\bigg]}\nonumber\\
            &\quad\ \ < \infty,
        \end{align*}
        where we find the constant $c_5>0$ by the fact that $\varepsilon_0< \varepsilon$ and since the exponential function asymptotically dominates any polynomial. This convergence was what was left to show and concludes the proof.
    \end{proof}
\end{theorem}

If one would not only consider invariant random complex-valued measures but also tempered distributions like in~\cite{Björklund_Byléhn_2025}, one could even move the differentiation operator from the test function to the invariant tempered distribution. Then one could, e.g., replace~\eqref{e: decomposition exponential moment} by an equation of the form
\begin{equation}
    K\Phi(f) = \sum_{q=0}^\infty \frac{1}{q!}\tilde{\Psi}_q(f).
\end{equation}
The next theorem is a slight improvement of Theorem~\ref{t:main theorem}, and just like it, it is derived from the previous Theorem~\ref{t:main theorem formula}. It not only allows us to ensure $p$-uniformity of the destination, and therefore the order of the asymptotic variance, but also allows us to obtain the coefficient of the leading term. Similarly, one could also derive a central limit theorem, since the destination $K\Phi$ follows a central limit theorem if and only if the dominant term of the expansion $\Psi_{q_0}$ follows a central limit theorem.

\begin{theorem}\label{t: main theorem rate}
Suppose that $\Phi$ is a locally square-integrable invariant random complex measure, and let $K$ be an invariant transport kernel. We distinguish the following two cases. 
\begin{enumerate}[label=(\roman*)] 
    \item Let $p\in[-d,\infty)$, and assume that either Condition~\ref{c:condition square-integrable general with p} or Condition~\ref{c:condition square-integrable Markov with p} holds. For $q\in\N_0$ with $q<\max(\frac{d+p}{2}, 1)$, define $\Psi_q$ like in Theorem~\ref{t:main theorem}. Assume that there is a $q_0<\max(\frac{d+p}{2}, 1)$ such that $\Psi_{q_0}$ is $(p-2q_0)$-uniform, and else that $\Psi_q$ is beyond \linebreak $(p-2q_0)$-uniform for all $q<\max(\frac{d+p}{2}, 1)$ with $q\neq q_0$. Further, assume that $f$ is compactly supported and Fourier-smooth with exponent $2(d+1)+\lfloor\max(p,0)\rfloor+\vartheta$ for some $\vartheta>0$. Then, if either limit exists,
    \begin{equation}\label{e: exact variance asymptotic}
    \lim_{r\to\infty} \frac{\BV[K\Phi(f_r)]}{r^{d-p}} = \lim_{r\to\infty} \frac{\BV\big[\Psi_{q_0}\big(f_r^{(q_0)}\big)\big]}{(q_0!)^2 r^{d-p}},
    \end{equation}
    $K\Phi$ is $p$-uniform, and if the limit is not $0$, then $K\Phi$ is solely $p$-uniform.
    
    \item Let $\varepsilon>0, q_0\in\N_0$, assume that Condition~\ref{c:condition square-integrable stealthy} holds, and assume that $\Psi_q$ is $\infty$-uniform with radius $\varepsilon$ for $q\in\N_0\setminus\{q_0\}$. Further, suppose that $f\in L^1(\R^d, \C)$ such that \linebreak $\supp(\hat{f})\subseteq \overline{B_\varepsilon}$. Then,
        \begin{equation}\label{e: exact variance stealthy}
            \BV\big[K\Phi(f)\big] = \frac{1}{q_0!^2} \BV\big[\Psi_{q_0}\big(f^{(q_0)}\big)\big].
        \end{equation}
    Hence, if $\Psi_{q_0}$ is $\infty$-uniform with radius $\varepsilon_0\in(0,\varepsilon)$, then $K\Phi$ also is $\infty$-uniform with radius $\varepsilon_0$, and if the above variance is not $0$ for a sequence $(f_n)_{n\in\N}$ such that $f_n\in L^1(\R^d, \C)$ and $\supp(\hat{f_n})\subseteq B_{\varepsilon_0+\frac{1}{n}}$ for $n\in\N$, then $K\Phi$ is $\infty$-uniform with maximal radius $\varepsilon_0$.
\end{enumerate}
\begin{proof}
We start with the first part and let $p\in[-d, \infty)$, and $q_0\in\N_0$ with $q_0<\max(\tfrac{d+p}{2},1)$. Further, we assume that either Condition~\ref{c:condition square-integrable general with p} or Condition~\ref{c:condition square-integrable Markov with p} holds, and that $\Psi_{q_0}$ is \linebreak $(p-2q_0)$-uniform, and else $\Psi_q$ is beyond $(p-2q)$-uniform for all $q<\max(\frac{d+p}{2}, 1)$ with \linebreak $q\neq q_0$. Finally, we suppose that $f$ is compactly supported and Fourier-smooth with exponent $2(d+1)+\lfloor\max(p,0)\rfloor+\vartheta$ for some $\vartheta>0$, and define $f_r:=f(\frac{\cdot}{r})$ for $r>0$. Note that Lemma~\ref{l: fourier smooth differentiation} yields that $f$ is $\lfloor \tfrac{d+\max(p,0)}{2}+1\rfloor$-times continuously differentiable and that the derivatives are sufficiently Fourier-smooth for the following conclusions. By Theorem~\ref{t:main theorem formula}, we know that, with convergence in $L^2$,
\begin{equation}
    K\Phi(f_r) = \sum_{q=0}^{\big\lfloor\tfrac{d+p}{2}\big\rfloor}\frac{1}{q!} \Psi_q\big(f_r^{(q)}\big) + o\Big(r^{\tfrac{d-p}{2}}\Big),
\end{equation}
where the last summand can be dropped if $\tfrac{d+p}{2}\in\N$. Then, using the Cauchy-Schwarz inequality like in the proof of Theorem~\ref{t:main theorem}, we can show that
\begin{equation}
    K\Phi(f_r) = \frac{1}{q_0!} \Psi_{q_0}\big(f_r^{(q_0)}\big) + o\Big(r^{\tfrac{d-p}{2}}\Big),
\end{equation}
which implies~\eqref{e: exact variance asymptotic} if either limit exists. Further, $K\Phi$ is $p$-uniform by an application of Theorem~\ref{t:main theorem}. Moreover, the fact that a non-zero limit implies that $K\Phi$ is not beyond $p$-uniform is a direct implication from Theorem~\ref{t:asymptotic variance test function}. Actually, divergence would yield the same result as a non-zero limit.

In the second part, we alternatively let $\varepsilon>0, q_0\in\N_0$, assume that Condition~\ref{c:condition square-integrable stealthy} holds, and assume that $\Psi_q$ is $\infty$-uniform with radius $\varepsilon$ for $q\in\N_0\setminus\{q_0\}$. Further, we suppose that $f\in L^1(\R^d, \C)$ such that $\hat{f}\in\cC_c^{d+1}$ and $\supp(\hat{f})\subseteq B_\varepsilon$. For more general $f$, the assertion then follows from a density argument, as~\eqref{e: exact variance stealthy} does not involve a limit. Proceeding with an application of Theorem~\ref{t:main theorem formula}, we obtain that
\begin{equation}
    K\Phi(f) = \sum_{q=0}^{\infty}\frac{1}{q!} \Psi_q\big(f^{(q)}\big),
\end{equation}
where we can apply that, by Theorem~\ref{t: stealthy characterization} and the $\infty$-uniformity assumption for $q\neq q_0$, the summands on the RHS are deterministic for $q\neq q_0$. Hence,~\eqref{e: exact variance stealthy} holds. Further, if for some $n\in\N$, we have $\supp(\hat{f})\subseteq \overline{B_{\varepsilon_0+\frac{1}{n}}}$, then
\begin{equation*}
    \BV\big[K\Phi(f)\big] = (2\pi)^{-d}\int |\hat{f}(k)|^2\, \hat\beta_{K\Phi}(dk) \leq (2\pi)^{-d}\|\hat{f}\|_\infty^2 \hat\beta(B_{\varepsilon_0+\frac{1}{n}}).
\end{equation*}
Thus if the variance on the left-most side is not $0$, $\hat\beta(B_{\varepsilon_0+\frac{1}{n}})$ cannot be $0$ either, and if one finds such a function $f$ for every $n\in\N$, then $K\Phi$ cannot be $\infty$-uniform with radius $\varepsilon_1$ for any $\varepsilon_1>\varepsilon_0$, which concludes the proof.
\end{proof}
\end{theorem}

Under the conditions of the preceding theorem, one can directly analyze the asymptotic behavior of $K\Phi$ by relating it to the leading order term of the approximation, which comes from some $\Psi_{q_0}$. For $q_0=0$, in particular when $p\leq-d+2$, one can even derive that $K\Phi$ is (beyond) $p$-uniform iff $\Psi_{q_0}$ is (beyond) $p$-uniform. However, for higher $q_0$, derivatives get involved. Since derivatives always integrate to $0$, there can be some cancellation such that $K\Phi$ actually is \linebreak $\tilde{p}$-uniform for a higher $\tilde{p}$ than $\Psi_{q_0}$. Additionally, if $d\geq2$, $\Psi_{q_0}$ is tensor-valued and the components can be correlated such that the combined variance on the RHS of~\eqref{e: exact variance asymptotic} is of lower order than the variances of the individual components of $\Psi_{q_0}$. However, if $d=1$, and if $\Psi_{q_0}$ is (pseudo-) ergodic, then Proposition~\ref{p:asymptotic variance test function zero integral} yields that, for every $p\in(-d, \infty)$, $K\Phi$ is (beyond) $p$-uniform iff $\Psi_{q_0}$ is (beyond $p$-uniform. For $\infty$-uniformity, the same problems apply when $q_0\geq 1$ and $d\geq 2$, but still, if $q_0=0$, or if $d=1$ and $\Psi_{q_0}$ is (pseudo-) ergodic, then $K\Phi$ is $\infty$-uniform with radius $\varepsilon_0$ iff $\Psi_{q_0}$ is $\infty$-uniform with radius $\varepsilon_0$.

Having finished the presentation of the central results of this paper, we proceed with giving two approaches for checking the assumptions of these theorems. We do not elaborate on the Conditions~\ref{c:condition square-integrable general with p} and~\ref{c:condition square-integrable stealthy}, as the simplifications are analogous to those made in Proposition~\ref{p:independent square-integrable} and Remark~\ref{r:independent square-integrable}. The first approach for checking the remaining assumptions works under a certain degree of decorrelation and is only applicable to $p$-uniformity with $p\leq 2$. It can be compared to an approach in~\cite{Klatt_Last_Lotz_Yogeshwaran_2025}, where a certain kind of $\beta$-mixing was assumed. None of the approaches is strictly stronger than the other, but the approach here yields $p$-uniformity and not only hyperuniformity, and while we have an additional moment condition here, the decorrelation assumption strictly is weaker, which can be seen in Proposition~\ref{p: alpha mixing condition}. There, it is also shown that our condition is strictly weaker than the one used in~\cite{Flimmel_2025}.

\begin{proposition}\label{p:decorrelation setting}
    Let $p\in[-d, 2]$. Suppose that $\Phi$ is a locally absolutely square-integrable invariant random measure, and suppose that $K$ is an invariant transport kernel. Assume that Condition~\ref{c:condition square-integrable general with p} or Condition~\ref{c:condition square-integrable Markov with p} is fulfilled. Further, assume that 
    \begin{align}\label{e: decorrelation assumption}
        \int \Bigg\| \BE_{0,y}^{\Phi}&\bigg[\int x^{\otimes q}\, K^\ast(y,dx)\times \overline{\int x^{\otimes q}\, K^\ast(0,dx)}\bigg] \nonumber\\
        &- \BE_0^{\Phi}\bigg[\int x^{\otimes q}\, K^\ast(0,dx)\bigg]\times \overline{\BE_0^{\Phi}\bigg[\int x^{\otimes q}\, K^\ast(0,dx)\bigg]} \Bigg\| \|y\|^{p-2q} \, \alpha_{\Phi}(dy) < \infty
    \end{align}
    for $q\in\N_0$ with $q<\max(\frac{d+p}{2}, 1)$, and where $\times$ denotes the component-wise product. If $\Phi$ is $p$-uniform and $p\leq 0$, then $K\Phi$ is also $p$-uniform. For $p<0$ the same holds for beyond $p$-uniformity. Further, if $K_0(\R^d)$ is deterministic, e.g., if $K$ is a probability kernel, $\Phi$ is $p$-uniform, and $p\leq 2$, then $K\Phi$ is also $p$-uniform. For $p<2$, the same holds for beyond $p$-uniformity.
\end{proposition}
\begin{remark}
    In the setting of  Proposition~\ref{p:decorrelation setting}, if $K$ and $\Phi$ are independent, then~\eqref{e: decorrelation assumption} simplifies to
    \begin{equation}
        \int \bigg\| \BC\bigg[\int x^{\otimes q}\, K^\ast(y,dx), \int x^{\otimes q}\, K^\ast(0,dx)\bigg]\bigg\| \|y\|^{p-2q} \, \alpha_{\Phi}(dy) < \infty,
    \end{equation}
    where the covariance is taken component-wise.
    In this case, the assertions even hold if $\Phi$ is allowed to be complex-valued.
\end{remark}
    \begin{proof}[Proof of Proposition~\ref{p:decorrelation setting}.]
        As Condition~\ref{c:condition square-integrable general with p} or Condition~\ref{c:condition square-integrable Markov with p} are fulfilled, it remains to check $(p-2q)$-uniformity of $\Psi_q$ for $q\in\N_0, q < \max(\tfrac{d+p}{2}, 1)$ to apply Theorem~\ref{t:main theorem}. Hence, for $q\in\N_0, q < \max(\tfrac{d+p}{2}, 1)$, we show that $Z_{q}\cdot\Phi$ is $(p-2q)$-uniform, where
        \begin{equation}
            Z_{q}(y) := \int x^{\otimes q} \, K^\ast(y, dx),\quad y\in\R^d.
        \end{equation}
        For $d\geq2, q\geq1$, note that the invariant complex random measure $Z_q\cdot \Phi$ is tensor valued, but by the definition of $p$-uniformity, it suffices to analyze its components.
        If $p\leq 0$ or if $q\geq 1$, we have $p-2q\leq 0$, and $(p-2q)$-uniformity can be derived from Proposition~\ref{p: uniformity with random field density} with an application of~\eqref{e: decorrelation assumption}. Further, if $K_0(\R^d)$ is deterministic, then $Z_{0}$ is also deterministic, and $Z_0\cdot\Phi$ inherits $p$-uniformity from $\Phi$ by Proposition~\ref{p: uniform linear combination}. The assertions concerning beyond $p$-uniformity follow similarly, which concludes the proof.
    \end{proof}

As is the case with independent and identically distributed (iid) perturbations (see Example~\ref{ex: general point cluster}), spatial decorrelation of the transport alone does not preserve $p$-uniformity for $p>2$. The following approach uses a different technique to circumvent this limitation. While it is a lot more restrictive that just decorrelation, it allows for the construction of many interesting examples in Subsections~\ref{ss: pp from invariant partitions} and~\ref{ss: pp from averaging sets}.

\begin{proposition}\label{p:moment setting}
    Let $p\in[-d, \infty)$. Suppose that $\Phi$ is a locally absolutely square-integrable invariant random complex measure, and suppose that $K,L$ are invariant transport kernels. Assume that $K,\Phi$ and $L,\Phi$ either both fulfill Condition~\ref{c:condition square-integrable general with p} or both fulfill Condition~\ref{c:condition square-integrable Markov with p}.
    Further, assume that 
    \begin{equation}\label{e:markov difference 0}
        \int x^{\otimes q}\, (K_y-L_y)(dx) = 0,\quad q\in\N_0, q<\max(\tfrac{d+p}{2}, 1), y\in\R^d.
    \end{equation}
    Then $K\Phi$ is (beyond) $p$-uniform iff $L\Phi$ is (beyond) $p$-uniform.

    In particular, if $L$ is deterministic, e.g., $L_0:=\BE_0^\Phi[K_0]$ when $\Phi$ is nonnegative, then $K\Phi$ is (beyond) $p$-uniform if $\Phi$ is (beyond) $p$-uniform. The converse holds if $L_0(\R^d)\neq 0$.
    \begin{proof}
        By assumption, the invariant transport kernel $\tilde{K}:=K-L$, together with $\Phi$, fulfills Condition~\ref{c:condition square-integrable general with p} or Condition~\ref{c:condition square-integrable Markov with p}. Further, using the Cauchy-Schwarz inequality
        \begin{equation}
            \BV[K\Phi(f)] = \BV[\tilde{K}\Phi(f) + L\Phi(f)] \leq 2 \BV[\tilde{K}\Phi(f)] + 2 \BV[L\Phi(f)]
        \end{equation}
        for some bounded $f:\R^d\to\C$ with compact support, and vice versa. Hence, if $\tilde{K}\Phi$ is (beyond) $p$-uniform, then $K\Phi$ is (beyond) $p$-uniform iff $L\Phi$ is (beyond) $p$-uniform. We define $\Psi_q$ as in~\eqref{e:Psi_q definition} with respect to $\tilde{K}$ and $\Phi$. Then, by~\eqref{e:markov difference 0}, $\Psi_q=0$ for $q\in\N_0, q < \max(\frac{d+p}{2}, 1)$, whereby $\Psi_q$ trivially is $\infty$-uniform. Hence, an application of Theorem~\ref{t:main theorem} concludes the first part of the proof. The second part then follows from the first, as if $L$ is deterministic, then $L\Phi$ inherits its $p$-uniformity from $\Phi$, and if further $L_0(\R^d)\neq0$, then $\Phi$ also inherits $p$-uniformity from $L\Phi$. See~\cite[Lemma 3.9]{Klatt_Last_Lotz_Yogeshwaran_2025} for the case that $L_0$ is a probability measure, which can be directly generalized to the case that $L_0(\R^d)\neq0$.
    \end{proof}
\end{proposition}

In Subsection~\ref{ss: pp from invariant partitions}, we appliy this comparison approach to a non-deterministic $L$ which fulfills $L\Phi=\lambda_d$. Hence, $L\Phi$ is trivially $\infty$-uniform. On the other hand, in Subsection~\ref{ss: pp from averaging sets} we choose a deterministic $L$. There, the conditions are even improved by combining the approach from Proposition~\ref{p:moment setting} with the decorrelation approach from Proposition~\ref{p:decorrelation setting}.

We conclude this subsection by showing that the Conditions~\ref{c:condition square-integrable general with p} and~\ref{c:condition square-integrable Markov with p} we assume in this subsection and, in particular, the main Theorem~\ref{t:main theorem} are basically sharp. To simplify we only consider the case $p\leq0$, i.e., show that Conditions~\ref{c:condition square-integrable general} and~\ref{c:condition square-integrable Markov} are sharp. For the more general case, one needs to replace the power of $d$ by $d+p$ in some places.

In the first remark, we show that the strongly log-dominating function $\rho$ is necessary for Condition~\ref{c:condition square-integrable general} to be sufficient in Theorem~\ref{t:main square-integrable} and that it cannot be dropped like in the more specialized  Condition~\ref{c:condition square-integrable Markov}. 

\begin{remark}\label{r:square-integrable sharpness}
In Theorem~\ref{t:main square-integrable}, if $K$ is not both of bounded total variation and independent of $\Phi$, then a variation of Condition~\ref{c:condition square-integrable general} does not ensure local square integrability of $K\Phi$ if one replaces the strongly log-dominating function $\rho:[0,\infty)\to[0,\infty)$ by one which satisfies
\begin{equation}\label{e:not strongly log-dominating}
    \int_0^\infty \frac{1}{r\rho(r) + 1} \, dr = \infty.
\end{equation}
While the weakened condition ensures the local square integrability of $\Psi_q$ in Theorem~\ref{t:main theorem} for $q\leq\frac{d}{2}$, it does not allow for the control of the remainder term in the approximation from Theorem~\ref{t:main theorem formula} which is used. Note that we even show that the slightly stronger condition 
\begin{equation}
    \BE\bigg[\int_{B_\varepsilon}|K|(y, \R^d)\,|\Phi|(dy) \int_{B_\varepsilon} \int \|x\|^d\rho(\|x\|)+1\, |K^\ast|(y, dx)\, |\Phi|(dy)\bigg] < \infty,
\end{equation}
for some $\varepsilon>0$,
does note suffice with respect to such a $\rho$.

We give an example for $K$ and $\Phi$, for which $K\Phi$ is not locally square-integrable, the variation of Condition~\ref{c:condition square-integrable general with p} is fulfilled for a function $\rho:[0,\infty)\to[0,\infty)$ that satisfies~\eqref{e:not strongly log-dominating}, and where $K$ does not have bounded total variation but is independent of $\Phi$. To get the other counterexample, one can simply choose $\tilde{K}(x):=\frac{K(x)}{K(x,\R^d)}$, $x\in\R^d$, and $\tilde{\Phi}(dx):= K(x, \R^d) \cdot \Phi(dx)$. Note that~\eqref{e:condition fulfilled} implies that this $\tilde\Phi$ is locally absolutely square-integrable.

Let $R$ be a nonnegative random variable and choose $U$ such that, conditioned on $R$,\linebreak $U\sim\cU([0,R)^d)$. Now choose $\Phi:=\lambda_d$ and 
\begin{equation}
    K(y):= \frac{ R^{\frac{d}{2}}}{\|R\lfloor \frac{y-U}{R}\rfloor+U-y\|^d\rho(\|R\lfloor \frac{y-U}{R}\rfloor+U-y\|)+1}\delta_{R\lfloor \frac{y-U}{R}\rfloor+U}, \quad y\in\R^d. 
\end{equation}
Then $K$ and $\Phi$ are both invariant under a suitable shift that fulfills $R(\theta_z\omega)=R(\omega)$, \linebreak $U(\theta_z\omega) = R(\omega)\lfloor \tfrac{U(\omega)-z}{R(\omega)}\rfloor, z\in\R^d, \omega\in\Omega$. Further, by construction, $K\Phi$ is concentrated on the randomly stretched lattice $R\Z^d+U$ and has a spatially constant but random density
\begin{equation}
    R^{\frac{d}{2}} \underbrace{\int_{[0,R)^d} \frac{1}{\|y\|^d\rho(\|y\|)+1} \, dy}_{M:=} =  R^{\frac{d}{2}} M
\end{equation}
with respect to it. From now onward, we choose $R$ such that $R\geq 1$, $\BE[M]<\infty$ and $\BE[M^2]=\infty$. This choice is possible since
\begin{equation*}
   \int_{[0,\infty)^d} \frac{1}{\|y\|^d\rho(\|y\|)+1} \, dy = \frac{d\kappa_d}{2^d}\int_0^\infty \frac{r^{d-1}}{r^d\rho(r)+1} \, dr \geq \frac{d\kappa_d}{2^{d+1}}\int_1^\infty \frac{1}{r(\rho(r)+1)} \, dr = \infty.
\end{equation*}
Now, using $R\geq 1$, we can calculate that
\begin{equation}
    \alpha_{K\Phi}(\{0\}) = \BE\bigg[\int_{[0,1)^d} K\Phi(\{y\}) \, K\Phi(dy) \bigg] = \BE\bigg[\frac{(R^\frac{d}{2} M)^2}{R^d}\bigg] = \BE[M^2] = \infty,
\end{equation}
which implies that $K\Phi$ is not locally square-integrable. Additionally, as $\Phi=\lambda_d$, it is deterministic and hence trivially independent of $K$. It remains to check the variant of Condition~\ref{c:condition square-integrable general}. By choice of $K$, we have
\begin{equation*}
    \int \|x\|^d\rho(\|x\|) + 1 \, |K^\ast|(y, dx) = R^\frac{d}{2}, \quad y\in\R^d.
\end{equation*}
Thus, for any $\varepsilon>0$,
\begin{align}
    \BE[|\Psi_{d, \rho}(B_\varepsilon)|^2] &\leq \BE\bigg[ \int_{B_\varepsilon}\,  |K^\ast|(y,\R^d)\, |\Phi|(dy) \int_{B_\varepsilon} \int \|x\|^d\rho(\|x\|) + 1\, |K^\ast|(y,dx) \, |\Phi|(dy)  \bigg] \nonumber\\ \label{e:condition fulfilled}
    &= \lambda_d(B_\varepsilon)^2\BE\bigg[\frac{R^\frac{d}{2} M}{R^d} R^\frac{d}{2} \bigg] = \lambda_d(B_\varepsilon)^2\BE[M] < \infty.
\end{align}
\end{remark}

Coming to the second Condition~\ref{c:condition square-integrable Markov}, the following remark essentially shows that the test function $x\mapsto\|x\|^d$ from the condition cannot be lowered in a meaningful way. This counterexample can also be found in~\cite[Proposition 4.1]{Dereudre_Flimmel_Huesmann_Leblé_2024}.

\begin{remark}\label{r:ind square-integrable sharpness}
In Theorem~\ref{t:main square-integrable}, if $K$ has bounded total variation and is independent of $\Phi$, then Condition~\ref{c:condition square-integrable Markov} is sharp in the sense that there exists sufficiently large class of probability kernels such that $K\lambda_d$ is locally square-integrable iff Condition~\ref{c:condition square-integrable Markov} holds.

Let $R$ be a nonnegative random variable and choose $U$ such that, conditioned on $R$,\linebreak $U\sim\cU([0,R)^d)$. Now choose $K(y):= \delta_{R\lfloor \frac{y-U}{R}\rfloor + U}$ for $y\in\R^d$. Like in the previous remark, the shift can be defined such that $K$ is invariant. Then we get that
\begin{equation}
    \BE\bigg[\int \|x\|^d \, K(0, dx)\bigg] = \BE[\|R\lfloor \tfrac{y-U}{R}\rfloor + U\|^d] = \BE\bigg[\frac{1}{R^d}\int_{[0,R)^d} \|x\|^d \, dx\bigg] = \frac{\kappa_d c_d}{2^{d+1}} \BE[R^d]
\end{equation}
for some dimensional constant $c_d$ with $1\leq c_d \leq d^d$. Additionally, by construction $K\lambda_d$ is concentrated on the randomly stretched lattice $R\Z^d+U$ and has the spatially constant but random density $R^d$ with respect to it. Hence, one can calculate that 
\begin{equation*}
    \kappa_d\BE[R^d\I\{R>2\}] \leq \BE[K\lambda_d(B_1)^2] \leq \kappa_d\BE[R^d] + 4^d.
\end{equation*}
This shows that both local square integrability of $K\lambda_d$ and Condition~\ref{c:condition square-integrable Markov} are equivalent to $\BE[R^d]<\infty$.
\end{remark}

In the case that $K$ is a displacement kernel independent of $\Phi$, Condition~\ref{c:condition square-integrable Markov} is equivalent to the condition
\begin{equation}\label{e:condition square-integrable general wrong?}
    \BE\bigg[\bigg(\int \|x\|^\frac{d}{2} + 1 \, K_0(dx)\bigg)^2\bigg]<\infty,
\end{equation}
which is, up to the missing $\rho$, the same condition as one could derive from Condition~\ref{c:condition square-integrable general}; see Remark~\ref{r:independent square-integrable}.
Still, for a general probability kernel $K$ independent of $\Phi$, it remains open if the weaker condition~\eqref{e:condition square-integrable general wrong?} is sufficient for local square integrability.

\subsection{Persistence of \textit{p}-uniformity under spatial interpolation}\label{ss:Persistence of uniformity under spatial interpolation}

It is clear that the assumptions made in Theorem~\ref{t:main theorem} are sufficient but not necessary. However, if one restricts the class of transports a little further, the same assumptions guarantee that the transport $K$ of the source $\Phi$ also fulfills a much stronger property regarding a spatial inter- and even extrapolation of the source $\Phi$ and destination $K\Phi$. In this setting, the assumptions come a lot closer to being necessary.

Here, we always assume that $\Phi$ is a locally square-integrable invariant complex random measure, and that $K$ is an invariant probability kernel. Further, we assume that almost surely, for $\Phi$-almost all $x\in\R^d$, $K_x$ has bounded support, i.e.,
\begin{equation}
    \BP(\Phi(\{x\in\R^d:\diam(\supp(K_x))=\infty)=0)=1.
\end{equation}
Importantly, the bound on the support does not need to be uniform, whereby almost all practical examples of transports, like most that we construct in Section~\ref{s:applications}, fulfill this condition. If we drop this assumption, the following results about $p$-uniformity still hold, but the interpretation as a continuous interpolation only holds in a weak sense.

For $\gamma\in\R\setminus\{0\}$, we define the transport kernels $K^{(\gamma)}$ by
\begin{equation}
    K^{(\gamma)\ast}_x(B) := K_x^\ast(\tfrac{1}{\gamma}B), \quad x\in\R^d, B\in\cB^d,
\end{equation}
and $K^{(0)\ast}:=\delta_0$, whereby $K^{(1)}=K$. The assumptions made guarantee that almost surely for $\Phi$-almost every $x\in\R^d$, the mass of $K_x^{(\gamma)}$ flows continuously with respect to $\gamma\in\R$. Hence, the random measures $K^{(\gamma)}\Phi$ continuously interpolate between $\Phi$ and $K\Phi$ for $\gamma\in[0,1]$, and extrapolate for $\gamma\in\R\setminus[0,1]$. It turns out that these modified transports also preserve\linebreak $p$-uniformity.

\begin{theorem}
    Let $p\in[-d, \infty]$ and $\varepsilon>0$. Assume that $K$ and $\Phi$ fulfill the general assumptions in this subsection and the assumptions made in Theorem~\ref{t:main theorem}. In particular, suppose that $K$ is a probability kernel and that $\Phi$ is (beyond) $p$-uniform. Then $K^{(\gamma)}\Phi$ is (beyond) $p$-uniform for every $\gamma\in\R$. In the case $p=\infty$,  $K^{(\gamma)}\Phi$ is $\infty$-uniform with radius $\varepsilon$ for every $\gamma\in\R$ such that Condition~\ref{c:condition square-integrable stealthy} is fulfilled with respect to $\tilde\varepsilon:=|\gamma|\varepsilon$, in particular, for any $\gamma\in[-1,1]$.
    \begin{proof}
        All assertions are basically a consequence of the fact that
        \begin{equation}\label{e: K gamma conversion}
            \int g(x)\, K^{(\gamma)\ast}(y, dx) = \int g(\gamma x)\, K^\ast(y, dx) 
        \end{equation}
        for $\gamma\in\R, y\in\R^d$ and a suitable function $g:\R^d\to\C$. It implies that Condition~\ref{c:condition square-integrable general with p} or~\ref{c:condition square-integrable Markov with p} holds with respect to $K^{(\gamma)}$ if it does with respect to $K$, where one also has to apply Lemma~\ref{l:stronlgy log-dominating modification} to deal with $\rho$. In Condition~\ref{c:condition square-integrable stealthy}, there is only a slight change of the exponent. It is also implied that $\Psi_q$ only differs by a factor of $\gamma^q$ if one defines it with respect to $K^{(\gamma)}$ instead of $K$, which does not affect $(p-2q)$-uniformity or the radius if $p=\infty$. In combination, this allows for an application of Theorem~\ref{t:main theorem} with $K$ replaced by $K^{(\gamma)}$ to conclude the proof.
    \end{proof}
\end{theorem}
Actually, in the case that $p=\infty$, we can drop the assumption that Condition~\ref{c:condition square-integrable stealthy} is fulfilled with respect to $\tilde\varepsilon:=|\gamma|\varepsilon$. Therefore, the assertion always holds for every $\gamma\in\R$. This fact relies on a slight generalization of Theorem~\ref{t:main theorem}, which we do not prove here. For details, see the comment right after the proof of Theorem~\ref{t:main theorem}.

Next, we see that preservation of $p$-uniformity not only under transport but also under spatial inter- or extrapolation comes very close to making the assumptions in Theorem~\ref{t:main theorem} necessary. One could say that one gets as close as one could reasonably expect.

\begin{theorem}
    Let $p\in[-d, \infty)$, assume that $K$ and $\Phi$ fulfill the general assumptions in this subsection, and assume that Condition~\ref{c:condition square-integrable general with p} or Condition~\ref{c:condition square-integrable Markov with p} holds. Further, define\linebreak $N:=\max(\lceil\frac{d+p}{2}\rceil, 1)$, and assume that there are $\gamma_1,..., \gamma_N\in\R$ with $\gamma_n\neq\gamma_m$ for\linebreak $n,m\in\{1,...,N\},n\neq m$ such that $K^{(\gamma_n)}\Phi$ is $p$-uniform for every $n\in\{1,..,N\}$. Then $K^{(\gamma)}\Phi$ is $p$-uniform for every $\gamma\in\R$, and as $r\to \infty$,
    \begin{equation}\label{e: Psi_q limit var}
        \frac{\BV\big[\Psi_q(f^{(q)}(\frac{\cdot}{r}))\big]}{r^d} = O\big(r^{-(p-2q)}\big), \quad q\in\N_0, q<\max(\tfrac{d+p}{2}, 1),
    \end{equation}
    with $\Psi_q$ defined as in Theorem~\ref{t:main theorem} and $f$ chosen as in the first part of Theorem~\ref{t:main theorem formula}. The same holds with respect to beyond $p$-uniformity, where $O$ gets replaced by $o$ in~\eqref{e: Psi_q limit var}.

    Alternatively, let $\varepsilon>0$, and assume that Condition~\ref{c:condition square-integrable stealthy} is fulfilled. Further, assume that there is a sequence $(\gamma_n)_{n\in\N}$ in $[-1, 1]\setminus\{0\}$ such that $\gamma_n\to 0$ as $n\to\infty$, and such that $K^{(\gamma_n)}\Phi$ is $\infty$-uniform with radius $\varepsilon$ for every $n\in\N$. Then $K^{(\gamma)}\Phi$ is $\infty$-uniform with radius $\varepsilon$ for every $\gamma\in\R$ such that Condition~\ref{c:condition square-integrable stealthy} is fulfilled with respect to $\tilde\varepsilon:=|\gamma|\varepsilon$, in particular $\gamma\in[-1,1]$. Additionally, $\Psi_q(f^{(q)})$ is deterministic for every $q\in\N_0$ with $f$ chosen as in the second part of Theorem~\ref{t:main theorem formula}.
    \begin{proof}
        Let $p\in[-d, \infty)$, and assume that Condition~\ref{c:condition square-integrable general with p} or Condition~\ref{c:condition square-integrable Markov with p} holds. Then, by Theorem~\ref{t:main theorem formula},~\eqref{e: K gamma conversion}, and with convergence in $L^2$, we have that
        \begin{equation}\label{e: K gamma f_r decomposition}
        K^{(\gamma)}\Phi(f_r) = \sum_{q=0}^{N-1}\frac{\gamma^q}{q!} \Psi_q\big(f_r^{(q)}\big) + o\Big(r^{\tfrac{d-p}{2}}\Big),\quad \gamma\in\R,
    \end{equation}
    as $r\to\infty$. The remainder term also fulfills some regularity with respect to $\gamma$, but we do not need it here. Hence, if $K^{(\gamma_n)}\Phi$ is $p$-uniform for every $n\in\{1,..,N\}$ for some $\gamma_n\in\R$, then
    \begin{equation}
        \begin{pmatrix}
            1 & \gamma_1 & \gamma_1^2 & \cdots & \gamma_1^{N-1}\\
            \vdots&\vdots&\vdots&&\vdots\\
            1 & \gamma_N & \gamma_N^2 & \cdots & \gamma_N^{N-1}
        \end{pmatrix}
        \begin{pmatrix}
            \Psi_0(f_r)\\
            \Psi_1(f_r')\\
            \vdots\\
            \frac{1}{(N-1)!}\Psi_{N-1}\big(f_r^{(N-1)}\big)
        \end{pmatrix}
        = O\Big(r^{\tfrac{d-p}{2}}\Big)
    \end{equation}
    as $r\to\infty$. Since the matrix is a Vandermonde matrix, and therefore invertible, as the $\gamma_n$ are assumed to be pairwise different for $n\in\{1,..., N\}$ (see, e.g.,~\cite[Example 0.9.11]{Horn_Johnson_1985}), this relation also implies that
    \begin{equation}
        \Psi_{n}\big(f_r^{(n)}\big) = O\Big(r^{\tfrac{d-p}{2}}\Big), \quad n\in\{0,...,N-1\}.
    \end{equation}
    Together with~\eqref{e: K gamma f_r decomposition} and Theorem~\ref{t:asymptotic variance test function}, we therefore obtain that $K^{(\gamma)}\Phi$ is $p$-uniform for every $\gamma\in\R$. Moreover, Remark~\ref{r: f_r formulas} yields~\eqref{e: Psi_q limit var}. Under the assumption of beyond $p$-uniformity, the same proof yields beyond $p$-uniformity of $K^{(\gamma)}\Phi$ for $\gamma\in\R$, and $O$ simply gets replaced by $o$, which concludes the first part of the proof.

    Now, instead let $\tilde{\varepsilon}\geq\varepsilon>0$, and assume that Condition~\ref{c:condition square-integrable stealthy} holds with respect to $\tilde\varepsilon$. Further, suppose that $f$ is chosen as in the second part of Theorem~\ref{t:main theorem formula}. Then, by Theorem~\ref{t:main theorem formula} and~\eqref{e: K gamma conversion}, with convergence in $L^2$,
    \begin{equation}\label{e: K gamma sequence}
        K^{(\gamma)}\Phi(f) = \sum_{q=0}^{\infty}\frac{\gamma^q}{q!} \Psi_q\big(f^{(q)}\big),\quad \gamma\in \big[-\tfrac{\tilde\varepsilon}{\varepsilon}, \tfrac{\tilde\varepsilon}{\varepsilon}\big].
    \end{equation}
    Now, assume that there is a sequence $(\gamma_n)_{n\in\N}$ in $[-1, 1]\setminus\{0\}$ such that $\gamma_n\to 0$ as $n\to\infty$, and such that $K^{(\gamma_n)}\Phi$ is $\infty$-uniform with radius $\varepsilon$ for every $n\in\N$. Then, 
    \begin{equation*}
        \frac{1}{m!}\Psi_m(f^{(m)}) = \gamma_n^{-m}\bigg(K^{(\gamma_n)}\Phi(f) - \sum_{q=0}^{m-1} \frac{\gamma_n^q}{q!}\Psi_q(f^{(q)}) \bigg) + \sum_{q=m+1}^\infty \frac{\gamma_n^{q-m}}{q!}\Psi_q(f^{(q)}), \quad m\in\N_0.
    \end{equation*}
    The Variance of the sum on the right converges to $0$ as $n\to\infty$, as $\gamma_n\to 0$, and because it was shown in the proof of Theorem~\ref{t:main theorem formula} that the sequence converges absolutely in $L^2$. Further, Theorem~\ref{t: stealthy characterization} implies that $K^{(\gamma_n)}\Phi(f)$ is deterministic for every $n\in\N$. An induction over $m\in\N_0$ then also yields that $\Psi_m(f^{(m)})$ also is deterministic for every $m\in\N_0$. The remaining part now follows from Theorem~\ref{t: stealthy characterization} and~\eqref{e: K gamma sequence}.
    \end{proof}
\end{theorem}

In the special case that $d=1$ and that $K$ and $\Phi$ are ergodic, these implications are actually equivalent to the $(p-2q)$-uniformity assumptions made in Theorem~\ref{t:main theorem}; see Proposition~\ref{p:asymptotic variance test function zero integral}. For higher dimensions, the picture is more complicated. In conclusion, the assumptions made in Theorem~\ref{t:main theorem} seem to almost be necessary for a transport to preserve $p$-uniformity from the source to the destination and along the way. For an even closer relation in both directions for $d\geq 2$, one probably needs to more directly analyze the second-moment properties of invariant tensor-valued random complex measures.  

\section{Applications}\label{s:applications}
\subsection{Point processes from invariant partitions}\label{ss: pp from invariant partitions}

In this subsection, we apply our theorems to random partitions of space, We thus obtain point processes that are $p$-uniform for arbitrarily high $p$. We begin by formally defining what we mean by a random partition of space.

\begin{definition}
Suppose that $\Phi$ is an invariant locally square-integrable simple point process and suppose that $\tau$ is an invariant allocation. We then call $(\Phi, \tau)$ an \textit{invariant random partition} if $\tau(x)\in\Phi$ for all $x\in\R^d$. Further, for $x\in\Phi$, we call $C(x):=\tau^{-1}(x)$ a \textit{cell} of $(\Phi,\tau)$. Finally, $V_0:= C(\tau(0))$ is called the \textit{zero-cell}.

Additionally, if $\lambda_d(C(x))=\gamma^{-1}$ for all $x\in\Phi$, where $\gamma$ is the intensity of $\Phi$, then we call $(\Phi,\tau)$ a \textit{random fair partition}; see~\cite{Klatt_Last_Lotz_Yogeshwaran_2025}.
\end{definition}

The underlying principle for the following theorem, from which all examples in this subsection are derived, was already discovered in physics in~\cite{Gabrielli_Joyce_Torquato_2008}. Our contribution is that we can apply our results from Section~\ref{s:main results} to control the effect of correlations in the partition. Therefore, we can formulate the statement rigorously and in greater generality, which allows for applications that were probably not in view when the idea was first stated. 

\begin{theorem}\label{t: partition persistence}
Let $p\in[-d,\infty)$. Suppose that $(\Phi,\tau)$ is an invariant random partition, and suppose that $\Psi, \Gamma$ are invariant locally square-integrable random measures. Further, assume that, for some $\vartheta>0$,
\begin{equation}\label{e:partition integrability}
    \BE_0^\Phi\bigg[\Phi(B_\vartheta)\bigg(\int_{C(0)} \sqrt{\|x\|^{d+\max(p,0)}\rho(\|x\|)+1}\,(\Psi+\Gamma)(dx)\bigg)^2\bigg] <\infty,
\end{equation}
and that
\begin{equation}\label{e:partition difference 0}
    \int_{V_0} x^{\otimes q}\, (\Psi-\Gamma)(dx) = 0,\quad q\in\N_0,q<\max(\tfrac{d+p}{2},1).
\end{equation}
Then $\Psi$ is (beyond) $p$-uniform iff $\Gamma$ is (beyond) $p$-uniform.
\begin{proof}
    At its core, this Theorem is derived from Proposition~\ref{p:moment setting}. For this matter, we define the transport kernels $K,L$ as
    \begin{equation}
        K(x, B) := \Psi(B\cap C(x)), \quad L(x, B) := \Gamma(B\cap C(x)),\quad x\in\R^d, B\in\cB^d.
    \end{equation}
    Then we can derive from~\eqref{e:partition difference 0} that~\eqref{e:markov difference 0} from Proposition~\ref{p:moment setting} holds. Further, the proof that~\eqref{e:partition integrability} implies Condition~\ref{c:condition square-integrable general with p} with respect to $K,\Phi$ and $L,\Phi$ is analogous to the proof of Proposition~\ref{p:independent square-integrable}. Hence, the assertion is implied by Proposition~\ref{p:moment setting}.
\end{proof}
\end{theorem}

\begin{remark}\label{r: partition persistence}
In the setting of Theorem~\ref{t: partition persistence}, if $(\Phi,\tau)$ is a random fair partition and $\Gamma=\gamma\lambda_d$ for some $\gamma>0$, one can replace condition~\eqref{e:partition integrability} by 
\begin{equation}\label{e:partition integrability markov}
    \BE\bigg[\int_{V_0}\|x\|^{d+\max(p,0)}\, (\Psi+\lambda_d)(dx)\bigg] < \infty.
\end{equation}

The main change in the proof of Theorem~\ref{t: partition persistence} is that we replace the source $\Phi$ by $\gamma\lambda_d$ and define $K,L$ instead by
\begin{equation}
    K(x, B) := \frac{\Psi(B\cap C(\tau(x)))}{\lambda_d(C(\tau(x)))}, \quad L(x, B) := \frac{\lambda_d(B\cap C(\tau(x)))}{\lambda_d(C(\tau(x)))},\quad x\in\R^d, B\in\cB^d.
\end{equation}
Then $L$ is a probability Kernel, and by~\eqref{e:partition difference 0} for $q=0$, $K$ is also a probability kernel. Both are trivially independent of the source $\gamma\lambda_d$. Therefore, we can use Condition~\ref{c:condition square-integrable Markov with p} instead of Condition~\ref{c:condition square-integrable general with p}. We derive it from~\eqref{e:partition integrability markov} using the fact that, in a fair partition, the zero-cell $V_0$ behaves under $\BP$ just like the typical cell, i.e., $C(0)$ under $\BP_0^\Phi$. 

\end{remark}

In the following, we construct point processes which are (beyond) $p$-uniform, where a high $p$ is of particular interest. Practically, we start with a random fair partition. Then we construct a point process by placing an equal number of points in each cell such that they satisfy the moment condition~\eqref{e:partition difference 0}, where $\Psi$ is the point process and $\Gamma$ is a multiple of $\lambda_d$. This condition is equivalent to the points, together with equal weights, forming a quadrature formula of degree $\max(\lceil\frac{d+p}{2}\rceil-1, 0)$ for their respective cell. Hence, the existence of such quadrature formulae is necessary to construct these point processes. While the existence of regular quadrature formulas can be easily derived from Carathéodory's theorem, the restriction to equal-weight quadrature formulas, which are also called averaging sets, makes the problem a lot harder. Still, the construction of averaging sets of rank $0$ or $1$ (or even $2$) is always possible, which gives rise to the two following examples. The first allows for the arbitrary placement of one point per cell.

These methods are also applicable to non-fair random partitions if one passes down the volume of the cells to the mass placed within; see also~\cite{Kim_Torquato_2019,Klatt_Last_Lotz_Yogeshwaran_2025}. However, one thus obtains a random measure and generally not a point process.

\begin{example}\label{ex:fair tiling any point}
Suppose that $(\Phi, \tau)$ is a random fair partition which satisfies
\begin{equation}
    \BE[\diam(V_0)^{\max(2,d)}] < \infty.
\end{equation}
Further, suppose that $F:\cB_b^d\to \bM^1(\R^d)$ is measurable and satisfies
\begin{align}
    F(B)(B)&= 1,\\
    F(B+x)&= F(B)(\cdot-x),\quad x\in\R^d, B\in\cB_b^d.
\end{align}
Now, choose $(Y_x)_{x\in\Phi}$ such that, conditioned on $(\Phi,\tau)$, $Y_x\sim F(C(x))$ for $x\in\Phi$ and $\{Y_x:x\in\Phi\}$ are independent. Then the point process
\begin{equation}
    \Psi:=\sum_{x\in\Phi} \delta_{Y_x}
\end{equation}
is beyond $(2-d)$-uniform.

This assertion follows from Theorem~\ref{t: partition persistence} and Remark~\ref{r: partition persistence} with $p:=2-d$ and $\gamma:=\lambda_d(V_0)^{-1}$. To this end, it is essential that the cell size $\lambda_d(V_0)$ is deterministic since the partition is fair.
\end{example}
For $d=1$, the preceding example allows for the construction of point processes which are beyond $1$-uniform, and for $d=2$, beyond $0$-uniform. However, Proposition~\ref{p: hyperfluctuating points in lattice} shows that the bound is sharp in general. Hence, for $d\geq 3$ it is possible that $\Psi$ is not $0$-uniform; see also~\cite{Dereudre_Flimmel_Huesmann_Leblé_2024}. Still, there is a simple way to construct hyperuniform point processes in dimensions $3$ and $4$ as shown in the following example, which restricts itself to placing the points in the centroids.

\begin{example}\label{ex:fair tiling centroid}
Suppose that $(\Phi, \tau)$ is a random fair partition which satisfies
\begin{equation}
    \BE[\diam(V_0)^{\max(4,d)}] < \infty.
\end{equation}
Now, choose 
\begin{equation}
    Y_x := \frac{1}{\lambda_d(C(x))} \int_{C(x)} z \, dz,\quad x\in\Phi.
\end{equation}
Then the point process
\begin{equation}
    \Psi:=\sum_{x\in\Phi} \delta_{Y_x}
\end{equation}
is beyond $(4-d)$-uniform.

The argument is analogous to that in Example~\ref{ex:fair tiling any point} with the alternate choice $p:=4-d$.
\end{example}

We can apply the preceding two examples to the so-called pinwheel tiling; see~\cite{Radin_1994} for the construction procedure. 

\begin{example}
    The pinwheel tiling is a fair partition of $\R^2$. If we apply Example~\ref{ex:fair tiling any point}, we can see that the arbitrary placement of one point in each cell results in a beyond $0$-uniform, i.e., hyperuniform, point process. This result supports the findings in~\cite{Sgrignuoli_Dal_Negro_2021}. Moreover, if the points are each placed independently and uniformly in the cells like in~\cite{Gabrielli_Jancovici_Joyce_Lebowitz_Pietronero_Sylos_Labini_2003}, we can even combine techniques with the hyperuniformerer from~\cite[Example 4.7]{Klatt_Last_Lotz_Yogeshwaran_2025} to obtain that the constructed point process is solely $2$-uniform, in particular, class I hyperuniform. Finally, the placement in the centroids even yields a beyond $2$-uniform point process by Example~\ref{ex:fair tiling centroid}. The maximal $p$ for which the centroids are $p$-uniform has not been found yet.
\end{example}

To guarantee that the procedure produces point processes which are beyond $0$-uniform in any dimension or to guarantee $p$-uniformity for higher $p$, one needs more points per cell. Proposition~\ref{p: hyperfluctuating centroids} shows that, even in a random fair partition with highly regular cells, the centroids can form a point process which is not $0$-uniform for $d\geq5$. This extension necessitates a slightly higher degree of regularity on the side of the random fair partition to guarantee the existence of a point process which fits the moment condition~\eqref{e:partition difference 0}. In particular, we require the volume of the convex hull of the cells to be uniformly bounded. If the cells of a fair partition are convex, they immediately satisfy this condition, and as $\lambda_d(\conv(V_0))\leq \kappa_d\diam(V_0)^d$, cells with uniformly bounded diameter also suffice. In Proposition~\ref{p:averaging set counterexample}, it can be seen why a restriction of this kind is necessary.

\begin{example}\label{ex:fair tiling high degree}
Let $C> 0, n,m\in\N$. Suppose that $(\Phi,\tau)$ is a random fair partition such that $\lambda_d(\partial V_0)=0$ and $\lambda_d(\conv(V_0))\leq C$ almost surely, and that
\begin{equation}\label{e:fair tiling high degree moment condition}
    \BE[\diam(V_0)^{\max(2m,d)}] < \infty.
\end{equation}
Let $\gamma$ be the intensity of $\Phi$, and define the set which includes all possible cells
\begin{equation}
    A_{\gamma, C}:=\{B\in\cB_b^d: \lambda_d(\partial B)=0, \lambda_d(B)=\gamma^{-1}, \lambda_d(\conv(V_0))\leq C\}.
\end{equation}
Further, suppose that $F:A_{\gamma, C}\to \bM^1(\R^{nd})$ is measurable and satisfies
\begin{align}\label{e: moment fixing}
    F(B)\bigg(\bigg\{(x_1,...x_n)\in B^n&:\frac{1}{n}\sum_{k=1}^n x_k^{\otimes q} = \gamma\int_B z^{\otimes q}\, dz, q\in\{0,...,m-1\}\bigg\}\bigg)= 1,\\
    F(B+x)&= F(B)(\cdot-x),\quad x\in\R^d, B\in A_{\gamma, C},
\end{align}
where $(y_1,...,y_n)-x := (y_1-x,...,y_n-x)$ for $ x,y_1,...,y_n\in\R^d$.

Now, chose $(Y_{x,k})_{x\in\Phi, k\in\{1,...,n\}}$ such that, conditioned on $(\Phi,\tau)$, $(Y_{x,1},... Y_{x,n})\sim F(C(x))$ for $x\in\Phi$ and $\{(Y_{x,1},... Y_{x,n}):x\in\Phi\}$ are independent. Then the point process
\begin{equation}
    \Psi:=\sum_{x\in\Phi}\sum_{k=1}^n\delta_{Y_{x,k}}
\end{equation}
is beyond $(2m-d)$-uniform. Note that Proposition~\ref{p:averaging set bound improved} guarantees that for any $d\in\N, C>0, m\in\N$, there exists an $N\in\N$ such that for all $n\geq N$, there exists a map $F$ that satisfies the above assumptions. Also, we can recover Examples~\ref{ex:fair tiling any point} and~\ref{ex:fair tiling centroid} with slightly stricter assumptions as the special cases $p=1$ and $p=2$ with $n=1$.

Again, the argument is analogous to that in Example~\ref{ex:fair tiling any point} with the alternate choice of\linebreak $p:=2m-d$.
\end{example}

Now that we have shown how to construct $p$-uniform point processes with an arbitrarily high $p$ from invariant random fair partitions, we shift our focus to the construction of the partitions themselves. One class of invariant fair partitions is obtained from lattices and their corresponding lattice tilings. We follow an approach closely connected to this one in Subsection~\ref{ss: pp from averaging sets}. However, lattice tilings are not isotropic, which is another property we are interested in. If one artificially makes them isotropic by randomly rotating them, they are no longer ergodic, which we also want them to be. Another class consists of fair tilings that can be generated from matching algorithms. This class actually contains isotropic and ergodic examples. A discussion of multiple options can be found in~\cite{Klatt_Last_Lotz_Yogeshwaran_2025}. However, the classical stable matching rarely yields fair tilings which fulfill the moment condition~\eqref{e:fair tiling high degree moment condition}; see~\cite{Hoffman_Holroyd_Peres_2005}. While there are alternatives like gravitational allocation for $d=3$~\cite{Chatterjee_Peled_Peres_Romik_2010}, which are regular enough for Example~\ref{ex:fair tiling any point} and~\ref{ex:fair tiling centroid}, their cell shapes may not satisfy the stricter assumptions of Example~\ref{ex:fair tiling high degree}. Finally, complete matchings are usually hard to simulate at a truly large scale, since most of the corresponding algorithms exhibit superlinear complexity.

In the following, we introduce such an invariant, isotropic, and ergodic random fair partition with convex cells in any dimension that has a finite approximation that can be simulated in linear time. However, a rigorous proof for the existence is beyond the scope of this paper, as it requires distinct techniques, and in our opinion, deserves its own proper investigation.

Our method is closely connected to the so-called \textit{STIT tessellations} (\textbf{st}able with respect to \textbf{it}erations). These arise from Markov chains on the space of partitions with convex cells, where in every step the cells are randomly and independently divided by hyperplanes and then scaled up to preserve the scale of the original partition. See~\cite{Nagel_Weiss_2003, Nagel_Weiss_2005, Nagel_Thäle_Weiß_2026} for further details. Our construction differs in that the Markov chain is constructed on the space of fair partitions with convex cells, i.e., in every step every cell is split exactly once into two parts of equal volume. We can also allow for the distribution of the split direction to depend on the cell shape, which was not allowed in this freedom in the original STIT models. Although the models are not a subclass of the STIT models and more akin to a variant of the underlying idea, we call these models the \textit{fair STIT} models to emphasize their similarity.

\begin{definition}[Splitting distribution]
    Let $\cK^d_1:=\{B\in\cB^d: B \textnormal{ convex}, \lambda_d(B)=1\}$ and 
    \begin{equation}
        \cF^d_1:=\{(\phi,\tau): (\phi,\tau) \textnormal{ is a fair partition of } \R^d, C(x)\in\cK^d_1 \text{ for } x\in\phi\}.
    \end{equation}
    Suppose that $F$ is a probability kernel from $\cK^d_1$ to $\partial B_1$ which is only shape dependent, i.e.,
    \begin{equation}
        F(B+x) = F(B),\quad B\in\cK^d_1, x\in\R^d.
    \end{equation}
    Then $F$ is called a \textit{splitting distribution}. 
\end{definition}

\begin{remark}\label{r: fair STIT}
    While we determine the splitting direction randomly according to a splitting distribution $F$, the location of the split is then uniquely determined since the newly created parts are required to have equal volume.
    Concretely, the measurable map $f:\cK^d_1\times\partial B_1\to \R$ with
    \begin{equation}
        \lambda_d(\{x\in B: \langle x,u\rangle \geq f(B,u)\}) = \lambda_d(\{x\in B: \langle x,u\rangle < f(B,u)\}) = \frac{1}{2},\quad B\in\cK^d_1, u\in\partial B_1,
    \end{equation}
    can be used to determine the location given a cell and splitting direction.
    Hence, we can define a probability kernel $\tilde{F}$ from $\cK^d_1$ to $\cK^d_1\times\cK^d_1$ by
    \begin{align}
        \tilde{F}(B) := \int \I\big\{\big(&\{x\in \sqrt[d]{2}B: \langle x,u\rangle \geq \sqrt[d]{2}f(B,u)\},\nonumber\\
        &\{x\in \sqrt[d]{2}B: \langle x,u\rangle < \sqrt[d]{2}f(B,u)\}\big)\in \cdot\big\} \, F(B,du),\quad B\in\cK_1^d.
    \end{align}
    It captures the distribution of the two newly created cells after splitting and rescaling.
    With a cell-wise application of $\tilde{F}$, we can extend this kernel to a probability kernel $P_F$ from $\cF^d_1$ to $\cF_1^d$. 
\end{remark}

\begin{definition}[Fair STIT]
    Suppose that $F$ is a splitting distribution and $P_F$ the corresponding probability kernel from $\cF^d_1$ to $\cF_1^d$ constructed in Remark~\ref{r: fair STIT}. Then we call a random fair partition $(\Phi,\tau)$ a \textit{fair STIT with splitting distribution} $F$ if it is invariant with respect to the probability kernel $P_F$.  To allow for the possibility of a unique invariant distribution, we also formally assume that
    \begin{equation}
        x = \int_{C(x)} z\, dz,\quad x\in\Phi.
    \end{equation}
\end{definition}

Note that not every splitting distribution gives rise to a fair STIT. On the other hand, the distribution of the fair STIT may not be unique if one exists. Also note that $\cF_1^1$ only contains the lattice tiling and shifted versions of it. However, $\cF_1^d$ is a lot richer for $d\geq2$.

In practice, we focus on the following splitting distribution, which is shape-independent and uniform.

\begin{example}[Fair STIT with a uniform splitting distribution]\label{ex: fair STIT}
    Let $F$ be the shape-independent splitting distribution $\cU(\partial B_1)$. Then we call a fair STIT with respect to $F$ a \textit{fair STIT with a uniform splitting distribution}. As stated before, we do not show here that such a fair STIT exists and is unique in distribution. Still, to help the reader better understand the object, we provide a detailed description of how to construct a step in the Markov chain.
    
    Let $(\phi,\tau)\in \cF_1^d$ be the current state. We denote the cells of $(\phi,\tau)$ by $(C(x))_{x\in\phi}$. Let $(U_x)_{x\in\phi}$ be a sequence of iid random vectors distributed like $\cU(\partial B_1)$, i.e., random unit vectors. These are the normal vectors of the hyperplanes, by which we split the cells. In a second step, we need to find the location of the planes such that they split the cells into parts of equal volume. For each $x\in\phi$ there is a unique random variable $Y_x\in\R$ which only depends on $U_x$ (and $C(x)$) such that $\lambda_d(\{z\in C(x):\langle z, U_x\rangle \geq Y_x)=\frac{1}{2}$. Hence, we independently split the cells in random directions into two equally sized parts. To preserve the volume, we scale up the whole process and obtain the random fair partition $(\tilde\Phi, \tilde\tau)$ as the next step of our Markov chain from
    \begin{align}
        X_{x,1} &:= \int z \I\{z\in \sqrt[d]{2}C(x), \langle z, U_x\rangle \geq \sqrt[d]{2}Y_x\}\, dz,\quad x\in\phi,\\
        X_{x,2} &:= \int z \I\{z\in \sqrt[d]{2}C(x), \langle z, U_x\rangle < \sqrt[d]{2}Y_x\}\, dz,\quad x\in\phi,\\
        \tilde{\Phi}&:=\sum_{x\in\phi} \delta_{X_{x,1}} + \delta_{X_{x,2}},\\
        \tilde\tau(z) &:= \begin{cases}
            X_{\tau(z), 1} : \langle z, U_{\tau(z)}\rangle \geq \sqrt[d]{2}Y_{\tau(z)},\\
             X_{\tau(z), 2} : \langle z, U_{\tau(z)}\rangle < \sqrt[d]{2}Y_{\tau(z)},
        \end{cases}\quad z\in\R^d.
    \end{align}

    Let us come back to the analysis of the fair STIT.
    One can expect that it is ergodic and isotropic from our construction with independent and isotropic splitting directions. It is harder to show that the cells satisfy~\eqref{e:fair tiling high degree moment condition}. We leave these questions open for now.

    In practice, we can just start with a fair partition and simulate multiple steps of the Markov chain and assume that the result is, in distribution, close to a supposed limit distribution. This approach has the additional advantage that it is easy to simulate at a large scale. As every cell is divided independently, the number of operations we need to obtain a snapshot with $2^n$ cells is proportional to $2^n$, $n\in\N_0$, i.e., the algorithm operates in linear time. Additionally, the algorithm can be parallelized perfectly.
\end{example}

We have implemented this algorithm in dimensions $2$ and $3$ and simulated samples of the fair STIT with a uniform splitting distribution with up to $2^{30}\approx 10^9$ cells. Then, for different $n,m\in\N$, we sampled $n$ points in each cell such that the first $m$ moments of the points and cells coincide as in \eqref{e: moment fixing}. Hence, by Example \ref{ex:fair tiling high degree}, the points form a point process that is beyond $(2m-d)$-uniform. The empirical structure factor of these samples is consistent with this theoretical result, as can be seen in Figure~\ref{fig: empirical structure factor}; indeed, our estimated exponent exceeds the lower bound as discussed in the caption of Figure~\ref{fig: empirical structure factor}. With samples containing $10^9$ points, we match the largest previously reported samples of non-crystalline hyperuniform point processes in two dimensions and surpass them in three dimensions; see~\cite{Shih_Casiulis_Martiniani_2024}. For further details about the simulation study, see Appendix~\ref{s: Simulation study}.

\begin{figure}[p]
    \centering
    \includegraphics[width=0.6574\textwidth]{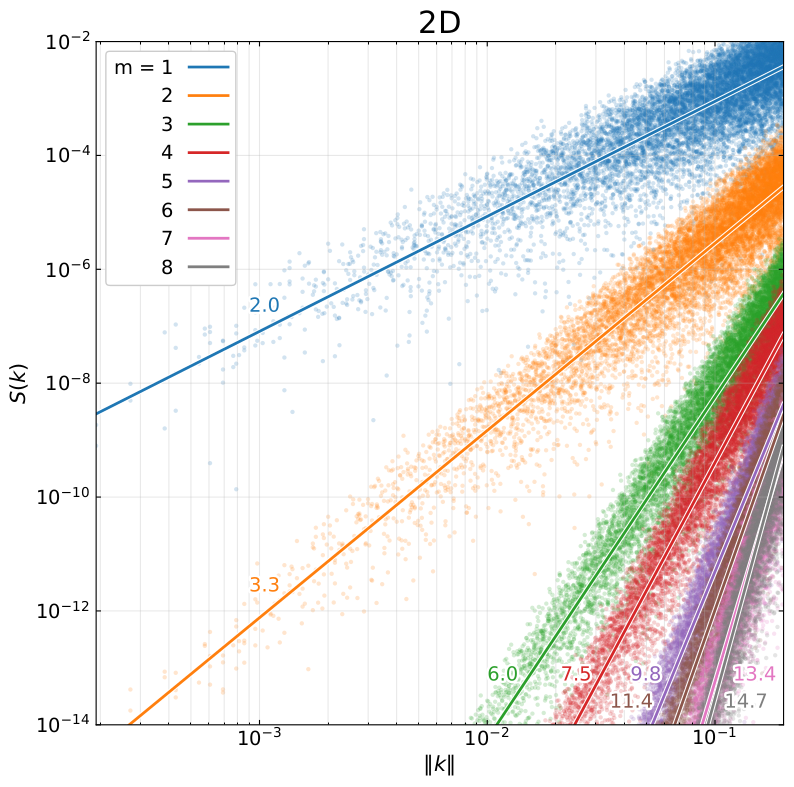}
    \hfill
    \includegraphics[width=0.3326\textwidth]{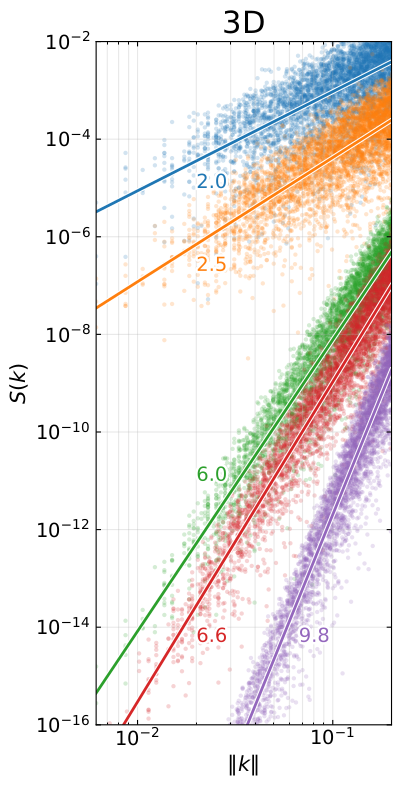}
    \caption{Empirical structure factors of the simulated point processes in dimensions $2$ (left) and $3$ (right) for different values of the number of fixed moments $m$. 
    Each sample contains about $10^9$ points; for the precise parameters including the number of points per cell $n$, see Table~\ref{tab: simulation} in Appendix~\ref{s: Simulation study}. Both scales are logarithmic, so a power-law with exponent $p$ in the structure factor corresponds to a linear law in this plot with slope $p$. To obtain estimates for $p$, we performed a linear regression for each sample and indicate the slope next to the line in the corresponding color. Note that these estimates are rough, especially for large $m$. Still, the estimated values of $p$ all exceed our lower bound of $2m-d$. Further, the estimates seem to be very close to $2m$ for odd $m$ but significantly smaller for even $m$, especially in dimension $3$. This observation indicates long-range correlations in the even moments of the cells of the fair STIT and fast decorrelation of the odd moments; see Theorem~\ref{t: main theorem rate}. Such a behavior could be explained by long-range correlations in the degree and direction of elongation of the cells.}
    \label{fig: empirical structure factor}
\end{figure}
\clearpage

\subsection{Point processes from moment-preserving clusters}\label{ss: pp from averaging sets}

The following constructions of point processes are also based on placing clusters of points, but no fair partitions are involved here. Instead, we base the construction on a progenitor point process which is $\tilde{p}$-uniform for a very high $\tilde{p}$, e.g., $\tilde{p}=\infty$. We then decrease $\tilde{p}$ by a controlled amount. While the previous construction yields nicer point processes for practical purposes, we can also derive sole $p$-uniformity here. A similar construction was done in~\cite{Lachièze-Rey_2025a}.

\begin{theorem}\label{t: independent avsets}
    Let $m \in \N_0$, and let $\mu$ be a random finite signed measure on $\R^d$. Suppose that $\Phi$ is a locally square-integrable invariant point process that is $2m$-uniform, and suppose that $K$ is a transport kernel such that $(K_x^\ast)_{x\in\Phi}$, conditioned on $\Phi$, is a sequence of independent random finite measures distributed like $\mu$. Further, assume that
    \begin{align}
        \BE\bigg[\bigg(&\int \|x\|^m+1\, |\mu|(dx)\bigg)^2\bigg] < \infty, \label{e: ind avsets moment}\\
        &\int x^{\otimes q}\, \mu(dx) \textnormal{ is deterministic for } q\in\N_0, q<m,\label{e: ind avsets fit}\\
        &\int x^{\otimes m}\, \mu(dx) \textnormal{ is not deterministic}.\label{e: ind avsets nonfit}
    \end{align}
    Then $K\Phi$ is a locally square-integrable invariant random measure and solely $2m$-uniform.

    \begin{proof}
        If $\mu$ was a random probability measure,~\cite[Theorem 5.13]{Klatt_Last_Lotz_Yogeshwaran_2025} would yield that $K\Phi$ is a locally square-integrable random measure and that
        \begin{equation}\label{e: spectral measure cluster process}
            \hat\beta_{K\Phi} = \big|\BE[\hat\mu]\big|^2 \cdot \hat\beta_\Phi + \alpha_\Phi(\{0\}) \BV[\hat\mu]\cdot \lambda_d.
        \end{equation}
        Similarly, this formula was also derived in \cite[Proposition 5]{Lachièze-Rey_2025a}, where $\mu$ is assumed to be a random finite counting measure. Here, we show this formula in the slightly more general setting that $\mu$ is a random finite signed measure.
        
        First, since $\Phi$ and $K$ are invariant, $K\Phi$ is also invariant. Further, it is locally square-integrable because of independence and $\BE\big[\|\mu\|^2\big]<\infty$, as
        \begin{align}
            \BE\big[|K\Phi(B_1)|^2\big] &= \iint \BE_{0,y}^\Phi\big[K(0, B_1-x), \overline{K(y, B_1-x)}\big]\,dx \,\alpha_\Phi(dy)\nonumber\\
            &= \iint \BE\big[\mu(B_1-x)\big]\BE\big[\overline{\mu(B_1-x-y)}\big]\,dx \,\alpha_\Phi(dy)\nonumber\\
            &\qquad + \alpha_\Phi(\{0\})\int\BV[\mu(B_1-x)]\,dx\nonumber\\
            &= \BE\big[|L\Phi(B_1)|^2\big] + \alpha_\Phi(\{0\})\int\BV[\mu(B_1-x)]\,dx\nonumber\\
            &< \infty,\label{e:alpha independent avsets}
        \end{align}
        where $L$ is an invariant deterministic transport kernel from $\R^d$ to $\R^d$ with $L^\ast(y)=\BE[\mu], y\in\R^d$. The fact that $L\Phi$ is locally square-integrable can be derived analogously to~\cite[Lemma 3.8]{Klatt_Last_Lotz_Yogeshwaran_2025}. Additionally, we can derive from~\eqref{e:alpha independent avsets} and $\BE[L\Phi]=\BE[K\Phi]$ that
        \begin{equation}
            \beta_{K\Phi} = \beta_{L\Phi} + \alpha_\Phi(\{0\}) (\BE[\mu\star\mu]-\BE[\mu]\star\BE[\mu]).
        \end{equation}
        Hence,
        \begin{align}
            \hat\beta_{K\Phi} &= \hat\beta_{L\Phi} + \alpha_\Phi(\{0\}) \BV[\hat\mu]\cdot \lambda_d\nonumber\\
            &= \big|\BE[\hat\mu]\big|^2 \cdot \hat\beta_\Phi + \alpha_\Phi(\{0\}) \BV[\hat\mu]\cdot \lambda_d, \label{e:bartlett spectral ind avset}
        \end{align}
        which proves~\eqref{e: spectral measure cluster process} in this more general setting. Further,
        \begin{equation}
            \hat\mu(k) = \int e^{-i\langle k,x\rangle}\, \mu(dx) = \sum_{q=0}^m \frac{(-ik)^{\otimes q}}{q!} \int x^{\otimes q} \, \mu(dx) + \int g(\langle k ,x\rangle)\, \mu(dx), \quad k\in\R^d,
        \end{equation}
        where
        \begin{equation}
            g(z) := e^{-iz} - \sum_{q=0}^m \frac{(-iz)^q}{q!} = \sum_{q=m+1}^\infty \frac{(-iz)^q}{q!},\quad z\in\R.
        \end{equation}
        Thus, there is some $c>0$ such that
        \begin{equation*}
            |g(z)| \leq c\min(|z|^{m}, |z|^{m+1}),\quad z\in\R.
        \end{equation*}
        This bound allows us to apply~\eqref{e: ind avsets moment} and the dominated convergence theorem to obtain that, for $k\in\R^d$,
        \begin{equation*}
            \frac{1}{\|k\|^{2m}}\BE\bigg[\bigg|\int g(\langle k,x\rangle)\, \mu(dx)\bigg|^2\bigg] = \BE\bigg[\bigg|\int \underbrace{\frac{g(\langle k,x\rangle)}{\|k\|^m}}_{|\cdot|\leq c \min(\|x\|^m, \|k\|\|x\|^{m+1})}\, \mu(dx)\bigg|^2\bigg] \xrightarrow{k\to 0} 0.
        \end{equation*}
        Hence, by~\eqref{e: ind avsets fit}, as $k\to 0$,
        \begin{equation}\label{e: ind avset var formula}
            \BV[\hat\mu(k)] = \frac{\|k\|^{2m}}{(m!)^2} \BV\bigg[\int \langle \tfrac{k}{\|k\|},x\rangle^m\, \mu(dx)\bigg] + o(\|k\|^{2m}).
        \end{equation}
        Applying~\eqref{e: ind avsets moment}, we can derive an upper bound for the asymptotics of $\BV[\hat\mu(k)]$ as $k\to 0$ by
        \begin{equation}
            \BV[\hat\mu(k)] \leq \frac{\|k\|^{2m}}{(m!)^2} \BE\bigg[\bigg|\int \|x\|^m\, |\mu|(dx)\bigg|^2\bigg] + o(\|k\|^{2m}) = O(\|k\|^{2m}).
        \end{equation}
        In conjunction with the assumption that $\Phi$ is $2m$-uniform and~\eqref{e:bartlett spectral ind avset}, this bound yields that $K\Phi$ is $2m$-uniform. It remains to show that it is not beyond $2m$-uniform. By~\eqref{e:bartlett spectral ind avset}, this is established if we can show that
        \begin{equation}
            \frac{1}{\varepsilon^{d}}\int_{B_\varepsilon} \BV[\hat\mu(k)]\, dk \neq o(\varepsilon^{2m})
        \end{equation}
        as $\varepsilon\to 0$. To this end, we leverage~\eqref{e: ind avset var formula}. By~\eqref{e: ind avsets nonfit} and the multilinear polarization identity, see~\cite{Mazur_Orlicz_1934}, there is an $s\in \partial B_1$ such that
        \begin{equation*}
            \BV\bigg[\int \langle s,x\rangle^m\, \mu(dx)\bigg] > 0.
        \end{equation*}
        As this function is continuous in $s$ by the dominated convergence theorem and~\eqref{e: ind avsets moment},
        \begin{equation*}
            \int_{\partial B_1} \BV\bigg[\int \langle s,x\rangle^m\, \mu(dx)\bigg] \, \sigma_d(ds) >0, 
        \end{equation*}
        where $\sigma_d$ is the unique rotation invariant probability measure on $\partial B_1$. Therefore, for $\varepsilon>0$,~\eqref{e: ind avset var formula} with an insertion of polar coordinates yields
        \begin{equation*}
            \liminf_{\varepsilon\to 0}\frac{1}{\varepsilon^{d+2m}}\int_{B_\varepsilon} \BV[\hat\mu(k)]\, dk
            = \frac{d\kappa_d}{(m!)^2(d+2m)}\int_{\partial B_1} \BV\bigg[\int \langle s,x\rangle^m\, \mu(dx)\bigg]\, \sigma_d(ds)
            >0,
        \end{equation*}
        which concludes the proof.
    \end{proof}
\end{theorem}

The following theorem is analogous to the previous theorem but is concerned with $\infty$-uniformity.

\begin{theorem}
    Let $\mu$ be a random finite measure on $\R^d$. Suppose that $\Phi$ is a locally square-integrable invariant point process that is $\infty$-uniform with radius $\varepsilon>0$, and that $K$ is a transport kernel such that $(K_x^\ast)_{x\in\Phi}$, conditioned on $\Phi$, is a sequence of independent random finite measures distributed like $\mu$. Further, assume that
    \begin{equation}
        \hat\mu(k) \textnormal{ is deterministic for all } k\in B_\varepsilon.
    \end{equation}
    Then $K\Phi$ is a locally square-integrable invariant random measure and $\infty$-uniform with radius $\varepsilon$.
    \begin{proof}
        The assertion follows directly from~\eqref{e: spectral measure cluster process} derived in the proof of Theorem~\ref{t: independent avsets}.
    \end{proof}
\end{theorem}

With this theorem, we can construct the following interesting example.
\begin{example}
    Suppose that $\Phi$ is a locally square-integrable invariant point process which is $\infty$-uniform with radius $\varepsilon>0$. Choose $f:\R^d\to(0,\infty)$ such that $\hat{f}$ has support on $B_\vartheta$ with $\vartheta\in(0,\varepsilon)$. Further, let $(\Phi_x)_{x\in\Phi}$, conditioned on $\Phi$, be a sequence of independent point processes distributed like $\Phi$. Then 
    \begin{equation}
        \Psi := \sum_{x\in\Phi}\sum_{y\in\Phi_x} f(y-x) \delta_y = \int \big(f(\cdot-x) \cdot\Phi_x\big)\, \Phi(dx)
    \end{equation}
    is a locally square-integrable invariant discrete random measure which is $\infty$-uniform with radius $\varepsilon-\vartheta$ and has atoms which are everywhere dense. Further, if $\Phi$ is ergodic then $\Psi$ also is ergodic.

    In particular, if $\Phi$ is the invariant lattice, this construction results in a discrete random measure, where in each cell of the original lattice, the locations of the new atoms are independent and uniformly distributed over the cell. However, for different cells the weights are highly correlated, which guarantees that $\infty$-uniform is preserved.

    By convoluting this discrete random measure with a suitable function, one can also construct an invariant random measure which is $\infty$-uniform while having a (smooth) density with respect to the Lebesgue measure. Another variant can be obtained by replacing the invariant lattice by the $\infty$-uniform point processes constructed in~\cite{Kurasov_Sarnak_2020, Olevskii_Ulanovskii_2020, Alon_Kummer_Kurasov_Vinzant_2025, Björklund_2026}.
\end{example}

Let us return from this small excursion about discrete random measures and nonnegative random fields to our main problem at hand concerning the construction of $p$-uniform point processes with high $p$. Theorem~\ref{t: independent avsets} gives us a clear direction and the next step is to find a suitable $\mu$ that is is of the form $\mu=\sum_{k=1}^n \delta_{X_k}$ for some $n\in\N$ and random vectors $X_1,...,X_n$ in $\R^d$, which guarantees that $K\Phi$ is a point process. Then~\eqref{e: ind avsets fit} clearly implies that $X_1,...,X_n$ almost surely form an averaging set of degree $m-1$ of some measure $\nu$, and~\eqref{e: ind avsets nonfit} implies that they do not almost surely form an averaging set of degree $m$ of any measure. See the comment after Remark~\ref{r: partition persistence} for a short explanation and Appendix~\ref{s: averaging sets} for a proper introduction of averaging sets.

\begin{example}\label{ex: general point cluster}
    Let $n,m \in \N$ and suppose that $\Phi$ is a locally square-integrable invariant point process which is $2m$-uniform. Suppose that $X_1,...,X_n$ almost surely form an averaging set of degree $m-1$ of some measure $\nu$, do not almost surely form an averaging set of degree $m$ of any measure, and that $\BE[X_k^{2m}]< \infty, k\in\{1,...,n\}$. Further, let $\big((X_{x,1},...,X_{x,n})_{x\in\Phi}\big)$, conditioned on $\Phi$, be an independent sequence with elements distributed like $(X_1,...,X_n)$.
    Then
    \begin{equation}
        \Psi :=\sum_{x\in\Phi}\sum_{k=1}^n \delta_{x+X_{x,k}}
    \end{equation}
    is a locally square-integrable invariant point process which is solely $2m$-uniform by Theorem~\ref{t: independent avsets}.

    Note that we recover that an iid-perturbed point process cannot be beyond $2$-uniform if the perturbation is not deterministic and has finite variance. This relation also holds without the assumption of finite variance; see~\cite[Corollary 2.1]{Lachièze-Rey_2025b}. The case that the unperturbed point process itself is beyond $2$-uniform is covered by choosing $m=1$, and otherwise it follows directly from~\eqref{e: spectral measure cluster process}.
\end{example}

A well-known averaging set is that of the nodes of the Chebyshev-Gauss quadrature of degree $2n-1$ for $n\in\N$:
\begin{equation}
    x_k = \cos\Big(\frac{2k-1}{2n}\pi\Big),\quad k\in\{1,...,n\}.
\end{equation}
They form an averaging set of degree $2n-1$ of the measure $\nu$ with the density
\begin{equation}
    f(x) = \I_{(-1,1)}(x) \frac{1}{\sqrt{1-x^2}}, x\in\R,
\end{equation}
with respect to the Lebesgue measure; see, e.g.,~\cite[Section 25]{Abramowitz_Stegun_1965}.

If we introduce an additional parameter $u\in\R$, which we can randomize as in the following example, we reduce the degree of the averaging set for almost all $u\in\R$ by $n$. However, in turn, we obtain a whole class of averaging sets of $\nu$ of the same degree.

\begin{example}
    Assume that $d=1$. Let $n \in \N$ and suppose that $\Phi$ is a locally square-integrable invariant point process which is $2n$-uniform. Further, suppose that $(U_x)_{x\in\Phi}$, conditioned on $\Phi$ is a sequence of iid random variables with distribution $\cU([0,2\pi))$. 
    Let
    \begin{equation}
        X_{x,k} := \cos\Big(2\pi\frac{k}{n} +U_x\Big),\quad x\in\Phi, k\in\{1,...,n\}.
    \end{equation}
    Then
    \begin{equation}
        \Psi :=\sum_{x\in\Phi}\sum_{k=1}^n \delta_{x+X_{x,k}}
    \end{equation}
    is a locally square-integrable invariant point process which is solely $2n$-uniform.

    This assertion follows from Example~\ref{ex: general point cluster} in conjunction with Theorem~\ref{t:averaging set trigonometric polynomial} and the fact that $\cos$ is a trigonometric polynomial of degree $1$.
\end{example}

While we only perturb the points of $\Phi$ in one spatial direction in the preceding example, we can generalize this concept to higher dimensions. There, the measure $\nu$ we approximate is an integral along a curve which is parameterized by trigonometric polynomials. For a definition of trigonometric polynomials; see Definition~\ref{d:trigonometric polynomials}.

\begin{example}\label{ex: trigonometric point cluster on line}
    Let $n \in \N$, and let $p_1,...,p_d$ be real-valued trigonometric polynomials with degrees $m_1,...,m_d\in\N_0$, and suppose that $\Phi$ is a locally square-integrable invariant point process which is $2\big\lfloor \frac{n}{\max(m_1,...,m_d,1)}\big\rfloor$-uniform. Further, suppose that $(U_x)_{x\in\Phi}$, conditioned on $\Phi$ is a sequence of iid random variables with distribution $\cU([0,2\pi))$. 
    Let
    \begin{equation}
        X_{x,k} := \begin{pmatrix}
            p_1\big(2\pi\frac{k}{n} +U_x\big)\\
            \vdots\\
            p_d\big(2\pi\frac{k}{n} +U_x\big)
        \end{pmatrix},\quad x\in\Phi, k\in\{1,...,n\}.
    \end{equation}
    Then
    \begin{equation}
        \Psi :=\sum_{x\in\Phi}\sum_{k=1}^n \delta_{x+X_{x,k}}
    \end{equation}
    is a locally square-integrable invariant point process which is $2\big\lfloor \frac{n}{\max(m_1,...,m_d,1)}\big\rfloor$-uniform, but not beyond $2\big\lfloor \frac{n}{\max(\textnormal{gcf}(m_1,n),...,\textnormal{gcf}(m_d,n))}\big\rfloor$-uniform, where $\textnormal{gcf}$ yields the greatest common factor of two integers.

    The argument is analogous to that in the previous example, but Theorem~\ref{t:averaging set multivariate trigonometric polynomial} is applied in its multivariate form this time.
\end{example}

The following special case of this example was already discussed in~\cite{Lachièze-Rey_2025a} for the case that $n$ is prime. We show that $n$ does not need to be prime for the assertion to hold.

\begin{example}
    Assume that $d=2$. Let $n \in \N$, and suppose that $\Phi$ is a locally square-integrable invariant point process which is $2n$-uniform. Further, suppose that $(U_x)_{x\in\Phi}$, conditioned on $\Phi$, is a sequence of iid random variables with distribution $\cU([0,2\pi))$. 
    Let
    \begin{equation}
        X_{x,k} := \begin{pmatrix}
            \sin\big(2\pi\frac{k}{n} +U_x\big)\\
            \cos\big(2\pi\frac{k}{n} +U_x\big)
        \end{pmatrix},\quad x\in\Phi, k\in\{1,...,n\}.
    \end{equation}
    Then
    \begin{equation}
        \Psi :=\sum_{x\in\Phi}\sum_{k=1}^n \delta_{x+X_{x,k}}
    \end{equation}
    is a locally square-integrable invariant point process which is solely $2n$-uniform.
    
    This assertion follows with the choice of $d:=2,p_1:=\sin, p_2:=\cos$ in Example~\ref{ex: trigonometric point cluster on line}, and since $\sin,\cos$ are trigonometric polynomials of degree $1$.
\end{example}

One can generalize this construction from placing points on a curve to placing them on a manifold with a higher dimension, which is parameterized by a multivariate trigonometric polynomial. We do not put a general statement here and instead refer the reader to Appendix~\ref{ss: Averaging sets of trigonometric polynomials} for the theoretical background. Instead, we present the following example, where the manifold is the surface of the unit ball.

\begin{example}\label{ex: trigonometric hyperspheres}
    Assume $d\geq 2$. Let $n \in \N$ and suppose that $\Phi$ is a locally square-integrable invariant point process which is $2n$-uniform. Further, suppose that $(U^{(x)})_{x\in\Phi}$, conditioned on $\Phi$, is a sequence of iid random vectors with distribution $\cU([0,2\pi)^{d-1})$. 
    Let
    \begin{equation}
        X_{x,k} := \begin{pmatrix}
            \sin\big(2\pi\frac{k_1}{n} +U^{(x)}_1\big)\\
            \cos\big(2\pi\frac{k_1}{n} +U^{(x)}_1\big)\sin\big(2\pi\frac{k_2}{n} +U^{(x)}_2\big)\\
            \vdots\hspace{4mm}\\
            \cos\big(2\pi\frac{k_1}{n} +U^{(x)}_1\big)\cdots \cos\big(2\pi\frac{k_{d-1}}{n} +U^{(x)}_{d-1}\big)
        \end{pmatrix},\quad x\in\Phi, k\in\{1,...,n\}^{d-1}.
    \end{equation}
    Then
    \begin{equation}
        \Psi :=\sum_{x\in\Phi}\sum_{k\in\{1,...,n\}^{d-1}} \delta_{x+X_{x,k}}
    \end{equation}
    is a locally square-integrable invariant point process which is solely $2n$-uniform.

    The argument is similar to those in the preceding examples, but instead of Theorem~\ref{t:averaging set trigonometric polynomial}, we leverage Theorem~\ref{t:averaging set multivariate trigonometric polynomial} and Proposition~\ref{p:averaging set multivariate trigonometric polynomial degree 1}.
\end{example}

However, the measure on the surface of the unit ball, which is approximated in the previous example, is not uniform for $d\geq 3$, whereby the constructed point process cannot be isotropic in this case. Averaging sets for the uniform distribution are called spherical designs in the literature; see, e.g.,~\cite{Seymour_Zaslavsky_1984}. It is known that spherical designs exist in any dimension $d\in\N$, and of any degree $m\in\N$, and that the minimum number of points needed grows like $m^{d-1}$ as $m\to\infty$; see~\cite{Bondarenko_Radchenko_Viazovska_2013}. In this sense, they are as efficient as the trigonometric parametrizations from Example~\ref{ex: trigonometric hyperspheres} with the added advantage that their class is rotation invariant. There also exist explicit constructions of spherical designs of arbitrarily high degree (see, e.g.,~\cite{Xiang_2022}), but the known constructions are not asymptotically efficient in terms of size in relation to degree for dimensions $d \geq 3$.

\begin{example}
    Let $n,m \in \N$. Suppose that $\Phi$ is a locally square-integrable invariant point process which is $2m$-uniform, and that $\{z_1,...,z_n\}\subset \partial B_1$ is a spherical design of degree $m-1$ but not of degree $m$. Further, suppose that $(A_x)_{x\in\Phi}$, conditioned on $\Phi$, is a sequence of iid random matrices with distribution $\cU(SO_d)$.
    Let
    \begin{equation}
        X_{x,k} := A_xz_k,\quad x\in\Phi, k\in\{1,...,n\}.
    \end{equation}
    Then
    \begin{equation}
        \Psi :=\sum_{x\in\Phi}\sum_{k=1}^n \delta_{x+X_{x,k}}
    \end{equation}
    is a locally square-integrable invariant point process which is solely $2m$-uniform. Further, if $\Phi$ is isotropic, then $\Psi$ is also isotropic. Note that such an isotropic $\Phi$ is constructed in Example~\ref{ex:fair tiling high degree}. 
    This example is a special case of Example~\ref{ex: general point cluster}.
\end{example}

\subsection{Persistence of \textit{p}-uniformity under displacements} \label{ss: Persistence of uniformity under displacements}

In this subsection, we analyze the effect of the application of a displacement field on $p$-uniformity of a random measure. Displacement fields are $\R^d$-valued random fields on $\R^d$. Sometimes they are the product of a disintegration of a transport. Otherwise, in particular when we suppose that the displacement field and the source are independent, we only consider random fields with c\`adl\`ag-paths for measurability reasons. As the random fields are multivariate, we reserve the index for the components of the random field, and we access the individual random vectors only using the argument. We refer to Subsection~\ref{ss: random complex measures} for a precise introduction.

In the following, unless explicitly stated otherwise, we always assume that $Z$ is an invariant $\R^d$-valued random field on $\R^d$, and that $\Phi$ is an invariant locally square-integrable random measure. We then define the invariant random measure $\Psi$ by
\begin{equation}\label{e: definition Psi displacement}
    \Psi(B) := \int \I\{x+Z(x)\in B\}\, \Phi(dx).
\end{equation}
Hence, $\Psi = K\Phi$, where $K$ is a probability kernel defined by $K(x):=\delta_{x+Z(x)}, x\in\R^d$.

We start with simple specializations of Theorems~\ref{t:main square-integrable} and~\ref{t:main theorem} to displacement fields.

\begin{theorem}
    Suppose that $Z$ is an invariant $\R^d$-valued random field on $\R^d$, that $\Phi$ is an invariant locally square-integrable random measure, and that $\Psi$ is defined as in~\eqref{e: definition Psi displacement}. Further, either assume that for some $\vartheta>0$ and strongly log-dominating function $\rho$,
    \begin{equation}\label{e:condition square-integrable RF}
        \BE_0^\Phi[\Phi(B_\vartheta) \|Z(0)\|^d \rho(\|Z(0\|))] <\infty,
    \end{equation}
    or assume that $Z$ and $\Phi$ are independent and that
    \begin{equation}
        \BE[\|Z(0)\|^d] < \infty.
    \end{equation}
    Then $\Psi$ is also locally square-integrable.
    \begin{proof}
        The assertion is a simple specialization of Theorem~\ref{t:main square-integrable} to the special case of displacement kernels with an application of Proposition~\ref{p:independent square-integrable} to obtain Condition~\ref{c:condition square-integrable general} from~\eqref{e:condition square-integrable RF}.
    \end{proof}
\end{theorem}

\begin{theorem}\label{t: main theorem displacement}
    Let $p\in[-d,\infty)$. Suppose that $Z$ is an invariant $\R^d$-valued random field on $\R^d$, that $\Phi$ is an invariant locally square-integrable random measure, and that $\Psi$ is defined as in~\eqref{e: definition Psi displacement}. Either assume that for some $\vartheta>0$ and strongly log-dominating function $\rho$,
    \begin{equation}\label{e: random field moment condition}
        \BE_0^\Phi[\Phi(B_\vartheta) \|Z(0)\|^{d+\max(p,0)} \rho(\|Z(0\|))] <\infty,
    \end{equation}
    or assume that $Z$ and $\Phi$ are independent and that
    \begin{equation}\label{e: random field moment condition independent}
        \BE[\|Z(0)\|^{d+\max(p,0)}] < \infty.
    \end{equation}
    Further, assume that $Z^q\cdot\Phi$ is (beyond) $(p-2q)$-uniform for $q\in\N_0$ with $q<\max(\frac{d+p}{2},1)$. Then $\Psi$ is (beyond) $p$-uniform.
    Alternatively, let $\varepsilon>0$ and assume that 
    \begin{equation}
        \BE_0^\Phi\big[\Phi(B_\vartheta) e^{2\varepsilon Z(0)}\big] < \infty,
    \end{equation}
    for some $\vartheta>0$, and that $Z^q\cdot\Phi$ is $\infty$-uniform with radius $\varepsilon$ for $q\in\N_0$. Then $\Psi$ also is $\infty$-uniform with radius $\varepsilon$.
    \begin{proof}
        The assertion is a simple specialization of Theorem~\ref{t:main theorem} to the special case of displacement kernels, where Condition~\ref{c:condition square-integrable general with p} can be obtained from~\eqref{e: random field moment condition} analogous to Proposition~\ref{p:independent square-integrable}.
    \end{proof}
\end{theorem}

In the assumptions of the preceding theorem, we can see that $p$-uniformity of the displaced random measure $\Psi$, up to a moment condition, is directly related to $p$-uniformity of $Z^q\cdot\Phi$ for specific $q\in\N_0$. If we are only interested in $p$-uniformity for $p\leq2$, we can leverage the following proposition to obtain criteria that are easy to check.

\begin{proposition}\label{p: uniformity with random field density}
Suppose that $\Phi$ is an invariant complex random measure that is locally absolutely square-integrable. Further, suppose that $Z$ is an invariant $\C$-valued random field on $\R^d$ such that
\begin{equation}
    \int \BE_{0,y}^{|\Phi|}[|Z(y)Z(0)|] \, \alpha_{|\Phi|}(dy) < \infty.
\end{equation}
Then $Z\cdot\Phi$ is locally absolutely square-integrable. Further, if $\Phi$ is nonnegative and $p$-uniform with $p\leq0$ and
\begin{equation}\label{e:condition hyperfluctuating random field}
    \int \big|\BE_{0,y}^{\Phi}[Z(y)\overline{Z(0)}] - |\BE_{0}^{\Phi}[Z(0)]|^2\big| (1+\|y\|)^{p} \, \alpha_{\Phi}(dy) < \infty,
\end{equation}
then $Z\cdot\Phi$ is also $p$-uniform. If $p<0$, beyond $p$-uniformity also transfers.
\end{proposition}
\begin{remark}\label{r: uniformity with random field density}
In the setting of Proposition~\ref{p: uniformity with random field density}, if $\Phi$ and $Z$ are also independent, then one can replace $\BE_{0,y}^{|\Phi|}$, $\BE_{0,y}^{\Phi}$, and $\BE_{0}^{\Phi}$ by $\BE$, which simplifies the second condition to
\begin{equation}
    \int \big|\BC[Z(y),Z(0)]\big| (1+\|y\|)^{p} \, \alpha_{\Phi}(dy) < \infty.
\end{equation}
This condition even works in the case that $\Phi$ is complex-valued.
\end{remark}
\begin{proof}[Proof of Proposition~\ref{p: uniformity with random field density}.]
Note that if $\Phi$ is nonnegative,~\eqref{e: alpha transport formula} and~\eqref{e: beta transport formula} yield
\begin{align}
    \alpha_{Z\cdot\Phi}(dy) &= \BE_{0,y}^{\Phi}[Z(y)\overline{Z(0)}]\, \alpha_\Phi(dy)\label{e:alpha RF formula},\\
    \beta_{Z\cdot\Phi}(dy) &= \big(\BE_{0,y}^{\Phi}[Z(y)\overline{Z(0)}]-\big|\BE_0^\Phi[Z(0)]\big|^2\big)\, \alpha_\Phi(dy) + \big|\BE_0^\Phi[Z(0)]\big|^2\, \beta_\Phi(dy).\label{e:beta RF formula}
\end{align}
Now the assertion on local absolute square-integrality is a consequence of~\eqref{e:alpha RF formula} applied to $|Z|\cdot|\Phi|$. For the result regarding $p$-uniformity, it suffices to consider~\eqref{e:beta RF formula}. The second term is unproblematic because $\Phi$ is $p$-uniform. The first term can be handled with, or at least as in, the last part of Proposition~\ref{p:mixing condition uniformity}.
\end{proof}

In this way, we can obtain the following proposition from Theorem~\ref{t: main theorem displacement}.

\begin{proposition}\label{p: random field low uniformity correlations}
    Let $p\in[-d, 2]$. Suppose that the moment condition~\eqref{e: random field moment condition} is satisfied, and that for $q\in\N$ with $q<\frac{d+p}{2}$,
    \begin{equation}\label{e: random field low uniformity condition}
        \int \frac{\big\|\BE_{0,y}^\Phi[Z(y)^{\otimes q}\times Z(0)^{\otimes q}] - \BE_{0}^\Phi[Z(0)^{\otimes q}]\times \BE_{0}^\Phi[Z(0)^{\otimes q}]\big\|}{1 + \|y\|^{2q-p}}\,\alpha_\Phi(dy) < \infty,
    \end{equation}
    where $\times$ denotes the component-wise product. 
    Then $\Psi$ is $p$-uniform if $\Phi$ is. If $p<2$, then beyond $p$-uniformity also transfers.

    Further, if $Z$ and $\Phi$ are independent, we can replace the moment condition by~\eqref{e: random field moment condition independent} and simplify~\eqref{e: random field low uniformity condition} to
    \begin{equation}\label{e:  random field low uniformity condition independent}
        \int \frac{\big\|\BC[Z(y)^{\otimes q}, Z(0)^{\otimes q}]\big\|}{1 + \|y\|^{2q-p}}\,\alpha_\Phi(dy) < \infty
    \end{equation}
    for $q\in\N$ with $q<\frac{d+p}{2}$, where the covariance is taken component-wise.
    \begin{proof}
        The assertions directly follow from Theorem~\ref{t: main theorem displacement} with an application of Proposition~\ref{p: uniformity with random field density} and Remark~\ref{r: uniformity with random field density} to obtain $(p-2q)$-uniformity of the components of $Z^q\cdot\Phi$ for the required $q\in\N$.
    \end{proof}
\end{proposition}

\begin{remark}
    Condition~\eqref{e:  random field low uniformity condition independent} from Proposition~\ref{p: random field low uniformity correlations} holds if there is a strongly log-dominating function $\rho$ and a $c>0$ such that
    \begin{equation}
        \big\|\BC[Z(y)^{\otimes q}, Z(0)^{\otimes q}]\big\| \leq c\|y\|^{-(d+p)+2q}\rho(\|y\|)^{-1}, \quad y\in \R^d\setminus\{0\}.
    \end{equation}
\end{remark}

Using commonly known bounds of the covariance in terms of mixing coefficients, which were also applied in~\cite{Flimmel_2025}, one can replace the conditions on the covariance in the preceding proposition by $\alpha$-mixing-type conditions. In that case, the moment assumptions have to be stronger as well.

\begin{proposition}\label{p: alpha mixing condition}
    Let $p\in[-d, 2]$. Suppose that there are $\gamma> d+\max(p,0), c>0$ such that
    \begin{align}
        \BE_{0}^{\Phi}[\|Z(0)\|^{\gamma}] &< \infty,\label{e: mixing moment 1}\\
        \BE_{0,y}^{\Phi}[\|Z(0)\|^{\gamma}] &\leq c,\quad y\in\R^d\label{e: mixing moment 2}\\
        \int \frac{\kappa(y)^{\frac{\gamma-2q}{\gamma}}}{1 + \|y\|^{2q-p}}\, \alpha_\Phi(dy) &< \infty,\quad q\in\N, q<\tfrac{d+p}{2}, \label{e:mixing condition}
    \end{align}
    where
    \begin{equation}
        \kappa(y) := \sup_{B_1,B_2\in\cB^d} \big|\BP_{0,y}^\Phi(Z(y)\in B_1, Z(0)\in B_2) - \BP_0^\Phi(Z(y)\in B_1)\BP_0^\Phi(Z(0)\in B_2)\big|,\quad y\in\R^d.
    \end{equation}
    Then $\Psi$ is $p$-uniform if $\Phi$ is. If $p<2$, then beyond $p$-uniformity also transfers.
\end{proposition}

\begin{remark}
    In Proposition~\ref{p: alpha mixing condition}, if $Z$ and $\Phi$ are independent, then $\kappa(y)$ is simply the \linebreak $\alpha$-mixing-coefficient of $Z(y)$ and $Z(0)$, $y\in\R^d$, and the two moment conditions simplify to
    \begin{equation}\label{e: mixing moment 3}
        \BE[\|Z(0)\|^{\gamma}] < \infty
    \end{equation}
    with only $\gamma\geq d+\max(p,0)$ being required.
    
    Further, the mixing condition~\eqref{e:mixing condition} holds if $d+p\leq 2$, or if $\gamma>d+p$ and
    \begin{equation}\label{e:mixing condition simplified 1}
        \int \kappa(y)^{\frac{(\gamma-2)(d+\vartheta)}{\gamma(d+p-2+\vartheta)}}\, \alpha_\Phi(dy) < \infty
    \end{equation}
    for some $\vartheta>0$, or if there are $\vartheta,\tilde{c}>0$ such that
    \begin{equation}\label{e:mixing condition simplified 2}
       \kappa(y) \leq \tilde{c}\|y\|^{-(d+p-2)\frac{\gamma}{\gamma-2}-\vartheta}, \quad y\in \R^d\setminus\{0\}.
    \end{equation}
    Condition~\ref{e:mixing condition simplified 1} implies condition~\ref{e:mixing condition} by the Hölder inequality, where all cases but $q=1$ drop out. The fact that condition~\ref{e:mixing condition simplified 2} implies condition~\ref{e:mixing condition simplified 1} can be seen by insertion.
    Note that the exponent in condition~\ref{e:mixing condition simplified 1} is always greater than $1$ if $p=0$ and $\vartheta$ is sufficiently small.
\end{remark}
    \begin{proof}[Proof of Proposition~\ref{p: alpha mixing condition}.]
        Let $\gamma> d+\max(p,0), c>0$ and assume that~\eqref{e: mixing moment 1},~\eqref{e: mixing moment 2}, and~\eqref{e:mixing condition} hold. Choose $\tilde{Z}:=Z_1^{q_1}\cdots Z_d^{q_d}$ for some $q\in\N_0^d$ with $\|q\|_1\in[1,\frac{d+p}{2})$. Let $r<0$ and define $\tilde{Z}^{(1)}:= \tilde{Z}\I\{|\tilde{Z}|\leq r\}$, $\tilde{Z}^{(2)}:= \tilde{Z}\I\{|\tilde{Z}|>r\}$, whereby $\tilde{Z}=\tilde{Z}^{(1)}+\tilde{Z}^{(2)}$. Then we can employ the triangle inequality and the fact that $(\tilde{Z}^{(2)})^a\leq (\tilde{Z}^{(2)})^{a+b}r^{-b}$, $a,b>0$, to obtain that there is a $\tilde{c}>0$ such that
        \begin{align*}
            \big|\BE_{0,y}^\Phi[\tilde{Z}(y)\tilde{Z}(0)]-\BE_0^\Phi[\tilde{Z}(0)]^2\big| &\leq \big|\BE_{0,y}^\Phi[\tilde{Z}^{(1)}(y)\tilde{Z}^{(1)}(0)]-\BE_0^\Phi[\tilde{Z}^{(1)}(0)]^2\big| \nonumber\\
            &\qquad  + \Big(\big|\BE_{0,y}^\Phi[\tilde{Z}^{(2)}(y)]\big|+\big|\BE_{0,y}^\Phi[\tilde{Z}^{(2)}(0)]\big|+2\big|\BE_{0}^\Phi[\tilde{Z}^{(2)}(0)]\big|\Big)r\nonumber\\
            &\qquad  + \big|\BE_{0,y}^\Phi[\tilde{Z}^{(2)}(y)^2]\big|+\big|\BE_{0,y}^\Phi[\tilde{Z}^{(2)}(0)^2]\big|+\big|\BE_{0}^\Phi[\tilde{Z}^{(2)}(0)]\big|^2\nonumber\\
            &\leq 2\kappa(y)r^2 + \tilde{c}r^{2-\frac{\gamma}{q}},\quad y\in\R^d.
        \end{align*}
        If $\kappa(y)>0$, choosing $r:=\kappa(y)^{-\frac{q}{\gamma}}$ yields that 
        \begin{equation*}
            \big|\BE_{0,y}^\Phi[\tilde{Z}(y)\tilde{Z}(0)]-\BE_0^\Phi[\tilde{Z}(0)]^2\big| \leq (2+\tilde{c}) \kappa(y)^{\frac{\gamma-2q}{\gamma}}.
        \end{equation*}
        If $\kappa(y)=0$, the bound is trivial. It remains to apply Proposition~\ref{p: random field low uniformity correlations} to obtain the main part of the assertion. Note that the moment condition~\eqref{e: random field moment condition} follows from~\eqref{e: mixing moment 2}.
    \end{proof}

For $p=0$ and with independence between $Z$ and $\Phi$, similar results have been derived before. However, our assumptions are basically the weakest. In~\cite[Corollary 1]{Flimmel_2025}, even if the perturbations are uniformly bounded, it is required that $\kappa(y)$ decays like $\|y\|^{-d-\vartheta}$ as $y\to\infty$ for some $\vartheta>0$. Here we only require a decay like $\|y\|^{-d+2-\vartheta}$ with some specific $\vartheta\in(0,2)$ depending on $\gamma$. This fact makes our result also more efficient than the one obtained in~\cite[Theorem  5.1]{Klatt_Last_Lotz_Yogeshwaran_2025} up to the basic moment condition~\eqref{e: mixing moment 3}. We can leverage this increased efficiency to obtain $p$-uniformity for some $p\in(0, 2]$ under the same assumptions that would simply yield beyond $0$-uniformity applying the previously derived theorems.

For a Gaussian displacement field independent of the source, the criterion from Proposition~\ref{p: random field low uniformity correlations} can be simplified differently to a condition on the decay of only the first-order correlations. Again, for the case $p=0$, similar conditions have been derived before, but the following are the most general. In~\cite[Theorem 2]{Flimmel_2025} and~\cite[Example 5.2]{Klatt_Last_Lotz_Yogeshwaran_2025}, it is required that the decay of the bound on the RHS of~\eqref{e: GRF covariance condition decay} is smaller by a factor of $\|y\|^{-2}$. The result we provide is basically sharp in the sense that for each $p\in[-d,2)$ there is a GRF (Gaussian random field) $Z$ such that $\Psi$ (with $\Phi:=\lambda_d$) is solely $p$-uniform and
\begin{equation*}
    \big\|\BC[Z(y), Z(0)]\big\| = O(\|y\|^{-(d+p)+2})
\end{equation*}
as $\|y\|\to\infty$. for $p=2$, the counterexample is not even $2$-uniform.

\begin{proposition}\label{p: gaussian mixing}
    Let $p\in[-d, 2]$. Assume that $Z$ is a Gaussian random field and assume that $Z$ and $\Phi$ are independent. Further, assume that
    \begin{equation}\label{e: GRF covariance condition}
        \int \frac{\big\|\BC[Z(y), Z(0)]\big\|}{1 + \|y\|^{2-p}}\,\alpha_\Phi(dy) < \infty,
    \end{equation}
    where the covariance is the full covariance matrix.
    Then $\Psi$ is $p$-uniform if $\Phi$ is, and if $p<2$, then beyond $p$-uniformity also transfers.
\end{proposition}
\begin{remark}
    In Proposition~\ref{p: gaussian mixing},~\eqref{e: GRF covariance condition} is implied by the existence of a strongly log-dominating function $\rho$ and a constant $c>0$ such that
    \begin{equation}\label{e: GRF covariance condition decay}
        \big\|\BC[Z(y), Z(0)]\big\| \leq c\|y\|^{-(d+p)+2}\rho(\|y\|)^{-1}, \quad y\in \R^d\setminus\{0\}.
    \end{equation}
\end{remark}
    \begin{proof}[Proof of Proposition~\ref{p: gaussian mixing}.]
        By Proposition~\ref{p: random field low uniformity correlations}, it suffices to show that there is a $c>0$ such that 
        \begin{equation}\label{e: gauss higher covariance bound}
            |\BC[Z_{j_1}(0)\cdots Z_{j_q}(0), Z_{j_1}(y)\cdots Z_{j_q}(y)]| \leq c \|\BC(Z(0), Z(y))\|, \quad y\in\R^d,
        \end{equation}
        for all $j\in\N_0^q, q\in\N, q<\tfrac{d+p}{2}$. Therefore, let $j\in\N_0^q, q\in\N, q<\tfrac{d+p}{2}$. Then, by Isserlis theorem, see~\cite{ISSERLIS_1918}, 
        \begin{equation*}
            \BC[Z_{j_1}(0)\cdots Z_{j_q}(0), Z_{j_1}(y)\cdots Z_{j_q}(y)] = \sum_{p\in P_{q}} \prod_{\{(k_1,l_1), (k_2, l_2)\}\in p} \BC[Z_{j_{k_1}}(l_1 y), Z_{j_{k_2}}(l_2 y)],
        \end{equation*}
        where $P_{q}$ is the set of partitions of $\{1,...,q\}\times\{0,1\}$ into pairs such that every partition contains at least one pair $\{(k_1,l_1), (k_2, l_2)\}$ with $l_1\neq l_2$. Since every summand contains at least one factor where $Z$ is evaluated one in $0$ and once in $y$, and since the remaining factors can then be bounded using the Cauchy-Schwarz inequality, there is a constant $c>0$ which depends on $q$ and $\|\BV[Z(0)]\|$ such that~\eqref{e: gauss higher covariance bound} holds. Finally, as there are only finitely many $q\in\N$ with respect to which the inequality has to hold, the constant $c>0$ can also be chosen independently of $q$.
    \end{proof}

From now on, we focus on the case that $Z$ and $\Phi$ are independent. In this case, we can directly calculate the Bartlett spectral measure of $Z\cdot \Phi$.

\begin{proposition}\label{p: ind random field convolution formula}
    Let $Z$ be a $\C$-valued random field on $\R^d$. Suppose that it has a finite second moment, and define $C:\R^d\to [0,\infty), y\mapsto \BC[Z(y), Z(0)]$. Further, assume that $Z$ and $\Phi$ are independent. Then, as $Z$ and $\Phi$ are independent,
    \begin{align}
        \beta_{Z\cdot\Phi} &= C \cdot \alpha_\Phi + \big|\BE[Z(0)]\big|^2\beta_\Phi,\\
        \hat\beta_{Z\cdot\Phi} &= (2\pi)^{-d}\hat{C} \ast \hat\alpha_\Phi + \big|\BE[Z(0)]\big|^2\hat\beta_\Phi.
    \end{align}
    \begin{proof}
        The formulas are a direct consequence of the more general formula~\eqref{e: beta transport formula}. We choose $K_x:= Z(x) \delta_x, x\in\R^d$ and use that $\Phi$ and $K$ are independent in this case.
    \end{proof}
\end{proposition}
Note that the second summand is $0$ in both formulas if $Z$ is centered.

Proposition~\ref{p: ind random field convolution formula} also shows how rare it is that $Z\cdot \Phi$ is beyond $0$-uniform, whereby Theorem~\ref{t: main theorem displacement} only yields that that $\Psi$ is $p$-uniform for some $p\leq2$ in most cases. We illustrate this fact for $d=1$ with the following proposition under very mild assumptions. The picture is similar for $d\geq2$, but the edge cases are harder to rule out neatly.

\begin{proposition}\label{p: d=1 p=2 limit example}
    Let $d=1$. Suppose that $Z$ is not deterministic, that it has a finite second moment, and define $C:\R\to [0,\infty), y\mapsto \BC[Z(0), Z(y)]$. Further, assume that $Z$ and $\Phi$ are independent, and that $\Phi$ is (pseudo-) ergodic and not $\infty$-uniform. Finally, either assume that there is an $a>0$ such that
    \begin{equation}\label{e:counterexample p=2 random field 1}
        \int e^{a|x|} |C(x)|\, dx <\infty,
    \end{equation}
    or assume that the following integral exists for some $n\in\N_0$, and that
    \begin{equation}\label{e:counterexample p=2 random field 2}
        \int x^{2n} C(x)\, dx \neq 0.
    \end{equation}
    Then $Z\cdot \Phi$ is not beyond $0$-uniform, and $\Psi$ is not beyond $2$-uniform.

    \begin{proof}
        First, we notice that~\eqref{e:counterexample p=2 random field 1} implies~\eqref{e:counterexample p=2 random field 2}, as $Z$ is not deterministic, and therefore $C\neq 0$. In turn,~\eqref{e:counterexample p=2 random field 2} implies that $\hat{C}$ is integrable and continuous, and that there is some $\varepsilon_0>0$ such that $\hat{C}(k)>0, k\in B_{2\varepsilon_0}\setminus\{0\}$. Further, because $\Phi$ is (pseudo-) ergodic and not\linebreak $\infty$-uniform, $\hat\beta_\Phi(B_{\varepsilon_0}\setminus\{0\}) = \hat\beta_\Phi(B_{\varepsilon_0})>0$, where we used Proposition~\ref{p:pseudo ergodic}. Hence, there is also an $\varepsilon_1\in(0,\tfrac{\varepsilon_0}{2})$ such that $\hat\beta_\Phi(B_{\varepsilon_0}\setminus B_{2\varepsilon_1})>0$. By continuity of $\hat{C}$, there is a $c>0$ such that $\hat{C}(k)\geq c, k\in B_{\varepsilon_0+\varepsilon_1}\setminus B_{\varepsilon_1}$. In total, we obtain that 
        \begin{equation*}
            (\hat{C}\ast\hat\alpha_\Phi)(k) = \int \hat{C}(k-s) \,\hat\alpha_\Phi(ds) \geq c\hat\beta_\Phi(B_{\varepsilon_0}\setminus B_{2\varepsilon_1}) >0, \quad k\in B_{\varepsilon_1}.
        \end{equation*}
        By the formula for $\hat\beta_{Z\cdot\Phi}$ from Proposition~\ref{p: ind random field convolution formula}, we can conclude that $\varepsilon^{-d}\hat\beta_{Z\cdot\Phi}(B_\varepsilon)$ is bounded from below, as $\varepsilon\to0$. Hence, $Z\cdot\Phi$ is not beyond $0$-uniform.
        
        For the final part of the proof, we can assume that $Z$ is centered, i.e., $\BE[Z(0)]=0$, as centering $Z$ does not change the distribution of $\Psi$. Let $f\in\cC_c^\infty(\R, [0,\infty))$ with $f\neq0$, and recall Definition~\ref{d: f_r}. From the more precise statement in Theorem~\ref{t:main theorem formula} compared to Theorem~\ref{t: main theorem displacement}, we know that, as $r\to\infty$,
        \begin{align}\label{e: covariance equation p=2 counterexample random field}
            \BV[\Psi(f_r)] &= \BV[\Phi(f_r) + (Z\cdot\Phi)(f_r')] + o(r^{-1})\nonumber\\
            &= \BV[\Phi(f_r)] + 2\BC[\Phi(f_r), (Z\cdot\Phi)(f_r')]  + \BV[(Z\cdot\Phi)(f_r')] + o(r^{-1}).
        \end{align}
        From the first part of the proof and Proposition~\ref{p:asymptotic variance test function zero integral} we know that\linebreak $\BV[(Z\cdot\Phi)(f_r')]=r^{-2}\BV[(Z\cdot\Phi)(f'(\tfrac{\cdot}{r}))]\neq o(r^{-1})$. Hence, if $\Psi$ was beyond $2$-uniform, whereby $\BV[\Psi(f_r)]=o(r^{-1})$ by Theorem~\ref{t:asymptotic variance test function}, the covariance term on the right-most side of~\eqref{e: covariance equation p=2 counterexample random field} would need to cancel the part which is not $o(r^{-1})$ since a variance cannot be negative. However, since $Z$ is centered and independent of $\Phi$, we know that the intensity of $Z\cdot\Phi$ is $0$, and
        \begin{align}
            \BC[\Phi(f_r), (Z\cdot\Phi)(f_r')] = \BE[\Phi(f_r) (Z\cdot\Phi)(f_r')]\nonumber
            =\BE[Z(0)]\BE[\Phi(f_r) \Phi(f_r')] = 0,\quad r>0.
        \end{align}
        Actually, one can even show that both factors on the right-most side are $0$.
        Thus, the cancellation is impossible, and $\Psi$ cannot be beyond $2$-uniform.
    \end{proof}
\end{proposition}

Let us now come back to the Gaussian case. Then we can construct examples such that $\Psi$ is $4$-uniform. However, at least if $d=1$, we can rigorously show that it is impossible that $\Psi$ is beyond $4$-uniform under very mild assumptions. In the physics literature, it was already established in~\cite{Gabrielli_2004} that $4$-uniformity is usually the limit in this case, but the precise conditions were not derived yet.

\begin{proposition}
    Let $d=1$. Suppose that $Z$ is a Gaussian random field, and define \linebreak $C:\R\to [0,\infty), y\mapsto \BC[Z(0), Z(y)]$. Assume that $\Phi\neq0$, and that $Z$ and $\Phi$ are independent. Further, assume that $\hat{C}$ is not purely discrete, which, e.g., is the case if $Z$ is ergodic and nondeterministic, or if $\limsup_{x\to\infty}|C(x)|<C(0)$. Then $Z^2\cdot\Phi$ is not beyond $0$-uniform and $\Psi$ is not beyond $4$-uniform.
    
    More precisely, if $Z$ is ergodic, $\Phi$ is solely $p_0$-uniform, $Z\cdot\Phi$ solely $p_1$-uniform, and $Z^2\cdot\Phi$ solely $p_2$-uniform, then $\Psi$ is solely $\min(p_0, p_1+2, p_2+4)$-uniform. We always have $p_2\geq 0$, and most of the time also $p_1\geq0$; see Proposition~\ref{p: d=1 p=2 limit example}. However, if, e.g., $\Phi = \lambda_1$ or $\Phi=\Z+U, U\sim\cU([0,1))$, then $p_1<0$ is possible, and we can even choose $Z$ such that $p_1$ exactly matches any value below $0$ while keeping $p_2=0$. In this way, $p$-uniformity for any $p\in[-1, 4]$ is obtainable exactly with an appropriate choice of $Z$ and $\Phi$.

    \begin{proof}
        Before we start with the main part of the proof, we show that $\hat{C}$ is not purely atomic if $Z$ is either ergodic and nondeterministic, or if $\limsup_{x\to\infty}|C(x)|<C(0)$. In both cases, we can assume that $Z$ is centered without loss of generality. In the first case, we can derive that $Z^2$ must be ergodic as well, and hence by Proposition~\ref{p:pseudo ergodic} and Isserlis theorem,
        \begin{equation*}
            0 = \hat\beta_{Z^2\cdot\lambda_1}(\{0\}) = \pi^{-1} (\hat{C}\ast\hat{C})(\{0\}) = \pi^{-1}\sum_{x\in\R} \hat{C}(\{x\})^2;
        \end{equation*}
        see~\cite{ISSERLIS_1918}.
        As $Z$ is nondeterministic, $\hat{C}\neq0$, whereby it must have a non-trivial non-atomic part. In the second case, we can leverage the fact that $\hat{C}$ is a finite measure, and if it were purely discrete, it would be almost periodic. Thereby, for every $\varepsilon>0$, there is an $l>0$ such that $|C(x+l)-C(x)|<\varepsilon, x\in\R$; see~\cite{Bohr_1925}. Note, however, that there is no equivalence, as the $\limsup$ can be equal to $C(0)$, even if $\hat{C}$ has no discrete component. This can happen with specific singular continuous measures, e.g., if $\hat{C}$ comes from $C(x)=\prod_{n=1}^\infty\cos(2^{-n^2}x)^2, x\in\R$.
        
        Let us now come to the main part of the proof. Choose $\mu\neq0$ such that it has bounded support and such that it is dominated by the non-atomic part of $\hat{C}$. From Proposition~\ref{p: ind random field convolution formula}, we can derive that
        \begin{equation}\label{e: beta Z^2Phi inequality}
            \hat{\beta}_{Z^2\cdot\Phi} \geq 2^{-1}\pi^{-2}\hat\alpha_\Phi(\{0\})(\hat{C}\ast\hat{C})\geq 2^{-1}\pi^{-2}\hat\alpha_\Phi(\{0\}) (\mu\ast\mu).
        \end{equation}
        Choose $\tilde{Z}$ as an invariant centered GRF with correlation function $\tilde{C}$ such that $\hat{\tilde{C}}=\mu$. Since $\Phi\neq0$, whereby $\hat\alpha_\Phi(\{0\})>0$, we can conclude from~\eqref{e: beta Z^2Phi inequality} that $Z^2\cdot\Phi$ is at most as $p$-uniform as $\tilde{Z}^2$. We obtain that the random field $\tilde{Z}^2$ is not beyond $0$-uniform from Theorem~\ref{t:asymptotic variance test function} since \linebreak $\BC[\tilde{Z}(0)^2, \tilde{Z}(y)^2]= 2 \tilde{C}(y)^2, y\in\R$ by Isserlis formula, see~\cite{ISSERLIS_1918}, and therefore
        \begin{equation*}
            \frac{\BV[(\tilde{Z}^2\cdot\lambda_1)(B_r)]}{r^d} = 2\int \max(1-\tfrac{|y|}{r}, 0) \tilde{C}(y)^2 \, dy \xrightarrow{r\to\infty} 2\int \tilde{C}(y)^2\, dy>0.
        \end{equation*} 
        We continue with the analysis of $\Psi$. Let $f\in \cC_c^\infty(\R, [0,\infty))$ with $f\neq 0$, and recall Definition~\ref{d: f_r}. If we apply the more precise Theorem~\ref{t:main theorem formula} instead of Theorem~\ref{t: main theorem displacement}, we obtain that, as $r\to\infty$,
        \begin{equation*}
            \BV[\Psi(f_r)] = \BV\big[\Phi(f_r) + (Z\cdot\Phi)(f_r') + (Z^2\cdot\Phi)(f_r'')\big] + o(r^{-3}).
        \end{equation*}
        As in Proposition~\ref{p: d=1 p=2 limit example}, it remains to analyze the covariances. 
        Using the independence of $Z$ and $\Phi$ and partial integration, we obtain that, as $r\to\infty$,
        \begin{align*}
            \BC\big[\Phi(f_r),(Z^2\cdot\Phi)(f_r'')\big] &= \int (f_r\star f_r'')(y) \BE[Z(y)^2]\, \beta_\Phi(dy)\nonumber\\
            &= -C(0)\BV[\Phi(f_r')]\nonumber\\
            &= O(\BV[\Phi(f_r)]r^{-2}), 
        \end{align*}
        applying Theorem~\ref{t:asymptotic variance test function} and Remark~\ref{r: f_r formulas} in the last step. Similarly, partial integration yields that $f_r\star f_r'$ and $f_r'\star f_r''$ are point symmetric around zero, whereby 
        \begin{equation*}
            \BC\big[\Phi(f_r),(Z\cdot\Phi)(f_r')\big] = \BC\big[(Z\cdot\Phi)(f_r'),(Z^2\cdot\Phi)(f_r'')\big] = 0,\quad r>0.
        \end{equation*}
        Hence, as $r\to\infty$,
        \begin{equation}\label{e: gaussian pert variance equation}
            \BV[\Psi(f_r)] = \BV[\Phi(f_r)] + \BV[(Z\cdot\Phi)(f_r')] + \BV[(Z^2\cdot\Phi)(f_r'')] + o(r^{-3}+\BV[\Phi(f_r)]).
        \end{equation}
        From this equation, we can obtain that $\Psi$ is not beyond $4$-uniform, as for $r>0$, by~\eqref{e: beta Z^2Phi inequality},
        \begin{align*}
            r^3\BV[(Z^2\cdot\Phi)(f_r'')] &= (2\pi)^{-d}\int |rk|^4|\hat{f}(rk)^2|\, \hat\beta_{Z^2\cdot\Phi}(dk)\nonumber\\
            &\geq (2\pi)^{-d}\hat\alpha_\Phi(\{0\}) \int |rk|^4|\hat{f}(rk)^2|\, (\mu\star\mu)(dk)\nonumber\\
            &= \hat\alpha_\Phi(\{0\}) \BV[(\tilde{Z}^2\cdot\lambda_1)(f''(\tfrac{\cdot}{r}))],
        \end{align*}
        and $\tilde{Z}^2$ is pseudo-ergodic, whereby Proposition~\ref{p:asymptotic variance test function zero integral} applies.

        Now, we can come to the last part of the proof, where we assume that $Z$ is ergodic, that $\Phi$ is solely $p_0$-uniform, $Z\cdot\Phi$ solely $p_1$-uniform, and $Z^2\cdot\Phi$ solely $p_2$-uniform. The fact that $\Psi$ is $\min(p_0, p_1+2, p_2+4)$-uniform directly follows from~\eqref{e: gaussian pert variance equation} and Remark~\ref{r: f_r formulas}. It remains to show that $\Psi$ is not beyond $\min(p_0, p_1+2, p_2+4)$-uniform. However, this fact also follows from~\eqref{e: gaussian pert variance equation}. If $p_0=-1$, it is clear that $\Psi$ is also solely $-1$-uniform, and if $p_0>-1$, then $\Psi$ is pseudo-ergodic. This property also transfers to $Z\cdot\Phi$ and $Z^2\cdot\Phi$ by Proposition~\ref{p: ind random field convolution formula}, as $Z$ is assumed to be ergodic, whereby $\hat{C}$ and $\hat{C}\ast\hat{C}$ do not have any atoms, and as $\BC[Z(0)^2, Z(y)^2]=2 C(y)^2 + 4 \BE[Z(0)]^2C(y)$ by Isserlis formula; see~\cite{ISSERLIS_1918}. Hence, the assertion follows with Proposition~\ref{p:asymptotic variance test function zero integral}.
    \end{proof}
\end{proposition}

The last two propositions suggest that Theorem~\ref{t: main theorem displacement}, at least in the independent case, is only applicable up to $4$-uniformity in most cases. While we were able to construct many examples for the more general transport kernels in subsections~\ref{ss: pp from invariant partitions} and~\ref{ss: pp from averaging sets}, we are not aware of any interesting examples such that $Z^q\cdot\lambda_d$ is $p$-uniform for high $p$ for all $q\in\{0,...,n\}$ and $n\geq2$.

\subsection{Lattices at positive temperature}\label{ss: Lattices at positive temperature}

As it has been shown in both physicists (see, e.g.,~\cite[Section V]{Kim_Torquato_2018}) and mathematics (see~\cite{Ruelle_1970, Dereudre_Flimmel_2024}), compressible particle systems are not hyperuniform, i.e., beyond $0$-uniform, at positive temperature. Physicists often model crystals at positive temperature as a lattices with displacements, known as phonons. Here, we extended the results from a mathematical perspective. 

We consider a point process and view the positive temperature state as a random displacement of some zero-temperature state. If these displacements were independent of the initial state and each other, like it is assumed in the so-called Einstein model~\cite{Einstein_1907}, the temperature could only affect $2$-uniformity; see Example~\ref{ex: general point cluster}. However, in the newer Debye model~\cite{Debye_1912}, the displacements are allowed to be correlated, which changes the picture.

More precisely, we consider a model from~\cite[Section V]{Kim_Torquato_2018}, where the lattice is modeled as a point process with springs along the lattice edges. There, it was already shown that this system, unlike the perfect lattice, is not beyond $0$-uniform. However, we add our results to their technique to make the proof mathematically tight. To this end, we consider only the case $d\geq 3$, as in lower dimensions one can no longer model a thermalized crystal as a stationarily displaced perfect crystal with a typical displacement $Z$ that satisfies $\BE[\|Z(0)\|^d] < \infty$; see, e.g.,~\cite[Section 137]{Landau_Lifshitz_2013}. This integrability is a fundamental assumption to make our techniques applicable; see Condition~\ref{c:condition square-integrable Markov}. On the other hand, the techniques used here should apply not only to the lattice model we consider but also to other Debye solids for $d\geq3$. In this context, the central assumption is that the dispersion relation is linear in $0$ asymptotically.

Coming back to the example at hand, in~\cite[Section V]{Kim_Torquato_2018} it was derived that if \linebreak $\Phi:=\Z^d+U, U\sim\cU([0,1)^d),$ is the stationary lattice, the transport to its positive temperature state can be facilitated by a transport kernel $K$ of the form
\begin{equation}
    K^\ast(x) := \delta_{Z(\lfloor x-U\rfloor)}, \quad x\in\R^d,
\end{equation}
where $Z$ is an $\R^d$-valued centered stationary random field on $\Z^d$ with uncorrelated components, where each component has the same correlation function $C$ that satisfies 
\begin{equation}\label{e: C thermalized crystal}
    \hat{C}(k) = \bigg(4\kappa\sum_{j=1}^d \sin\big(\tfrac{k_j}{2}\big)^2\bigg)^{-1},\quad k\in [0,2\pi)^d\setminus\{0\},
\end{equation}
where $\kappa>0$ is the spring constant used. To allow for the calculation of the correlation of the second and higher moments of $Z$, which we need for $d\geq 5$, we additionally assume that $Z$ is a Gaussian random field. Recall Definition~\ref{d: f_r}.

\begin{proposition}
    Under the assumptions previously made, $K\Phi$, i.e., the lattice at positive temperature, is solely $0$-uniform. In particular, for $f\in \cC_c^\infty$,
    \begin{equation}
        \frac{\BV[K\Phi(f_r)]}{r^d} \xrightarrow{r\to\infty} \frac{\|f\|_2}{\kappa} > 0.
    \end{equation}
    If $\|f\|_2=1$, the above limit corresponds to the structure factor in $0$.
    \begin{proof}
        Since $Z$ is a Gaussian random field, Condition~\ref{c:condition square-integrable stealthy} is fulfilled for any $\varepsilon>0$. Therefore, $\Psi_q$, as defined in Theorem~\ref{t:main theorem}, is well defined for any $q\in\N_0$. For $q=0$, the analysis is simple since $K$ is a probability kernel and therefore $\Psi_0=\Phi$, whereby it is $\infty$-uniform. We continue with $q=1$, which turns out to contribute the leading order term in the approximation. Since the $d$ components of the vector-valued random measure are uncorrelated, it suffices to analyze the individual components, which we call $\Psi_{1,j}$ for $l\in\{1,...,d\}$. From the assumptions and Proposition~\ref{p: ind random field convolution formula}, we directly obtain that
        \begin{equation}
            \hat\beta_{\Psi_{1,l}}(dk) = \bigg(4\kappa\sum_{j=1}^d \sin\big(\tfrac{k_j}{2}\big)^2\bigg)^{-1}\, dk.
        \end{equation}
        Note that, as $k\to 0$,
        \begin{equation}\label{e: C thermalized crystal asym}
            4\kappa\sum_{j=1}^d \sin\big(\tfrac{k_j}{2}\big)^2 = \kappa \|k\|^2 + O(\|k\|^4).
        \end{equation}
        Suppose that $f\in\cC_c^\infty$. Then, $f$ is differentiable, and as $r\to\infty$,
        \begin{align*}
            \frac{\BV[\Psi_{1,l}(\partial_l f_r)]}{r^{d}} &= (2\pi)^{-d}r^d\int |k_l|^2|\hat{f}(rk)|^2\, \hat\beta_{\Psi_{1,1}}(dk)\nonumber\\
            &= (2\pi)^{-d}r^d\int \frac{|k_l|^2|\hat{f}(rk)|^2}{4\kappa\sum_{j=1}^d \sin\big(\tfrac{k_j}{2}\big)^2}\, dk\nonumber\\
            &= (2\pi)^{-d}r^d\int \frac{|k_l|^2|\hat{f}(rk)|^2}{\kappa\|k\|^2}\, dk + o(1)\nonumber\\
            &= (2\pi)^{-d}\int \frac{|k_l|^2|\hat{f}(k)|^2}{\kappa\|k\|^2}\, dk + o(1),
        \end{align*}
        where we use in the third step that
        \begin{equation*}
            \int_{[0,2\pi)^d} \frac{1}{\sum_{j=1}^d \sin\big(\tfrac{k_j}{2}\big)^2}\, dk < \infty,
        \end{equation*}
        as $d\geq 3$.
        Hence,
        \begin{equation*}
            \frac{\BV\big[\Psi_{1}(f_r')\big]}{r^{d}} = \sum_{l=1}^d\frac{\BV[\Psi_{1,l}(\partial_l f_r)]}{r^{d}} \xrightarrow{r\to\infty} \frac{\|f\|_2}{\kappa},
        \end{equation*}
        since by Plancherel's theorem $(2\pi)^{-d}\|\hat{f}\|_2 = \|f\|_2$. It remains to show that $\Psi_q$ is beyond \linebreak $(-2q)$-uniform for $2 \leq q < \frac{d}{2}$. Then the assertion follows from an application of Theorem~\ref{t: main theorem rate}. As there is nothing to show for $d\leq 4$, let us assume $d\geq 5$, and choose $q\geq2$, $q_1, ..., q_d\in\{0,...,q\}$ such that $q_1+...+q_q=q$. Since $Z$ is Gaussian and centered, one can calculate using Isserlis formula, see~\cite{ISSERLIS_1918}, that there are coefficients $a_0,..., a_{\lceil\frac{q}{2}\rceil-1}\geq 0$ such that
        \begin{equation*}
            \BC[Z_1(0)^{q_1}\cdots Z_d(0)^{q_d}, Z_1(y)^{q_1}\cdots Z_d(y)^{q_d}] = \sum_{m=0}^{\lceil\frac{q}{2}\rceil-1} a_m C(y)^{q-2m}, \quad y\in\Z^d.
        \end{equation*}
        Hence by Proposition~\ref{p: ind random field convolution formula}, the $(q_1,...,q_d)$-component $\Psi_{q;q_1,...,q_d}$ of $\Psi_q$ satisfies
        \begin{equation}
            \hat\beta_{\Psi_{q;q_1,...,q_d}}(dk) = \sum_{m=0}^{\lceil\frac{q}{2}\rceil-1} a_m (2\pi)^{-(q-2m-1)d}\hat{C}^{\ast(q-2m)}(k), \quad k\in(-\pi, \pi)^d,
        \end{equation}
        where $\hat{C}^{\ast(q-2m)}$ is the $(q-2m)$-times convolution on the $d$-torus of $\hat{C}$ with itself. In~\eqref{e: C thermalized crystal}, one can see that $\hat{C}$ has exactly one pole and is otherwise bounded. This pole lies in $0$ and has degree $2$ by~\eqref{e: C thermalized crystal asym}. Therefore, and because $d\geq 5$, the function $\hat{C}$ is square-integrable and $\hat{C}^{\ast m}$ is bounded for $m\geq 2$. For even $q$, this fact directly implies that $\Psi_{q;q_1,...,q_d}$ is $0$-uniform. For odd $q$, one can add what was shown in the case $q=1$ to establish that $\Psi_{q;q_1,...,q_d}$ is \linebreak$(-2)$-uniform. In every case, since $q_1,...,q_d$ were arbitrary, $\Psi_q$ is $(-2)$-uniform, and therefore also beyond $(-2q)$-uniform, which concludes the proof.
    \end{proof}
\end{proposition}

\section{Prospects}\label{s: prospects}

In this paper, we have focused on conditions under which $p$-uniformity is preserved under transport, and have then applied these results to answer open questions. A further question is whether $p$-uniform random measures are close to the perfectly uniform Lebesgue measure in terms of some transport cost. If one considers a transport as having low cost whenever it fulfills the conditions of Theorem~\ref{t:main theorem}, then the question is trivial by the following proposition.

\begin{proposition}\label{p:main theorem inversion}
Let $p\in[-d,\infty]$. Suppose that $\Phi$ is an invariant locally absolutely square-integrable random complex measure which is (beyond) $p$-uniform. Then there exists an invariant transport kernel $K$ such that $\Phi = K\lambda_d$, and such that $K$ and $\lambda_d$ satisfy the conditions of Theorem~\ref{t:main theorem}.
\begin{proof}
Such a transport kernel $K$ is explicitly constructed in Appendix~\ref{ss:Reduction to Lebesgue}. There, it is named $L$, and we also do so in this proof. As in~\eqref{e:L}, we choose it as
\begin{equation}\label{e:L2}
    L(x,B) := \kappa_d^{-1}\Phi(B\cap(B_1+x)), \quad x\in\R^d, B\in\cB^d.
\end{equation}
By Lemma~\ref{l: simplification to lebesgue}, it is invariant and $\Phi=L\lambda_d$. It also trivially fulfills Condition~\ref{c:condition square-integrable stealthy} with respect to $\Phi$ and any $\varepsilon>0$. To apply Theorem~\ref{t:main theorem}, it remains to analyze the components of $\Psi_q$ for $q\in\N_0$ as defined in~\eqref{e:Psi_q definition}. Let $q\in\N_0$ and suppose that $q_1,...,q_d\in\{0,...,q\}$ such that $q_1+\dots +q_d=q$. Then,
\begin{equation*}
    \int x_1^{q_1}\cdots x_d^{q_d}\, L^\ast(y, dx) = \frac{1}{\kappa_d}\int_{B_1+y}(x_1-y_1)^{q_1}\cdots(x_d-y_d)^{q_d} \, \Phi(dx).
\end{equation*}
Therefore, for $f\in\cC_c^\infty$ the $(q_1,...,q_d)$-component $\Psi_{q;q_1,...,q_d}$ of $\Psi_q$ satisfies
\begin{align*}
    \Psi_{q;q_1,...,q_d}(f) &= \int g(y) \int x^q\, L^\ast(y, dx) \, dx\nonumber\\
    &= \frac{1}{\kappa_d}\iint_{B_1+x} f(y) (x_1-y_1)^{q_1}\cdots(x_d-y_d)^{q_d} \, dy \, \Phi(dx).
\end{align*}
The inner integral is a convolution of $f$ with a bounded function with bounded support, and therefore is in $\cC_c^\infty$ as a function of $x$. Now, an application of Theorem~\ref{t:asymptotic variance test function} and Theorem~\ref{t: stealthy characterization} yields that if $\Phi$ is (beyond) $p$-uniform (with some radius), then $\Psi_{q;q_1,...,q_d}$ also is, which concludes the proof.
\end{proof}
\end{proposition}

Therefore, Theorem~\ref{t:main theorem} and Proposition~\ref{p:main theorem inversion} together yield a characterization of $p$-uniformity in terms of an invariant transport with source $\lambda_d$.

\begin{theorem}
    Let $\Phi$ be a random complex measure. Then $\Phi$ is invariant, locally absolutely square-integrable, and (beyond) $p$-uniform iff there exists an invariant transport kernel $K$ such that $\Phi = K\lambda_d$, and such that $K$ and $\lambda_d$ satisfy the conditions of Theorem~\ref{t:main theorem}.
\end{theorem}

This characterization of $p$-uniformity shows that the conditions in Theorem~\ref{t:main theorem} are well-chosen. Still, it is not that powerful since the transport constructed in Proposition~\ref{p:main theorem inversion} is not mass-preserving and does not allow for the construction of a spatial interpolation like in Subsection~\ref{ss:Persistence of uniformity under spatial interpolation}. If, instead, we restrict the kernel $K$ to be a probability kernel, then our moment conditions on the transport are closely connected to the Wasserstein distance; see Definition~\ref{d: Wasserstein distance}. In~\cite{Lachièze-Rey_Yogeshwaran_2024, Butez_Dallaporta_García-Zelada_2024, Huesmann_Leblé_2026}, it was shown that for $d=2$, up to edge cases, all hyperuniform random measures have a finite $2$-Wasserstein distance to the Lebesgue measure. However, for $d\neq2$ the picture is different. For $d=1$, $p$-uniformity is not sufficient for a finite $1$-Wasserstein distance to the Lebesgue measure if $p<1$; see~\cite{Elboim_Spinka_Yakir_2025}. For $d\geq3$, the converse fails since not even a $\infty$-Wasserstein distance to the Lebesgue measure is close to implying hyperuniformity; see~\cite[Theorem 1]{Dereudre_Flimmel_Huesmann_Leblé_2024} or Proposition~\ref{p: hyperfluctuating points in lattice}. For this reason, we propose the following distance to (almost) characterize $p$-uniformity. It involves strong restrictions on higher-order moments if the parameter is large. We restrict ourselves to the distance to the Lebesgue measure because the sharpness of Proposition~\ref{p: wasserstein conditions general} shows that only very regular invariant random measures are compatible with a condition on a Wasserstein-type distance.

\begin{definition}
    Let $p\geq 1$. Suppose that $\Phi$ is an invariant random (nonnegative) measure with finite intensity $\gamma$. Then 
    \begin{align}
        d_p(\Phi, \gamma\lambda_d) := \inf\bigg\{\BE\bigg[\int\|x\|^p\, &(K_0+L_0)(dx)\bigg] : \nonumber\\ 
        &L,K \text{ invariant probability kernels}, \nonumber\\
        &\Phi = \gamma K\lambda_d, \lambda_d = L\lambda_d,\nonumber\\
        &\int x^{\otimes q}\, (K_0-L_0)(dx)=0\text{ for } q\in\N, q\leq\frac{p}{2}-1\bigg\}.
    \end{align}
    Being even more restrictive, one could also define
    \begin{align}
        \tilde{d}_p(\Phi, \gamma\lambda_d) := \inf\bigg\{\BE\bigg[\int\|x\|^p\, &K_0(dx)\bigg] : \nonumber\\ 
        &K \text{ invariant probability kernel}, \Phi = \gamma K\lambda_d, \nonumber\\
        &\int x^{\otimes q}\, K_0(dx)\text{ deterministic for } q\in\N, q\leq\frac{p}{2}-1\bigg\}.
    \end{align}
\end{definition}

It is easy to prove that $d_p(\Phi, \gamma\lambda_d)\leq 2\tilde{d}_p(\Phi, \gamma\lambda_d)$.
Using this or a similar definition, we expect a conjecture of the following kind to hold:

\begin{conjecture}
    Let $p\geq -d+2$. Suppose that $\Phi$ is an invariant locally absolutely square-integrable random (nonnegative) measure with intensity $\gamma$. Then $d_{d+p}(\Phi, \gamma\lambda_d)<\infty$ if $\Phi$ is 
    $(p+\vartheta)$-uniform for some $\vartheta>0$.
\end{conjecture}

Compared to the preceding Proposition~\ref{p:main theorem inversion}, this conjecture is no complete inversion of Theorem~\ref{t:main theorem}, as $d_{d+p}(\Phi, \gamma\lambda_d)<\infty$ only implies that $\Phi$ is beyond $(2\lfloor \frac{d+p}{2}\rfloor-d)$-uniform for $p\geq0$, and for $p<0$ not even local square integrability of $\Phi$ is guaranteed. However, a complete inversion is actually not possible under these restrictions.

Let us now motivate this conjecture. For $d=1$, it was shown in~\cite{Elboim_Spinka_Yakir_2025} that $p$-uniformity does not suffice to conclude that $d_{d+p}(\Phi, \gamma\lambda_d)<\infty$ for $p<1$. On the other hand, in~\cite{Butez_Dallaporta_García-Zelada_2024} it was shown that $(-d+2+\vartheta)$-uniformity implies $d_{2}(\Phi, \gamma\lambda_d)<\infty$ for $d\geq 2, \vartheta>0$. For $d=1$, their bound involves a logarithmic term that one might be able to avoid via a smoother test function, which would extend the applicability of their result to $d=1$. For the special case of $d=2$, the same was also shown in~\cite{Lachièze-Rey_Yogeshwaran_2024, Huesmann_Leblé_2026}. For higher $d+p$, there are no results available in the literature so far, since a distance of our type, to the best of our knowledge, has not been considered yet, and a classical Wasserstein distance no longer suffices at that point, as discussed above. Still, we can prove a weaker version of the conjecture, where we assume that $\Phi$ is not only $p$-uniform for some $p<\infty$ but also $\infty$-uniform.

\begin{theorem}\label{t:main theorem stealthy markov inversion}
Suppose that $\Psi$ is an invariant locally absolutely square-integrable random measure of intensity $\gamma$ which is $\infty$-uniform. Then there exist invariant transport kernels $K,L$ such that $K,L$ are probability kernels, $\Psi = \gamma K\lambda_d$, $\lambda_d=L\lambda_d$, and
\begin{align}
    \BE\bigg[\int \|x\|^q\, (K_0+L_0)(dx)\bigg] &<\infty,\\
    \int x^{\otimes q}\, (K_0-L_0)(dx) &= 0,\quad q\in\N.
\end{align}
We can even choose $L$ to be deterministic, which implies that $\int x^q\, K_0(dx), q\in\N$ is also deterministic. Hence, $d_{p}(\Phi, \gamma\lambda_d)\leq 2\tilde{d}_{p}(\Phi, \gamma\lambda_d)<\infty$ for all $p\in[1,\infty)$.
\begin{proof}
Let $\varepsilon>0$ such that $\Phi$ is $\infty$-uniform with radius $\varepsilon$. Further, we can assume that $\gamma>0$, as the case $\gamma=0$ is trivial.
First, we choose $g:\R^d\to\C$ such that $\hat{g}\in\cC^\infty$ and $\hat{g}(k)=0$ for $k\in\R^d$ with $\|k\|\geq \frac{\varepsilon}{2}$. Then we chose $f:= \lambda_d(|g|^2)^{-1}|g|^2$, where we note that the Paley-Wiener theorem implies that $\lambda_d(|g|^2)<\infty$; see, e.g.,~\cite[Chapter III Section 4]{Stein_Weiss_1971}. By construction, $f$ is nonnegative and satisfies $\hat{f}\in\cC^\infty$ with $\hat{f}(k)=0$ for $k\in\R^d$ with $\|k\|\geq \varepsilon$, and $\lambda_d(f)=1$. 
Now we can define the transport kernels $K,L$ by
\begin{align}
    K(x,B) &:= \gamma^{-1} \int_B f(x-y) \, \Psi(dy), \\
    L(x,B) &:= \int_B f(x-y) \, dy, \quad x\in\R^d, B\in\cB^d.
\end{align}
We can see that $K$ is invariant, as
\begin{align*}
    K(\theta_z,x,B) &= \gamma^{-1} \int_B f(x-y) \, \Psi(\theta_z,dy) \nonumber\\
    &= \gamma^{-1} \int_{B+z} f(x+z-y) \, \Psi(dy) \nonumber\\
    &= K(x+z,B+z),\quad x,z\in\R^d, B\in\cB^d.
\end{align*}
The invariance of $L$ follows in the same way. Further, we can calculate that
\begin{equation*}
    \BE\bigg[\int \|x\|^q\, K_0(dx)\bigg]= \gamma^{-1} \BE\bigg[\int\|x\|^q f(x)\, \Psi(dx)\bigg] = \int\|x\|^q f(x)\, dx < \infty,
\end{equation*}
where we used that $f(x)$ decays faster than any polynomial as $\|x\|\to\infty$ by the Paley-wiener theorem; see, e.g.,~\cite[Chapter III Section 4]{Stein_Weiss_1971}. The same holds for $L_0$ instead of $K_0$. Additionally, we have that
\begin{equation*}
    \BV\bigg[\int x^{\otimes q}\, K_0(dx)\bigg] =\gamma^{-2} \BV\bigg[\int \underbrace{x^{\otimes q} f(x)}_{f_{q}(x):=}\, \Psi(dx)\bigg] = \gamma^{-2}\int |\hat{f}_{q}(k)|^2\, \hat\beta_\Psi(dk)
    =0,
\end{equation*}
for $q\in\N_0$,where we used that $\hat{f}_{q}=(\hat{f})^{(q)}$, and that $\hat{f}(k)=0$ for $k\in\R^d$ with $\|k\|\geq \varepsilon$. Combined with the fact that 
\begin{equation*}
    \BE\bigg[\int x^{\otimes q}\, K_0(dx)\bigg] =\gamma^{-1} \BE\bigg[\int x^{\otimes q} f(x)\, \Psi(dx)\bigg]= \int x^{\otimes q} f(x)\, dx = \int x^{\otimes q}\, L_0(dx),
\end{equation*}
for $q\in\N_0$,
this equation yields that
\begin{equation*}
    \int x^{\otimes q}\, K_0(dx) = \int x^{\otimes q}\, L_0(dx),\quad q\in\N_0,
\end{equation*}
where both sides are deterministic and for $q=0$ also equal to $1$.
Finally, Fubini's theorem yields
\begin{equation*}
    \gamma K\lambda_d(B) = \iint_B f(y-x) \, \Psi(dy)\, dx = \Psi(B),
\end{equation*}
and in the same way $\lambda_d=L\lambda_d$, which concludes the proof.
\end{proof}
\end{theorem}

In particular, we obtain that every invariant locally absolutely square-integrable random measure which is $\infty$-uniform has a finite $p$-Wasserstein distance to the Lebesgue measure for any $p\geq 1$. If the random measure is a point process, we can derive from this fact that it is an $L^p$-perturbed lattice for any $p\geq1$; see~\cite{Huesmann_2016}. However, at least for $d\geq3$, this fact is not surprising, since even the Poisson process exhibits this property there; see~\cite{Huesmann_Sturm_2013}.

Analogous transports can also be constructed if one replaces $\gamma\lambda_d$ by the invariant lattice $\gamma^{-1}(\Z^d+U)$ with $U\sim\cU([0,1)^d)$, where $\gamma^{-1}$ is not a density but a stretching factor. In this case, we can choose $L$ to depend only on $U$.

\ifsplit

    %

\else


    \appendix

\appendix
\section{Auxiliary results}
In this section, we provide results which are fundamental to the proofs of the main theorems from Section~\ref{s:main results}.
\subsection{Taylor approximation}
The following two results provide the two ways we approximate or represent our test functions. The first purely relies on the classical Taylor approximation. Recall Definition~\ref{d: f_r} and that $f^{(q)}$ denotes the (possibly tensor-valued) $q$-th derivative of $f\in\cC^q$ for $q\in\N_0$.
\begin{lemma}\label{l:taylor}
Let $\gamma\geq 0$. Suppose that $f\in C^{\lceil \gamma \rceil}(\R^d)$ with $f(x) = 0$ for all $x\in B_1^c$. Then there are mappings $R: (0,\infty) \times \R^d \times  \R^d \to \C$ and $g: (0,\infty) \times \R^d \to [0,1]$ such that for all $r>0,x,y\in\R^d$,
\begin{align}
    f_r(y+x) &= \sum_{q=0}^{\lfloor \gamma \rfloor} \frac{1}{q!} f_r^{(q)}(y) x^{\otimes q} + R(r,y,x),\label{e:taylor1}\\
    |R(r,y,x)| &\leq \frac{2\|f^{(\lfloor \gamma \rfloor)}\|_\infty + \|f^{(\lceil \gamma \rceil)}\|_\infty}{\lfloor\gamma\rfloor!}\frac{\|x\|^\gamma}{r^\gamma} g(r, x),\label{e:taylor2}\\
    g(r, x) &\xrightarrow{r\to\infty} 0, \label{e:taylor4}
\end{align}
and for all $r>0,x\in\R^d,y\in B_r^c$
\begin{equation}
    R(r,y,x) = f_r(y+x).\label{e:taylor3}
\end{equation}
\begin{proof}
First of all, one can define $R$ simply by
\begin{equation}\label{e:taylor11}
    R(r,y,x) := f_r(y+x) - \sum_{q=0}^{\lfloor \gamma \rfloor} \frac{1}{q!} f_r^{(q)}(y) x^{\otimes q},\quad r>0, x,y\in\R^d.
\end{equation}
Then~\eqref{e:taylor1} is fulfilled and~\eqref{e:taylor3} directly follows from the fact that $f(x) = 0$ for all $x\in B_1^c$ and therefore also $f^{(q)}(x) = 0$ for all $x\in B_1^c, q\in\N_0$.
If $\gamma\notin \N_0$, i.e., $\lceil\gamma\rceil = \lfloor \gamma\rfloor + 1$, by Taylor's theorem and Remark~\ref{r: f_r formulas}, we know that
\begin{equation*}
    R(r,y,x) = \frac{1}{(\lceil \gamma \rceil-1)! r^{\lceil \gamma \rceil}} \int_0^1 (1-t)^{\lceil \gamma \rceil-1} f^{(\lceil \gamma \rceil)}(\tfrac{tx+y}{r}) x^{\otimes \lceil \gamma \rceil} \ dt,\quad r>0,x,y\in\R^d,
\end{equation*}
where we performed a Taylor-approximation of degree $\lfloor \gamma \rfloor$. This formula directly gives rise to the bound
\begin{equation}\label{e:taylor21}
    |R(r,y,x)| \leq \frac{\|f^{(\lceil \gamma \rceil)}\|_\infty}{\lceil \gamma \rceil!} \frac{\|x\|^{\lceil \gamma \rceil}}{r^{\lceil \gamma \rceil}}.
\end{equation}
similarly, if $\lfloor \gamma \rfloor \geq 1$, a Taylor-approximation of degree $\lfloor \gamma \rfloor - 1$ yields
\begin{equation}
    R(r,y,x) = \frac{1}{(\lfloor \gamma \rfloor-1)! r^{\lfloor \gamma \rfloor}} \int_0^1 (1-t)^{\lfloor \gamma \rfloor-1} f^{(\lfloor \gamma \rfloor)}(\tfrac{tx+y}{r}) x^{\otimes \lfloor \gamma \rfloor} \ dt - \frac{1}{\lfloor \gamma \rfloor! r^{\lfloor \gamma \rfloor}} f^{(\lfloor \gamma \rfloor)}(\tfrac{y}{r}),
\end{equation}
from which we can derive
\begin{equation}\label{e:taylor22}
    |R(r,y,x)| \leq \frac{1}{\lfloor \gamma\rfloor!} \frac{\|x\|^{\lfloor \gamma\rfloor}}{r^{\lfloor \gamma\rfloor}} \sup\big\{ \|f^{(\lfloor \gamma \rfloor)}(y)-f^{(\lfloor \gamma \rfloor)}(z) \| : y,z\in\R^d, \|y-z\|\leq \tfrac{\|x\|}{r}\big\},
\end{equation}
for $r>0,x,y\in\R^d$.
The supremum is finite and goes to $0$ as $r\to\infty$ for every $x\in\R^d$, since $f^{(\lfloor\gamma\rfloor)}$ is continuous and compactly supported and therefore also uniformly continuous.
For $\lfloor \gamma \rfloor = 0$, this bound is trivially obtained from~\eqref{e:taylor11}. A combination of~\eqref{e:taylor21} and~\eqref{e:taylor22} yields 
\begin{equation*}
    |R(r,y,x)| \leq \frac{2\|f^{(\lfloor \gamma \rfloor)}\|_\infty + \|f^{(\lceil \gamma \rceil)}\|_\infty}{\lfloor\gamma\rfloor!}\frac{\|x\|^\gamma}{r^\gamma} g(r, x),
\end{equation*}
with $g$, for $r>0, x\in\R^d$, defined as
\begin{align*}
    g(r, x) := \min\bigg(&\frac{\|x\|^{\lceil\gamma\rceil-\gamma}}{r^{\lceil\gamma\rceil-\gamma}}, \nonumber\\
    &\frac{r^{\gamma-\lfloor\gamma\rfloor}}{\|x\|^{\gamma-\lfloor\gamma\rfloor}}\sup\bigg\{ \frac{\|f^{(\lfloor \gamma \rfloor)}(y)-f^{(\lfloor \gamma \rfloor)}(z)\|}{2\|f^{(\lfloor \gamma \rfloor)}\|_\infty} : y,z\in\R^d, \|y-z\|\leq \tfrac{\|x\|}{r}\bigg\}\bigg),
\end{align*}
whereby it fulfills~\eqref{e:taylor4}.
\end{proof}
\end{lemma}

The following lemma is also connected to Taylor's theorem, but the main point of interest, which is~\eqref{e:derivative l1 bound}, is a consequence of a Paley-Wiener-type argument.

\begin{lemma}\label{l: taylor series}
    Let $\varepsilon>0$. Suppose that $f\in L^1(\R^d, \C)$ is chosen such that $\hat{f}\in\cC_c^{d+1}$ with $\hat{f}(k)=0$ for all $k\in B_\varepsilon^c$. Then
    \begin{equation}\label{e: taylor series f}
        f(y+ x) = \sum_{q=0}^\infty \frac{1}{q!} f^{(q)}(y) x^{\otimes q}, \quad x,y\in\R^d,
    \end{equation}
    and there is a $c>0$ such that
    \begin{equation}\label{e:derivative l1 bound}
        \int\sup_{u\in B_1} \|f^{(q)}(x+u)\|\, dx \leq c (q+1)^{d+1}\varepsilon^q,\quad q\in\N_0.
    \end{equation}
    \begin{proof}
        The first part of the assertion, i.e., that~\eqref{e: taylor series f} holds, directly follows from the Paley-Wiener theorem, as it implies that $f$ is an entire function; see, e.g.,~\cite[Chapter III Section 4]{Stein_Weiss_1971}. Additionally, we use a Paley-Wiener-type argument in the following to show that there exists a $c>0$ such that 
        \begin{equation}
            \|f^{(q)}(x)\| \leq c (q+1)^{d+1}\varepsilon^q\frac{1}{\max(\|x\|^{d+1}, 1)}.
        \end{equation}
        Then the second part follows as well.
        
        Let $\varepsilon>0$, suppose that $f\in L^1(\R^d, \C)$ such that $\hat{f}\in\cC^{d+1}_c$ with $\hat{f}(k)=0$ for all $k\in B_\varepsilon^c$, and define $f_\varepsilon:=f(\tfrac{\cdot}{\varepsilon})$. We obtain that $\hat{f_\varepsilon}(k)=\varepsilon^d \hat{f}(\varepsilon k), k\in\R^d$, and hence that $\hat{f_\varepsilon}(k)=0$ for all $k\in B_1^c$. Further, $f^{(q)}(x) = \varepsilon^{q} f_\varepsilon^{(q)}(\varepsilon x)$ for $q\in\N_0, x\in\R^d$, and
        \begin{equation*}
            \frac{1}{\max(\varepsilon^{d+1}\|x\|^{d+1}, 1)} \leq \max(\varepsilon^{-(d+1)}, 1)\frac{1}{\max(\|x\|^{d+1}, 1)}, \quad x\in\R^d.
        \end{equation*}
        Since $\max(\varepsilon^{-(d+1)}, 1)$ can be a part of the constant, this transformation allows us to assume $\varepsilon=1$ without loss of generality. Now let $q\in\N_0$, $i\in\{1,...,d\}$, and $x\in\R^d$ with $x_i\neq0$. Then, using Plancherel's theorem and $d+1$ times partial integration,
        \begin{align*}
            \|f^{(q)}(x)\| &= \frac{1}{(2\pi)^d}\bigg\|\int \widehat{f^{(q)}}(k) e^{-i\langle k, x\rangle}\, dk\bigg\|\nonumber\\
            &= \frac{1}{(2\pi)^d}\bigg\|\int_{B_1} k^{\otimes q} \hat{f}(k) e^{-i\langle k, x\rangle}\, dk\bigg\|\nonumber\\
            &= \frac{1}{(2\pi)^d|x_i|^{d+1}} \bigg\|\int_{B_1} \partial_i^{d+1}\big(s\mapsto s^{\otimes q} \hat{f}(s)\big)(k) e^{-i\langle k, x\rangle}\, dk\bigg\|\nonumber\\
            &\leq \frac{\kappa_d}{(2\pi)^d|x_i|^{d+1}} \sup_{k\in B_1}\bigg\|\partial_i^{d+1}\big(s\mapsto s^{\otimes q} \hat{f}(s)\big)(k)\bigg\|.
        \end{align*}
        Further, for $k\in B_1$, we have
        \begin{align*}
            \bigg\|\partial_i^{d+1}\big(s\mapsto s^{\otimes q} \hat{f}(s)\big)(k)\bigg\| &= \bigg\|\sum_{n=0}^{d+1} \binom{d+1}{n}\partial_i^{n}\big(s\mapsto s^{\otimes q}\big)(k) \partial_i^{d+1-n}\hat{f}(k)\bigg\|\nonumber\\
            &= \bigg\|\sum_{n=0}^{\min(d+1, q)} \binom{d+1}{n}\frac{q!}{(q-n)!}\big(e_i^{\otimes n}\otimes k^{\otimes (q-n)}\big) \partial_i^{d+1-n}\hat{f}(k)\bigg\|\nonumber\\
            &\leq (q+1)^{d+1} \sum_{n=0}^{d+1} \binom{d+1}{n} \bigg\|\partial_i^{d+1-n}\hat{f}(k)\bigg\|.
        \end{align*}
        The fact that $\hat{f}\in\cC^{d+1}$ implies that there is a $c_1>0$ independent of $i$ and $k$ such that
        \begin{equation*}
            \sum_{n=0}^{d+1} \binom{d+1}{n} \bigg\|\partial_i^{d+1-n}\hat{f}(k)\bigg\| \leq c_1, \quad k\in B_1.
        \end{equation*}
        Therefore, for $x\in\R^d\setminus\{0\}$,
        \begin{equation*}
            \|f^{(q)}(x)\| \leq \frac{\kappa_d c_1(q+1)^{d+1}}{(2\pi)^d\|x\|_\infty^{d+1}} \leq \frac{\sqrt{d}\kappa_d c_1(q+1)^{d+1}}{(2\pi)^d\|x\|^{d+1}}.
        \end{equation*}
        It remains to find a bound for small $x$, which is much simpler, as
        \begin{equation*}
            \|f^{(q)}(x)\| = \frac{1}{(2\pi)^d}\bigg\|\int_{B_1} k^{\otimes q} \hat{f}(k) e^{-i\langle k, x\rangle}\, dk\bigg\|\leq \frac{\kappa_d}{(2\pi)^d}\sup_{k\in B_1} \|\hat{f}(k)\| =: c_2 < \infty.
        \end{equation*}
        Combined, we obtain that
        \begin{align*}
            \|f^{(q)}(x)\| \leq \bigg(c_2 + \frac{\sqrt{d}\kappa_d c_1}{(2\pi)^d}\bigg)(q+1)^{d+1} \frac{1}{\max(\|x\|^{d+1}, 1)} ,\quad, x\in\R^d,
        \end{align*}
        which is what was left to show.
    \end{proof}
\end{lemma}

\subsection{Reduction to Lebesgue}\label{ss:Reduction to Lebesgue}
In this subsection, we show why one can sometimes assume that the source of the transport is $\lambda_d$ without loss of generality. This fact is of particular importance if the transport kernel and the source are allowed to be dependent, as the dependence drops out with this simplification.
In the following, we always suppose that $\Phi$ is a locally square-integrable invariant random complex measure and that $K$ is an invariant transport kernel. Further, we define the random transport kernel $L$ by
\begin{equation}\label{e:L}
    L(x,B) := \kappa_d^{-1}\Phi(B\cap(B_1+x)), \quad x\in\R^d, B\in\cB^d.
\end{equation}
and define $\tilde{K}:=KL$, i.e.,
\begin{equation}\label{e:tilde K}
    \tilde{K}(x, B) = \int K(z, B)\, L(x, dz), \quad x\in\R^d, B\in\cB^d.
\end{equation}
We can use $\tilde{K}$ to simplify many problems from Section~\ref{s:main results} to the case that $\Phi=\lambda_d$. The way becomes obvious through the following lemma.

\begin{lemma}\label{l: simplification to lebesgue}
    The transport kernels $L$ and $\tilde{K}$ are invariant, $\Phi = L\lambda_d$, and therefore also $K\Phi=\tilde{K}\lambda_d$.
    \begin{proof}
        The transport kernel $L$ is invariant, as for $x,z\in\R^d, B\in\cB^d$,
        \begin{equation*}
            L(\theta_z, x, B) = \kappa_d^{-1} \Phi(\theta_z, B\cap (B_1+x)) = \kappa_d^{-1} \Phi((B+z)\cap (B_1+x+z)) = L(x+z, B+z). 
        \end{equation*}
        Then the invariance of $\tilde{K}$ directly follows from the fact that $\tilde{K}=KL$, and that $K$, $L$ are invariant.
        Further, $\Phi=L\lambda_d$ by construction, as
        \begin{equation}
            L\lambda_d(B) = \int L(x,B)\, dx \nonumber\\
            = \kappa_d^{-1} \iint_B \I\{y\in B_1+x\} \,\Phi(dy)\, dx \nonumber\\
            = \Phi(B), \quad B\in\cB^d.
        \end{equation}
        The last equality holds by $K\Phi = KL\lambda_d=\tilde{K}\lambda_d$.
    \end{proof}
\end{lemma}

To make this simplification useful, we now show that the integrability conditions transfer from $K, \Phi$ to $\tilde{K}, \lambda_d$.
\begin{lemma}\label{l:condition reduction lebesgue}
    Suppose that $f:\R^d\to [0,\infty)$ is measurable, and define $\Psi_f$ by
    \begin{equation}
        \Psi_f(dy) := \bigg(\int f(x)\, |K^\ast|(y, dx)\bigg)\, |\Phi|(dy),\quad y\in\R^d.
    \end{equation}
    Then 
    \begin{equation}
        \BE\bigg[\bigg(\int f(x)\, |\tilde{K}_0|(dx)\bigg)^2\bigg] < \infty
    \end{equation}
    if $\Psi$ is locally absolutely square-integrable. The converse is true if $K$ and $\Phi$ are nonnegative, as then
    \begin{equation}
        \Psi_f(B_1) = \kappa_d\int f(x) \tilde{K}_0(dx).
    \end{equation}
    Moreover, the following statements hold.
    \begin{enumerate}[label=(\roman*)]
        \item If Condition~\ref{c:condition square-integrable general with p} holds with respect to $p\in[-d, \infty)$ and a strongly log-dominating function $\rho$, then
        \begin{equation}
            \BE\bigg[\bigg(\int \sqrt{\|x\|^{d+\max(p,0)}\rho(\|x\|)+1}\, |\tilde{K}_0|(dx)\bigg)^2\bigg] < \infty.
        \end{equation}
        \item If Condition~\ref{c:condition square-integrable Markov with p} holds with respect to $p\in[-d, \infty)$, then
        \begin{equation}
            \BE\bigg[\bigg(\int \sqrt{\|x\|^{d+\max(p,0)}+1}\, |\tilde{K}_0|(dx)\bigg)^2\bigg] < \infty.
        \end{equation}
        \item If Condition~\ref{c:condition square-integrable stealthy} holds with respect to $\varepsilon>0$, then
        \begin{equation}
            \BE\bigg[\bigg(\int e^{\varepsilon\|x\|}\, |\tilde{K}_0|(dx)\bigg)^2\bigg] < \infty.
        \end{equation}
    \end{enumerate}
    \begin{proof}
        Let $f:\R^d\to[0,\infty)$ be measurable. By construction, we have $|\tilde{K}_0|\leq |K||L_0|= \kappa_d^{-1}|K||\I_{B_1}\cdot\Phi|$. Therefore,
        \begin{equation*}
            \int f(x) |\tilde{K}_0|(dx) \leq \kappa_d^{-1}\int_{B_1}\int f(x) |K|(y, dx)\, |\Phi|(dy) = \kappa_d^{-1} \Psi_f(B_1).
        \end{equation*}
        The inequality becomes an equality if $K$ and $\Phi$ are nonnegative. The remaining assertions follow with an appropriate choice of $f$.
    \end{proof}
\end{lemma}

\subsection{Square integrability}

The main focus of this subsection is to bound a certain remainder term which comes up if one applies the Taylor approximation from Lemma~\ref{l:taylor}. However, we first show a small lemma that allows us to assume that strongly log-dominating functions do not grow too quickly, which makes sense, as they are supposed to be slowly increasing functions.
\begin{lemma}\label{l:stronlgy log-dominating modification}
Let $\rho$ be a strongly log-dominating function. Then there exists a strongly log-dominating function $\tilde\rho\leq \rho$ such that 
\begin{equation}\label{e:subexponential growth}
    \tilde\rho(2r) \leq 4 \tilde\rho(r), \quad r\geq0.
\end{equation}
\begin{proof}
Define $a_0:= \rho(0)$ and iteratively define
\begin{equation*}
    a_n := \min(4a_{n-1}, \rho(2^n-1)), \quad n\in\N.
\end{equation*}
Now, let $\tilde\rho:[0,\infty)\to[0,\infty)$ and choose
\begin{equation*}
    \tilde\rho(r) := a_{\lfloor\log_2(r+1)\rfloor},\quad r\geq 0.
\end{equation*}
Then, by construction and because $\rho$ is increasing, $\tilde\rho$ is increasing and
\begin{equation*}
    \tilde\rho(r) = a_{\lfloor\log_2(r+1)\rfloor} \leq \rho(2^{\lfloor\log_2(r+1)\rfloor}-1) \leq \rho(r), \quad r\geq0.
\end{equation*}
Additionally,~\eqref{e:subexponential growth} holds, as
\begin{equation*}
     \tilde\rho(2r) = a_{\lfloor\log_2(2r+1)\rfloor} \leq a_{\lfloor\log_2(r+1)\rfloor+1} \leq 4 a_{\lfloor\log_2(r+1)\rfloor} = 4 \tilde\rho(r), \quad r\geq0.
\end{equation*}
It remains to show that $\tilde\rho$ fulfills~\eqref{e:strongly log-dominating}. We have that
\begin{equation}\label{e:series bound}
    \int_0^\infty \frac{1}{r\tilde\rho(r)+1} \, dr \leq \sum_{n=0}^\infty \frac{2^n}{na_n + 1}.
\end{equation}
Choose $N_0\in\N$ such that $(n+1)a_n \geq 2$ for $n\geq N_0$. This is possible, as $a_n\to\infty$ as $n\to\infty$ because $\rho(r)\to\infty$ as $r\to\infty$. Note that this implies $4(n+1)a_n+1\geq 3((n+1)a_n+1)$ for $n\geq N_0$. Hence, for $N\geq N_0$,
\begin{align*}
    \sum_{n=0}^N \frac{2^n}{na_n + 1} &\leq \sum_{n=1}^N \frac{2^n}{4na_{n-1} + 1} + \sum_{n=0}^N \frac{2^n}{n\rho(2^n-1) + 1}\nonumber\\
    &\leq \sum_{n=N_0+1}^N \frac{2^n}{4na_{n-1} + 1} + \underbrace{ \sum_{n=1}^{N_0} \frac{2^n}{4na_{n-1} + 1} + 1 + 2\int_0^\infty \frac{1}{r\rho(r)+1}\, dr}_{=: C < \infty}\nonumber\\
    &\leq \frac{2}{3} \sum_{n=N_0+1}^N \frac{2^{n-1}}{na_{n-1} + 1} + C\nonumber\\
    &\leq \frac{2}{3} \sum_{n=0}^N \frac{2^n}{na_n + 1} + C.
\end{align*}
Therefore, the RHS of~\eqref{e:series bound} is bounded by $3C<\infty$, which concludes the proof.
\end{proof}
\end{lemma}
The first version deals with the case that one assumes Condition~\ref{c:condition square-integrable general with p}.
\begin{lemma}\label{l:remainder2case1}
Suppose that $\Phi$ is an invariant locally absolutely square-integrable random measure, and that $K$ is an invariant transport kernel. Let $p\geq 0$ and assume that Condition~\ref{c:condition square-integrable general with p} is satisfied.
Then $K\Phi$ is locally absolutely square-integrable. Further, for $r>0$,
\begin{equation}\label{e:remainder2case1}
    \lim_{r\to\infty} \frac{1}{r^{d-p}} \BE\bigg[\bigg(\int_{B_{2r}^c} |K|(y, B_r) \, |\Phi|(dy) \bigg)^2\bigg] = 0. 
\end{equation}
\begin{proof}
First of all, by Lemma~\ref{l:stronlgy log-dominating modification}, without loss of generality, we can assume that $\rho$ satisfies
\begin{equation}
    \rho(2r) \leq 4 \rho(r), \quad r\geq 0.
\end{equation}
Further, by Lemmas~\ref{l: simplification to lebesgue} and~\ref{l:condition reduction lebesgue}, there is an invariant nonnegative transport kernel $L$ such that $L(y)$ has support on $\overline{B_1}+y$ for $y\in\R^d$, that $|\Phi|=L\lambda_d$, and that the invariant nonnegative transport kernel $\tilde{K}:=|K|L$ satisfies
\begin{equation}\label{e: tilde K integrability}
    \BE\bigg[\bigg(\int \sqrt{\|x\|^{d+p}\rho(\|x\|)+1}\, \tilde{K}_0(dx)\bigg)^2\bigg] < \infty.
\end{equation}
Let $r>2$. Then,
\begin{align*}
    \int_{B_{2r}^c} |K|(y, B_r) \, |\Phi|(dy) &= \iint_{B_{2r}^c} |K|(z, B_r)\, L(y, dz) \, dy\nonumber\\
    &\leq \int_{B_{2r-1}^c}\int |K|(z, B_r) \, L(y, dz) \, dy\nonumber\\
    &= \int_{B_{2r-1}^c} \tilde{K}(y, B_r)\, dy,
\end{align*}
which simplifies the problem. Now this inequality and a two-times application of the Cauchy-Schwarz inequality yield
\begin{align}
    \BE\bigg[\bigg(\int_{B_{2r}^c} |K|(y, B_r) \, |\Phi|(dy) \bigg)^2\bigg] &\leq \BE\bigg[\bigg(\int_{B_{2r-1}^c} \tilde{K}(y, B_r) \, dy \bigg)^2\bigg] \nonumber\\
    &= \int_{B_{2r-1}^c}\int_{B_{2r-1}^c} \BE\big[\tilde{K}(y, B_r)\tilde{K}(z, B_r)\big] \,dy\,dz \nonumber\\
    &\leq \int_{B_{2r-1}^c}\int_{B_{2r-1}^c} \sqrt{\BE\big[\tilde{K}(y, B_r)^2\big]\BE\big[\tilde{K}(z, B_r)^2\big]} \,dy\,dz \nonumber\\
    &= \bigg(\int_{B_{2r-1}^c}\sqrt{\BE\big[\tilde{K}(y, B_r)^2\big]} \,dy\bigg)^2\nonumber\\
    &= \Bigg(\int_{B_{2r-1}^c}\frac{\sqrt{(\|y\|^{d+p}\rho(\|y\|)+1)\BE\big[\tilde{K}(y, B_r)^2\big]}}{\sqrt{\|y\|^{d+p}\rho(\|y\|)+1}} \,dy\Bigg)^2 \nonumber\\
    &\leq \int_{B_{2r-1}^c} \frac{1}{\|y\|^{d+p}\rho(\|y\|)+1} \, dy\nonumber\\
    &\quad \ \ \times\int_{B_{2r-1}^c} (\|y\|^{d+p}\rho(\|y\|)+1) \BE\big[\tilde{K}(y, B_r)^2\big] \, dy.
\end{align}
For the first factor, we get
\begin{align*}
    \int_{B_{2r-1}^c}\frac{1}{\|y\|^{d+p}\rho(\|y\|)+1} \, dy &= d\kappa_d \int_{2r-1}^\infty \frac{s^{d-1}}{s^{d+p}\rho(s)+1} \, ds \nonumber\\
    &\leq \frac{d\kappa_d}{2^p r^p} \underbrace{\int_{2r-1}^\infty \frac{1}{s\rho(s)+s^{-(d+p-1)}}\, ds}_{a(r):=}.
\end{align*}
Because $\rho$ is strongly log-dominating, we can derive that $a(r)\to 0$ as $r\to\infty$ by the theorem of dominated convergence.
The second factor can be bounded by
\begin{align*}
    &\int_{B_{2r-1}^c}(\|y\|^{d+p}\rho(\|y\|)+1)\BE\big[\tilde{K}(y, B_r)^2\big] \, dy\nonumber\\
    &\quad\ \ = \int_{B_{2r-1}^c}(\|y\|^{d+p}\rho(\|y\|)+1)\BE\big[\tilde{K}_0(B_r-y)^2\big] \, dy\nonumber\\
    &\quad\ \ = \BE\bigg[\iiint(\|y\|^{d+p}\rho(\|y\|)+1)\,  \I\{y\in B_{2r-1}^c, x, z\in B_r-y\}\, \tilde{K}_0(dx)\, \tilde{K}_0(dz)\, dy\bigg]\nonumber\\
    &\quad\ \ \leq \BE\bigg[\iiint(\|y\|^{d+p}\rho(\|y\|)+1)\,  \I\{y\in B_{3\min(\|x\|, \|z\|)}\cap (B_r - x) \cap (B_r-z)\}\nonumber\\
    &\hspace{12cm}\tilde{K}_0(dx)\, \tilde{K}_0(dz)\, dy\bigg]\nonumber\\
    &\quad\ \ \leq \BE\bigg[\iint\Big(\big(3\min(\|x\|, \|z\|)\big)^{d+p}\rho\big(3\min(\|x\|, \|z\|)\big)+1\Big) \lambda_d\big((B_r - x) \cap (B_r-z)\big)\nonumber\\
    &\hspace{12cm}\phantom{\, dy} \tilde{K}_0(dx)\, \tilde{K}_0(dz)\bigg]\nonumber\\
    &\quad\ \ \leq \kappa_d 3^{d+p+3}r^d\BE\bigg[\iint\min\Big(\|x\|^{d+p}\rho(\|x\|)+1, \|z\|^{d+p}\rho(\|z\|)+1\Big) \,\tilde{K}_0(dx)\, \tilde{K}_0(dz)\bigg]\nonumber\\
    &\quad\ \ \leq \kappa_d 3^{d+p+3}r^d\BE\bigg[\bigg(\int\sqrt{\|x\|^{d+p}\rho(\|x\|)+1} \,\tilde{K}_0(dx)\bigg)^2\bigg]
\end{align*}
In combination, we obtain that~\eqref{e:remainder2case1} holds, as
\begin{align*}
    &\frac{1}{r^{d-p}} \BE\bigg[\bigg(\int_{B_{2r}^c} |K|(y, B_r) \, |\Phi|(dy) \bigg)^2\bigg]\nonumber\\
    &\quad\ \ \leq d\kappa_d^2 3^{d+p+3} \underbrace{\BE\bigg[\bigg(\int\sqrt{\|x\|^{d+p}\rho(\|x\|)+1} \,\tilde{K}_0(dx)\bigg)^2\bigg]}_{<\infty \text{ by }~\eqref{e: tilde K integrability}} a(r)\nonumber\\
    &\quad \ \ \xrightarrow{r\to\infty} 0.
\end{align*}
Finally, we can conclude that $K\lambda_d$ is locally absolutely square-integrable, as with a similar argument regarding the support of $L$, we see that
\begin{align*}
    \BE\bigg[\bigg(\int_{B_{2}} |K|(y, B_1) \, |\Phi|(dy) \bigg)^2\bigg] &\leq \BE\bigg[\bigg(\int_{B_{3}} \tilde{K}(y, B_1) \, dy \bigg)^2\bigg]\nonumber\\
    &\leq \kappa_d^2 3^{2d} \BE\big[\tilde{K}_0(\R^d)^2\big]\nonumber\\
    &< \infty.
\end{align*}
\end{proof}
\end{lemma}

The second version deals with the case that one assumes Condition~\ref{c:condition square-integrable Markov with p} and leverages a proof which is fundamentally different from the previous and a generalization of a proof in~\cite{Dereudre_Flimmel_Huesmann_Leblé_2024}.

\begin{lemma}\label{l:remainder2case2}
Suppose that $\Phi$ is a locally absolutely square-integrable invariant random complex measure, and suppose that $K$ is an invariant transport kernel. Let $p\geq0$ and assume that Condition~\ref{c:condition square-integrable Markov with p} is satisfied. Then $K\Phi$ is locally absolutely square-integrable. Further, for $r>0$,
\begin{equation}\label{e:remainder2case2}
    \lim_{r\to\infty} \frac{1}{r^{d-p}} \BE\bigg[\bigg(\int_{B_{2r}^c} |K|(y, B_r) \, |\Phi|(dy) \bigg)^2\bigg] = 0. 
\end{equation}
\begin{proof}
As $K$ has bounded total variation, there is a $b>0$ such that $|K|(y, \R^d)\leq b, y\in\R^d$. 
For $r>0$,
\begin{align}
    \BE\bigg[\bigg(\int_{B_{2r}^c} |K|(y, B_r) \, |\Phi|(dy) \bigg)^2\bigg] &= \BE\bigg[\int_{B_{2r}^c}\int_{B_{2r}^c} |K|(z, B_r)|K|(y, B_r) \, |\Phi|(dz)\, |\Phi|(dy) \bigg]\nonumber\\
    &\leq 2 \BE\bigg[\int_{B_{2r}^c}\int_{B_{\|y\|}} |K|(z, B_r)|K|(y, B_r) \, |\Phi|(dz)\, |\Phi|(dy) \bigg]\nonumber\\
    &\leq 2b \BE\bigg[\int_{B_{2r}^c} |K|(y, B_r) |\Phi|(B_{\|y\|})\, |\Phi|(dy) \bigg]\nonumber\\
    &= 2b \BE\bigg[\int_{B_{2r}^c} \int_{B_r-y} |\Phi|(B_{\|y\|}) \,|K_0|(dx) \, |\Phi|(dy) \bigg]\nonumber\\
    &\leq 2b \BE\bigg[\int_{B_{r}^c} \int_{B_r-x} |\Phi|(B_{\|y\|})\, |\Phi|(dy)\,|K_0|(dx) \bigg]\nonumber\\
    &\leq 2b \BE\bigg[\int_{B_{r}^c} \BE\big[|\Phi|(B_{2\|x\|}) |\Phi|(B_r-x)\big]\,|K_0|(dx) \bigg],
\end{align}
where we first utilized symmetry and then twice the independence of $K$ and $\Phi$. Further, for $r\geq1,x\in B_r^c$,
\begin{align*}
    \BE\big[|\Phi|(B_{2\|x\|}) |\Phi|(B_r-x)\big] &\leq c \bigg(\frac{2\|x\|}{r}\bigg)^d \BE\big[|\Phi|(B_r)^2\big]\nonumber\\
    &\leq c^3 \bigg(\frac{2\|x\|}{r}\bigg)^d r^{2d}\BE\big[|\Phi|(B_1)^2\big]\nonumber\\
    &=  c^3 2^d \BE\big[|\Phi|(B_1)^2\big] r^d \|x\|^d,
\end{align*}
where $c\geq 2$ is a dimensional constant that is determined by how efficiently one can cover a larger sphere with multiple smaller spheres. In combination, we obtain,
\begin{align*}
    \frac{1}{r^{d-p}} \BE\bigg[\bigg(\int_{B_{2r}^c} K(y, B_r) \, |\Phi|(dy) \bigg)^2\bigg] &\leq 2^{d+1} b c^3 \BE\big[|\Phi|(B_1)^2\big] \BE\bigg[\int_{B_{r}^c} r^p\|x\|^d \,|K_0|(dx) \bigg]\nonumber\\
    &\leq 2^{d+1} b c^3 \BE\big[|\Phi|(B_1)^2\big] \BE\bigg[\int_{B_{r}^c} \|x\|^{d+p} \,|K_0|(dx) \bigg]\nonumber\\
    &\xrightarrow{r\to\infty} 0,
\end{align*}
where the last convergence is ensured by the local absolute square integrability of $\Phi$, Condition~\ref{c:condition square-integrable Markov with p}, and the dominated convergence theorem. As
\begin{equation*}
    \BE\bigg[\bigg(\int_{B_1}|K|(y, B_1) \, |\Phi|(dy)\bigg)^2\bigg] \leq b^2\BE\big[|\Phi|(B_1)^2\big] < \infty,
\end{equation*}
we can also conclude that $K\Phi$ is locally absolutely square-integrable.
\end{proof}
\end{lemma}

\section{Fourier-smoothness}\label{s: Fourier-smoothness}

We recall that by Definition~\ref{d: fourier smoooth} a function $f\in L^1(\R^d, \C)$ is called \textit{Fourier-smooth with exponent} $p\in[-d,\infty)$ if there is a $c>0$ such that
\begin{equation}\label{e: fourier smooth 12}
    |\hat{f}(k)| \leq c \frac{1}{(1+\|k\|)^{\frac{d+p}{2}}},\quad k\in\R^d,
\end{equation}
and if either $p>0$ or, as $r\to\infty$,
\begin{equation}\label{e: fourier smooth 22}
    \int_{B_2^c} \sup_{s\in B_1} |\hat{f}(r(k+s))|^2\, dk = O(r^{-(d+p)}).
\end{equation}
Note that~\eqref{e: fourier smooth 22} follows directly from~\eqref{e: fourier smooth 12} if $p>0$.
This definition is an extension of the definition from~\cite{Björklund_Hartnick_2024}, where it was introduced for indicator functions with exponents limited to $(0, 1]$. While the extension to non-indicator functions is trivial, the non-positive exponents require the additional assumption of~\eqref{e: fourier smooth 22} to make sure that the assertions of Theorem~\ref{t:asymptotic variance test function} hold also in this case. There may be a more optimal formulation of the additional condition, but in the following, we see that ours at least allows for the correct classification of the important $\I_{[0,1)^d}$.

\begin{examples}\label{ex: fourier smooth}
    The indicator function of the unit ball $\I_{B_1}$ is Fourier-smooth with exponent $1$ and not any higher exponent; see~\cite[Remark 3.5]{Björklund_Hartnick_2024} and~\cite[Theorem 2.16]{Iosevich_Liflyand_2014}. The indicator function of the unit cube $\I_{[0,1)^d}$ is Fourier-smooth with exponent $-d+2$ and not any higher exponent by Lemma~\ref{l: fourier smooth multiplication}. Hence, the latter cannot be used to reliably detect hyperuniformity in general; see Theorem~\ref{t:asymptotic variance test function} and also~\cite{Kim_Torquato_2017}. Further, Schwartz functions are Fourier-smooth with any exponent. In particular, these include the functions in $C_c^\infty(\R^d, \C)$.
\end{examples}

To obtain the Fourier-smoothness of $\I_{[0,1)^d}$, we show the following more general lemma that covers all functions that are the product of one-dimensional functions.

\begin{lemma}\label{l: fourier smooth multiplication}
    Let $f\in L^1(\R, \C)$ be Fourier-smooth with exponent $p\in[-1, \infty)$, and define $g:\R^d\to\C$ by
    \begin{equation}
        g(x) := \prod_{j=1}^d f(x_j), \quad x\in\R^d.
    \end{equation}
    Then $g$ is Fourier-smooth with exponent $p-(d-1)$.
    \begin{proof}
        Since $f$ is Fourier-smooth with exponent $p$, there is a $c>0$ such that
        \begin{equation*}
            |\hat{f}(k)| \leq c \frac{1}{(1+ |k|)^\frac{1+p}{2}}, \quad k\in\R.
        \end{equation*}
        Therefore,
        \begin{equation*}
            |\hat{g}(k)| \leq c^d \prod_{j=1}^d \frac{1}{(1+|k_j|)^\frac{1+p}{2}} \leq c^d\frac{1}{(1+\|k\|_\infty)^\frac{1+p}{2}} \leq c^d\frac{1}{(1+\|k\|)^\frac{d+p-(d-1)}{2}}, \quad k\in\R^d,
        \end{equation*}
        which shows the first part of the asserted Fourier-smoothness of $g$.

        For the second part, we can calculate that, for $r>0$,
        \begin{align*}
            \int_{B_2^c} \sup_{s\in B_1} |\hat{g}(r(k+s))|^2\, dk &= \int_{B_2^c} \sup_{s\in B_1} \prod_{j=1}^d|\hat{f}(r(k_j+s_j))|^2\, dk\nonumber\\
            &\leq d \int_{B_2^c} \I\{k_d\leq 2\} \sup_{s\in B_1} \prod_{j=1}^d|\hat{f}(r(k_j+s_j))|^2\, dk\nonumber\\
            &\quad \ \ + \bigg(\int_{(-2,2)^c}\sup_{s\in(-1,1)}|\hat{f}(r(k+s))|^2\, dk\bigg)^d\nonumber\\
            &\leq 4cd \int_{B_2^c}  \sup_{s\in B_1} \prod_{j=1}^{d-1}|\hat{f}(r(k_j+s_j))|^2\, dk\nonumber\\
            &\quad \ \ + \bigg(\int_{(-2,2)^c}\sup_{s\in(-1,1)}|\hat{f}(r(k+s))|^2\, dk\bigg)^d.
        \end{align*}
        Note that $B_2$ in the second-to-last integral does not refer to the ball in $\R^d$ but in $\R^{d-1}$. By the assumption that $f$ is Fourier-smooth with exponent $p$, the second summand of the bound is $O(r^{-(d+dp)})$ as $r\to\infty$, and since $p\geq-1$, this fact also implies that it is $O(r^{-(d+p-(d-1))})$. For the first summand of the bound, we can use induction to see that it also is \linebreak $O(r^{-(1+p)})=O(r^{-(d+p-(d-1))})$, which concludes the proof.
    \end{proof}
\end{lemma}

A simple way to obtain functions that are Fourier smooth with high exponents is provided by the following lemma.

\begin{lemma}
    suppose that $f,g\in L^1(\R^d, \C)$ are Fourier-smooth with exponents $p,q\in[-d,\infty)$. Then $f\ast g$ is Fourier-smooth with exponent $d+p+q$.
    \begin{proof}
        Let $c$ be the joint constant such that~\eqref{e: fourier smooth 12} holds for $f,g$ with their respective constants.
        Then the first part of the Fourier-smoothness of $f\ast g$ is obtained by
        \begin{equation*}
            |\widehat{(f\ast g)}(k)| = |\hat{f}(k)\hat{g}(k)| \leq c^2\frac{1}{(1+\|k\|)^\frac{d+(d+p+q)}{2}},\quad k\in\R^d. 
        \end{equation*}
        The second part follows from
        \begin{align*}
            \int_{B_2^c} \sup_{s\in B_1} |\widehat{(f\ast g)}(r(k+s))|^2\, dk &\leq \int_{B_2^c} \sup_{s\in B_1} |\hat{f}(r(k+s))|^2|\hat{g}(r(k+s))|^2\, dk\nonumber\\
            &\leq c\int_{B_2^c} \frac{\sup_{s\in B_1} |\hat{f}(r(k+s))|^2}{(1+r(\|k\|-1))^{d+q}}\, dk\nonumber\\
            &\leq cr^{-(d+q)}\int_{B_2^c} \sup_{s\in B_1} |\hat{f}(r(k+s))|^2\, dk\nonumber\\
            &= O(r^{-(d+d+p+q)})
        \end{align*}
        as $r\to\infty$.
    \end{proof}
\end{lemma}

Finally, Fourier-smoothness with a high exponent implies differentiability, and the derivative also inherits some of the Fourier-smoothness.
\begin{lemma}\label{l: fourier smooth differentiation}
    Suppose that $f\in L^1(\R^d, \C)$ is Fourier-smooth with exponent $p\in (d, \infty)$. Then $f$ is $\lceil \tfrac{p-d}{2}-1\rceil$-times continuously differentiable and every component of the $q$-th derivative $f^{(q)}$ is Fourier-smooth with exponent $p-2q$ for $q\in\{0,..., \lceil \tfrac{p-d}{2}-1\rceil\}$.
    \begin{proof}
        Since $f$ is Fourier-smooth with exponent $p>d$, we know that, by~\eqref{e: fourier smooth 12},
        \begin{equation*}
            (1+\|k\|)^{ \big\lceil\tfrac{p-d}{2}-1\big\rceil}|\hat{f}(k)|\leq c \frac{(1+\|k\|)^{ \big\lceil\tfrac{p-d}{2}-1\big\rceil}} {(1+\|k\|)^{\frac{d+p}{2}}} = c \frac{1}{(1+\|k\|)^{d+\big(\tfrac{d+p}{2}-\big\lceil\tfrac{p+d}{2}-1\big\rceil\big)}}, \quad k\in\R^d,
        \end{equation*}
        for some $c>0$. Since the exponent in the denominator of the right-most side is strictly greater than $d$, this fact implies that $k\mapsto (1+\|k\|)^{ \big\lceil\tfrac{p-d}{2}-1\big\rceil}|\hat{f}(k)|$ is integrable, whereby $f$ is \linebreak $\lceil \tfrac{p-d}{2}-1\big\rceil$-times continuously differentiable. The Fourier-smoothness of the components of $f^{(q)}$ then is a direct consequence of the fact that
        \begin{equation*}
            \|\widehat{f^{(q)}}(k)\|=\|k^{\otimes q}\hat{f}(k)\|\leq \|k\|^q |\hat{f}(k)|, \quad k\in\R^d.
        \end{equation*}
    \end{proof}
\end{lemma}

\section{Averaging sets}\label{s: averaging sets}

Averaging sets are a special class of quadrature formula where the weights are all equal. The term was coined in~\cite{Seymour_Zaslavsky_1984} as a generalization of so-called spherical $t$-designs, which are averaging sets for the uniform surface integral on a sphere; see also~\cite{Delsarte_Goethals_Seidel_1977, Bondarenko_Radchenko_Viazovska_2013}. Our main contribution here lies in Appendix~\ref{ss: Averaging sets of sets}, where we show that for fixed dimension and degree, there is a size such that there exist averaging sets of that size of all convex sets; see Definition~\ref{d: av set of set} for the properties of such an averaging set. This fact is relevant to ensure the existence of the point processes constructed in Subsection~\ref{ss: pp from invariant partitions}. Then in Appendix~\ref{ss: Averaging sets of trigonometric polynomials}, we generalize Chebyshev–Gauss quadrature to approximate image measures of arbitrary trigonometric polynomials. This technique is applied in Subsection~\ref{ss: pp from averaging sets}. There, we also use spherical formulas, but the relevant results about them have already been provided in~\cite{Bondarenko_Radchenko_Viazovska_2013}, so we do not explicitly introduce them here.

\begin{definition}
    Let $n\in\N$, $\mu\neq0$ be a finite measure on $\R^d$, $J$ be an index set, \linebreak $\{f_j:j\in J\}\subseteq L^1(\mu)$, and $x_1,...,x_n\in\R^d$. We say that $x_1,...,x_n$ form an \textit{averaging set} of $\mu$ and $\{f_j:j\in J\}$ of size $n$ if
    \begin{equation}
        \frac{1}{\mu(\R^d)}\int f_j(x)\, \mu(dx) = \frac{1}{n} \sum_{k=1}^n f_j(x_k),\quad j\in J.
    \end{equation}
    If we do not say otherwise, we require that $x_1,...,x_n$ are pairwise different. Then we also say that $\{x_1,...,x_n\}$ is an averaging set of $\mu$ and $\{f_j:j\in J\}$ of size $n$.
\end{definition}

\begin{definition}
    Let $n\in\N, p\in\N_0$, $\mu\neq0$ be a finite measure on $\R^d$, and $x_1,...,x_n\in\R^d$. We say that $x_1,...,x_n$ form an averaging set of $\mu$ of order $p$ and size $n$ if they form an averaging set of $\mu$ and the polynomials of degree up to $p$, i.e.,
    \begin{equation}
        \frac{1}{\mu(\R^d)}\int x^{\otimes q}\, \mu(dx) = \frac{1}{n} \sum_{k=1}^n x_k^{\otimes q},\quad q\in\{0,...,p\}.
    \end{equation}
    Implicitly, we assume that the polynomials of degree up to $p$ are integrable with respect to $\mu$.
\end{definition}

Before we continue, we make a quick detour to the complex plane, where averaging sets with size equal to the degree exist for all measures and degrees by the Girard-Newton formulas.

\begin{theorem}\label{t: complex averaging sets}
    Let $p\in\N$, and suppose that $\mu$ is a finite measure on $\C$ for which the polynomials up to degree $n$ are integrable. Then there exist $z_1,...,z_{p}\in\C^d$ which form a complex averaging set of $\mu$ of order $p$ and size $p$. These are unique up to permutation.
    \begin{proof}
        The assertion is directly implied by the upcoming Lemma~\ref{l: Newton's identities}.
    \end{proof}
\end{theorem}

Note that, $z_1,...,z_{p}$ are not necessarily pairwise different, and if $\mu$ is concentrated on $\R$, there is no guarantee that $z_1,...,z_p\in\R$.

\begin{lemma}\label{l: Newton's identities}
    The mapping
    \begin{equation}
        F:\C^d\to\C^d, z\mapsto \begin{pmatrix}
            z_1+...+z_d\\
            z_1^2+...+z_d^2\\
            \vdots\\
            z_1^d+...+z_d^d
        \end{pmatrix}, \quad z\in\C^d,
    \end{equation}
    is surjective. Additionally, up to permutation of the arguments, it is injective.
    \begin{proof}
        By the Girard-Newton formulae, see, e.g.,~\cite{Mead_1992}, we know that there is a bijective map $f:\C^d\to\C^d$ such that
        \begin{equation*}
            f(p(z)) = e(z), \quad z\in\C^d,
        \end{equation*}
        where $p,e:\C^d\to \C^d$ with components
        \begin{align*}
            p_j(z) &= \sum_{k=1}^d z_k^j,\\
            e_j(z) &= \sum_{\substack{k\in\{1,...,d\}^j\\k_i\neq k_l\textnormal{ for } i\neq l}} z_{k_1}\cdots z_{k_j},\quad z\in\C^d, j\in\{1,...,d\}.
        \end{align*}
        Hence, $F(z) = m$ is equivalent to $e(z) = f(m)$ for $z,m\in\C^d$. Here, we can apply that
        \begin{equation*}
            \prod_{k=1}^d (x-z_k) = x^d + \sum_{k=1}^d (-1)^k e_k(z)x^{d-k}, \quad x\in\C, z\in\C^d.
        \end{equation*}
        Thus, for $m\in\C^d$, the relation $e(z_1,...,z_d) = f(m)$ is fulfilled iff $z_1,...,z_d\in\C$ are the zeros of the polynomial
        \begin{equation*}
            x^d + \sum_{k=1}^d (-1)^k f_k(m)x^{d-k},\quad x\in\C,
        \end{equation*}
        by a comparison of coefficients. The assertion then follows from the existence of all roots of polynomials in $\C$ and their uniqueness.
    \end{proof}
\end{lemma}

To conclude the excursion, we give the following examples, where we calculate the optimal complex averaging sets from Theorem~\ref{t: complex averaging sets}. In general, one can derive exact formulas for up to $4$ points, and if the measure is more regular, e.g., real and symmetric, up to $9$ points. Additionally, for measures on the real line, all configurations up to $3$ points are also real. This no longer holds for $4$ points, as one can see in Proposition~\ref{p:averaging set counterexample}.

\begin{examples}
    Suppose that $\mu$ is the measure with unit intensity on the unit disk in $\C$. Then the optimal complex averaging set from Theorem~\ref{t: complex averaging sets} just contains $0$ multiple times.

    If instead, we suppose that $\mu$ is the measure with support on $[-1,1]$ and the density \linebreak $x\mapsto\frac{1}{\sqrt{1-x^2}}, x\in(-1,1)$, then we recover the nodes from Chebyshev-Gauss quadrature by uniqueness of the solution up to permutation.

    Finally, if we suppose that $\mu$ is the Lebesgue measure restricted to $[0,1]$ like in classical quadrature, then for any number of points up to $7$, and additionally for $9$ points, with the assistance of computer algebra software, we showed that the points of the optimal averaging sets are all real. For $8$ points, we showed that the solution also involves complex numbers which do not lie on the real line. For more points, we can no longer utilize closed-form solutions for the zeros of symmetric polynomials. However, one can calculate approximate solutions with high enough accuracy to ensure that they are not purely real, which we did for many more. What we can observe suggests that there do not exist any purely real solutions beyond $9$ points.
\end{examples}

\subsection{Averaging sets of sets}\label{ss: Averaging sets of sets}

In this subsection, we no longer consider averaging sets of arbitrary measures, but in particular of restrictions of the Lebesgue measure. 

\begin{definition}\label{d: av set of set}
    Let $n\in\N, p\in\N_0$, $B\in\cB^d$ with $\lambda_d(B)\in(0,\infty)$, and $x_1,...,x_n\in\R^d$. We say that $x_1,...,x_n$ form an \textit{averaging set} of $B$ of order $p$ and size $n$ if they form an averaging set of $\lambda_d|_B$ of order $p$ and size $n$, i.e.,
    \begin{equation}
        \frac{1}{\lambda_d(B)}\int_B x^{\otimes q}\, dx = \frac{1}{n} \sum_{k=1}^n x_k^{\otimes q},\quad q\in\{0,...,p\}.
    \end{equation}
\end{definition}

The following lemma shows that these averaging sets are preserved under bijective linear maps and translations.

\begin{lemma}\label{l:averaging set linear map}
    Let $B\in\cB^d$ with $\lambda_d(B)\in(0,\infty)$, $n\in\N$, $p\in\N_0$, and $F:\R^d\to\R^d$ be an affine transformation. Suppose that  $x_1,...,x_n\in\R^d$ form an averaging set of $B$ of degree $p$. Then $F(x_1),...,F(x_n)$ form an averaging set of $F(B)$ of degree $p$.
    \begin{proof}
        There is an invertible matrix $A\in\R^{d\times d}$ and a vector $b\in\R^d$ such that \linebreak$F(x) = Ax+b, x\in\R^d$. Consequently, 
        \begin{equation*}
            \frac{1}{\lambda_d(F(B))}\int_{F(B)} y^{\otimes q}\, dy = \frac{|\det(A)|}{\lambda_d(F(B))}\int_B (Ax+b)^{\otimes q} \, dx = \frac{1}{\lambda_d(B)}\int_B (Ax+b)^{\otimes q} \, dx
        \end{equation*}
        for $q\in\{0,...,p\}$.
        As every component of $x\mapsto (Ax+b)^{\otimes q}$ is a polynomial of degree $q$, the assertion follows from the fact that $x_1,...,x_n$ form an averaging set of $B$ of degree $p$, since
        \begin{equation*}
            \frac{1}{\lambda_d(B)}\int_B (Ax+b)^{\otimes q} \, dx = \frac{1}{n} \sum_{k=1}^n (Ax_k+b)^{\otimes q} = \frac{1}{n} \sum_{k=1}^n F(x_k)^{\otimes q},\quad q\in\{0,...,p\}.
        \end{equation*}
    \end{proof}
\end{lemma}

It was shown in~\cite[Corollary 2]{Seymour_Zaslavsky_1984} that for any $p\in\N_0$ and bounded, connected, open set $B\subseteq \R^d$, there is an $N\in\N$ such that for any $n\geq N$, there is an averaging set $x_1,...,x_n$ of $B$ of degree $p$ and size $n$. However, for our purposes in Subsection~\ref{ss: pp from invariant partitions} we need to control how $N$ depends on $B$. We provide such an extension in the following theorem.

\begin{theorem}\label{t:averaging set bound}
For all $d\in\N, C>0$, and $p\in\N_0$, there exists an $N\in\N$ such that for any $n\geq N$, $B\in\cB^d_b$ with $\lambda_d(\partial B)=0$ and $\frac{\diam(B)^d}{\lambda_d(B)}\leq C$, there exists an averaging set $x_1,...,x_n$ of $B$ of degree $p$ and size $n$, where $x_1,...,x_n\in B$ and $x_1,...,x_n$ are pairwise different.
\begin{proof}
Let $n\in \N, B\subset \R^d$ be an open set such that and $\frac{\diam(B)^d}{\lambda_d(B)}\leq C$. The assertion also follows for sets with a boundary that is a null set because a null set does not contribute to the integral. Further, without loss of generality, we can assume that $\lambda_d(B)=1$ and $0\in B$, as we can otherwise linearly scale it up or down to change the volume to $1$ and translate it to move $0$ into the interior. This scaling does not change $\frac{\diam(B)^d}{\lambda_d(B)}$ and preserves averaging sets by Lemma~\ref{l:averaging set linear map}. As a final simplification, we assume that $C\geq1$, whereby $\diam(B)\leq C$ and $B\subseteq B_C$.

To construct the averaging set of $B$ of degree $p$ we are looking for, we need to find pairwise different points $x_1,...,x_n\in B$ such that
\begin{equation}
    \frac{1}{n} \sum_{k=1}^n x_k^{\otimes q} = \int_B z^{\otimes q} \, dz,\quad q\in\{0,...,p\},
\end{equation}
for a suitable $n\in\N$.
Because the Lebesgue measure is translation invariant, without loss of generality, we can assume $0\in B$.
If we decompose the tensor-valued equations into multiple real-valued equations, then we obtain equations of the form
\begin{equation}
    \frac{1}{n} \sum_{k=1}^n f_q(x_k) = \int_B f_q(z) \, dz,\quad q\in\{0,...,\tilde{p}\},
\end{equation}
where $\tilde{p} = \binom{p+d}{d}$ and $f_q\in\cC(B, \R)$ are multivariate monomials with coefficient $1$ and therefore also $\|f_q\|_\infty\leq C^p$ for $q\in\{1,...,\tilde{p}\}$. We additionally specify that $f_0=1$. Now, we define the following scalar product:
\begin{equation}
    \langle f, g \rangle_B := \int_B f(x)g(x)\, dx,\quad f,g\in\cC(B,\R).
\end{equation}
As $\{f_q : q\in\{0,...,\tilde{p}\}\}$ are linearly independent with respect to this scalar product, we can construct an orthonormal base with respect to it. First, we can simply define $\tilde{f}_0:=1$. Then, we iteratively define
\begin{align}
    g_q &:=f_q - \sum_{k=0}^{q-1} \langle f_q, \tilde{f}_k\rangle_B \tilde{f}_k\\
    \tilde{f}_q &:= \frac{1}{\sqrt{\langle g_q, g_q\rangle_B}}g_q, \quad q\in\{1,...,\tilde{p}\}.
\end{align}

For the following arguments, it is important that $\|\tilde{f}_q\|_\infty\leq c_q$, where the constant $c_q>0$ can only depend on $d,a,C$, and $p$. We show this by induction over $q\in\{0,...,\tilde{p}\}$. For $q=0$, the assertion holds. Now, assume that the assertion holds for all $k<q$, i.e., there are $c_k>0$ only dependent on $d,C$, and $p$ such that $\|\tilde{f}_k\|_\infty\leq c_k$ for $k<q$. Then,
\begin{equation*}
    \|g_q\|_\infty \leq \|f_q\|_\infty + \sum_{k=0}^{q-1} |\langle f_q, \tilde{f}_k\rangle_B| \|\tilde{f}_k\|_\infty \leq C^p + C^{\frac{p}2{}} \sum_{k=0}^{q-1} c_k.
\end{equation*}
As the upper bound only depends on $d,C$, and $p$, it suffices to find a positive lower bound for $\langle g_q, g_q\rangle_B$ only dependent on $d,C$, and $p$ to get the bound on $\|\tilde{f}_q\|_\infty$ we are looking for. To this end, we again note that $f_q$ is a monomial with a coefficient of $1$. The polynomial $g_q$ keeps this coefficient of $1$ in front of the monomial it inherits from $f_q$ because $f_k$ is a different monomial for $k\neq q$. Hence,
\begin{equation*}
    \langle g_q,g_q\rangle_B \geq \inf \int_{A_h} h(x)^2\, dx,
\end{equation*}
where the infimum is taken over all polynomials $h$ on $B_C$ such that the coefficient in front of the monomial from $f_q$ is $1$, and where we choose $A_h:=\{h^2\leq Q_h^{-1}(1)\}\cap B_C$ with \linebreak $Q_h(y):=\lambda_d(\{h^2\leq y\}\cap B_C)$ for $y\in[0,\infty)$. $Q_h$ is invertible since these polynomials are continuous and never constant, and for non-constant polynomials, all level sets are null sets with respect to Lebesgue measure. The set $A_h$ can be seen as a worst-case scenario for $B$ with respect to every polynomial $h$. Now, we choose a sequence $(h_m)_{m\in\N}$ of these polynomials such that
\begin{equation}\label{e:avset hm convergence}
    \int_{A_{h_m}} h_m(x)^2\, dx \xrightarrow{m\to\infty} \inf \int_{A_h} h(x)^2\, dx.
\end{equation}
Because $A_{h_m}$ is compact, $A_{h_m}\subseteq B_C$, and $\lambda_d(A_{h_m})=1$ for $m\in\N$, we know that $(A_{h_m})_{m\in\N}$ contains a subsequence that converges to some compact $\tilde{B}\subseteq B_C$ with $\lambda_d(\tilde{B})=1$ in the Hausdorff metric. Hence, without loss of generality $A_{h_m}\to \tilde{B}$ as $m\to\infty$ in the Hausdorff metric and therefore also $\I_{A_{h_m}}\to \I_{\tilde{B}}$ as $m\to\infty$ in $L^1(B_C)$. It is also easy to derive that $h_m$ itself must be a bounded sequence in the space of multivariate polynomials on $B_C$ with degree up to the degree of $f_q$, as otherwise it would have a subsequence that diverges almost everywhere, which contradicts~\eqref{e:avset hm convergence}. Thus, there exists a subsequence of $(h_m)_{m\in\N}$ and a polynomial $\tilde{h}$ such that $h_m\to \tilde{h}$ as $m\to\infty$ in $L^\infty(B_C)$. Because $h_m$ has the coefficient $1$ in front of the monomial from $f_q$ for all $m\in\N$, $\tilde{h}$ inherits this property and particularly is not the zero polynomial. Therefore,
\begin{equation*}
    \inf \int_{B_h} h(x)^2\, dx \xleftarrow{m\to\infty} \int_{B_{h_m}} h_m(x)^2\, dx \xrightarrow{m\to\infty} \int_{\tilde{B}} \tilde{h}(x)^2\, dx >0.
\end{equation*}
As the right limit only depends on $d,C$, and $p$, the induction is complete, and we have shown that there exist $c_q>0$ that only depend on $d,C$, and $p$ such that $\|\tilde{f}_q\|_\infty\leq c_q$ for all $q\in\{0,...,\tilde{p}\}$. From now on, let $c:=\max\{c_q:q\in\{0,...,\tilde{p}\}\}$.

For the next step, choose $\varepsilon_{i,j}$ for $i\in\{1,...,\tilde{p}\}, j\in\{1,...,2^{\tilde{p}}\}$ such that
\begin{equation}
    \big\{(\varepsilon_{1,j},...,\varepsilon_{\tilde{p},j}) : j\in\{1,...,2^{\tilde{p}}\}\big\} = \bigg\{(a_1,...,a_{\tilde{p}}):a_1,...,a_{\tilde{p}}\in\Big\{-\frac{1}{2\tilde{p}c},\frac{1}{2\tilde{p}c}\Big\}\bigg\}.
\end{equation}
Then define 
\begin{align}
    l_j &:= \sum_{i=1}^{\tilde{p}} \varepsilon_{i,j} \tilde{f}_i,\\
    \tilde{l}_j &:= 1 + l_j - \int_B l_j(x)\, dx,\quad j\in\{1,...,2^{\tilde{p}}\}.\label{e:avset 1 integral l}
\end{align}
By construction, we obtain
\begin{align*}
    \|l_j\|_\infty &\leq \frac{\tilde{p}c}{2\tilde{p}c} = \frac{1}{2},\\
    \Big\|\tilde{l}_j-1\Big\|_\infty &\leq 2\|l_j\|_\infty \leq 1,\quad j\in\{1,...,\tilde{p}\}.
\end{align*}
Hence, $\tilde{l}_j\in\cC(B, [0,2])$, and from definition~\eqref{e:avset 1 integral l}, we can directly derive that\linebreak$\int_B \tilde{l}_j(x)\, dx=1$ for $j\in\{1,...,\tilde{p}\}$. Additionally, we can see that
\begin{align*}
    \int_B\tilde{l}_j(x) \tilde{f}_q(x)\, dx &= \sum_{i=1}^{\tilde{p}} \varepsilon_{i,j} \langle \tilde{f}_i, \tilde{f}_q\rangle_B + \bigg( 1 - \int_B l_j(x)\, dx\bigg) \langle\tilde{f}_0, \tilde{f}_q\rangle_B \nonumber\\
    &= \varepsilon_{q,j},\quad j\in\{1,...,2^{\tilde{p}}\},q\in\{1,...,\tilde{p}\}.
\end{align*}
This fact allows us to conclude that
\begin{equation*}
    (\varepsilon_{1, j}, ..., \varepsilon_{\tilde{p}, j}) \in \conv\Big(\Big\{\big(\tilde{f}_1(x), ..., \tilde{f}_{\tilde{p}}(x)\big):x\in B\Big\}\Big), \quad j\in\{1,...,2^{\tilde{p}}\}.
\end{equation*}
Consequently, there exist $y_{1,j},...,y_{\tilde{p}+1,j}\in B$ for all $j\in\{1,...,2^{\tilde{p}}\}$ such that
\begin{equation*}
    (\varepsilon_{1, j}, ..., \varepsilon_{\tilde{p}, j}) \in \conv\Big(\Big\{\big(\tilde{f}_1(y_{q,j}), ..., \tilde{f}_{\tilde{p}}(y_{q,j})\big): q\in\{1,...,\tilde{p}+1\}\}\Big\}\Big)
\end{equation*}
by Carathéodory's theorem; see~\cite{Carathéodory_1911}.
In total, we can conclude that
\begin{equation}\label{e:avset centroid pos convhull distance}
    \Big[-\frac{1}{2\tilde{p}c},\frac{1}{2\tilde{p}c}\Big]^{\tilde{p}}\subseteq \conv\Big(\Big\{\big(\tilde{f}_1(y_{q,j}), ..., \tilde{f}_{\tilde{p}}(y_{q,j})\big): q\in\{1,...,\tilde{p}+1\}, j\in\{1,..., 2^{\tilde{p}}\}\Big\}\Big).
\end{equation}
Now, we can almost apply~\cite[Proposition 7.1]{Seymour_Zaslavsky_1984}. However, it requires finding $\tilde{p}+1$ points such that $0$ is in the interior of the convex hull which is spanned by the evaluation of $(\tilde{f}_1, ..., \tilde{f}_{\tilde{p}})$ at these points. However, we have $2^{\tilde{p}}(\tilde{p}+1)$ points and cannot, at least not directly from~\eqref{e:avset centroid pos convhull distance}, guarantee that a subset of $\tilde{p}+1$ points fulfills this property with a uniform bound on the distance of $0$ from the boundary of the generated convex hull.

Still, by Carathéodory's theorem and~\eqref{e:avset centroid pos convhull distance}, the convex hulls of all subsets of \linebreak $\big\{\big(\tilde{f}_1(y_{q,j}), ..., \tilde{f}_{\tilde{p}}(y_{q,j})\big): q\in\{1,...,\tilde{p}+1\}, j\in\{1,..., 2^{\tilde{p}}\}\}$ with $\tilde{p}+1$ elements form a covering of $\big[-\frac{1}{2\tilde{p}c},\frac{1}{2\tilde{p}c}\big]^{\tilde{p}}$. Hence, there is at least one choice such that the volume of the convex hull intersected with the box is at least $\frac{1}{(\tilde{p}c)^{\tilde{p}}\binom{2^{\tilde{p}}(\tilde{p}+1)}{\tilde{p}+1}}$. We denote these chosen $y_{q,j}$'s by $\tilde{y}_1, ..., \tilde{y}_{\tilde{p}+1}$. Because $\conv(\{\tilde{y}_1, ..., \tilde{y}_{\tilde{p}+1}\})\cap \big[-\frac{1}{2\tilde{p}c},\frac{1}{2\tilde{p}c}\big]^{\tilde{p}}$ is also convex, we can therefore derive from~\cite{Sewell_2024} that the intersection must contain a ball of radius $r:=\frac{1}{2\tilde{p}^2c\binom{2^{\tilde{p}}(\tilde{p}+1)}{\tilde{p}+1}}$. We denote the center of this ball by $x$. As $-x\in \big[-\frac{1}{2\tilde{p}c},\frac{1}{2\tilde{p}c}\big]^{\tilde{p}}$ and by~\eqref{e:avset centroid pos convhull distance}, we can see that there are $\gamma_{q,j}\in[0,1]$ for $q\in\{1,...,\tilde{p}+1\}, j\in\{1,..., 2^{\tilde{p}}\}$ such that $\sum_{q=1}^{\tilde{p}+1}\sum_{j=1}^{2^{\tilde{p}}} \gamma_{q,j}=1$ and 
\begin{equation*}
    -x = \sum_{q=1}^{\tilde{p}+1}\sum_{j=1}^{2^{\tilde{p}}} \gamma_{q,j} \big(\tilde{f}_1(y_{q,j}), ..., \tilde{f}_{\tilde{p}}(y_{q,j})\big).
\end{equation*}
Further, because $\|\tilde{f}_q\|_\infty\leq c$ for $q\in\{1,...,\tilde{p}\}$, we can change all of the $\gamma_{q,j}$'s by amounts of up to $\frac{r}{3\sqrt{\tilde{p}}c(\tilde{p}+1) 2^{\tilde{p}}}$ without moving the linear combination away from $-x$ by more than $\frac{r}{3}$. Hence, there are $\tilde\gamma_{q,j}\in[0,1]\cap\Q$ for $q\in\{1,...,\tilde{p}+1\}, j\in\{1,..., 2^{\tilde{p}}\}$ such that $\sum_{q=1}^{\tilde{p}+1}\sum_{j=1}^{2^{\tilde{p}}} \tilde\gamma_{q,j}=1$ and 
\begin{equation*}
    \sum_{q=1}^{\tilde{p}+1}\sum_{j=1}^{2^{\tilde{p}}} \tilde\gamma_{q,j} \big(\tilde{f}_1(y_{q,j}), ..., \tilde{f}_{\tilde{p}}(y_{q,j})\big)\in -x + B_{\frac{r}{3}},
\end{equation*}
where the $\tilde\gamma_{q,j}$'s can be chosen to have any common denominator greater than $N_1:=\big\lceil\frac{3\sqrt{\tilde{p}}c(\tilde{p}+1) 2^{\tilde{p}}}{r}\big\rceil$. Using that $B$ is open and that $\tilde{f}_q$ is continuous for $q\in\{1,...,\tilde{p}\}$, we can conclude that, for any $n>N_1$, we can find different $z_1,...,z_n\in B$ such that
\begin{equation}
    \tilde{x}:= -\frac{1}{n}\sum_{k=1}^n \big(\tilde{f}_1(z_k), ..., \tilde{f}_{\tilde{p}}(z_k)\big) \in x + B_{\frac{r}{2}}.
\end{equation}
As 
\begin{equation*}
    \tilde{x} + B_{\frac{r}{2}} \subseteq x + B_r \subseteq \conv\Big(\Big\{\big(\tilde{f}_1(y_k),...,\tilde{f}_{\tilde{p}}(y_k)\big):k\in\{1,...,\tilde{p}+1\}\Big\}\Big),
\end{equation*}
we can now apply Theorem~\cite[Proposition 7.1]{Seymour_Zaslavsky_1984} to find the remaining $\tilde{z}_1,...,\tilde{z}_n\in B$ such that
\begin{equation}
    \tilde{x}= \frac{1}{n}\sum_{k=1}^n \big(\tilde{f}_1(\tilde{z}_k), ..., \tilde{f}_{\tilde{p}}(\tilde{z}_k)\big)
\end{equation}
for any $n>N_2:=\big\lceil\frac{4\tilde{p}^\frac{3}{2}c}{r}\big\rceil$. Note the fact that $B$ is not assumed to be connected, which is formally required in~\cite[Proposition 7.1]{Seymour_Zaslavsky_1984}. However, we can extend the polynomials to the set $B_C$, which is connected. To not change the underlying problem, we consider averaging sets on $B_C$ with respect to the measure $\I_B\cdot \lambda_d|_{B_C}$. This extension guarantees that one can connect any finite set of points by a simple path, and~\eqref{e:avset centroid pos convhull distance} guarantees that $(\tilde{f}_1,..,\tilde{f}_{\tilde{p}})(B)$ does not lie a lower-dimensional affine subspace of $\R^{\tilde{p}}$. Finally, because $B$ is open and because the argument in~\cite[Section 6]{Seymour_Zaslavsky_1984} allows us to find the remaining points arbitrarily close to $y_1,...,y_{\tilde{p}+1}$, we can still guarantee that the points $\tilde{z}_1,..., \tilde{z}_n$ lie not only in $B_C$ but also in $B$. In total, we obtain that for any $n\in\N, 2n> N:=2 \max(N_1,N_2)$, there are pairwise distinct $z_1,...,z_n,\tilde{z}_1,...,\tilde{z}_n\in B$ such that
\begin{equation}
    \frac{1}{2n}\sum_{k=1}^n \big(\tilde{f}_1(\tilde{z}_k), ..., \tilde{f}_{\tilde{p}}(\tilde{z}_k)\big) + \big(\tilde{f}_1(z_k), ..., \tilde{f}_{\tilde{p}}(z_k)\big) = \frac{1}{2}(\tilde{x}-\tilde{x}) = 0 = \int_B \big(\tilde{f}_1(z), ..., \tilde{f}_{\tilde{p}}(z)\big)\, dz.
\end{equation}
Further, $\tilde{f}_0=1$, and we have
\begin{equation}
    \frac{1}{2n}\sum_{k=1}^n \tilde{f}_0(\tilde{z}_k) + \tilde{f}_0(z_k) = 1 = \int_B \tilde{f}_0(z)\, dz.
\end{equation}
Hence, $z_1,...,z_n,\tilde{z}_1,...,\tilde{z}_n$ form an averaging set for the functions $\tilde{f}_0,\tilde{f_1},...,\tilde{f}_{\tilde{p}}$, and because these form a basis of the polynomials on $B$ with degree less than $p$, it is also an averaging set of degree $p$. For 
a collection of $2n+1>N$ points, one needs to replace the point $-\tilde{x}$ by $-\frac{n}{n+1}\tilde{x}$, and has to use one more point to fit $-\frac{n}{n+1}\tilde{x}$ than to fit $\tilde{x}$. Then the additional factor cancels out.
\end{proof}
\end{theorem}

We should add that one can drop the assumption that $\lambda_d(\partial B)=0$ if one does not require the points of the averaging set to lie in $B$.

\begin{remark}
    Theorem~\ref{t:averaging set bound} handles the existence of an averaging set of some set $B$, degree $p$, and size $n$. To use such an averaging set in a construction like in Example~\ref{ex:fair tiling high degree}, we also need to find the averaging set measurably in terms of $B$. One option is to try random initializations and use gradient descent to possibly converge to the solution. Since we showed that the solution exists, and because the squares of polynomials are locally convex around their zeros, there is a random but finite number of re-initializations after which a solution can be found in the limit of the gradient descent. Similarly, one can apply Newton's method instead of gradient descent. This option is also the one we picked for the simulations in Subsection~\ref{ss: pp from invariant partitions}.
\end{remark}

In the proof of the preceding Theorem~\ref{t:averaging set bound}, after some simplifications, one first finds a quadrature formula by Carathéodory's theorem, and then replaces it by an averaging set with points close to the original nodes, applying~\cite[Proposition 7.1]{Seymour_Zaslavsky_1984}. An alternative way is to start with an approximate averaging set right away, e.g., by using iid points. If one then uses a base of the polynomials that has orthonormal derivatives, one can make sure that this approximate averaging set lies in a neighborhood of an exact averaging set. In this way, one can also find bounds for the minimum number of points of the averaging set. This technique was applied in~\cite{Wagner_Volkmann_1991, Bondarenko_Viazovska_2008, Bondarenko_Radchenko_Viazovska_2013}. In~\cite{Bondarenko_Radchenko_Viazovska_2013}, the authors even derived the best known bounds with respect to the dimension for averaging sets on the surface of spheres, i.e., spherical designs. However, if one applies this technique to our setting, the assumptions on $B$ have to be a lot stricter to guarantee that the points of the resulting averaging set actually lie in $B$. On the other hand, one can expect the bound to be a lot better.

In the following proposition, we show that one can weaken the assumptions in Theorem~\ref{t:averaging set bound} by bounding the volume of the convex hull of the set instead of the diameter.

\begin{proposition}\label{p:averaging set bound improved}
    For all $d\in\N, C>0, p\in\N_0$, there exists an $N\in\N$ such that for all $n\geq N$, $B\in\cB^d_b$ such that $\lambda_d(\partial B)=0$, $\frac{\lambda_d(\conv(B))}{\lambda_d(B)}\leq C$, there exists an averaging set $x_1,...,x_n$ of $B$ of degree $p$ and size $n$, where $x_1,...,x_n\in B$ and $x_1,...,x_n$ are pairwise different.
\end{proposition}

\begin{remark}
    Proposition~\ref{p:averaging set bound improved} implies that for every $d\in\N, p\in\N_0$, there is an $N\in\N$ such that for all $n\geq N$ and convex $B\in\cB^d_b$ there exists an averaging set $x_1,...,x_n$ of $B$ of degree $p$ and size $n$, where $x_1,...,x_n\in B$ and $x_1,...,x_n$ are pairwise different.
\end{remark}
    
\begin{proof}[Proof of Proposition~\ref{p:averaging set bound improved}]
    Let $C>0$ and suppose that $B\in\cB^d$ with $\frac{\lambda_d(\conv(B))}{\lambda_d(B)}\leq C$. Then there is a minimal volume ellipsoid $E\in\cB^d_b$ such that $B\subset E$. It is called the outer Löwner–John ellipsoid and also satisfies $\frac{1}{d}E\subset \conv(B)$; see, e.g.,~\cite{Henk_2012}. Hence, $\lambda_d(E)\leq d^d\lambda_d(\conv(B))\leq Cd^d\lambda_d(B)$. Therefore, there is a volume-preserving affine transformation $F:\R^d\to\R^d$ such that $F(B)\subseteq F(E)\subseteq B_r$ with $r = \sqrt[d]{\frac{C\lambda_d(B)}{\kappa_d}}d$. Thus, $\frac{\diam(F(B))^d}{\lambda_d(F(B))} \leq \frac{C(2d)^d \lambda_d(B)}{\kappa_d \lambda_d(F(B))} = \frac{C(2d)^d}{\kappa_d}$. Now, the application of Theorem~\ref{t:averaging set bound} yields the desired averaging set of $F(B)$, which can be mapped to an averaging set of $B$ by $F^{-1}$ by Lemma~\ref{l:averaging set linear map}.
\end{proof}

Finally, we show that one cannot find averaging sets of equal order and uniformly bounded size for all bounded sets. One needs to make some regularity assumption, as we do in Theorem~\ref{t:averaging set bound} and Proposition~\ref{p:averaging set bound improved}.

\begin{proposition}\label{p:averaging set counterexample}
    Define $B_{m, p}:=[0, 1-m^{-p}]\cup [m, m+m^{-p}]$ for $m\in\N, p>0$. Then, for every $n\in\N$, there is an $m\in\N$ such that there exists no averaging set of $B_{m, 1}$ of size less than $n$ and degree $1$ if we require the points to lie in $B_m$. If we allow the points of the averaging set to lie in $\R$, then for every $n\in\N$, there is an $m\in\N$ such that there exists no averaging set of $B_{m, 2}$ of size less than $n$ and degree $3$. Hence, in the previous Proposition~\ref{p:averaging set bound improved}, we cannot drop the assumption that $\frac{\lambda_d(\conv(B))}{\lambda_d(B)}\leq C$ for some $C>0$. The same holds for $\frac{\diam(B)^d}{\lambda_d(B)}\leq C$ in Theorem~\ref{t:averaging set bound}.
    \begin{proof}
        Let $n,m\in\N$, and $x_1,...,x_n\in B_{m,1}$. Assume that $x_1,...,x_n$ form an averaging set of $B_{m,1}$ of degree $1$. Then,
        \begin{equation*}
            \frac{1}{n} \sum_{k=1}^n x_k = \int_{B_{m,1}} x\, dx = \frac{3}{2}-\frac{1}{m}+\frac{1}{m^2} > 1-\frac{1}{m}.
        \end{equation*}
        Hence, $x_k\in[m,m+\frac{1}{m}]$ for at least one $k\in\{1,...,n\}$. This fact, together with $x_k\geq0$ for $k\in\{1,...,n\}$ and the previous bound, implies that
        \begin{equation*}
            \frac{3}{2} \geq\frac{1}{n} \sum_{k=1}^n x_k \geq \frac{m}{n}.
        \end{equation*}
        In total, we obtain that $n\geq \frac{2}{3}m$, which proves the first assertion. 

        Now, we instead assume that $x_1,...,x_n$ form an averaging set of $B_{m,2}$ of degree $3$, and that $x_1,...,x_n\in\R$ are not restricted to $B$. Then
        \begin{equation*}
            \frac{1}{n} \sum_{k=1}^n x_k^3 = \int_{B_{m,2}} x^3\, dx = m + \frac{1}{4} + \frac{1}{2m^2} + \frac{3}{2m^4} +\frac{1}{m^5} - \frac{1}{m^6} +\frac{1}{2m^8} \geq m.
        \end{equation*}
        Hence, $|x_k|\geq \sqrt[3]{m}$ for at least one $k\in\{1,...,n\}$. Therefore,
        \begin{equation*}
            \frac{m^\frac{2}{3}}{n} \leq \frac{1}{n} \sum_{k=1}^n x_k^2 = \int_{B_{m,2}} x^2\, dx = \frac{4}{3} - \frac{1}{m^2} +\frac{1}{m^3} + \frac{1}{m^4} \leq \frac{7}{3}.
        \end{equation*}
        In total, we obtain that $n\geq \frac{3}{7}m^\frac{2}{3}$, which proves the second assertion.
    \end{proof}
\end{proposition}

If one considers the optimal complex averaging set for $B_{m,2}$ from Theorem~\ref{t: complex averaging sets}, one can see that for symmetry reasons, it must have the form $a_m, b_m+ic_m, b_m-ic_m$ with $a_m,b_m,c_m\in\R$ for $m\in\N$. As $m\to\infty$, we can further calculate that $b_m\to-\infty$ with $|c_m|\sim \sqrt{3}|b_m|$ and $a_m\sim -2b_m$. Hence, while $\lambda_1(B_{m,2})$ is constant, the three points move further and further away from each other, and two of them also further away from the real line. This could be a heuristic explanation for the effect that, on the real line, one needs more and more points to form an averaging set.

\subsection{Averaging sets of trigonometric polynomials}\label{ss: Averaging sets of trigonometric polynomials}

\begin{definition}\label{d:trigonometric polynomials}
    A function $p:\R\to\C$ is called a \textit{trigonometric polynomial} of degree $m\in\N_0$ if there are $a_{-m},...,a_m\in\C$ such that $a_{-m}\neq 0$ or $a_m\neq 0$, and
    \begin{equation}
        p(x) = \sum_{k=-m}^m a_k e^{ikx},\quad x\in\R.
    \end{equation}
    Similarly, a function $p:\R^l\to\C$ is called a \textit{multivariate trigonometric polynomial} of degree $(m_1,...,m_l)\in\N_0^l$ if there are $a_{k_1,...,k_l}\in\C$, $k_1,...,k_l\in\Z$ with $|k_j|\leq m_j$ for $j\in\{1,...,l\}$, such that
    \begin{equation}
        p(x_1,...,x_l) = \sum_{\substack{k_1,...,k_l\in\Z,\\|k_j|\leq m_j \textnormal{ for } j\in\{1,...,l\}}} a_{k_1,...,k_l} e^{ik_1x_1+...+ik_lx_l},\quad x\in\R^l,
    \end{equation}
     and such that, for all $j\in\{1,...,l\}$, there are $k_1,...,k_l\in\Z$ with $a_{k_1,...,k_l}\neq 0$ and $|k_j|=m_j$.
\end{definition}

\begin{definition}
    Let $n,l\in\N, m\in\N_0$, $p_1,...,p_d:\R^l\to\R$ be real-valued (multivariate) trigonometric polynomials, and $x_1,...,x_n\in\R^d$. We say that $x_1,...,x_n$ form an \textit{averaging set} of $(p_1,...,p_d)$ of order $m$ and size $n$ if they form an averaging set of
    \begin{equation}
        (p_1,...,p_d)(\lambda_l|_{[0,2\pi)^l}) = \int_{[0,2\pi)^l} \I\{(p_1(x),...,p_d(x))\in\cdot\}\, dx
    \end{equation}
    of order $m$ and size $n$, i.e.,
    \begin{equation}
        \frac{1}{(2\pi)^l}\int_{[0,2\pi)^l} (p_1(x),...,p_d(x))^{\otimes q}\, dx = \frac{1}{n} \sum_{k=1}^n x_k^{\otimes q},\quad q\in\{0,...,m\}.
    \end{equation}
\end{definition}

Here, we do not require the points to be pairwise different in general, and use the set notation with a slight abuse of notation, as the averaging sets could be multisets at times. Further, similarly to how the averaging set $x_1,...,x_n$ of some set $B\in\cB_b^d$ was usually required to be a subset of $B$ in the previous subsection, here the averaging sets always lie in the image of the trigonometric polynomials. 

\begin{proposition}\label{p:trig poly avset}
    Let $m\in\N_0$, and suppose that $p$ is a trigonometric polynomial of degree $m$. Then, for all $n\geq m+1$,
    \begin{equation}\label{e:trig poly avset fit}
        \frac{1}{2\pi}\int_0^{2\pi} p(x)\, dx = \frac{1}{n} \sum_{j=1}^n p(2\pi j/n+u),\quad u\in\R,
    \end{equation}
    but if $m\in\N$,
    \begin{equation}\label{e:trig poly avset nonfit}
        u\mapsto \frac{1}{m} \sum_{j=1}^m p(2\pi j/m+u), \quad u\in\R,
    \end{equation}
    is not constant.

\begin{proof}
    Let $m\in\N_0$, $n\in\N$, and $u\in\R$. We show that~\eqref{e:trig poly avset fit} holds for a trigonometric monomial $p(x)=e^{imx}$. The general case then follows from symmetry and linearity. Also, for $m=0$, both sides are equal to $1$, whereby the assertion holds. Hence, without loss of generality, we can assume $m\in\N$. In this case,
    \begin{equation*}
        \int_0^{2\pi} e^{imx}\, dx=0,
    \end{equation*}
    and what remains to show is that
    \begin{equation*}
        \sum_{j=1}^n e^{im(2\pi j/n+u)} = 0.
    \end{equation*}
    Further, as
    \begin{equation*}
        \sum_{j=1}^n e^{im(2\pi j/n+u)} = e^{imu} \sum_{j=1}^n e^{2\pi i jm/n},
    \end{equation*}
    it suffices to show that
    \begin{equation*}
        \sum_{j=1}^n e^{2\pi i jm/n} = 0.
    \end{equation*}
    Now let $l\in\N$ be the greatest common factor of $n$ and $m$. Then,
    \begin{equation*}
        \sum_{j=1}^n e^{2\pi i jm/n} = \sum_{k=0}^{l-1}\sum_{j=1}^{n/l} e^{2\pi i (kn/l+j)m/n} = \sum_{k=0}^{l-1} \underbrace{e^{2\pi ikm/l}}_{=1}\sum_{j=1}^{n/l} e^{2\pi ijm/n} = l\sum_{j=1}^{n/l} e^{2\pi ij(m/l)/(n/l)}.
    \end{equation*}
    Hence, it suffices to show the assertion for $n,m$ coprime, as dividing by a common factor keeps the property that $n\geq m+1\geq 2$. In that case, we can reorder the sum and perform an index shift to obtain
    \begin{equation*}
        \sum_{j=1}^n e^{2\pi i jm/n} = \sum_{j=1}^n e^{2\pi i j/n} = e^{2\pi i/n} \sum_{j=1}^n e^{2\pi i j/n}.
    \end{equation*}
    This relation concludes the first part of the proof, as $e^{2\pi i/n}\neq 1$, since $n\geq2$, whereby the right equation can only hold if both sides are $0$.

    For the second part, let $p$ be a general trigonometric polynomial with coefficients \linebreak $a_{-m},...,a_m\in\C$. Now, we can apply~\eqref{e:trig poly avset fit} to the different trigonometric monomials of $p$ to obtain
    \begin{equation*}
        \frac{1}{m} \sum_{j=1}^m p(2\pi j/m+u) = a_0 + \frac{a_{-m}}{m} \sum_{j=1}^m e^{-2\pi ij-imu} + \frac{a_m}{m} \sum_{j=1}^m e^{2\pi ij+imu} = a_0 + a_{-m} e^{-imu} + a_m e^{imu}.
    \end{equation*}
    This equation yields the second assertion, as $u\mapsto e^{-imu}$ and $u\mapsto e^{imu}$ are linearly independent and $a_{-m}\neq0$ or $a_m\neq0$.
\end{proof}
\end{proposition}
Also note that, in the case that $p$ is real valued, the application of the intermediate value theorem to the last equation in the proof yields that there exist specific $u\in\R$ such that $\{2\pi j/n+u:j\in\{1,...,m\}\}$ forms an averaging set of $p$ of order $1$. Further, a slight refinement yields that there even exists a $u\in\R$ such that $\{2\pi j/n+u:j\in\{1,...,\lceil\frac{m+1}{2}\rceil\}\}$ forms an averaging set of $p$ of order $1$. This fact is connected to the derivation of the original Chebyshev-Gauss quadrature.

\begin{theorem}\label{t:averaging set trigonometric polynomial}
    Let $n\in\N,m\in\N_0$, and suppose that $p$ is a real-valued trigonometric polynomial of degree $m$. Then 
    \begin{equation}
        \{p(2\pi j/n+u):j\in\{1,...,n\}\}
    \end{equation}
    is an averaging set of $p$ of order $\lceil\frac{n}{m}\rceil-1$ for all $u\in\R$, but not of order $\frac{n}{\textnormal{gcf}(n,m)}$ for almost all $u\in\R$.

    Further, suppose that $p_1,..., p_d$ are real-valued trigonometric polynomials of degrees \linebreak $m_1,...,m_d\in\N_0$. Then 
    \begin{equation}
        \big\{\big(p_1(2\pi j/n+u),...,p_d(2\pi j/n+u)\big):j\in\{1,...,n\}\big\}
    \end{equation}
    is an averaging set of $(p_1,...,p_d)$ of order $\lceil\frac{n}{\max(m_1,...,m_d)}\rceil-1$ for all $u\in\R$, but not of order $\frac{n}{\max(\textnormal{gcf}(n,m_1),...,\textnormal{gcf}(n,m_d))}$ for almost all $u\in\R$.
    \begin{proof}
        We show only the first assertion. The second follows in just the same way. We assume that $m\in\N$, as the assertion is trivial for $m=0$. First, we have to show that
        \begin{equation}
            \frac{1}{2\pi} \int_0^{2\pi} p(x)^k \, dx = \frac{1}{n}\sum_{j=1}^n p(2\pi j/n+u)^k,\quad k\in\{0,...,\lceil\tfrac{n}{m}\rceil-1\}, u\in\R.
        \end{equation}
        This follows from Proposition~\ref{p:trig poly avset} and the fact that $p^k$ is a trigonometric polynomial of degree $km\leq (\lceil\tfrac{n}{m}\rceil-1) m\leq n-1$. For the second part, let $l:=\frac{m}{\textnormal{gcf}(n,m)}\in\N$. Note that $p^{\frac{n}{\textnormal{gcf}(n,m)}}$ is a trigonometric polynomial of degree $\frac{n}{\textnormal{gcf}(n,m)}m=ln$. Further,
        \begin{equation*}
            \frac{1}{n}\sum_{j=1}^n e^{iln(2\pi j/n+u)} = e^{ilnu} \frac{1}{n}\sum_{j=1}^n e^{2\pi ilj} = e^{ilnu},\quad u\in\R.
        \end{equation*}
        As $p^{\frac{n}{\textnormal{gcf}(n,m)}}$ is real-valued and of degree $ln$, this monomial has non-zero coefficient. Hence,
        \begin{equation*}
            \frac{1}{n}\sum_{j=1}^n p(2\pi j/n+u)^{ln} = a e^{ilnu} + q(u), \quad u\in\R,
        \end{equation*}
        where $a\in\C\setminus\{0\}$ and $q$ is a trigonometric polynomial which is linearly independent of the monomial. Hence, all level sets of the RHS are null sets.
    \end{proof}
\end{theorem}

\begin{theorem}\label{t:averaging set multivariate trigonometric polynomial}
    Let $l,n_1,...,n_l\in\N, m_1,..,m_l\in\N_0$, and suppose that $p$ is a real-valued multivariate trigonometric polynomial of degree $(m_1,...,m_l)$. Then
    \begin{equation}
        \big\{p(2\pi j_1/n_1+u_1, ..., 2\pi j_l/n_l+u_l):j_k\in\{1,...,n_k\} \textnormal{ for } k\in\{1,...,l\}\big\}
    \end{equation}
    is an averaging set of $p$ of degree $\min_{j\in\{1,...,l\}}(\lceil\frac{n_l}{m_l}\rceil)-1$ for all $u\in\R^l$.

    Further, suppose that $p_1,...,p_d$ are real-valued multivariate trigonometric polynomials with degrees bounded by $(m_1,..., m_l)$. Then 
    \begin{equation}
        \left\{\begin{pmatrix}
            p_1(2\pi j_1/n_1+u_1, ..., 2\pi j_l/n_l+u_l)\\
            \vdots\\
            p_d(2\pi j_1/n_1+u_1, ..., 2\pi j_l/n_l+u_l)
        \end{pmatrix}:j_k\in\{1,...,n_k\} \textnormal{ for } k\in\{1,...,l\}\right\}
    \end{equation}
    is an averaging set of $(p_1,...,p_d)$ of degree $\min_{j\in\{1,...,l\}}(\lceil\frac{n_l}{m_l}\rceil)-1$ for all $u\in\R^l$.
    \begin{proof}
        Again, we only show the first assertion, as the second follows in the same way. Additionally, we only show it for a monomial with positive exponents, as the general case can be derived by symmetry and linearity afterwards. Hence, $p(x_1,...,x_l) = e^{im_1x_1+...+im_lx_l}$ for $x\in\R^l$. In this case,
        \begin{equation*}
            \frac{1}{(2\pi)^l}\int_{[0,2\pi)^l} e^{im_1x_1+...+im_lx_l}\, dx = \prod_{k=1}^l\frac{1}{2\pi}\int_0^{2\pi} e^{im_kx}\, dx,
        \end{equation*}
        and
        \begin{equation*}
            \frac{1}{n_1\cdots n_l}\sum_{\substack{j_k\in\{1,...,n_k\}\\\textnormal{for }k\in\{1,...,l\}}} e^{im_1(2\pi j_1/n_1+u_1)+...+im_l(2\pi j_l/n_l+u_l)} = \prod_{k=1}^l \frac{1}{n_k}\sum_{j=1}^{n_k} e^{im_k(2\pi j/n_k+u_k)},\quad u\in\R^l.
        \end{equation*}
        Now, we can derive the equality from the univariate case in Theorem~\ref{t:averaging set trigonometric polynomial}.
    \end{proof}
\end{theorem}

In the multivariate case, it is harder to calculate the exact degree of the averaging set, as it depends on many properties of the trigonometric polynomial(s). However, in the following case, we can still get a precise statement.

\begin{proposition}\label{p:averaging set multivariate trigonometric polynomial degree 1}
    Suppose that $l,n\in\N$, and that $p_1,...,p_d$ are real-valued trigonometric polynomials of degrees $m^{(1)},...m^{(d)}\in\{0,1\}^l$. Then
    \begin{equation}\label{e:averaging set degree 01}
        \left\{\begin{pmatrix}
            p_1(2\pi j_1/n+u_1, ..., 2\pi j_l/n+u_l)\\
            \vdots\\
            p_d(2\pi j_1/n+u_1, ..., 2\pi j_l/n+u_l)
        \end{pmatrix}:j\in\{1,...,n\}^l\right\}
    \end{equation}
    is an averaging set of $(p_1,...,p_d)$ of degree $n-1$ for all $u\in\R^l$, and not of any higher degree for almost all $u\in\R^l$ if at least one trigonometric polynomial is not constant.

    \begin{proof}
        An application of Theorem~\ref{t:averaging set multivariate trigonometric polynomial} directly yields that the set in~\eqref{e:averaging set degree 01} is an averaging set of $(p_1,...,p_d)$ of degree $m-1$.
        Now, we have to show that this does not hold for any higher degree.
        In the proof of the previous Theorem~\ref{t:averaging set multivariate trigonometric polynomial}, we saw that the problem factorizes for multivariate monomials. Here, $p_k^n$ also always contains a monomial of the form $e^{im^{(k)}_1nx_1+...+m^{(k)}_lnx_l}$ for \linebreak $k\in\{1,..,d\}$. If $p_k$ is not constant, then $m^{(k)}_j=1$ for some $j\in\{1,...,l\}$. Without loss of generality, we assume $m^{(1)}_1=1$. Then,
        \begin{equation*}
            \frac{1}{(2\pi)^l}\int_{[0,2\pi)^l} e^{im^{(k)}_1nx_1+...+m^{(k)}_lnx_l}\, dx = \prod_{k=1}^l\frac{1}{2\pi}\int_0^{2\pi} e^{im_knx}\, dx = 0,
        \end{equation*}
        and 
        \begin{equation*}
            \frac{1}{n^l}\sum_{j\in\{1,...,n\}^k} e^{im^{(1)}_1n(2\pi j_1/n+u_1)+...+im^{(1)}_ln(2\pi j_l/n+u_l)} = \prod_{k=1}^l \frac{1}{n}\sum_{j=1}^{n} e^{2\pi im^{(1)}_k j+im^{(1)}_knu_k},\quad u\in\R^l.
        \end{equation*}
        On the RHS of the second equation, the first factor is equal to $e^{inu_1 }$, and the others are either equal to $1$ or $e^{inu_k }$. In total, the product is not $0$ for almost all $u\in\R^l$, which concludes the proof.
    \end{proof}
\end{proposition}

\section{Hyperfluctuating examples}\label{s: Hyperfluctuating examples}

The goal of this section is to construct novel hyperfluctuating point processes, i.e., point processes which are not $0$-uniform. Our particular interest lies in those examples that one may expect to be hyperuniform using heuristic arguments, but which are actually hyperfluctuating instead. We explore these in Appendix~\ref{ss: Hyperfluctuating point processes from the lattice partition}. These constructions rely on the fact that there are hyperfluctuating thinnings of the stationary lattice, which we construct in Appendix~\ref{ss: Hyperfluctuating thinnings of the stationary lattice}. Further, the point processes we construct are ergodic. This restriction directly excludes all examples where the only source of hyperfluctuation is an atom of the Bartlett spectral measure in the origin. These point processes would be trivial to construct.

\subsection{Hyperfluctuating thinnings of the stationary lattice}\label{ss: Hyperfluctuating thinnings of the stationary lattice}

We first show the existence of hyperfluctuating thinnings of the stationary lattice for $d=1$ and then extend the construction to dimensions $d\geq 2$. For $d=1$, the construction relies on renewal processes. 
\begin{definition}[Renewal process]
    Suppose that $\Phi\neq0$ is a stationary point process on $\R$, and that $\BQ$ is a probability distribution on $[0,\infty)$. Further, suppose that $X_1,X_2,...$ are random variables on $[0,\infty)$ such that $X_1\leq X_2\leq\ldots$ and
    \begin{equation}
        \Phi|_{[0,\infty)} = \sum_{n=1}^\infty \delta_{X_n}.
    \end{equation}
    Then $\Phi$ is called a \textit{(stationary) renewal process} with renewal distribution $\BQ$ if $(X_{n+1}-X_n)_{n\in\N}$ is a family of independent random variables with the common distribution $\BQ$.
\end{definition}
It was first shown by J. L. Doob that such stationary renewal processes exist iff the expected renewal time is finite and not almost surely 0, i.e.,
\begin{equation}\label{e:renewal existence condition}
    0 <\int x \, \BQ(dx) < \infty,
\end{equation}
see~\cite{Doob_1948}.
Hence, from now on, we always assume that $\BQ$ is a probability distribution on $[0,\infty)$ that satisfies~\eqref{e:renewal existence condition} and that $\Phi$ is a stationary renewal process with renewal distribution $\BQ$. Further, for simplicity of notation, we let $\mu := \int x\,\BQ(dx)$. We also know that $\Phi$ has an intensity of $\frac{1}{\mu}$. For the analysis of $p$-uniformity of such renewal processes, we can rely on results by M. S. Bartlett, where he already provided a formula for the density of the part of the Bartlett spectral measure which is continuous with respect to the Lebesgue measure
\begin{equation}\label{e:renewal bartlet formula}
    g(k) = \frac{1}{\mu}\bigg(1+2\Re\bigg(\frac{\hat\BQ(k)}{1-\hat\BQ(k)}\bigg)\bigg),\quad k\in\{s\in\R:\hat\BQ(s)\neq 1\},
\end{equation}
see~\cite[Section 6]{Bartlett_1963}. In the points where this density is not defined, there may lie atoms, as in the case of the stationary lattice ($\BQ=\delta_1$). However, by Proposition~\ref{p:pseudo ergodic}, there never lies an atom at the origin since stationary renewal processes are ergodic by the independence of the lengths of the renewal intervals. In the following proposition, we see that the renewal distribution of the renewal process we are looking for cannot have a finite second moment; see also~\cite[Section 2 Theorem 5.2]{Gut_2009}.
\begin{proposition}
    Suppose that
    \begin{equation}\label{e:renewal finite second moment}
        \int x^2\,  \BQ(dx) < \infty,
    \end{equation}
    and choose $\sigma^2:= \int x^2\,  \BQ(dx)-\mu^2$.
    Then, as $k\to 0$,
    \begin{equation}
        g(k) = \frac{\sigma^2}{\mu^3} + o(1).
    \end{equation}
    Hence, $\Phi$ is solely $0$-uniform with 
    \begin{equation}
        \frac{\BV[\Phi(B_r)]}{\lambda_d(B_r)} \xrightarrow{r\to\infty} \frac{\sigma^2}{\mu^3}.
    \end{equation}
    \begin{proof}
        By Taylor's theorem and~\eqref{e:renewal finite second moment}, as $k\to 0$,
        \begin{equation*}
            \hat\BQ(k) = 1-i\mu k - \frac{1}{2}(\sigma^2+\mu^2)k^2 + o(k^2).
        \end{equation*}
        Therefore, as $k\to 0$,
        \begin{align*}
            \Re(1-\hat\BQ(k)) &= \frac{1}{2}(\sigma^2+\mu^2) k^2 + o(k^2),\\
            \Im(1-\hat\BQ(k)) &= \mu k + o(k^2).
        \end{align*}
        These can be inserted in~\eqref{e:renewal bartlet formula} to obtain, as $k\to 0$,
        \begin{align*}
            g(k) &= \frac{1}{\mu}\bigg(1+2\Re\bigg(\frac{\hat\BQ(k)}{1-\hat\BQ(k)}\bigg)\bigg) \nonumber\\
            &= \frac{1}{\mu} \bigg(2\Re\bigg(\frac{1}{1 - \hat\BQ(k)}\bigg) - 1\bigg) \nonumber\\
            &= \frac{1}{\mu}  \bigg(2\frac{\frac{1}{2}(\sigma^2+\mu^2) k^2 + o(k^2)}{\frac{1}{4}(\sigma^2+\mu^2)^2 k^4 + o(k^4) + \mu^2 k^2 + o(k^3)} - 1\bigg) \nonumber\\
            &= \frac{1}{\mu} \bigg(\frac{\sigma^2+\mu^2}{\mu^2} - 1 + o(1)\bigg) \nonumber\\
            &= \frac{\sigma^2}{\mu^3} + o(1).
        \end{align*}
    \end{proof}
\end{proposition}
However, if the second moment of the renewal distribution is infinite, we can construct stationary renewal processes which are not $0$-uniform. Because we want to construct a renewal process which can be interpreted as a thinning of the stationary lattice, we need $\BQ(\N)=1$. That is why we leverage the following class of probability distributions.
\begin{definition}[Zeta distribution]
    Let $s>1$. Suppose that
    \begin{equation}
        \BQ = \frac{1}{\zeta(s)} \sum_{n=1}^\infty \frac{1}{n^s} \delta_n,
    \end{equation}
    where
    \begin{equation}
        \zeta(s) := \sum_{n=1}^\infty \frac{1}{n^s}.
    \end{equation}
    Then $\BQ$ is a \textit{zeta distribution} with parameter $s$.
\end{definition}
It is clear that the first moment is finite but the second is not iff $s\in(2,3]$. This choice leads to the following example.
\begin{proposition}\label{p: d=1 hyperfluctuationg lattice thinning}
    Suppose that $\BQ$ is a zeta distribution with parameter $s\in(2,3)$. Then $\Phi$ is solely $(s-3)$-uniform with
    \begin{equation}\label{e: d=1 hyperfluctuationg lattice thinning variance limit}
        \frac{\BV[\Phi([0,r))]}{r^{4-s}} \xrightarrow{r\to\infty} \frac{\zeta(s)^2\Gamma(s-4)}{2\zeta(s-1)^3\Gamma(s)} > 0.
    \end{equation} 
    If $\BQ$ is a zeta distribution with parameter $s=3$, then $\Phi$ is $p$-uniform for any $p<0$, but not $0$-uniform.
    \begin{proof}
        First, suppose that $s\in(2,3)$. Then Jonqui\`ere's expansion yields that, as $k\to 0$,
        \begin{equation*}
            \hat\BQ(k) = 1 -i\frac{\zeta(s-1)}{\zeta(s)}k + \frac{\Gamma(1-s)}{\zeta(s)} (ik)^{s-1} + O(k^2); 
        \end{equation*}
        see, e.g.,~\cite[Subsection 25.12]{Olver_Lozier_Boisvert_Clark_2010} or~\cite[Section 9]{Wood_1992}. Hence, as $k\to0$,
        \begin{align*}
            \Re(1-\hat\BQ(k)) &= -\frac{\Gamma(1-s)}{\zeta(s)} \sin\Big(\frac{s\pi}{2}\Big) |k|^{s-1} + O(k^2),\\
            \Im(1-\hat\BQ(k)) &= \frac{\zeta(s-1)}{\zeta(s)}k + o(k).
        \end{align*}
        If we apply these approximations to~\eqref{e:renewal bartlet formula}, we obtain, as $k\to0$,
        \begin{align}
            g(k) &= \frac{\zeta(s)}{\zeta(s-1)}\bigg(1+2\Re\bigg(\frac{\hat\BQ(k)}{1-\hat\BQ(k)}\bigg)\bigg)\nonumber\\
            &= \frac{\zeta(s)}{\zeta(s-1)} \bigg(2\Re\bigg(\frac{1}{1 - \hat\BQ(k)}\bigg) - 1\bigg) \nonumber\\
            &= \frac{\zeta(s)}{\zeta(s-1)} \bigg(\frac{-\frac{\Gamma(1-s)}{\zeta(s)} \sin(\frac{s\pi}{2}) |k|^{s-1} + O(k^2)}{\frac{\zeta(s-1)^2}{\zeta(s)^2}k^2 + o(k^2)} - 1\bigg)\nonumber\\
            &= \underbrace{-\frac{\zeta(s)^2\Gamma(1-s)\sin(\frac{s\pi}{2})}{\zeta(s-1)^3}}_{>0} |k|^{s-3} + O(1).
        \end{align}
        Now we can calculate that
        \begin{align*}
            \lim_{r\to\infty} \frac{\BV[\Phi([0,r))]}{r^{4-s}} &= \lim_{r\to\infty} (2\pi)^{-1}r^{s-4} \int |\hat\I_{[0,r)}(k)|^2 \, \hat\beta_\Phi(dk) \nonumber\\
            &= \lim_{r\to\infty} - (2\pi)^{-1}r^{s-2} \int_0^\infty \frac{\sin(\tfrac{rk}{2})^2}{(\tfrac{rk}{2})^2} \frac{\zeta(s)^2\Gamma(1-s)\sin(\frac{s\pi}{2})}{\zeta(s-1)^3} k^{s-3} \, dk \nonumber\\
            &= -\frac{2^{s-2}\zeta(s)^2\Gamma(1-s)\sin(\frac{s\pi}{2})}{2\pi\zeta(s-1)^3} 
            \int_0^\infty \sin(k)^2 k^{s-5}\, dk\nonumber\\
            &= \frac{\zeta(s)^2}{\pi\zeta(s-1)^3} \Gamma(s-4) \Gamma(1-s)\sin(\tfrac{s\pi}{2}) \cos(\tfrac{s\pi}{2})\nonumber\\
            &= \frac{\zeta(s)^2\Gamma(s-4)}{2\zeta(s-1)^3\Gamma(s)},
        \end{align*}
        where we first used the variance formula~\eqref{e: spectral covariance formula}, then substitution, the Mellin transform for the last integral; see, e.g.,~\cite[Section 13]{Bracewell_2000}, and the gamma reflection identity and an addition theorem in the final step.
        This limit finishes the first part of the proof. Let us now assume that $s=3$. Then another formula gives the expression
        \begin{equation*}
            \hat\BQ(k) = 1 -i\frac{\zeta(3)}{\zeta(2)}k  - \frac{k^2}{2} \bigg(\frac{3}{2} - \ln(ik)\bigg) + o(k^2),
        \end{equation*}
        see~\cite[Section 9]{Wood_1992}.
        A similar calculation yields
        \begin{equation}
            g(k) = \frac{\zeta(2)}{\zeta(3)}\ln\Big(\frac{1}{|k|}\Big) + O(1),
        \end{equation}
        as $k\to 0$, which also proves this part of the assertion.
    \end{proof}
\end{proposition}
Now, we can lift these hyperfluctuating point processes to higher dimensions by constructing hyperplane intersection processes. Suppose that $\BQ$ is a zeta distribution with parameter \linebreak $s\in(2,3)$. Let $\Phi_1,...,\Phi_d$ be independent stationary renewal processes with renewal distribution $\BQ$, and define the hyperplane intersection process $\Psi$ on $\R^d$ by
\begin{equation}\label{e: d>=1 hyperfluctuationg lattice thinning}
    \Psi := \Phi_1 \otimes \cdots \otimes \Phi_d.
\end{equation}
\begin{proposition}\label{p: d>=1 hyperfluctuationg lattice thinning}
    The hyperplane intersection process $\Psi$, defined as in~\eqref{e: d>=1 hyperfluctuationg lattice thinning}, is solely \linebreak $(-d+(s-2))$-uniform with
    \begin{equation}\label{e:asymptotic variance of hyperfluctuating lattice thinning d >=1}
        \frac{\BV[\Psi([0,r)^d)]}{r^{2d-(s-2)}} \xrightarrow{r\to\infty} \frac{d\zeta(s)^2\Gamma(s-4)}{2\zeta(s-1)^{2d+1}\Gamma(s)} > 0.
    \end{equation}
    Hence, for any $ d\in\N,p\in(-d,-d+1)$, there is an $s\in(2,3)$ such that $\Psi$ is a solely $p$-uniform thinning of the stationary lattice.
    Further, the distribution of $\Psi$ is not only invariant under translation but also under all isometries of the lattice.
    \begin{proof}
        By Proposition~\ref{p: d=1 hyperfluctuationg lattice thinning}, we know that
        \begin{equation*}
        \frac{\BV[\Phi_1([0,r)])]}{r^{4-s}} \xrightarrow{r\to\infty} \frac{\zeta(s)^2\Gamma(s-4)}{2\zeta(s-1)^3\Gamma(s)} > 0.
        \end{equation*}
        As $\BE[\Phi_1([0,r)])] = \frac{1}{\zeta(s-1)}r$,
        This relation implies that
        \begin{equation*}
            \frac{\BV[\Phi_1([0,r))]}{\BE[\Phi_1([0,r))]^2} \xrightarrow{r\to\infty} 0.
        \end{equation*}
        Therefore, as $r\to\infty$,
        \begin{align*}
            \BV[\Psi([0,r)^d)] &= \BV[\Phi_1([0,r))\cdots\Phi_d([0,r))] \nonumber\\
            &= \BE\big[\Phi_1([0,r))^2\big]^d - \big(\BE[\Phi_1([0,r))]^2\big)^d \nonumber\\
            &= d \big(\BE[\Phi_1([0,r))]^2\big)^{d-1} \BV[\Phi_1([0,r))] + o(\big(\BE[\Phi_1([0,r))]^2\big)^{d-1} \BV[\Phi_1([0,r))]) \nonumber\\
            &= \underbrace{\frac{d\zeta(s)^2\Gamma(s-4)}{2\zeta(s-1)^{2d+1}\Gamma(s)}}_{>0}r^{2d-(s-2)} + o(r^{2d-(s-2)}).
        \end{align*}
        Finally, since $s-2<1$ and $\I_{[0,1)^d}$ is Fourier-smooth with exponent $d-2$ by Example~\ref{ex: fourier smooth}, we can leverage Theorem~\ref{t:asymptotic variance test function} to obtain that $\Psi$ is solely $(-d+(s-2))$-uniform.
    \end{proof}
\end{proposition}

\subsection{Hyperfluctuating point processes from the lattice partition}\label{ss: Hyperfluctuating point processes from the lattice partition}

Now we are ready to construct our first example of an ergodic point process which has exactly one point in each cell of a stationary lattice tiling, but is hyperfluctuating for $d\geq 3$. Besides the ergodicity, such an example was also given in~\cite{Dereudre_Flimmel_Huesmann_Leblé_2024}.

Let $p\in(0,1)$. Suppose that $\Phi$ is a stationary thinning of the stationary lattice $\Z^d+U$, $U\sim \cU([0,1)^d)$, which is solely $(-d+p)$-uniform. Further, assume that the distribution of $\Psi$ is invariant under isometries of the lattice. Note that such a point process was constructed in~\eqref{e: d>=1 hyperfluctuationg lattice thinning} and proven to have these properties in Proposition~\ref{p: d>=1 hyperfluctuationg lattice thinning}. Define
\begin{equation}\label{e: hyperfluctuating points in lattice}
    \Psi := \sum_{z\in\Z^d} \I\{\Phi(\{z+U\})=1\} \delta_{z+\frac{1}{2}e_1+U} + \I\{\Phi(\{z+U\})=0\} \delta_{z-\frac{1}{2}e_1+U}.
\end{equation}

\begin{proposition}\label{p: hyperfluctuating points in lattice}
    The point process $\Psi$, defined as in~\eqref{e: hyperfluctuating points in lattice}, is solely $(-d+2+p)$-uniform. This degree is negative iff $d\geq 3$.
    \begin{proof}
        Let $\tilde\Phi:=\Z^d+U$ and define the invariant transport kernel $K$ by
        \begin{equation}
            K_x := \I\{\Phi(\{x\})=1\} \delta_{x+\frac{1}{2}e_1} + \I\{\Phi(\{x\})=0\} \delta_{x-\frac{1}{2}e_1}, \quad x\in\R^d.
        \end{equation}
        Note that therefore $\Psi = K\tilde\Phi$, and
        \begin{equation}
            K^\ast_x = \delta_{(\Phi(\{x\})-\frac{1}{2})e_1}, \quad x\in\R^d.
        \end{equation}
        Now we apply Theorem~\ref{t:main theorem} with the parameter $\tilde{p}:=-d+2+p$. As $K$ is a probability kernel and the transport distance is bounded by $\frac{1}{2}$, Condition~\ref{c:condition square-integrable general with p} is obviously fulfilled. Let us now define $\Psi_q$ for $q\in\N_0$ like in~\eqref{e:Psi_q definition}, whereby
        \begin{equation}
            \Psi_q(dy) = \Big(\Big(\Phi(\{y\})-\frac{1}{2}\Big)e_1\Big)^{\otimes q} \tilde\Phi(dy).
        \end{equation}
        Hence, for $q\in2\N_0$, we have $\Psi_q = \frac{1}{2^q}e_1^{\otimes q}\tilde\Phi$, which implies that it is $\infty$-uniform. Moreover, for $q\in2\N_0+1$, we have $\Psi_q = \frac{1}{2^{q-1}}e_1^{\otimes q}\Phi-\frac{1}{2^q}e_1^{\otimes q}\tilde\Phi$, which implies that it is $(-d+p)$-uniform. In total, Theorem~\ref{t:main theorem} yields that $\Psi$ is $(-d+2+p)$-uniform, where the dominant contribution comes from $\Psi_1$. We improve this result by applying Theorem~\ref{t: main theorem rate} to show that $\Psi$ is not beyond $(-d+2+p)$-uniform. Let $0\neq f\in\cC_c^\infty(\R^d, [0,\infty))$ be isotropic and define $f_r:=f(\frac{\cdot}{r})$ for $r>0$. Now we have to show that
        \begin{equation}\label{e: Psi_1 limit condition hyperfluctuating points}
            \limsup_{r\to\infty} \frac{\BV\big[\Psi_1\big(f_r'\big)\big]}{r^{2d-(2+p)}} > 0.
        \end{equation}
        Using what we calculated so far and the variance formula~\eqref{e: spectral covariance formula}, we obtain that, as $r\to\infty$,
        \begin{align*}
            \frac{\BV\big[\Psi_1\big(f_r'\big)\big]}{r^{2d-(2+p)}} &= \frac{1}{4^{q-1}r^{2d-(2+p)}}\BV[\Phi(\partial_1 f_r)] + o(1)\nonumber\\
            &= \frac{r^{2+p}}{4^{q-1}(2\pi)^d}\int k_1^2 |\hat{f}(rk)|^2\,\hat\beta_\Phi(dk) + o(1)\nonumber\\
            &= \frac{r^{p}}{4^{q-1}d(2\pi)^d}\int \|rk\|^2 |\hat{f}(rk)|^2\,\hat\beta_\Phi(dk) + o(1),
        \end{align*}
        where we applied the invariance of $f$ and the distribution of $\Phi$ under isometries of the lattice in the last step. The function $k\mapsto \|k\|\hat{f}(k)$ is isotropic, and while we formally do not satisfy the conditions in Proposition~\ref{p:asymptotic variance test function zero integral}, we can use an argument similar to it to show that if
        \begin{equation*}
            \limsup_{r\to\infty} r^{p}\int \|rk\|^2 |\hat{f}(rk)|^2\,\hat\beta_\Phi(dk) = 0,
        \end{equation*}
        then $\Phi$ must be beyond $(-d+p)$-uniform. However, we assumed that this is not the case, whereby the $\limsup$ must be positive. This fact concludes the proof, as~\eqref{e: Psi_1 limit condition hyperfluctuating points} has been proven.
    \end{proof}
\end{proposition}

The next example may be even more surprising. There, every point of the point process lies in the centroid of a stationary random fair tiling which is highly regular in terms of its possible cell shapes. Namely, there are only two possible cell shapes which are both rectangular and even congruent. Here, we have to require the dimension to be at least $5$ to make the example hyperfluctuating. For $d=2$, it was used as an example for a hyperuniform point process in~\cite[Section B]{Gabrielli_Joyce_Torquato_2008}.

Let $p\in(0,1)$ and assume $d\geq 2$. Suppose that $\Phi$ is the point process called $\Psi$ in Proposition~\ref{p: d>=1 hyperfluctuationg lattice thinning} with parameter $s:=2+p$. Define
\begin{align}\label{e: hyperfluctuating centroids}
    \Psi := \sum_{z\in\Z^d} &\I\{\Phi(\{z+U\})=1\} \big(\delta_{z+\frac{1}{2}e_1+U} + \delta_{z-\frac{1}{2}e_1+U}\big) \nonumber\\
    &+ \I\{\Phi(\{z+U\})=0\} \big(\delta_{z+\frac{1}{2}e_2+U} + \delta_{z-\frac{1}{2}e_2+U}\big).
\end{align}
\begin{proposition}\label{p: hyperfluctuating centroids}
    The point process $\Psi$, defined as in~\eqref{e: hyperfluctuating centroids}, is solely \linebreak$(-d+4+p)$-uniform. This degree is negative iff $d\geq 5$.
\begin{proof}
    The proof is very similar to that of Proposition~\ref{p: hyperfluctuating points in lattice}. First of all, we note that $\Phi$ is $(-d+p)$-uniform, as shown in Proposition~\ref{p: d>=1 hyperfluctuationg lattice thinning}. Let $\tilde\Phi:=\Z^d+U$ and define the invariant transport kernel $K$ by
        \begin{equation}
            K_x := \I\{\Phi(\{x\})=1\} (\delta_{x+\frac{1}{2}e_1} +\delta_{x-\frac{1}{2}e_1}) + \I\{\Phi(\{x\})=0\} (\delta_{x+\frac{1}{2}e_2} +\delta_{x-\frac{1}{2}e_2}), \quad x\in\R^d.
        \end{equation}
        Note that therefore $\Psi = K\tilde\Phi$, and
        \begin{equation}
            K^\ast_x = \delta_{\frac{1}{2}e_{2-\Phi(\{x\})}} + \delta_{-\frac{1}{2}e_{2-\Phi(\{x\})}}, \quad x\in\R^d.
        \end{equation}
        Now we apply Theorem~\ref{t:main theorem} with the parameter $\tilde{p}:=-d+2+p$. As $K$ is a transport kernel with constant mass and the transport distance is bounded by $\frac{1}{2}$, Condition~\ref{c:condition square-integrable general with p} is obviously fulfilled. Let us now define $\Psi_q$ for $q\in\N_0$ like in~\eqref{e:Psi_q definition}, whereby
        \begin{equation}
            \Psi_q(dy) = \bigg(\Big(\frac{1}{2}e_{2-\Phi(\{x\})}\Big)^{\otimes q} + \Big(-\frac{1}{2}e_{2-\Phi(\{x\})}\Big)^{\otimes q}\bigg) \tilde\Phi(dy).
        \end{equation}
        Hence, for $q\in2\N_0$ we have $\Psi_q = \frac{1}{2^{q-1}}(e_1^{\otimes q}-e_2^{\otimes q})\Phi + \frac{1}{2^{q-1}}e_2^{\otimes q}\tilde\Phi$, which implies that it is \linebreak $(-d+p)$-uniform for $q\in2\N$ and even of degree $\infty$ for $q=0$. Moreover, for $q\in2\N_0+1$, we have $\Psi_q = 0$, which implies that it is $\infty$-uniform. In total, Theorem~\ref{t:main theorem} yields that $\Psi$ is \linebreak $(-d+4+p)$-uniform, where the dominant contribution comes from $\Psi_2$. We improve this result by applying Theorem~\ref{t: main theorem rate} to show that $\Psi$ is not beyond $(-d+4+p)$-uniform. Let \linebreak $0\neq f\in\cC_c^\infty(\R, [0,\infty))$ and define $\tilde{f}:\R^d\to[0,\infty), \tilde{f}(x):=f(x_1)\cdots f(x_d), x\in\R^d$. Further, define $f_r:=f(\frac{\cdot}{r})$ and $\tilde{f}_r:=\tilde{f}(\frac{\cdot}{r})$ for $r>0$. Now we have to show that
        \begin{equation}
            \limsup_{r\to\infty} \frac{\BV\big[\Psi_2\big(\tilde{f}_r''\big)\big]}{r^{2d-(4+p)}} > 0.
        \end{equation}
        Let $\Phi_1,...,\Phi_d$ be the independent renewal processes from~\eqref{e: d>=1 hyperfluctuationg lattice thinning} such that $\Phi = \Phi_1\otimes\cdots\otimes \Phi_d$. As derivatives integrate to $0$, we can calculate that, as $r\to\infty$,
        \begin{align*}
            \frac{\BV\big[\Psi_2\big(\tilde{f}_r''\big)\big]}{r^{2d-(4+p)}} &= \frac{1}{4^{q-1}r^{2d-(4+p)}}\BV\big[\Phi\big(\partial_1^2 \tilde{f}_r-\partial_2^2 \tilde{f}_r\big)\big] + o(1)\nonumber\\
            &= \frac{1}{4^{q-1}r^{2d-(4+p)}}\BE\Big[\Phi\big(\partial_1^2 \tilde{f}_r-\partial_2^2 \tilde{f}_r\big)^2\Big] + o(1)\nonumber\\ 
            &= \frac{1}{4^{q-1}r^{2d-(4+p)}}\BE\Big[\big(\Phi_1(f_r'')\Phi_2(f_r)-\Phi_1(f_r)\Phi_2(f_r'')\big)^2\Big] \BE\big[\Phi_1(f_r)^2\big]^{d-2} + o(1)\nonumber\\ 
            &= \frac{2r^p}{4^{q-1}}\Big(\BE\big[\Phi_1(f_r'')^2\big]\BE\big[\Phi_1(f_r)^2\big]-\BE\big[\Phi_1(f_r'')\Phi_1(f_r)\big]^2\Big) \frac{\BE\big[\Phi_1(f_r)^2\big]^{d-2}}{r^{2d-4}} + o(1).
        \end{align*}
        Since
        \begin{align*}
            \frac{\BE\big[\Phi_1(f_r)^2\big]^{d-2}}{r^{2d-4}} &= \frac{\big(\BV[\Phi_1(f_r)]+\BE[\Phi_1([0,1))]^2\lambda_1(f)^2r^2\big)^{d-2}}{r^{2d-4}} \nonumber\\
            &\xrightarrow{r\to\infty} \BE[\Phi_1([0,1))]^{2d-4}\lambda_1(f)^{2d-4}>0,
        \end{align*}
        we can continue with
        \begin{align*}
            &r^p \Big(\BE\big[\Phi_1(f_r'')^2\big]\BE\big[\Phi_1(f_r)^2\big]-\BE\big[\Phi_1(f_r'')\Phi_1(f_r)\big]^2\Big)\nonumber\\
            &\quad\ \ = r^p\big(\BV[\Phi_1(f''_r)]\BE[\Phi_1(f_r)]^2 + \BV[\Phi_1(f''_r)]\BV[\Phi_1(f_r)] - \BC[\Phi_1(f_r''),\Phi_1(f_r)]^2\big)\nonumber\\
            &\quad\ \ =\frac{1}{r^{4-p}}\Big(\BV[\Phi_1(f''(\tfrac{\cdot}{r}))]\BE[\Phi_1(f_r)]^2\nonumber\\
            &\qquad \qquad \quad\ \ \ + \underbrace{\BV[\Phi_1(f''(\tfrac{\cdot}{r}))]\BV[\Phi_1(f_r)] - \BC[\Phi_1(f''(\tfrac{\cdot}{r})),\Phi_1(f_r)]^2}_{= O(r^{4-2p})}\Big)\nonumber\\
            &\quad\ \ = r^p \BV[\Phi_1(f''(\tfrac{\cdot}{r}))] \underbrace{\frac{\BE[\Phi_1(f_r)]^2}{r^4}}_{= \BE[\Phi_1([0,1))]^2\lambda_1(f)^2 + o(1)} + o(1),
        \end{align*}
        as $r\to\infty$.
        For the remaining part, we can leverage Proposition~\ref{p:asymptotic variance test function zero integral}. As $\Phi_1$ is a one-dimensional ergodic point process, it implies that 
        \begin{equation*}
            \limsup_{r\to\infty} r^p \BV[\Phi_1(f''(\tfrac{\cdot}{r}))] = 0
        \end{equation*}
        can only hold if $\Phi_1$ is beyond $(-d+p)$-uniform, which is not the case by Proposition~\ref{p: d=1 hyperfluctuationg lattice thinning}. Hence, the $\limsup$ must be positive, which concludes our proof. A more precise calculation using~\eqref{e: d=1 hyperfluctuationg lattice thinning variance limit} would even yield the existence of the limit.
\end{proof}
\end{proposition}

\section{Simulation study}\label{s: Simulation study}

To demonstrate the practical applicability of the theoretical results of this paper, we implemented the proposed construction of $p$-uniform point processes from Subsection~\ref{ss: pp from invariant partitions} in dimensions $2$ and $3$. First, we implemented the algorithm from Example~\ref{ex: fair STIT} to generate large samples of the fair STIT with a uniform splitting distribution, which is a fair partition with convex cells. As the starting shape, we chose a box with equal side lengths such that, later on, we can map everything seamlessly onto the flat $d$-torus (i.e., we can use periodic boundary conditions). Directly approximating a section of a fair STIT on $\R^d$, as described in Example~\ref{ex: fair STIT}, was avoided to eliminate boundary effects when estimating the structure factor of the point process. In the second step, we choose specific values of $n,m\in\N$ and randomly sample $n$ points in each cell such that the first $m$ moments of the cell and the points coincide, i.e.,
\begin{equation}\label{e: cell points moment equation}
    \frac{1}{n} \sum_{j=1}^n X_j^{\otimes q} = \frac{1}{\lambda_d(C)} \int_C z^{\otimes q}\, dz,\quad q\in\{0,...,m-1\},
\end{equation}
where $X_1,...,X_n$ are the points and $C$ is the cell. 
To this end, we randomly initialize the points and then apply Newton's method until the points satisfy \eqref{e: cell points moment equation} up to a prescribed tolerance; see~\cite[Chapter 7]{Quarteroni_Sacco_Saleri_2006} for details on Newton's method. Only if $n=m=1$, we specifically place the point at a uniformly random location inside the cell, and if $n=1, m=2$, we take a shortcut and directly place the point at the centroid. See Table~\ref{tab: simulation} for the values of $n,m$, for which we carried out the simulations. Since the total number of points in each sample is the product of $n$ and the number of cells, which is a power of $2$ by construction, not all samples can be chosen to have exactly the same total number of points. Therefore, we chose the number of cells of each sample such that it contains between $10^9$ and $2\cdot 10^9$ points and then scaled the bounding box such that the point process has intensity $1$. 

\begin{table}[hb]
\centering
\renewcommand{\arraystretch}{1.2}
\begin{tabular}{ccccc}
\hline\hline
\multicolumn{5}{c}{\textbf{2D}} \\
\hline
\makecell{\vspace{-4mm}\\ Number of\\ fixed moments\\ $m$\\ \vspace{-4mm}} &
\makecell{\vspace{-4mm}\\ Number of\\ points per cell\\ $n$\\ \vspace{-4mm}} &
\makecell{\vspace{-4mm}\\ Number of\\ cells\\ \phantom{a}\\ \vspace{-4mm}} &
\makecell{\vspace{-4mm}\\ Approx. total\\ number of points\\ \phantom{a}\\ \vspace{-4mm}} &
\makecell{\vspace{-4mm}\\ Estimated\\ scaling exponent\\ $\hat{p}$ \\ \vspace{-4mm}}\\
\hline
1 & 1  & $2^{30}$ & $1.07 \cdot 10^{9}$ & 2.0 \\
2 & 1  & $2^{30}$ & $1.07 \cdot 10^{9}$ & 3.3 \\
3 & 3  & $2^{29}$ & $1.61 \cdot 10^{9}$ & 6.0 \\
4 & 5  & $2^{28}$ & $1.34 \cdot 10^{9}$ & 7.5 \\
5 & 9  & $2^{27}$ & $1.21 \cdot 10^{9}$ & 9.8 \\
6 & 12 & $2^{27}$ & $1.61 \cdot 10^{9}$ & 11.4 \\
7 & 19 & $2^{26}$ & $1.28 \cdot 10^{9}$ & 13.4 \\
8 & 24 & $2^{26}$ & $1.61 \cdot 10^{9}$ & 14.7 \\
\hline\hline
\end{tabular}

\vspace{0.5cm}
\begin{tabular}{ccccc}
\hline\hline
\multicolumn{5}{c}{\textbf{3D}} \\
\hline
\makecell{\vspace{-4mm}\\ Number of\\ fixed moments\\ $m$\\ \vspace{-4mm}} &
\makecell{\vspace{-4mm}\\ Number of\\ points per cell\\ $n$\\ \vspace{-4mm}} &
\makecell{\vspace{-4mm}\\ Number of\\ cells\\ \phantom{a}\\ \vspace{-4mm}} &
\makecell{\vspace{-4mm}\\ Approx. total\\ number of points\\ \phantom{a}\\ \vspace{-4mm}} &
\makecell{\vspace{-4mm}\\ Estimated\\ scaling exponent\\ $\hat{p}$ \\ \vspace{-4mm}}\\
\hline
1 & 1  & $2^{30}$ & $1.07 \cdot 10^{9}$ & 2.0 \\
2 & 1  & $2^{30}$ & $1.07 \cdot 10^{9}$ & 2.5 \\
3 & 4  & $2^{28}$ & $1.07 \cdot 10^{9}$ & 6.0 \\
4 & 7  & $2^{28}$ & $1.88 \cdot 10^{9}$ & 6.6 \\
5 & 14  & $2^{27}$ & $1.88 \cdot 10^{9}$ & 9.8 \\
\hline\hline
\end{tabular}
\caption{Overview of the parameters for which we generated and evaluated samples together with the estimated scaling exponent of the structure factor.}
\label{tab: simulation}
\end{table}\clearpage

For each point sample, we calculate the empirical structure factor via the aforementioned embedding onto the flat $d$-torus. Therefore, we define
\begin{equation}
    S_j := \frac{1}{|\Phi|} \bigg| \sum_{x\in\Phi} e^{-i\langle k_j, x\rangle}\bigg|^2 
\end{equation}
for $k_j\in\R^d$, where all components of $k_j$ are integer multiples of $u:=\frac{2\pi}{\sqrt[d]{|\Phi|}}$ (as in \cite{Klatt_Last_Henze_2022}), $\Phi$ is the point sample, $|\Phi|$ is the size of the sample, and $j$ is in the index set $J$. We choose $|J|$ to be approximately $10000$ in dimension $2$ and approximately $5000$ in dimension $3$. The values of $k_j$ are chosen (randomly) such that $\|k_j\|\in[u, 0.2]$ and such that every sub-interval of the form $[lu, (l+1)u),$ $l\in\N,$ contains the same number of $\|k_j\|$'s (where possible). Since we know that the underlying point process is isotropic and ergodic, we assume that $S_j$ does not depend, up to noise, on the randomly chosen direction of $k_j$ but only on $\|k_j\|$.

All floating-point operations are calculated at double precision. Since $S_j$ can be close to $0$ for small $\|k_j\|$, there is a threshold below which our calculation is dominated by the numerical error. 
This threshold depends on system-size, precision, and dimension. 
Here, we choose to disregard all values of the empirical structure factor, where $S_j<10^{-14}$ in dimension $2$ and where $S_j<10^{-16}$ in dimension $3$.
To avoid a bias in the estimation of the structure factor, we define a cutoff in terms of $\|k_j\|$.
To this end, we first perform a log-log-transformation, i.e., we 
consider $\log(\|k_j\|)$ and $\log(S_j)$.
Since we expect that $S(k)\approx c\|k\|^p$ close to the origin for some $c,p>0$, 
this transform can be expected to yield $\log(S_j) \approx p\log(\|k_j\|) + \log(c)$.
We smooth $\log(S_j)$ by performing local linear regressions, using a specific form of the LOWESS algorithm; see~\cite{Cleveland_1979} for details on LOWESS.
We then determine the cutoff to be the largest $\|k_j\|$, where the smoothed $\log(S_j)$ is smaller than $\log(10^{-14})$ in dimension $2$ and $\log(10^{-16})$ in dimension $3$. Under the assumption that the empirical structure factor follows an exponential distribution with the exact structure factor as an expectation, as it was shown in specific circumstances in~\cite{Klatt_Last_Henze_2022}, the linear smoothing of the logarithmic data $\log(S_j)$ via local linear regression is biased by a constant $-\gamma$, where $\gamma\approx 0.58$ is the Euler-Mascheroni constant. We compensate this bias by the appropriate choice of the cutoff parameters $10^{-14}$ and $10^{-16}$. While the cutoff thresholds were determined in terms of the smoothed data, the subsequent calculations again rely on the raw (unsmoothed) data.

The truncated log-log data $\log(S_j)$, up to noise, now seems to closely follow a linear law, which suggests that the structure factor of the point process approximately follows a power-law in the range of considered wave vectors $k$. Finally, we perform a linear regression on the truncated log-log data $\log(S_j)$ to access the slope, which corresponds to the scaling exponent $p$ of this power-law. Since the constant bias of the linear regression, which was noted above, does not influence the slope, this estimation is unbiased. The results can be found in Figure~\ref{fig: empirical structure factor} and Table~\ref{tab: simulation}. Since we evaluated the empirical structure factor only at a limited number of wave vectors, these estimates should be treated as rough, especially for large $m$, where the range of evaluation was significantly truncated. Still, the estimates fit our expectations. A more detailed discussion of the observed values of $\hat p$ can be found in the caption of Figure~\ref{fig: empirical structure factor}.

The code was written in Python and optimized using Numba for improved performance and the parallel execution of parts of the algorithm. The remaining details can be found in the code that will be made available together with the generated data in open-access repositories alongside the publication of this paper.

\fi

\section*{Acknowledgments}
The authors are very grateful to Günter Last for his guidance and helpful discussions. The authors also thank D. Yogeshwaran, Raphaël Lachièze-Rey, Daniela Flimmel, Peter Grabner, Gabriel Mastrilli, Michael Björklund, and Markus Kiderlen for the helpful scientific exchange. This work was supported by the Deutsche Forschungsgemeinschaft (DFG, German Research Foundation) through the SPP 2265, under grant number KL 3391/2-2, as well as by the Initiative and Networking Fund of the Helmholtz Association under the call Helmholtz Young Investigator Groups (VH-NG-19-34, DataMat).
This work is also supported by the Helmholtz Association Initiative and Networking Fund on the HAICORE@KIT partition.

\addcontentsline{toc}{section}{Acknowledgements}


\printbibliography[heading=bibintoc]

@article{Bartlett_1963, title={The Spectral Analysis of Point Processes}, volume={25}, ISSN={00359246}, url={http://www.jstor.org/stable/2984295}, abstractNote={The spectral analysis of stationary point processes in one dimension is developed in some detail as a statistical method of analysis. The asymptotic sampling theory previously established by the author for a class of doubly stochastic Poisson processes is shown to apply also for a class of clustering processes, the spectra of which are contrasted with those of renewal processes. The analysis is given for two illustrative examples, one an artificial Poisson process, the other of some traffic data. In addition to testing the fit of a clustering model to the latter example, the analysis of these two examples is used where possible to check the validity of the sampling theory.}, number={2}, journal={Journal of the Royal Statistical Society. Series B (Methodological)}, author={Bartlett, M. S.}, year={1963}, pages={264–296} }

@book{Berg_Forst_2012, series={Ergebnisse der Mathematik und ihrer Grenzgebiete. 2. Folge}, title={Potential Theory on Locally Compact Abelian Groups}, ISBN={9783642661280}, url={https://books.google.de/books?id=Lxv-CAAAQBAJ}, publisher={Springer Berlin Heidelberg}, author={Berg, C. and Forst, G.}, year={2012}, collection={Ergebnisse der Mathematik und ihrer Grenzgebiete. 2. Folge} }

@book{Bracewell_2000, series={Circuits and systems}, title={The Fourier Transform and Its Applications}, ISBN={9780073039381}, url={https://books.google.de/books?id=ZNQQAQAAIAAJ}, publisher={McGraw Hill}, author={Bracewell, R.N.}, year={2000}, collection={Circuits and systems} }

@article{Doob_1948, title={Renewal Theory From the Point of View of the Theory of Probability}, volume={63}, ISSN={00029947, 10886850}, url={http://www.jstor.org/stable/1990567}, DOI={10.2307/1990567}, number={3}, journal={Transactions of the American Mathematical Society}, author={Doob, J. L.}, year={1948}, pages={422–438} }

@book{Gut_2009, address={New York, NY}, series={Springer Series in Operations Research and Financial Engineering}, title={Stopped Random Walks: Limit Theorems and Applications}, ISBN={9780387878355}, url={http://dx.doi.org/10.1007/978-0-387-87835-5}, abstractNote={Limit Theorems for Stopped Random Walks -- Renewal Processes and Random Walks -- Renewal Theory for Random Walks with Positive Drift -- Generalizations and Extensions -- Functional Limit Theorems -- Perturbed Random Walks; Classical probability theory provides information about random walks after a fixed number of steps. For applications, however, it is more natural to consider random walks evaluated after a random number of steps. Stopped Random Walks: Limit Theorems and Applications shows how this theory can be used to prove limit theorems for renewal counting processes, first passage time processes, and certain two-dimensional random walks, as well as how these results may be used in a variety of applications. The present second edition offers updated content and an outlook on further results, extensions and generalizations. A new chapter introduces nonlinear renewal processes and the theory of perturbed random walks, which are modeled as random walks plus “noise”. This self-contained research monograph is motivated by numerous examples and problems. With its concise blend of material and over 300 bibliographic references, the book provides a unified and fairly complete treatment of the area. The book may be used in the classroom as part of a course on “probability theory”, “random walks” or “random walks and renewal processes”, as well as for self-study. From the reviews: “The book provides a nice synthesis of a lot of useful material.” --American Mathematical Society “...[a] clearly written book, useful for researcher and student.” --Zentralblatt MATH}, publisher={Springer New York}, author={Gut, Allan}, year={2009}, collection={Springer Series in Operations Research and Financial Engineering} }

@article{Henk_2012, title={Löwner-John ellipsoids}, journal={Documenta Mathematica}, author={Henk, Martin}, year={2012}, month=jan }

@article{Cleveland_1979, title={Robust Locally Weighted Regression and Smoothing Scatterplots}, volume={74}, url={https://doi.org/10.1080/01621459.1979.10481038}, DOI={10.1080/01621459.1979.10481038}, number={368}, journal={Journal of the American Statistical Association}, publisher={Taylor & Francis}, author={Cleveland, William S.}, year={1979}, pages={829–836} }

@book{Horn_Johnson_1985, title={Matrix Analysis}, publisher={Cambridge University Press}, author={Horn, Roger A. and Johnson, Charles R.}, year={1985} }

@book{Iosevich_Liflyand_2014, series={Mathematics and Statistics}, title={Decay of the Fourier Transform: Analytic and Geometric Aspects}, ISBN={9783034806251}, url={https://books.google.de/books?id=-CCsBAAAQBAJ}, publisher={Springer Basel}, author={Iosevich, A. and Liflyand, E.}, year={2014}, collection={Mathematics and Statistics} }

@article{ISSERLIS_1918, title={On a Formula for the Product-Moment Coefficient of any Order of a Normal Frequency Distribution in any Number of Variables}, volume={12}, ISSN={0006-3444}, url={https://doi.org/10.1093/biomet/12.1-2.134}, DOI={10.1093/biomet/12.1-2.134}, number={1–2}, journal={Biometrika}, author={Isserlis, L.}, year={1918}, month=nov, pages={134–139} }

@article{Janson_2021, title={The space D in several variables: random variables and higher moments}, volume={127}, url={https://www.mscand.dk/article/view/128971}, DOI={10.7146/math.scand.a-128971}, number={3}, journal={Mathematica Scandinavica}, author={Janson, Svante}, year={2021}, month=nov }

@book{Kallenberg_2017, title={Random Measures, Theory and Applications}, volume={77}, ISBN={978-3-319-41596-3}, DOI={10.1007/978-3-319-41598-7}, author={Kallenberg, Olav}, year={2017}, month=jan }

@book{Kallenberg_2021, series={Probability theory and stochastic modelling}, title={Foundations of Modern Probability}, ISBN={9783030618735}, url={https://books.google.de/books?id=mAVfzgEACAAJ}, publisher={Springer}, author={Kallenberg, O.}, year={2021}, collection={Probability theory and stochastic modelling} }

@book{Khoshnevisan_2002, address={New York, Berlin, Heidelberg [u.a.]}, series={Sprin\-ger monographs in mathematics}, title={Multiparameter processes: an introduction to random fields}, ISBN={0387954597}, archiveLocation={QA274.45}, url={https://zbmath.org/?q=an:1005.60005}, publisher={Springer}, author={Khoshnevisan, Davar}, year={2002}, collection={Springer monographs in mathematics} }

@book{Last_Penrose_2018, series={Institute of Mathematical Statistics Textbooks}, title={Lectures on the Poisson Process}, ISBN={9781107088016}, url={https://books.google.de/books?id=JRs3DwAAQBAJ}, publisher={Cambridge University Press}, author={Last, G. and Penrose, M.}, year={2018}, collection={Institute of Mathematical Statistics Textbooks} }

@article{Mazur_Orlicz_1934, title={Grundlegende Eigenschaften der polynomischen Operationen. Erste Mitteilung}, volume={5}, url={http://eudml.org/doc/218130}, number={1}, journal={Studia Mathematica}, author={Mazur, S. and Orlicz, W.}, year={1934}, pages={50–68}}

@article{Mead_1992, title={Newton’s Identities}, volume={99}, ISSN={00029890, 19300972}, url={http://www.jstor.org/stable/2324242}, DOI={10.2307/2324242}, number={8}, journal={The American Mathematical Monthly}, author={Mead, D. G.}, year={1992}, pages={749–751} }

@book{Olver_Lozier_Boisvert_Clark_2010, title={The NIST Handbook of Mathematical Functions}, publisher={Cambridge University Press, New York, NY}, author={Olver, Frank and Lozier, Daniel and Boisvert, Ronald and Clark, Charles}, year={2010}, month=may}

@book{Quarteroni_Sacco_Saleri_2006, series={Texts in Applied Mathematics}, title={Numerical Mathematics}, ISBN={978-3-540-34658-6}, url={https://books.google.de/books?id=31m4ahn_KfkC}, publisher={Springer Berlin Heidelberg}, author={Quarteroni, A. and Sacco, R. and Saleri, F.}, year={2006}, collection={Texts in Applied Mathematics} }

@article{Sewell_2024, title={Simple bounds for the inradius and $\varepsilon$-inner neighbourhood of a convex body}, url={https://arxiv.org/abs/2401.03593}, number={arXiv:2401.03593}, publisher={arXiv}, author={Sewell, Benedict}, year={2024}, month=feb }

@book{Stein_Weiss_1971, title={Introduction to Fourier Analysis on Euclidean Spaces (PMS-32)}, ISBN={9780691080789}, url={http://www.jstor.org/stable/j.ctt1bpm9w6}, abstractNote={The authors present a unified treatment of basic topics that arise in Fourier analysis. Their intention is to illustrate the role played by the structure of Euclidean spaces, particularly the action of translations, dilatations, and rotations, and to motivate the study of harmonic analysis on more general spaces having an analogous structure, e.g., symmetric spaces.}, publisher={Princeton University Press}, author={Stein, Elias M. and Weiss, Guido}, year={1971} }

@book{Wood_1992, address={University of Kent, Canterbury, UK}, title={The Computation of Polylogarithms}, url={http://www.cs.kent.ac.uk/pubs/1992/110}, number={15-92*}, institution={University of Kent, Computing Laboratory}, author={Wood, David}, year={1992}, month=jun, pages={182–196} }

@book{Yaglom_1987, address={New York}, title={Correlation theory of stationary and related random functions.}, url={https://zbmath.org/?q=an:0685.62077}, publisher={Springer,}, author={Yaglom, Akiva M.}, year={1987} }

@article{Stehr_Kiderlen_2020, title={Asymptotic variance of Newton–Cotes quadratures based on randomized sampling points}, volume={52}, DOI={10.1017/apr.2020.41}, number={4}, journal={Advances in Applied Probability}, author={Stehr, Mads and Kiderlen, Markus}, year={2020}, pages={1284–1307} }

@book{Abramowitz_Stegun_1965, series={Applied mathematics series}, title={Handbook of Mathematical Functions: With Formulas, Graphs, and Mathematical Tables}, ISBN={9780486612720}, url={https://books.google.de/books?id=MtU8uP7XMvoC}, publisher={Dover Publications}, author={Abramowitz, M. and Stegun, I.A.}, year={1965}, collection={Applied mathematics series} }

@article{Bondarenko_Radchenko_Viazovska_2013, title={Optimal asymptotic bounds for spherical designs}, volume={178}, ISSN={0003486X}, url={http://www.jstor.org/stable/23470798}, abstractNote={In this paper we prove the conjecture of Korevaar and Meyers: for each N ≥ c d t d , there exists a spherical t-design in the sphere S d consisting of N points, where c d is a constant depending only on d.}, number={2}, journal={Annals of Mathematics}, author={Bondarenko, Andriy and Radchenko, Danylo and Viazovska, Maryna}, year={2013}, pages={443–452} }

@article{Bondarenko_Viazovska_2008, title={Spherical Designs via Brouwer Fixed Point Theorem}, volume={24}, DOI={10.1137/080738313}, journal={SIAM Journal on Discrete Mathematics}, pages = {207-217}, author={Bondarenko, Andriy and Viazovska, Maryna}, year={2008}, month=nov }

@article{Carathéodory_1911, title={Über den variabilitätsbereich der fourier’schen konstanten von positiven harmonischen funktionen}, volume={32}, ISSN={0009-725X}, url={https://doi.org/10.1007/BF03014795}, DOI={10.1007/BF03014795}, number={1}, journal={Rendiconti del Circolo Matematico di Pa\-ler\-mo (1884-1940)}, author={Carathéodory, C.}, year={1911}, month=dec, pages={193–217} }

@article{Delsarte_Goethals_Seidel_1977, title={Spherical codes and designs}, volume={6}, ISSN={1572-9168}, url={https://doi.org/10.1007/BF03187604}, DOI={10.1007/BF03187604}, number={3}, journal={Geometriae Dedicata}, author={Delsarte, P. and Goethals, J. M. and Seidel, J. J.}, year={1977}, month=sep, pages={363–388} }

@article{Seymour_Zaslavsky_1984, title={Averaging sets: A generalization of mean values and spherical designs}, volume={52}, ISSN={0001-8708}, url={https://www.sciencedirect.com/science/article/pii/0001870884900227}, DOI={https://doi.org/10.1016/0001-8708(84)90022-7}, number={3}, journal={Advances in Mathematics}, author={Seymour, P. D. and Zaslavsky, Thomas}, year={1984}, pages={213–240} }

@article{Wagner_Volkmann_1991, title={On averaging sets}, volume={111}, ISSN={1436-5081}, url={https://doi.org/10.1007/BF01299278}, DOI={10.1007/BF01299278}, abstractNote={For any set Φ={f1,f2,...,fs} ofC3-functions on the interval [−1, 1], and for any weight functionw(x) satisfyingL1≥w(x)≥L2(1−|x|)β(L1,L2>0, β≥0) and$$int_{ - 1}^1 {w(x)dx = 1} $$, we give a constructive proof for the existence of quadrature formulas of the type$$frac{1}{n}sumlimits_{j = 1}^n {f_mu (x_j )} int_{ - 1}^1 {f_mu (x)w(x)dx} {text{    }}(mu = 1,2,{text{ }} ldots ,s)$$for sufficiently largen, −1<x1<x2<...<xn<1. Assuming the orthonormality of the derivativesf′1,f′2,...,f′s with respect to the weight functionw(x), we obtain explicit bounds for the numbern of interpolation points for which such formulas exist. As an application to combinatorics, we prove the existence ofd-dimensional sphericalt-designs of sizen for eachn>cd·t12d4,cd>0 a constant.}, number={1}, journal={Monatshefte für Mathematik}, author={Wagner, Gerold and Volkmann, Bodo}, year={1991}, month=mar, pages={69–78} }

@article{Xiang_2022, title={Explicit spherical designs}, volume={5}, DOI={10.5802/alco.213}, journal={Algebraic Combinatorics}, author={Xiang, Ziqing}, year={2022}, month=may, pages={347–369} }

@article{Bohr_1925, title={Zur Theorie der fast periodischen Funktionen}, volume={45}, ISSN={1871-2509}, url={https://doi.org/10.1007/BF02395468}, DOI={10.1007/BF02395468}, number={1}, journal={Acta Mathematica}, author={Bohr, Harald}, year={1925}, month=jul, pages={29–127} }

@article{Debye_1912, title={Zur Theorie der spezifischen Wärmen}, volume={344}, url={https://onlinelibrary.wiley.com/doi/abs/10.1002/andp.19123441404}, DOI={https://doi.org/10.1002/andp.19123441404}, number={14}, journal={Annalen der Physik}, author={Debye, P.}, year={1912}, pages={789–839} }

@article{Einstein_1907, title={Die Plancksche Theorie der Strahlung und die Theorie der spezifischen Wärme}, volume={327}, url={https://onlinelibrary.wiley.com/doi/abs/10.1002/andp.19063270110}, DOI={https://doi.org/10.1002/andp.19063270110}, number={1}, journal={Annalen der Physik}, author={Einstein, A.}, year={1907}, pages={180–190} }

@book{Landau_Lifshitz_2013, title={Statistical Physics: Volume 5}, ISBN={9780080570464}, url={https://books.google.de/books?id=VzgJN-XPTRsC}, publisher={Butterworth-Heinemann}, author={Landau, L.D. and Lifshitz, E.M.}, year={2013} }

@article{Chatterjee_Peled_Peres_Romik_2010, title={Gravitational allocation to Poisson points}, journal={Annals of mathematics}, author={Chatterjee, Sourav and Peled, Ron and Peres, Yuval and Romik, Dan}, year={2010}, pages={617–671} }

@article{Lothar_Heinrich_Hendrik_Schmidt_Volker_Schmidt_2006, title={Central limit theorems for Poisson hyperplane tessellations}, volume={16}, url={https://doi.org/10.1214/105051606000000033}, DOI={10.1214/105051606000000033}, abstractNote={We derive a central limit theorem for the number of vertices of convex polytopes induced by stationary Poisson hyperplane processes in ℝd. This result generalizes an earlier one proved by Paroux [Adv. in Appl. Probab. 30 (1998) 640–656] for intersection points of motion-invariant Poisson line processes in ℝ2. Our proof is based on Hoeffding’s decomposition of U-statistics which seems to be more efficient and adequate to tackle the higher-dimensional case than the “method of moments” used in [Adv. in Appl. Probab. 30 (1998) 640–656] to treat the case d=2. Moreover, we extend our central limit theorem in several directions. First we consider k-flat processes induced by Poisson hyperplane processes in ℝd for 0≤k≤d−1. Second we derive (asymptotic) confidence intervals for the intensities of these k-flat processes and, third, we prove multivariate central limit theorems for the d-dimensional joint vectors of numbers of k-flats and their k-volumes, respectively, in an increasing spherical region.}, number={2}, journal={The Annals of Applied Probability}, author={Lothar Heinrich and Hendrik Schmidt and Volker Schmidt}, year={2006}, month=may, pages={919–950} }

@article{Nagel_Weiss_2003, title={Limits of Sequences of Stationary Planar Tessellations}, volume={35}, ISSN={00018678}, url={http://www.jstor.org/stable/1428276}, abstractNote={In order to increase the variety of feasible models for random stationary tessellations (mosaics), two operations acting on tessellations are studied: superposition and iteration (the latter is also referred to as nesting). The superposition of two planar tessellations is the superposition of the edges of the cells of both tessellations. The iteration of tessellations means that one tessellation is chosen as a “frame” tessellation. The single cells of this “frame” are simultaneously and independently subdivided by cut-outs of tessellations of an independent and identically distributed sequence of tessellations. In the present paper, we investigate the limits for sequences of tessellations that are generated by consecutive application of superposition or iteration respectively. Sequences of (renormalised) superpositions of stationary planar tessellations converge weakly to Poisson line tessellations. For consecutive iteration the notion of stability of distributions is adapted and necessary conditions are formulated for those tessellations which may occur as limits.}, number={1}, journal={Advances in Applied Probability}, author={Nagel, Werner and Weiss, Viola}, year={2003}, pages={123–138} }

@book{Nagel_Thäle_Weiß_2026, title={Random Tessellations Generated by Division of Polytopes}, ISBN={9783112222409}, url={https://doi.org/10.1515/9783112222409}, DOI={doi:10.1515/9783112222409}, publisher={De Gruyter}, author={Nagel, Werner and Thäle, Christoph and Weiß, Viola}, year={2026} }

@article{Nagel_Weiss_2005, title={Crack STIT tessellations: characterization of stationary random tessellations stable with respect to iteration}, volume={37}, DOI={10.1239/aap/1134587744}, number={4}, journal={Advances in Applied Probability}, author={Nagel, Werner and Weiss, Viola}, year={2005}, pages={859–883} }

@article{Radin_1994, title={The Pinwheel Tilings of the Plane}, volume={139}, ISSN={0003486X, 19398980}, url={http://www.jstor.org/stable/2118575}, DOI={10.2307/2118575}, number={3}, journal={Annals of Mathematics}, author={Radin, Charles}, year={1994}, pages={661–702} }

@article{Elboim_Spinka_Yakir_2025, title={Optimal matchings of randomly perturbed lattices}, url={https://arxiv.org/abs/2506.16873}, number={arXiv:2506.16873}, publisher={arXiv}, author={Elboim, Dor and Spinka, Yinon and Yakir, Oren}, year={2025}, month=jun}

@article{Erbar_Huesmann_Jalowy_Müller_2025, title={Optimal transport of stationary point processes: Metric structure, gradient flow and convexity of the specific entropy}, volume={289}, ISSN={0022-1236}, url={https://www.sciencedirect.com/science/article/pii/S0022123625001569}, DOI={https://doi.org/10.1016/j.jfa.2025.110974}, abstractNote={We develop a theory of optimal transport for stationary random measures with a focus on stationary point processes and construct a family of distances on the set of stationary random measures. These induce a natural notion of interpolation between two stationary random measures along a shortest curve connecting them. In the setting of stationary point processes we leverage this transport distance to give a geometric interpretation for the evolution of infinite particle systems with stationary distribution. Namely, we characterise the evolution of infinitely many Brownian motions as the gradient flow of the specific relative entropy w.r.t. the Poisson point process. Further, we establish displacement convexity of the specific relative entropy along optimal interpolations of point processes and establish a stationary analogue of the HWI inequality, relating specific entropy, transport distance, and a specific relative Fisher information.}, number={4}, journal={Journal of Functional Analysis}, author={Erbar, Matthias and Huesmann, Martin and Jalowy, Jonas and Müller, Bastian}, year={2025}, pages={110974} }

@article{Hoffman_Holroyd_Peres_2005, title={A Stable Marriage of Poisson and Lebesgue}, volume={34}, url={https://api.semanticscholar.org/CorpusID:8938347}, journal={Annals of Probability}, author={Hoffman, Christopher and Holroyd, Alexander E. and Peres, Yuval}, year={2005}, pages={1241–1272} }

@article{Huesmann_2016, title={Optimal transport between random measures}, volume={52}, url={https://doi.org/10.1214/14-AIHP634}, DOI={10.1214/14-AIHP634}, abstractNote={Nous analysons le problème du transport optimal entre deux mesures aléatoires et équivariantes et démontrons des conditions qui garantissent l’existence d’une solution de type Monge. En outre, nous démontrons que l’équivariance apparaît naturellement dans ce contexte en prouvant que les couplages optimaux classiques dans des ensembles bornés convergent vers le couplage optimal dans tout l’espace. Finalement nous démontrons des conditions suffisantes pour que le coût au sens Lp soit fini en introduisant une métrique appropriée.}, number={1}, journal={Annales de l’Institut Henri Poincaré, Probabilités et Statistiques}, author={Huesmann, Martin}, year={2016}, pages={196–232} }

@article{Huesmann_Sturm_2013, title={Optimal transport from Lebesgue to Poisson}, volume={41}, url={https://doi.org/10.1214/12-AOP814}, DOI={10.1214/12-AOP814}, abstractNote={This paper is devoted to the study of couplings of the Lebesgue measure and the Poisson point process. We prove existence and uniqueness of an optimal coupling whenever the asymptotic mean transportation cost is finite. Moreover, we give precise conditions for the latter which demonstrate a sharp threshold at d=2. The cost will be defined in terms of an arbitrary increasing function of the distance. The coupling will be realized by means of a transport map (“allocation map”) which assigns to each Poisson point a set (“cell”) of Lebesgue measure 1. In the case of quadratic costs, all these cells will be convex polytopes.}, number={4}, journal={The Annals of Probability}, author={Huesmann, Martin and Sturm, Karl-Theodor}, year={2013}, pages={2426–2478} }

@article{Kantorovitch_1958, title={On the Translocation of Masses}, volume={5}, url={https://doi.org/10.1287/mnsc.5.1.1}, DOI={10.1287/mnsc.5.1.1}, abstractNote={The following paper is reproduced from a Russian journal of the character of our own Proceedings of the National Academy of Sciences, Comptes Rendus (Doklady) de I’Académie des Sciences de I’URSS, 1942, Volume XXXVII, No. 7–8. The author is one of the most distinguished of Russian mathematicians. He has made very important contributions in pure mathematics in the theory of functional analysis, and has made equally important contributions to applied mathematics in numerical analysis and the theory and practice of computation. Although his exposition in this paper is quite terse and couched in mathematical language which may be difficult for some readers of Management Science to follow, it is thought that this presentation will: (1) make available to American readers generally an important work in the field of linear programming, (2) provide an indication of the type of analytic work which has been done and is being done in connection with rational planning in Russia, (3) through the specific examples mentioned indicate the types of interpretation which the Russians have made of the abstract mathematics (for example, the potential and field interpretations adduced in this country recently by W. Prager were anticipated in this paper). It is to be noted, however, that the problem of determining an effective method of actually acquiring the solution to a specific problem is not solved in this paper. In the category of development of such methods we seem to be, currently, ahead of the Russians.—A. Charnes, Northwestern Technological Institute and The Transportation Center.}, number={1}, journal={Management Science}, author={Kantorovitch, L.}, year={1958}, pages={1–4} }

@article{Last_Thorisson_2009, title={Invariant transports of stationary random measures and mass-stationarity}, volume={37}, url={https://doi.org/10.1214/08-AOP420}, DOI={10.1214/08-AOP420}, abstractNote={We introduce and study invariant (weighted) transport-kernels balancing stationary random measures on a locally compact Abelian group. The first main result is an associated fundamental invariance property of Palm measures, derived from a generalization of Neveu’s exchange formula. The second main result is a simple sufficient and necessary criterion for the existence of balancing invariant transport-kernels. We then introduce (in a nonstationary setting) the concept of mass-stationarity with respect to a random measure, formalizing the intuitive idea that the origin is a typical location in the mass. The third main result of the paper is that a measure is a Palm measure if and only if it is mass-stationary.}, number={2}, journal={The Annals of Probability}, author={Last, Günter and Thorisson, Hermann}, year={2009}, pages={790–813} }

@book{Villani_2008, series={Grundlehren der mathematischen Wissenschaften}, title={Optimal Transport: Old and New}, ISBN={9783540710493}, url={https://books.google.de/books?id=NZXiNAEACAAJ}, publisher={Springer Berlin Heidelberg}, author={Villani, C.}, year={2008}, collection={Grundlehren der mathematischen Wissenschaften} }

@article{Batten_Stillinger_Torquato_2008, title={Classical disordered ground states: Super-ideal gases and stealth and equi-luminous materials}, volume={104}, ISSN={0021-8979, 1089-7550}, url={http://aip.scitation.org/doi/10.1063/1.2961314}, DOI={10.1063/1.2961314}, note={tex.ids= batten_classical_2008-1}, number={33}, journal={Journal of Applied Physics}, publisher={American Institute of Physics}, author={Batten, Robert D. and Stillinger, Frank H. and Torquato, Salvatore}, year={2008}, month=aug, pages={033504}}

@article{Beck_1987, title={Irregularities of distribution. I}, volume={159}, url={https://doi.org/10.1007/BF02392553}, DOI={10.1007/BF02392553}, number={none}, journal={Acta Mathematica}, author={Beck, József}, year={1987}, pages={1–49} }

@article{Björklund_Hartnick_2024, title={Hyperuniformity and non-hyperuniformity of quasicrystals}, volume={389}, ISSN={1432-1807}, url={https://doi.org/10.1007/s00208-023-02647-1}, DOI={10.1007/s00208-023-02647-1}, abstractNote={We develop a general framework to study hyperuniformity of various mathematical models of quasicrystals. Using this framework we provide examples of non-hyperuniform quasicrystals which unlike previous examples are not limit-quasiperiodic. Some of these examples are even anti-hyperuniform or have a positive asymptotic number variance. On the other hand we establish hyperuniformity for a large class of mathematical quasicrystals in Euclidean spaces of arbitrary dimension. For certain models of quasicrystals we moreover establish that hyperuniformity holds for a generic choice of the underlying parameters. For quasicrystals arising from the cut-and-project method we conclude that their hyperuniformity depends on subtle diophantine properties of the underlying lattice and window and is by no means automatic.}, number={1}, journal={Mathematische Annalen}, author={Björklund, Michael and Hartnick, Tobias}, year={2024}, month=may, pages={365–426} }

@article{Gabrielli_2004, title={Point processes and stochastic displacement fields}, volume={70}, url={https://link.aps.org/doi/10.1103/PhysRevE.70.066131}, DOI={10.1103/PhysRevE.70.066131}, abstractNote={The effect of a stochastic displacement field on a statistically independent point process is analyzed. Stochastic displacement fields can be divided into two large classes: spatially correlated and uncorrelated. For both cases exact transformation equations for the two-point correlation function and the power spectrum of the point process are found, and a detailed study of them with important paradigmatic examples is done. The results are general and in any dimension. Particular attention is devoted to the kind of large-scale correlations that can be introduced by the displacement field and to the realizability of arbitrary “superhomogeneous” point processes.}, number={66}, journal={Physical Review E}, author={Gabrielli, Andrea}, year={2004}, month=dec, pages={066131} }

@article{Gabrielli_Joyce_Torquato_2008, title={Tilings of space and superhomogeneous point processes}, volume={77}, url={https://link.aps.org/doi/10.1103/PhysRevE.77.031125}, DOI={10.1103/PhysRevE.77.031125}, abstractNote={We consider the construction of point processes from tilings, with equal-volume tiles, of d-dimensional Euclidean space Rd. We show that one can generate, with simple algorithms ascribing one or more points to each tile, point processes which are “superhomogeneous” (or “hyperuniform”)—i.e., for which the structure factor S(k) vanishes when the wave vector k tends to zero. The exponent γ characterizing the leading small-k behavior, S(k→0)∝kγ, depends in a simple manner on the nature of the correlation properties of the specific tiling and on the conservation of the mass moments of the tiles. Assigning one point to the center of mass of each tile gives the exponent γ=4 for any tiling in which the shapes and orientations of the tiles are short-range correlated. Smaller exponents in the range 4−d<γ<4 (and thus always superhomogeneous for d≤4) may be obtained in the case that the latter quantities have long-range correlations. Assigning more than one point to each tile in an appropriate way, we show that one can obtain arbitrarily higher exponents in both cases. We illustrate our results with explicit constructions using known deterministic tilings, as well as some simple stochastic tilings for which we can calculate S(k) exactly. Our results provide an explicit analytical construction of point processes with γ>4. Applications to condensed matter physics, and also to cosmology, are briefly discussed.}, number={33}, journal={Physical Review E}, author={Gabrielli, A. and Joyce, M. and Torquato, S.}, year={2008}, month=mar, pages={031125} }

@article{Ghosh_Lebowitz_2017, title={Fluctuations, large deviations and rigidity in hyperuniform systems: A brief survey}, volume={48}, url={http://link.springer.com/10.1007/s13226-017-0248-1}, DOI={10.1007/s13226-017-0248-1}, number={44}, journal={Indian Journal of Pure and Applied Mathematics}, author={Ghosh, Subhroshekhar and Lebowitz, Joel L.}, year={2017}, pages={609--631} }

@article{Ghosh_Lebowitz_2018, title={Generalized Stealthy Hyperuniform Processes: Maximal Rigidity and the Bounded Holes Conjecture}, volume={363}, ISSN={1432-0916}, url={https://doi.org/10.1007/s00220-018-3226-5}, DOI={10.1007/s00220-018-3226-5}, number={10}, journal={Duke Mathematical Journal}, author={Ghosh, Subhroshekhar and Peres, Yuval}, year={2017}, month=jul, pages={1789–1858} }

@article{Ghosh_Peres_2017, title={Rigidity and tolerance in point processes: Gaussian zeros and Ginibre eigenvalues}, volume={166}, url={https://doi.org/10.1215/00127094-2017-0002}, DOI={10.1215/00127094-2017-0002}, abstractNote={Let Π be a translation-invariant point process on the complex plane C, and let D⊂C be a bounded open set. We ask the following: What does the point configuration Πout obtained by taking the points of Π outside D tell us about the point configuration Πin of Π inside D? We show that, for the Ginibre ensemble, Πout determines the number of points in Πin. For the translation-invariant zero process of a planar Gaussian analytic function, we show that Πout determines the number as well as the center of mass of the points in Πin. Further, in both models we prove that the outside says “nothing more” about the inside, in the sense that the conditional distribution of the inside points, given the outside, is mutually absolutely continuous with respect to the Lebesgue measure on its supporting submanifold.}, number={10}, journal={Duke Mathematical Journal}, author={Ghosh, Subhroshekhar and Peres, Yuval}, year={2017}, month=jul, pages={1789–1858} }

@article{Kim_Torquato_2017, title={Effect of window shape on the detection of hyperuniformity via the local number variance}, volume={2017}, ISSN={1742-5468}, url={http://stacks.iop.org/1742-5468/2017/i=1/a=013402?key=crossref.cac505492aa6e5a6f522d023badc668e}, DOI={10.1088/1742-5468/aa4f9d}, number={11}, journal={Journal of Statistical Mechanics: Theory and Experiment}, author={Kim, Jaeuk and Torquato, Salvatore}, year={2017}, month=jan, pages={013402}}

@article{Kim_Torquato_2018, title={Effect of imperfections on the hyperuniformity of many-body systems}, volume={97}, url={https://link.aps.org/doi/10.1103/PhysRevB.97.054105}, DOI={10.1103/PhysRevB.97.054105}, abstractNote={A hyperuniform many-body system is characterized by a structure factor S(k) that vanishes in the small-wave-number limit or equivalently by a local number variance σ2N(R) associated with a spherical window of radius R that grows more slowly than Rd in the large-R limit. Thus, the hyperuniformity implies anomalous suppression of long-wavelength density fluctuations relative to those in typical disordered systems, i.e., σ2N(R)∼Rd as R→∞. Hyperuniform systems include perfect crystals, quasicrystals, and special disordered systems. Disordered hyperuniform systems are amorphous states of matter that lie between a liquid and crystal [S. Torquato et al., Phys. Rev. X 5, 021020 (2015)], and have been the subject of many recent investigations due to their novel properties. In the same way that there is no perfect crystal in practice due to the inevitable presence of imperfections, such as vacancies and dislocations, there is no “perfect” hyperuniform system, whether it is ordered or not. Thus, it is practically and theoretically important to quantitatively understand the extent to which imperfections introduced in a perfectly hyperuniform system can degrade or destroy its hyperuniformity and corresponding physical properties. This paper begins such a program by deriving explicit formulas for S(k) in the small-wave-number regime for three types of imperfections: (1) uncorrelated point defects, including vacancies and interstitials, (2) stochastic particle displacements, and (3) thermal excitations in the classical harmonic regime. We demonstrate that our results are in excellent agreement with numerical simulations. We find that “uncorrelated” vacancies or interstitials destroy hyperuniformity in proportion to the defect concentration p. We show that “uncorrelated” stochastic displacements in perfect lattices can never destroy the hyperuniformity but it can be degraded such that the perturbed lattices fall into class III hyperuniform systems, where σ2N(R)∼Rd−α as R→∞ and 0<α<1. By contrast, we demonstrate that certain “correlated” displacements can make systems nonhyperuniform. For a perfect (ground-state) crystal at zero temperature, increase in temperature T introduces such correlated displacements resulting from thermal excitations, and thus the thermalized crystal becomes nonhyperuniform, even at an arbitrarily low temperature. It is noteworthy that imperfections in disordered hyperuniform systems can be unambiguously detected. Our work provides the theoretical underpinnings to systematically study the effect of imperfections on the physical properties of hyperuniform materials.}, number={55}, journal={Physical Review B}, author={Kim, Jaeuk and Torquato, Salvatore}, year={2018}, month=feb, pages={054105} }

@article{Oğuz_Socolar_Steinhardt_Torquato_2017, title={Hyperuniformity of quasicrystals}, volume={95}, url={https://link.aps.org/doi/10.1103/PhysRevB.95.054119}, DOI={10.1103/PhysRevB.95.054119}, abstractNote={Hyperuniform systems, which include crystals, quasicrystals, and special disordered systems, have attracted considerable recent attention, but rigorous analyses of the hyperuniformity of quasicrystals have been lacking because the support of the spectral intensity is dense and discontinuous. We employ the integrated spectral intensity Z(k) to quantitatively characterize the hyperuniformity of quasicrystalline point sets generated by projection methods. The scaling of Z(k) as k tends to zero is computed for one-dimensional quasicrystals and shown to be consistent with independent calculations of the variance, σ2(R), in the number of points contained in an interval of length 2R. We find that one-dimensional quasicrystals produced by projection from a two-dimensional lattice onto a line of slope 1/τ fall into distinct classes determined by the width of the projection window. For a countable dense set of widths, Z(k)∼k4; for all others, Z(k)∼k2. This distinction suggests that measures of hyperuniformity define new classes of quasicrystals in higher dimensions as well.}, number={55}, journal={Physical Review B}, author={Oğuz, Erdal C. and Socolar, Joshua E. S. and Steinhardt, Paul J. and Torquato, Salvatore}, year={2017}, month=feb, pages={054119} }

@article{Torquato_2018, title={Hyperuniform states of matter}, volume={745}, ISSN={03701573}, url={https://linkinghub.elsevier.com/retrieve/pii/S037015731830036X}, DOI={10.1016/j.physrep.2018.03.001}, abstractNote={Hyperuniform states of matter are correlated systems that are characterized by an anomalous suppression of long-wavelength (i.e., large-length-scale) density fluctuations compared to those found in garden-variety disordered systems, such as ordinary fluids and amorphous solids. All perfect crystals, perfect quasicrystals and special disordered systems are hyperuniform. Thus, the hyperuniformity concept enables a unified framework to classify and structurally characterize crystals, quasicrystals and the exotic disordered varieties. While disordered hyperuniform systems were largely unknown in the scientific community over a decade ago, now there is a realization that such systems arise in a host of contexts across the physical, materials, chemical, mathematical, engineering, and biological sciences, including disordered ground states, glass formation, jamming, Coulomb systems, spin systems, photonic and electronic band structure, localization of waves and excitations, self-organization, fluid dynamics, number theory, stochastic point processes, integral and stochastic geometry, the immune system, and photoreceptor cells. Such unusual amorphous states can be obtained via equilibrium or nonequilibrium routes, and come in both quantum-mechanical and classical varieties. The connections of hyperuniform states of matter to many different areas of fundamental science appear to be profound and yet our theoretical understanding of these unusual systems is only in its infancy. The purpose of this review article is to introduce the reader to the theoretical foundations of hyperuniform ordered and disordered systems. Special focus will be placed on fundamental and practical aspects of the disordered kinds, including our current state of knowledge of these exotic amorphous systems as well as their formation and novel physical properties. © 2018 Elsevier B.V. All rights reserved.}, journal={Physics Reports}, author={Torquato, Salvatore}, year={2018}, month=jun, pages={1–95}}

@article{Torquato_Kim_Klatt_Car_Steinhardt_2026, title={Hyperuniformity of Weighted Particle Systems}, volume={16}, url={https://link.aps.org/doi/10.1103/fr99-qh7h}, DOI={10.1103/fr99-qh7h}, abstractNote={Hyperuniform particle arrangements are characterized by a local number variance within a spherical window of radius �� that grows more slowly than the volume of the window, i.e., ����, in ��-dimensional Euclidean space. We generalize this concept to describe the large-scale behavior of particle systems in which particles carry weights: internal degrees of freedom such as scalars (charges and masses), vectors (electric dipole moments, velocities, and torques), pseudovectors (spins and angular momenta), directors (bond orientations), tensors (quadrupole moments), or extrinsic local attributes (Voronoi-cell characteristics). The underlying hyperuniform arrangement may be ordered (crystals and quasicrystals) or disordered, the latter of which has been extensively studied for its novel properties. Our generalization extends hyperuniformity from fluctuations in particle positions to fluctuations in the spatial distribution of weights and examines how weighted fluctuations compare to their unweighted counterparts. We derive generalized weighted pair correlation functions, autocovariance functions, and spectral functions and show their relation to formulas for the local variance in weighted many-particle systems. Then, we apply our formalism to determine the hyperuniformity or nonhyperuniformity of bond-orientational ordered phases, dipolar liquid water, Voronoi-cell volumes, and certain ionic liquids in various Euclidean space dimensions. We demonstrate that hyperuniformity in the particle system does not necessarily translate to hyperuniformity of the weighted system. In fact, cases exist where a hyperuniform particle system becomes antihyperuniform when weighted and others where nonhyperuniform or antihyperuniform particle systems yield hyperuniform weighted systems. This theoretical framework provides a road map for quantifying large-scale fluctuations in weighted many-particle systems, offering a powerful tool for identifying systems with novel physical properties.}, number={1}, journal={Physical Review X}, publisher={American Physical Society}, author={Torquato, Salvatore and Kim, Jaeuk and Klatt, Michael A. and Car, Roberto and Steinhardt, Paul J.}, year={2026}, month=mar, pages={011042} }

@article{Torquato_Stillinger_2003, title={Local density fluctuations, hyperuniformity, and order metrics}, volume={68}, ISSN={1063-651X, 1095-3787}, url={http://link.aps.org/doi/10.1103/PhysRevE.68.041113}, DOI={10.1103/PhysRevE.68.041113}, note={tex.ids= TorquatoStillinger2003 number-of-pages: 25}, number={44}, journal={Physical Review E}, publisher={American Physical Society}, author={Torquato, Salvatore and Stillinger, Frank H.}, year={2003}, month=oct, pages={041113}}

@article{Torquato_Zhang_Stillinger_2015, title={Ensemble Theory for Stealthy Hyperuniform Disordered Ground States}, volume={5}, url={https://link.aps.org/doi/10.1103/PhysRevX.5.021020}, DOI={10.1103/PhysRevX.5.021020}, abstractNote={It has been shown numerically that systems of particles interacting with isotropic “stealthy” bounded long-ranged pair potentials (similar to Friedel oscillations) have classical ground states that are (counterintuitively) disordered, hyperuniform, and highly degenerate. Disordered hyperuniform systems have received attention recently because they are distinguishable exotic states of matter poised between a crystal and liquid that are endowed with novel thermodynamic and physical properties. The task of formulating an ensemble theory that yields analytical predictions for the structural characteristics and other properties of stealthy degenerate ground states in d-dimensional Euclidean space Rd is highly nontrivial because the dimensionality of the configuration space depends on the number density ρ and there is a multitude of ways of sampling the ground-state manifold, each with its own probability measure for finding a particular ground-state configuration. The purpose of this paper is to take some initial steps in this direction. Specifically, we derive general exact relations for thermodynamic properties (energy, pressure, and isothermal compressibility) that apply to any ground-state ensemble as a function of ρ in any d, and we show how disordered degenerate ground states arise as part of the ground-state manifold. We also derive exact integral conditions that both the pair correlation function g2(r) and structure factor S(k) must obey for any d. We then specialize our results to the canonical ensemble (in the zero-temperature limit) by exploiting an ansatz that stealthy states behave remarkably like “pseudo”-equilibrium hard-sphere systems in Fourier space. Our theoretical predictions for g2(r) and S(k) are in excellent agreement with computer simulations across the first three space dimensions. These results are used to obtain order metrics, local number variance, and nearest-neighbor functions across dimensions. We also derive accurate analytical formulas for the structure factor and thermal expansion coefficient for the excited states at sufficiently small temperatures for any d. The development of this theory provides new insights regarding our fundamental understanding of the nature and formation of low-temperature states of amorphous matter. Our work also offers challenges to experimentalists to synthesize stealthy ground states at the molecular level.}, note={tex.ids: torquato_ensemble_2015}, number={22}, journal={Physical Review X}, publisher={American Physical Society}, author={Torquato, S. and Zhang, G. and Stillinger, F.H.}, year={2015}, month=may, pages={021020} }

@article{Zachary_Torquato_2009, title={Hyperuniformity in point patterns and two-phase random heterogeneous media}, volume={2009}, ISSN={1742-5468}, url={https://doi.org/10.1088\%2F1742-5468\%2F2009\%2F12\%2Fp12015}, DOI={10.1088/1742-5468/2009/12/P12015}, abstractNote={Hyperuniform point patterns are characterized by vanishing infinite-wavelength density fluctuations and encompass all crystal structures, certain quasiperiodic systems, and special disordered point patterns (Torquato and Stillinger 2003 Phys. Rev. E 68 041113). This paper generalizes the notion of hyperuniformity to include also two-phase random heterogeneous media. Hyperuniform random media do not possess infinite-wavelength volume fraction fluctuations, implying that the variance in the local volume fraction in an observation window decays faster than the reciprocal window volume as the window size increases. For microstructures of impenetrable and penetrable spheres, we derive an upper bound on the asymptotic coefficient governing local volume fraction fluctuations in terms of the corresponding quantity describing the variance in the local number density (i.e., number variance). Extensive calculations of the asymptotic number variance coefficients are performed for a number of disordered (e.g., quasiperiodic tilings, classical ‘stealth’ disordered ground states, and certain determinantal point processes), quasicrystal, and ordered (e.g., Bravais and non-Bravais periodic systems) hyperuniform structures in various Euclidean space dimensions, and our results are consistent with a quantitative order metric characterizing the degree of hyperuniformity. We also present corresponding estimates for the asymptotic local volume fraction coefficients for several lattice families. Our results have interesting implications for a certain problem in number theory.}, number={1212}, journal={Journal of Statistical Mechanics: Theory and Experiment}, author={Zachary, Chase E. and Torquato, Salvatore}, year={2009}, month=dec, pages={P12015}}

@article{Alon_Kummer_Kurasov_Vinzant_2025, title={Higher dimensional Fourier quasicrystals from Lee–Yang varieties}, volume={239}, ISSN={1432-1297}, url={https://doi.org/10.1007/s00222-024-01307-8}, DOI={10.1007/s00222-024-01307-8}, abstractNote={In this paper, we construct Fourier quasicrystals with unit masses in arbitrary dimensions. This generalizes a one-dimensional construction of Kurasov and Sarnak. To do this, we employ a class of complex algebraic varieties avoiding certain regions in ${mathbb{C}}^{n}$, which generalize hypersurfaces defined by Lee–Yang polynomials. We show that these are Delone almost periodic sets that have at most finite intersection with every discrete periodic set.}, number={1}, journal={Inventiones mathematicae}, author={Alon, Lior and Kummer, Mario and Kurasov, Pavel and Vinzant, Cynthia}, year={2025}, month=jan, pages={321–376} }

@article{Brauchart_Grabner_Kusner_2019, title={Hyperuniform Point Sets on the Sphere: Deterministic Aspects}, volume={50}, ISSN={1432-0940}, url={https://doi.org/10.1007/s00365-018-9432-8}, DOI={10.1007/s00365-018-9432-8}, abstractNote={The notion of hyperuniformity originally introduced as a measure of regularity of infinite point sets in Euclidean space is generalized and extended to sequences of finite point sets on the sphere. It is shown that hyperuniformity implies uniform distribution. Furthermore, it is shown that Quasi-Monte Carlo design sequences with strength at least $$frac{d+1}{2}$$and especially sequences of spherical designs of optimal growth order are hyperuniform.}, number={1}, journal={Constructive Approximation}, author={Brauchart, Johann S. and Grabner, Peter J. and Kusner, Wöden}, year={2019}, month=aug, pages={45–61} }

@article{Gabrielli_Jancovici_Joyce_Lebowitz_Pietronero_Sylos_Labini_2003, title={Generation of primordial cosmological perturbations from statistical mechanical models}, volume={67}, pages={043506}, ISSN={1089-4918}, url={http://dx.doi.org/10.1103/PhysRevD.67.043506}, DOI={10.1103/physrevd.67.043506}, number={4}, journal={Physical Review D}, author={Gabrielli, A. and Jancovici, B. and Joyce, M. and Lebowitz, J. L. and Pietronero, L. and Sylos Labini, F.}, year={2003}, month=feb }

@article{Huesmann_Leblé_2026, title={The link between hyperuniformity, Coulomb energy and Wasserstein distance to Lebesgue for two-dimensional point processes}, volume={7}, DOI={10.2140/pmp.2026.7.123}, journal={Probability and Mathematical Physics}, author={Huesmann, Martin and Leblé, Thomas}, year={2026}, month=jan, pages={123–173} }

@article{Kim_Torquato_2019, title={New tessellation-based procedure to design perfectly hyperuniform disordered dispersions for materials discovery}, volume={168}, ISSN={13596454}, url={https://linkinghub.elsevier.com/retrieve/pii/S1359645419300412}, DOI={10.1016/j.actamat.2019.01.026}, abstractNote={Disordered hyperuniform dispersions are exotic amorphous two-phase materials characterized by an anomalous suppression of long-wavelength volume-fraction ﬂuctuations, which endows them with novel physical properties. While such unusual materials have received considerable attention, a stumbling block has been an inability to create large samples that are truly hyperuniform due to current computational and experimental limitations. To overcome such barriers, we introduce a new and simple construction procedure that guarantees perfect hyperuniformity for very large sample sizes. This methodology involves tessellating space into cells and then inserting a particle into each cell such that the local-cell particle packing fractions are identical to the global packing fraction. We analytically prove that such dispersions are perfectly hyperuniform in the inﬁnite-sample-size limit. Our methodology enables a remarkable mapping that converts a very large nonhyperuniform disordered dispersion into a perfectly hyperuniform one, which we numerically demonstrate in two and three dimensions. A similar analysis also establishes the hyperuniformity of the famous Hashin-Shtrikman multiscale dispersions, which possess optimal transport and elastic properties. Our hyperuniform designs can be readily fabricated using modern photolithographic and 3D printing technologies. The exploration of the enormous class of hyperuniform dispersions that can be designed and tuned by our tessellation-based methodology paves the way for accelerating the discovery of novel hyperuniform materials.}, journal={Acta Materialia}, author={Kim, J. and Torquato, S.}, year={2019}, month=apr, pages={143–151}}

@article{Klatt_Last_2022, title={On strongly rigid hyperfluctuating random measures}, volume={59}, ISSN={00219002, 14756072}, url={https://www.jstor.org/stable/48795200}, abstractNote={In contrast to previous belief, we provide examples of stationary ergodic random measures that are both hyperfluctuating and strongly rigid. Therefore we study hyperplane intersection processes (HIPs) that are formed by the vertices of Poisson hyperplane tessellations. These HIPs are known to be hyperfluctuating, that is, the variance of the number of points in a bounded observation window grows faster than the size of the window. Here we show that the HIPs exhibit a particularly strong rigidity property. For any bounded Borel set B, an exponentially small (bounded) stopping set suffices to reconstruct the position of all points in B and, in fact, all hyperplanes intersecting B. Therefore the random measures supported by the hyperplane intersections of arbitrary (but fixed) dimension, are also hyperfluctuating. Our examples aid the search for relations between correlations, density fluctuations, and rigidity properties.}, number={4}, journal={Journal of Applied Probability}, author={Klatt, Michael Andreas and Last, Günter}, year={2022}, pages={948–961}}

@article{Kurasov_Sarnak_2020, title={Stable polynomials and crystalline measures}, volume={61}, ISSN={0022-2488}, url={https://doi.org/10.1063/5.0012286}, DOI={10.1063/5.0012286}, abstractNote={Explicit examples of positive crystalline measures and Fourier quasicrystals are constructed using pairs of stable polynomials, answering several open questions in the area.}, number={8}, journal={Journal of Mathematical Physics}, author={Kurasov, P. and Sarnak, P.}, year={2020}, month=aug, pages={083501} }

@article{Ruelle_1970, title={Superstable interactions in classical statistical mechanics}, volume={18}, ISSN={1432-0916}, url={https://doi.org/10.1007/BF01646091}, DOI={10.1007/BF01646091}, abstractNote={We consider classical systems of particles inv dimensions. For a very large class of pair potentials (superstable lower regular potentials) it is shown that the correlation functions have bounds of the form$$varrho (x_1 ,...,x_n ) leqq xi ^n$$. Using these and further inequalities one can extend various results obtained by Dobrushin and Minlos [3] for the case of potentials which are non-integrably divergent at the origin. In particular it is shown that the pressure is a continuous function of the density. Infinite system equilibrium states are also defined and studied by analogy with the work of Dobrushin [2a] and of Lanford and Ruelle [11] for lattice gases.}, number={2}, journal={Communications in Mathematical Physics}, author={Ruelle, D.}, year={1970}, month=jun, pages={127–159} }

@article{Sgrignuoli_Dal_Negro_2021, title={Hyperuniformity and wave localization in pinwheel scattering arrays}, volume={103}, url={https://link.aps.org/doi/10.1103/PhysRevB.103.224202}, DOI={10.1103/PhysRevB.103.224202}, number={22}, journal={Physical Review B}, author={Sgrignuoli, F. and Dal Negro, L.}, year={2021}, month=jun, pages={224202} }

@article{Shih_Casiulis_Martiniani_2024, title={Fast generation of spectrally shaped disorder}, volume={110}, url={https://link.aps.org/doi/10.1103/PhysRevE.110.034122}, DOI={10.1103/PhysRevE.110.034122}, number={3}, journal={Phys. Rev. E}, author={Shih, Aaron and Casiulis, Mathias and Martiniani, Stefano}, year={2024}, month=sep, pages={034122} }

@article{Sodin_Tsirelson_2002, title={Random complex zeroes, I. Asymptotic normality}, volume={144}, DOI={10.1007/BF02984409}, abstractNote={We consider three models (elliptic, flat and hyperbolic) of Gaussian random analytic functions distinguished by invariance of their zeroes distribution. Asymptotic normality is proven for smooth functionals (linear statistics) of the set of zeroes. Introduction and the main result Zeroes of random polynomials and other analytic functions were studied by mathematicians and physicists under various assumptions on random coefficients. One class of models introduced not long ago by Bogomolny, Bohigas and Leboeuf [5, 6], Kostlan [16], and Shub and Smale [23] has a remarkably}, journal={Israel Journal of Mathematics}, author={Sodin, Mikhail and Tsirelson, Boris}, pages={125-149}, year={2002}, month=nov }

@article{Torquato_Kim_Klatt_2021, title={Local Number Fluctuations in Hyperuniform and Nonhyperuniform Systems: Higher-Order Moments and Distribution Functions}, volume={11}, url={https://link.aps.org/doi/10.1103/PhysRevX.11.021028}, DOI={10.1103/PhysRevX.11.021028}, journal={American Physical Society}, author={Torquato, Salvatore and Kim, Jaeuk and Klatt, Michael A.}, year={2021}, month=may, pages={021028–1--27} }

@article{Dereudre_Flimmel_2024,  title={Non-hyperuniformity of Gibbs point processes with short-range interactions}, volume={61}, DOI={10.1017/jpr.2024.21}, number={4}, journal={Journal of Applied Probability}, author={Dereudre, David and Flimmel, Daniela}, year={2024}, pages={1380–1406} }

@article{Klatt_Last_Yogeshwaran_2020, title={Hyperuniform and rigid stable matchings}, volume={57}, rights={© 2020 Wiley Periodicals, Inc.}, ISSN={1098-2418}, url={https://onlinelibrary.wiley.com/doi/abs/10.1002/rsa.20923}, DOI={10.1002/rsa.20923}, abstractNote={We study a stable partial matching τ of the d-dimensional lattice with a stationary determinantal point process Ψ on Rd with intensity α>1. For instance, Ψ might be a Poisson process. The matched points from Ψ form a stationary and ergodic (under lattice shifts) point process Ψτ with intensity 1 that very much resembles Ψ for α close to 1. On the other hand Ψτ is hyperuniform and number rigid, quite in contrast to a Poisson process. We deduce these properties by proving more general results for a stationary point process Ψ, whose so-called matching flower (a stopping set determining the matching partner of a lattice point) has a certain subexponential tail behavior. For hyperuniformity, we also additionally need to assume some mixing condition on Ψ. Furthermore, if Ψ is a Poisson process then Ψτ has an exponentially decreasing truncated pair correlation function.}, number={22}, journal={Random Structures \& Algorithms}, author={Klatt, Michael Andreas and Last, Günter and Yogeshwaran, D.}, year={2020}, month=apr, pages={439–473}}

@article{Soshnikov_2000, title={Determinantal random point fields}, volume={55}, ISSN={0036-0279}, url={https://doi.org/10.1070/RM2000v055n05ABEH000321}, DOI={10.1070/RM2000v055n05ABEH000321}, abstractNote={This paper contains an exposition of both recent and rather old results on determinantal random point fields. We begin with some general theorems including proofs of necessary and sufficient conditions for the existence of a determinantal random point field with Hermitian kernel and of a criterion for weak convergence of its distribution. In the second section we proceed with examples of determinantal random fields in quantum mechanics, statistical mechanics, random matrix theory, probability theory, representation theory, and ergodic theory. In connection with the theory of renewal processes, we characterize all Hermitian determinantal random point fields on  and  with independent identically distributed spacings. In the third section we study translation-invariant determinantal random point fields and prove the mixing property for arbitrary multiplicity and the absolute continuity of the spectra. In the last section we discuss proofs of the central limit theorem for the number of particles in a growing box and of the functional central limit theorem for the empirical distribution function of spacings.}, number={5}, journal={Russian Mathematical Surveys}, author={Soshnikov, A}, year={2000}, month=oct, pages={923} }

@article{Björklund_2026, title={Stealthy point processes and lattice induction}, number={arXiv:2607.25616}, publisher={arXiv}, url={https://arxiv.org/abs/2607.25616}, author={Björklund, Michael}, year={2026}, month=jul}

@article{Björklund_Byléhn_2024, title={Hyperuniformity of random measures on Euclidean and hyperbolic spaces}, url={https://arxiv.org/abs/2405.12737}, number={arXiv:2405.12737}, publisher={arXiv}, author={Björklund, Michael and Byléhn, Mattias}, year={2024}, month=may }

@article{Björklund_Byléhn_2025, title={Hyperuniformity and hyperfluctuations of random measures in commutative spaces}, url={https://arxiv.org/abs/2503.01567}, number={arXiv:2503.01567}, publisher={arXiv}, author={Björklund, Michael and Byléhn, Mattias}, year={2025}, month=mar }

@article{Butez_Dallaporta_García-Zelada_2024, title={On the Wasserstein distance between a hyperuniform point process and its mean}, url={http://arxiv.org/abs/2404.09549}, DOI={10.48550/arXiv.2404.09549}, note={arXiv:2404.09549 [math]}, number={arXiv:2404.09549}, publisher={arXiv}, author={Butez, Raphael and Dallaporta, Sandrine and García-Zelada, David}, year={2024}, month=apr }

@misc{Coste_2021, title={Order, fluctuations, rigidities}, url={https://scoste.fr/assets/survey_hyperuniformity.pdf}, author={Coste, Simon}, year={2021}, month=jul }

@article{Dereudre_Flimmel_Huesmann_Leblé_2024, title={(Non)-hyperuniformity of perturbed lattices}, url={http://arxiv.org/abs/2405.19881}, DOI={10.48550/arXiv.2405.19881}, note={arXiv:2405.19881 [math]}, number={arXiv:2405.19881}, publisher={arXiv}, author={Dereudre, David and Flimmel, Daniela and Huesmann, Martin and Leblé, Thomas}, year={2024}, month=may }

@article{Flimmel_2025, title={Fitting regular point patterns with a hyperuniform perturbed lattice}, url={https://arxiv.org/abs/2503.12179}, number={arXiv:2503.12179}, publisher={arXiv}, author={Flimmel, Daniela}, year={2025}, month=mar }

@article{Klatt_Last_Henze_2022, title={A genuine test for hyperuniformity}, rights={All rights reserved}, url={http://arxiv.org/abs/2210.12790}, DOI={10.48550/arXiv.2210.12790}, note={arXiv:2210.12790 [cond-mat, stat]}, number={arXiv:\allowbreak2210.12790}, publisher={arXiv}, author={Klatt, Michael A. and Last, Günter and Henze, Norbert}, year={2022}, month=oct}

@article{Klatt_Last_Lotz_Yogeshwaran_2025, title={Invariant transports of stationary random measures: asymptotic variance, hyperuniformity, and examples}, url={https://arxiv.org/abs/2506.05907}, number={arXiv:2506.05907}, publisher={arXiv}, author={Klatt, Michael A. and Last, Günter and Lotz, Luca and Yogeshwaran, D.}, year={2025}, month=jun }

@article{Krishnapur_Yogeshwaran_2024, title={Stationary random measures: Covariance asymptotics, variance bounds and central limit theorems}, url={https://arxiv.org/abs/2411.08848}, number={arXiv:2411.08848}, publisher={arXiv}, author={Krishnapur, Manjunath and Yogeshwaran, D.}, year={2024}, month=nov }

@article{Lachièze-Rey_2025a, title={Rigidity of random stationary measures and applications to point processes}, url={https://arxiv.org/abs/2409.18519}, number={arXiv:2409.18519}, publisher={arXiv}, author={Lachièze-Rey, Raphaël}, year={2025}, month=feb }

@article{Lachièze-Rey_2025b, title={Hyperuniform random measures, transport and rigidity}, url={https://arxiv.org/abs/2510.18392}, number={arXiv:\allowbreak2510.18392}, publisher={arXiv}, author={Lachièze-Rey, Raphaël}, year={2025}, month=oct }

@article{Lachièze-Rey_Yogeshwaran_2024, title={Hyperuniformity and optimal transport of point processes}, url={http://arxiv.org/abs/2402.13705}, DOI={10.48550/arXiv.2402.13705}, note={arXiv:2402.13705 [math-ph]}, number={arXiv:2402.13705}, publisher={arXiv}, author={Lachièze-Rey, Raphaël and Yogeshwaran, D.}, year={2024}, month=mar }

@article{Mastrilli_Błaszczyszyn_Lavancier_2024, title={Estimating the hyperuniformity exponent of point processes}, url={https://arxiv.org/abs/2407.16797}, number={arXiv:2407.16797}, publisher={arXiv}, author={Mastrilli, Gabriel and Błaszczyszyn, Bartłomiej and Lavancier, Frédéric}, year={2024}, month=jul }

@article{Olevskii_Ulanovskii_2020, title={A Simple Crystalline Measure}, url={https://arxiv.org/abs/2006.12037}, number={arXiv:2006.12037}, publisher={arXiv}, author={Olevskii, Alexander and Ulanovskii, Alexander}, year={2020}, month=jun }

@article{Roca_2023, title={Hyperuniformity and Number Rigidity of Inflation Tilings}, url={https://arxiv.org/abs/2310.20517}, number={arXiv:\allowbreak2310.20517}, publisher={arXiv}, author={Roca, Daniel}, year={2023}, month=oct }

@article{Thomassey_Lachièze-Rey_Shapira_2026, title={Regularization of a stationary point process by a stationary increments perturbation}, url={https://arxiv.org/abs/2602.19773}, number={arXiv:2602.19773}, publisher={arXiv}, author={Thomassey, Loïc and Lachièze-Rey, Raphaël and Shapira, Assaf}, year={2026}, month=feb }

@article{Uche_Stillinger_Torquato_2004, title={Constraints on collective density variables: Two dimensions}, volume={70}, url={https://link.aps.org/doi/10.1103/PhysRevE.70.046122}, DOI={10.1103/PhysRevE.70.046122}, number={4}, journal={Physical Review E}, author={Uche, Obioma U. and Stillinger, Frank H. and Torquato, Salvatore}, year={2004}, month=oct, pages={046122} }

@article{Uche_Torquato_Stillinger_2006, title={Collective coordinate control of density distributions}, volume={74}, url={https://link.aps.org/doi/10.1103/PhysRevE.74.031104}, DOI={10.1103/PhysRevE.74.031104}, number={3}, journal={Physical Review E}, author={Uche, Obioma U. and Torquato, Salvatore and Stillinger, Frank H.}, year={2006}, month=sep, pages={031104} }

@article{gabrielli_glass-like_2002, title = {Glass-like Universe: {{Real-space}} Correlation Properties of Standard Cosmological Models}, author = {Gabrielli, Andrea and Joyce, Michael and Sylos Labini, Francesco}, year = 2002, journal = {Phys. Rev. D}, volume = {65}, pages = {083523}, publisher = {American Physical Society}, doi = {10.1103/PhysRevD.65.083523}, urldate = {2023-04-18} }

@article{galliano_two-dimensional_2023, title = {Two-{{Dimensional Crystals}} Far from {{Equilibrium}}}, author = {Galliano, Leonardo and Cates, Michael E. and Berthier, Ludovic}, year = 2023, journal = {Phys. Rev. Lett.}, volume = {131}, pages = {047101}, publisher = {American Physical Society}, doi = {10.1103/PhysRevLett.131.047101}, urldate = {2024-07-03} }

@article{gorsky_engineered_2019, ids = {gorsky_engineered_2019-1}, title = {Engineered Hyperuniformity for Directional Light Extraction}, author = {Gorsky, S. and Britton, W. A. and Chen, Y. and Montaner, J. and Lenef, A. and Raukas, M. and Dal Negro, L.}, year = 2019, journal = {APL Photonics}, volume = {4}, pages = {110801}, publisher = {American Institute of Physics}, doi = {10.1063/1.5124302}, urldate = {2019-11-11}, langid = {english} }

@article{haberko_transition_2020, title = {Transition from Light Diffusion to Localization in Three-Dimensional Amorphous Dielectric Networks near the Band Edge}, author = {Haberko, Jakub and {Froufe-P{\'e}rez}, Luis S. and Scheffold, Frank}, year = 2020, journal = {Nat. Commun.}, volume = {11}, pages = {4867}, publisher = {Nature Publishing Group}, doi = {10.1038/s41467-020-18571-w}, urldate = {2021-03-01}, langid = {english} }

@article{klatt_wave_2022, title = {Wave Propagation and Band Tails of Two-Dimensional Disordered Systems in the Thermodynamic Limit}, author = {Klatt, Michael A. and Steinhardt, Paul J. and Torquato, Salvatore}, year = 2022, journal = {Proc. Natl. Acad. Sci.}, volume = {119}, pages = {e2213633119}, publisher = {Proceedings of the National Academy of Sciences}, doi = {10.1073/pnas.2213633119}, urldate = {2022-12-20} }

@article{leble_two-dimensional_2026, title = {The Two-Dimensional One-Component Plasma Is Hyperuniform}, author = {Lebl{\'e}, Thomas}, year = 2026, journal = {Duke Math. J.}, volume = {175}, pages = {763--901}, publisher = {Duke University Press}, doi = {10.1215/00127094-2025-0036}, urldate = {2026-05-16}, langid = {english} }

@article{maire_hyperuniform_2025, title = {Hyperuniform {{Interfaces}} in {{Nonequilibrium Phase Coexistence}}}, author = {Maire, Rapha{\"e}l and Galliano, Leonardo and Plati, Andrea and Berthier, Ludovic}, year = 2025, journal = {Phys. Rev. Lett.}, volume = {135}, pages = {227102}, publisher = {American Physical Society}, doi = {10.1103/4b8v-4sbh}, urldate = {2026-01-27} }

@article{rissone_long-range_2021, title = {Long-{{Range Anomalous Decay}} of the {{Correlation}} in {{Jammed Packings}}}, author = {Rissone, Paolo and Corwin, Eric I. and Parisi, Giorgio}, year = 2021, journal = {Phys. Rev. Lett.}, volume = {127}, pages = {038001}, publisher = {American Physical Society}, doi = {10.1103/PhysRevLett.127.038001}, urldate = {2021-07-29} }

@article{torquato_hyperuniformity_2016, ids = {ma_random_2017}, title = {Hyperuniformity and Its Generalizations}, author = {Torquato, Salvatore}, year = 2016, journal = {Phys. Rev. E}, volume = {94}, pages = {022122}, publisher = {American Physical Society}, doi = {10.1103/PhysRevE.94.022122}, urldate = {2020-02-10}, langid = {english} }

@article{vanoni_effective_2026, title = {Effective {{Delocalization}} in the {{One-Dimensional Anderson Model}} with {{Stealthy Disorder}}}, author = {Vanoni, Carlo and Karcher, Jonas and Rechtsman, Mikael C. and Altshuler, Boris L. and Steinhardt, Paul J. and Torquato, Salvatore}, year = 2026, journal = {Phys. Rev. Lett.}, volume = {136}, pages = {150404}, publisher = {American Physical Society}, doi = {10.1103/nl7n-bqn4}, urldate = {2026-04-27} }

@article{wilken_random_2021, title = {Random {{Close Packing}} as a {{Dynamical Phase Transition}}}, author = {Wilken, Sam and Guerra, Rodrigo E. and Levine, Dov and Chaikin, Paul M.}, year = 2021, journal = {Phys. Rev. Lett.}, volume = {127}, pages = {038002-1--6}, publisher = {American Physical Society}, doi = {10.1103/PhysRevLett.127.038002}, urldate = {2021-07-29} }

@article{yu_engineered_2021, title = {Engineered Disorder in Photonics}, author = {Yu, Sunkyu and Qiu, Cheng-Wei and Chong, Yidong and Torquato, Salvatore and Park, Namkyoo}, year = 2021, journal = {Nat. Rev. Mater.}, volume = {6}, pages = {226--243}, publisher = {Nature Publishing Group}, doi = {10.1038/s41578-020-00263-y}, urldate = {2021-04-17}, langid = {english} }

@article{zheng_disordered_2020, title = {Disordered Hyperuniformity in Two-Dimensional Amorphous Silica}, author = {Zheng, Yu and Liu, Lei and Nan, Hanqing and Shen, Zhen-Xiong and Zhang, Ge and Chen, Duyu and He, Lixin and Xu, Wenxiang and Chen, Mohan and Jiao, Yang and Zhuang, Houlong}, year = 2020, journal = {Sci. Adv.}, volume = {6}, pages = {eaba0826}, doi = {10.1126/sciadv.aba0826}, urldate = {2020-04-21}, langid = {english} }

\end{document}